\documentclass[12pt, a4paper]{amsart}
\usepackage[margin=1in]{geometry} 
\usepackage{amsmath,amsthm,amssymb,mathrsfs,graphicx,mathtools,tikz,tikz-cd}
\usetikzlibrary{positioning}
\usepackage{lscape}
\usepackage[all,2cell]{xy} \UseAllTwocells
\usepackage{ulem}
\usepackage[pagebackref,colorlinks]{hyperref}

\usepackage{enumerate,enumitem}
\setlist[enumerate, 1]{itemindent=1em, label={\normalfont(\arabic*)}}

\newcommand{\Rmnum}[1]{\uppercase\expandafter{\romannumeral #1}}

\newcommand{\aA}{\mathcal{A}}
\newcommand{\bB}{\mathcal{B}}
\newcommand{\cC}{\mathcal{C}}
\newcommand{\dD}{\mathcal{D}}

\newcommand{\gG}{\mathcal{G}}
\newcommand{\hH}{\mathcal{H}}

\newcommand{\lL}{\mathcal{L}}

\newcommand{\tT}{\mathcal{T}}

\newcommand{\ccC}{\mathscr{C}}

\newcommand{\one}{\mathbf{1}}

\newcommand{\mT}{{\normalfont\textrm{T}}}

\newcommand{\bL}{{\normalfont\textbf{L}}}

\newcommand{\bp}{{\normalfont\textbf{p}}}

\newcommand{\RHom}{\operatorname{RHom}}
\newcommand{\nin}{{\operatorname{in}}}
\newcommand{\ex}{{\operatorname{ex}}}
\newcommand{\Id}{{\operatorname{Id}}}
\newcommand{\dg}{{\operatorname{dg}}}

\renewcommand{\-}{\text{\normalfont -}}

\newcommand{\dgCAT}{\operatorname{dgCAT}}
\newcommand{\op}{{\operatorname{op}}}

\newcommand{\Ob}{{\operatorname{Ob}}}

\newcommand{\Span}{\operatorname{Span}}

\newcommand{\cone}{\operatorname{cone}}

\newcommand{\hHom}{\hH\hspace{-0.22em}\operatorname{om}}

\newcommand{\alg}{\operatorname{alg}}
\newcommand{\ac}{\operatorname{ac}}
\newcommand{\per}{{\normalfont\text{per}}}
\newcommand{\HH}{\operatorname{HH}}
\newcommand{\Ho}{\operatorname{Ho}}

\newcommand{\rep}{\operatorname{rep}}

\newcommand{\dgcat}{\operatorname{dgcat}}

\everymath{\displaystyle}

\theoremstyle{plain}
\newtheorem{theorem}{{T{\footnotesize HEOREM}}}[section]
\newtheorem{lemma}[theorem]{{L{\footnotesize EMMA}}}
\newtheorem{proposition}[theorem]{{P{\footnotesize ROPOSITION}}}
\newtheorem{corollary}[theorem]{{C{\footnotesize OROLLARY}}}

\theoremstyle{definition}

\newtheorem{remark}[theorem]{{R{\footnotesize EMARK}}}

\hypersetup{colorlinks, linkcolor=blue}

\usepackage{xcolor}

\begin{document}
	
\title{Lax functoriality of Hochschild cochain complex$^*$ \footnote{* This work was supported by the National Natural Science Foundation of China (Grant No. 12371043).}}
\author{Yang Han and Xukun Wang}
\maketitle

\begin{abstract}
    Unlike the Hochschild chain complex of an algebra, the Hochschild cochain complex of an algebra is not functorial. Nonetheless, we show that the Hochschild cochain complex of an algebra even a dg category is of lax functoriality, i.e., there exists a lax functor from bicategory of dg categories to bicategory of $B_\infty$-algebras which sends every dg category to its Hochschild cochain complex. This result is a homotopy version of the lax functoriality of center of an algebra obtained by Davydov, Kong, Runkel, Grady, Oren, et al, in the more general context of dg categories, and extends the restricted functoriality of Hochschild cochain complex of a dg category obtained by Keller to global lax functoriality. 
\end{abstract}



\tableofcontents


\section{Introduction}
Throughout this paper, $k$ is a fixed field. Unless stated otherwise, all vector spaces, algebras, categories and functors are over $k$. It is clear that the Hochschild chain complex $C_\bullet(A)=(A^{\otimes_k\bullet+1},b_{\bullet})$ of an algebra $A$ is functorial on $A$, i.e., Hochschild chain complex $C_{\bullet}$ is a functor from the category of algebras to the category of complexes of vector spaces. Furthermore, Hochschild homology $HH_{\bullet}(A)$ is also functorial on $A$. However, Hochschild cochain complex $C^{\bullet}(A)$ and Hochschild cohomology $HH^{\bullet}(A)$ are not the case. In fact, even the center $Z(A)$ of an algebra $A$, which is isomorphic to the zeroth Hochschild cohomology $HH^0(A)$ of $A$, is not functorial on $A$. For this, it suffices to observe the $2\times 2$ upper triangular matrix algebra $A={\tiny \begin{pmatrix} k&k\\ 0&k \end{pmatrix}}$ which has a commutative semisimple subalgebra $S={\tiny \begin{pmatrix} k&0\\ 0&k \end{pmatrix}}$. Obviously, $S$ is isomorphic to the factor algebra $A/{\rm rad}A$ of $A$ modulo its Jacobson radical ${\rm rad} A={}${\tiny $\begin{pmatrix} 0&k\\ 0&0 \end{pmatrix}$}. Moreover, the composition of the injective homomorphism $i:S\hookrightarrow A$ and the surjective homomorphism $p:A\to A/{\rm rad}A\cong S$ is the identity automorphism $1_S:S\to S$. Assume on the contrary that the center $Z$ is a functor from the category of algebras to the category of vector spaces. Then $Z(p)Z(i)=Z(pi)=Z(1_S)=1_{Z(S)}$. Note that $Z(A)\cong k$ and $Z(S)=S$. Thus ${\rm rank} Z(p)\le 1$ and ${\rm rank} Z(i)\le 1$. It contradicts to ${\rm rank}(Z(p)Z(i))=2$.
Nontheless, in 2011, Davydov, Kong and Runkel proved that the center of an algebra is of lax functoriality (Ref. \cite[Theorem 4.12]{DavydovKongRunkel11}), i.e., there exists a lax functor from the category of algebras, which is a bicategory with discrete Hom categories, to the bicategory of algebras, cospans of algebras, and morphisms of cospans of algebras, sending every algebra $A$ to its center $Z(A)$.
Moreover, in 2021, Grady and Oren showed that there exists a lax functor from the category of rings, which is a bicategory with discrete Hom categories, to the bicategory of rings, bimodules, and bimodule morphisms, sending every ring $R$ to its center $Z(R)$ (Ref. \cite[Theorem 4.4]{GradyOren21}).
Now that the Hochschild cochain complex of an algebra is its ``homotopy center'', it is natural to ask whether the Hochschild cochain complex of an algebra is of lax functoriality, i.e., whether there exists a lax functor from bicategory of algebras to bicategory of $B_\infty$-algebras sending every algebra to its Hochschild cochain complex?

In the more general context of dg (=differential graded) categories, there are two extreme categories of small dg categories: One is the category ${\rm dgcat}$ consisting of all small dg categories and all dg functors, the other is the category ${\rm dgcat}_{\rm id}$ consisting of all small dg categories and all identity dg functors. The Hochschild cochain complex $C^\bullet(\aA)$ is not a functor from ${\rm dgcat}$ to the category $\cC k$ of complexes of vector spaces, but it is obviously a functor from ${\rm dgcat}_{\rm id}$ to $\cC k$. It means that Hochschild cochain complex can be of {\it restricted functoriality} if we consider only special dg functors. A very nice intermediate category of small dg categories is the category ${\rm dgcat}_{\rm ff}$ consisting of all small dg categories and all fully faithful dg functors. In 2003, Keller found that Hochschild cochain complex $C^\bullet$ is a contravariant functor from ${\rm dgcat}_{\rm ff}$ to the category $B_\infty$ of $B_\infty$-algebras (Ref. \cite[4.3, Paragraph 1]{Keller03}). Furthermore, he discovered that Hochschild cochain complex $C^\bullet$ is a contravariant functor from the category ${\rm Hmo}_{\rm ff}$ to the homotopy category ${\rm Ho}B_\infty$ of $B_\infty$-algebras (Ref. \cite[5.4]{Keller06}). Here, the category {\rm Hmo} is the localization of ${\rm dgcat}$ with respect to Morita equivalences, and ${\rm Hmo}_{\rm ff}$ is the subcategory of ${\rm Hmo}$ whose objects are all objects of {\rm Hmo} and whose morphisms are the isomorphism classes $[X]$ of dg $\aA$-$\bB$-bimodules $X$ such that $X$ is perfect as a right dg $\bB$-module and the derived tensor product functor $-\otimes^{\bf L}_{\aA}X: \per\aA\to\per\bB$ is fully faithful.
It implies that if two small dg categories are derived equivalent then their Hochschild cochain complexes are isomorphic in ${\rm Ho}B_\infty$.
Inspired by Keller's results, it is natural to ask what kind of functoriality the Hochschild cochain complex of a small dg category is of, once we consider all dg functors, or more general, all dg bimodules? 

In this paper, we will answer the above two questions and show that the Hochschild cochain complex of a small dg category is of lax functoriality, i.e., there exists a lax functor from bicategory of small dg categories to bicategory of $B_\infty$-algebras which sends every small dg category to its Hochschild cochain complex (Theorem~\ref{Thm-LaxFuntor-DerCat-B-Inf}). This result is a homotopy version of the lax functoriality of center of an algebra obtained by Davydov, Kong, Runkel, Grady, Oren, et al, in the more general context of dg categories, and extends the restricted functoriality of Hochschild cochain complex of a small dg category obtained by Keller to global lax functoriality. Whereas the lax functoriality of center of an algebra has applications in topological quantum field theory and conformal field theory \cite{Davydov10,DavydovKongRunkel11,DavydovKongRunkel15,GradyOren21}, we hope our results are useful in these fields as well.

In fact, we will construct various lax functors and functors to display the lax functoriality and restricted functoriality of Hochschild cochain complexes and Hochschild cohomology of algebras and dg categories. The overview diagram $(\bigstar)$ below provides a global perspective.
$$
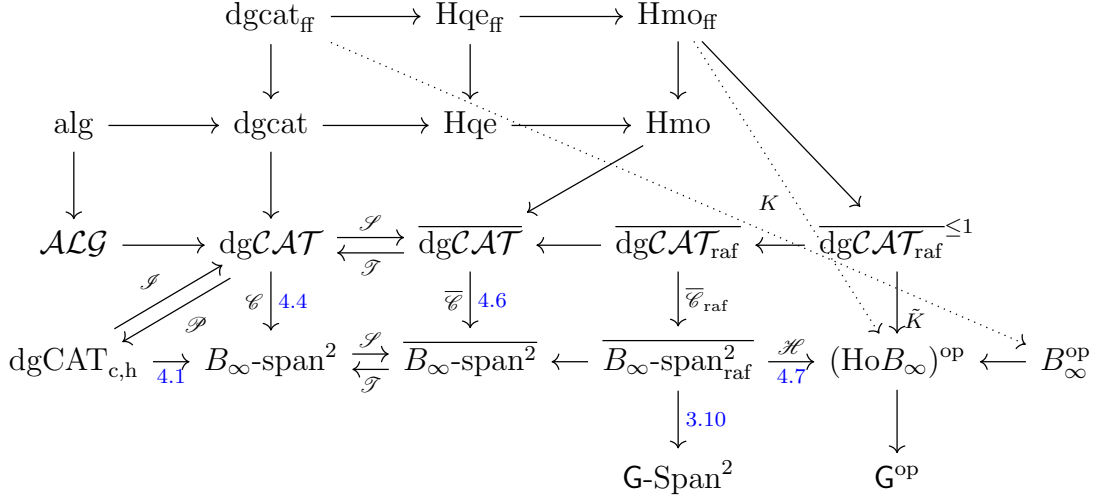
\begin{tikzcd}[column sep=small]
	& {\rm dgcat}_{\rm ff} \arrow[r] \arrow[d] \arrow[rrrrddd,dotted] & {\rm Hqe}_{\rm ff} \arrow[r] \arrow[d] & {\rm Hmo}_{\rm ff} \arrow[d] \arrow[rddd,dotted,"K"'] \arrow[ddr] & &  \\
	{\rm alg} \arrow[r] \arrow[d] & {\rm dgcat} \arrow[r] \arrow[d] & {\rm Hqe} \arrow[r] & {\rm Hmo} \arrow[dl] & & \\	
	\aA\lL\gG \arrow[r] & \dg\cC\aA\tT \arrow[dl,shift left=1,"\mathscr{P}"] \arrow[r,shift left=1,"\mathscr{S}"] \arrow[d,"\ref{Thm-LaxFuntor-DerCat-B-Inf}","\mathscr{C}"'] & \overline{\dg\cC\aA\tT} \arrow[l,shift left=1,"\mathscr{T}"] \arrow[d,"\ref{Cor-overline-ccC}","\overline{\mathscr{C}}"'] & \overline{\dg\cC\aA\tT_{\rm raf}} \arrow[l] \arrow[d,"\overline{\mathscr{C}}_{\rm raf}"] &  \overline{\dg\cC\aA\tT_{\rm raf}}^{\le 1} \arrow[l] \arrow[d,near end,"\tilde{K}"] & \\
	\dgCAT_{\rm c,h} \arrow[r,"\ref{Thm-LaxFuntor-HomotopyCat-B-Inf}"'] \arrow[ur,shift left=1,"\mathscr{I}"] & B_\infty\-{\rm span}^2 \arrow[r,shift left=1,"\mathscr{S}"] & \overline{B_\infty\-{\rm span}^2} \arrow[l,shift left=1,"\mathscr{T}"] & \overline{B_\infty\-{\rm span}^2_{\rm raf}} \arrow[l] \arrow[r,"\mathscr{H}","\ref{Prop-BInfOspan2-HBInf}"'] \arrow[d,"\ref{Cor-LaxFunctor-B8span2-GSpan2}"] & ({\rm Ho}B_\infty)^\op \arrow[d] & B_\infty^\op \arrow[l] \\
	& & & \mathsf{G}\-{\rm Span}^2 & \mathsf{G}^\op &
\end{tikzcd}
$$
In the diagram above, the vertices on the upmost two rows and rightmost two columns are categories which also can be viewed as bicategories with discrete Hom categories. The arrows between these vertices are functors which also can be viewed as strict functors when these vertices are regarded as bicategories. All other vertices are bicategories and all other arrows are lax functors. 
Since the composition of (lax) functors is still a (lax) functor, every path of (lax) functors in the diagram above gives a (lax) functor. For example, there exist a functor ${\rm dgcat}_{\rm ff}\to \mathsf{G}^\op$ sending every small dg category $\aA$ to its Hochschild cohomology $HH^\bullet(\aA)$, and a lax functor ${\rm alg} \to B_\infty\-{\rm span}^2$ sending every algebra $A$ to its Hochschild cochain complex $C^\bullet(A)$.

Those two dotted arrows in the diagram above are the functors ${\rm dgcat}_{\rm ff} \to B_\infty^\op$ and $K: {\rm Hmo}_{\rm ff} \to ({\rm Ho}B_\infty)^\op$ discovered by Keller in \cite[4.3, Paragraph 1]{Keller03} and \cite[5.4, Page 180]{Keller06} respectively.  
For these functors, only fully faithful dg functors or the dg bimodules inducing fully faithful derived tensor product functors are considered.
In order to delete the fully faithful restrictions, we have to consider bicategories and lax functors. Eventually we obtain the lax functors from {\rm alg}, $\aA\lL\gG$, {\rm dgcat}, $\dg\cC\aA\tT$, {\rm Hqe} and {\rm Hmo} to $B_\infty\-{\rm span}^2$ respectively.
Moreover, we extend Keller's functor $K$ to the functor $\tilde{K}$. 

The paper is organized as follows: In section 2, we will fix some notations on the Hochschild cochain complexes of dg categories. Besides, we will give some new results which are necessary preparations for constructing lax functors of Hochschild cochain complexes of dg categories. In section 3, we will construct some bicategories of dg categories such as $\dgCAT_{\rm c,h}$ and $\dg\cC\aA\tT$, and some bicategories of $B_\infty$-algebras or Gerstenhaber algebras such as $B_\infty\-{\rm span}^2$ and $\mathsf{G}\-{\rm Span}^2$, which are the source and target bicategories of our objective (co)lax functors. Indeed, we will provide a general construction of a bicategory from a model category (Theorem~\ref{Thm-Const-Bicat}), which can be applied to the model categories $B_\infty$ of $B_\infty$-algebras and $\mathsf{G}$ of Gerstenhaber algebras to obtain the target bicategories $B_\infty\-{\rm span}^2$ of $B_\infty$-algebras and $\mathsf{G}\-{\rm Span}^2$ of Gerstenhaber algebras respectively. Furthermore, we will show that the cohomology functor $H:B_\infty\to \mathsf{G}$ induces a colax functor $B_\infty\-{\rm span}^2\to \mathsf{G}\-{\rm Span}^2$ (Corollary~\ref{Cor-Colax functor-Bspan2-GSpan2}). Restricted to the sub-bicategory $\overline{B_\infty\-{\rm span}^2_{\rm raf}}$ of 1-skeleton $\overline{B_\infty\-{\rm span}^2}$ of bicategory $B_\infty\-{\rm span}^2$, we will obtain a lax functor $\overline{B_\infty\-{\rm span}^2_{\rm raf}} \to \mathsf{G}\-{\rm Span}^2$ (Corollary~\ref{Cor-LaxFunctor-B8span2-GSpan2}). In section 4, we will construct some lax functors from some bicategories of small dg categories to some bicategories of $B_\infty$-algebras or Gerstenhaber algebras which send small dg categories to their Hochschild cochain complexes or Hochschild cohomologies. First of all, we will construct a lax functor from bicategory $\dgCAT_{\rm c,h}$ to bicategory $B_\infty\-{\rm span}^2$ (Theorem~\ref{Thm-LaxFuntor-HomotopyCat-B-Inf}). During the construction, Keller's upper triangular matrix dg category (Ref. \cite[4.5, Paragraph 1]{Keller03}) plays a crucial role. Next, since the bicategories $\dgCAT_{\rm c,h}$ and $\dg\cC\aA\tT$ are biequivalent, we will get a lax functor from bicategory $\dg\cC\aA\tT$ to bicategory $B_\infty\-{\rm span}^2$ (Theorem~\ref{Thm-LaxFuntor-DerCat-B-Inf}). Note that every bicategory is biequivalent to its 1-skeleton. Passing from bicategories to their 1-skeletons, we will obtain a lax functor from bicategory $\overline{\dg\cC\aA\tT}$ to bicategory $\overline{B_\infty\-{\rm span}^2}$ (Corollary~\ref{Cor-overline-ccC}). The 1-skeleton $\overline{\dg\cC\aA\tT}$ of the bicategory $\dg\cC\aA\tT$ is closely related to the homotopy category ${\rm Hmo}$. Indeed, ${\rm Hmo}$ is a sub-bicategory of $\overline{\dg\cC\aA\tT}$. Last but not least, we will complete the overview diagram $(\bigstar)$ by defining the remain (bi)categories and (lax) functors and checking the commutativity of some diagrams. In three appendices, we will provide the detailed proofs of two lemmas and one fact which are used in the main text.

\section{Hochschild cochain complexes of dg categories}

In this section, we will fix some notations on the Hochschild cochain complexes of small dg categories. Moreover, we will give some new results which are necessary preparations for constructing lax functors of Hochschild cochain complexes of dg categories.

\subsection{Dg categories and dg modules.} In this subsection, we fix some notations on dg categories and dg modules. For more knowledge on dg categories and dg modules, we refer to \cite{Keller06}. 

\medspace

\noindent{\bf Dg categories.} Denote by $\gG k$ (resp. $\cC k$) the category of $\mathbb{Z}$-graded $k$-vector spaces (resp. complexes of $k$-vector spaces). Write $\otimes$ for $\otimes_k$. The {\it tensor product} $V\otimes V'$ of the complexes $V$ and $V'$ of $k$-vector spaces is the graded $k$-vector space 
$$V\otimes V'=\bigoplus\limits_{n\in\mathbb{Z}} (V\otimes V')^n := \bigoplus\limits_{n\in\mathbb{Z}} \bigoplus\limits_{i+j=n} (V^i\otimes V'^j)$$ 
with differential $d_{V\otimes V'}:=d_V\otimes 1_{V'}+1_V\otimes d_{V'}$. Then $(\cC k,\otimes,k)$ is a closed symmetric monoidal category for which the {\it internal Hom complex} 
$$\hHom_k(V,V'):= \bigoplus\limits_{n\in\mathbb{Z}} \gG k(V,s^nV')$$ with differential $d_{\hHom_k(V,V')}(f):=d_{V'}\circ f-(-1)^{|f|}f\circ d_V$, where the {\it shift functor} $s:\gG k\to\gG k, V\mapsto sV, (f:V\to V')\mapsto(sf:sV\to sV')$, and $(sV)^n:=V^{n+1}, (sf)^n:=f^{n+1}$, for all $n\in\mathbb{Z}$.

A {\it dg category} is a $\cC k$-enriched category (Ref. \cite{Kelly82}), i.e., a category $\aA$ whose Hom spaces are complexes of $k$-vector spaces and whose compositions $\aA(A',A'')\otimes\aA(A,A')\to\aA(A,A''), g\otimes f\mapsto gf,$ are morphisms of complexes of $k$-vector spaces for all $A,A',A''\in\aA$. A key example of dg categories is the dg category $\cC_\dg k$ of complexes of $k$-vector spaces whose objects are all complexes of $k$-vector spaces and whose Hom spaces $\cC_\dg k(V,V'):=\hHom_k(V,V')$ for all complexes $V$ and $V'$ of $k$-vector spaces. The {\it opposite dg category} $\aA^\op$ of a dg category $\aA$ has the same objects as $\aA$, the Hom spaces $\aA^\op(A,A'):=\aA(A',A)$, and the composition $f\circ_{\aA^\op} g:=(-1)^{mn}g\circ_\aA f$, for all $A,A',A''\in\aA$, $f\in\aA^\op(A',A)^m$ and $g\in\aA^\op(A'',A')^n$.

Let $\aA$ and $\bB$ be dg categories. A {\it dg functor} $F:\aA\to\bB$ is given by a function $F:\Ob(\aA)\to\Ob(\bB)$ and morphisms of complexes $F_{AA'}:\aA(A,A')\to\bB(FA,FA')$ for all $A,A'\in\aA$ which are compatible with the composition and the units. The {\it category of small dg categories} dgcat has all small dg categories as objects and all dg functors between them as morphisms. The {\it tensor product} $\aA\otimes\bB$ of two small dg categories $\aA$ and $\bB$ is the small dg category having the set of objects $\Ob(\aA)\times \Ob(\bB)$, the Hom spaces $$(\aA\otimes\bB)((A,B),(A',B')):= \aA(A,A')\otimes\bB(B,B'),$$ and the natural compositions and units. Let $\aA,\bB\in{\rm dgcat}$ and $F,G\in{\rm dgcat}(\aA,\bB)$. A {\it graded natural transformation $\alpha:F\Rightarrow G$ of degree $n$} from $F$ to $G$ is a set of morphisms  $\alpha=\{\alpha_A\in\bB(FA,GA)^n\}_{A\in\aA}$ such that $Gf\circ\alpha_A=(-1)^{mn}\alpha_{A'}\circ Ff$ for all $A,A'\in\aA$ and $f\in\aA(A,A')^m$. 
$$\begin{tikzcd}
	FA \arrow[r,"\alpha_A"] \arrow[d,"Ff"] & GA \arrow[d,"Gf"] \\
	FA' \arrow[r,"\alpha_{A'}"] & GA' 
\end{tikzcd}$$
The {\it complex of graded natural transformations} $\hHom(F,G)$ from $F$ to $G$ is the graded vector space 
$$\hHom(F,G)=\bigoplus\limits_{n\in\mathbb{Z}}\hHom(F,G)^n,$$ 
where $\hHom(F,G)^n$ is the vector space of all graded natural transformations of degree $n$ from $F$ to $G$, with the differential $d_{\hHom(F,G)}$ given by $d_{\hHom(F,G)}(\alpha) = \{d_{\bB(FA,GA)}(\alpha_A)\}_{A\in\aA}$. Then $({\rm dgcat},\otimes,k)$ is a closed symmetric monoidal category for which the internal Hom dg category $\hHom(\aA,\bB)$ has all dg functors from $\aA$ to $\bB$ as objects and Hom spaces $\hHom(\aA,\bB)(F,G):=\hHom(F,G)$.

The {\it homotopy category} $H^0\aA$ of a dg category $\aA$ is the category with the same objects as $\aA$ and Hom spaces 
$$(H^0\aA)(A,A'):=H^0(\aA(A,A')),$$ 
the 0-th cohomology of the complex $\aA(A,A')$, for all $A,A'\in\aA$.
A {\it quasi-equivalence} from a dg cateory $\aA$ to a dg category $\bB$ is a dg functor $F:\aA\to\bB$ such that $F_{AA'}:\aA(A,A')\to\bB(FA,FA')$ is a quasi-isomorphism for all $A,A'\in\aA$ and the induced functor $H^0F:H^0\aA\to H^0\bB, A\mapsto FA, [f]\in H^0\aA(A,A')\mapsto [Ff]\in H^0\bB(FA,FA')$, is an equivalence.

\medspace

\noindent{\bf Dg modules.} Let $\aA$ be a small dg category. A {\it left dg $\aA$-module} is a dg functor $\aA\to \cC_\dg k$. A {\it right dg $\aA$-module} is a dg functor $\aA^\op\to \cC_\dg k$. 

The {\it dg category $\cC_\dg\aA$ of right dg $\aA$-modules} has all right dg $\aA$-modules as objects and Hom spaces $\cC_\dg\aA(X,X'):=\hHom(X,X')$ for all right dg $\aA$-modules $X$ and $X'$. The {\it category of right dg $\aA$-modules} $\cC\aA := Z^0(\cC_\dg\aA)$, where $Z^0\aA$ is the {\it 0-th cocycle category} of the dg category $\aA$ whose objects are all objects of $\aA$ and whose Hom spaces $(Z^0\aA)(A,A'):=Z^0(\aA(A,A')):={\rm Ker}d^0_{\aA(A,A')}$, the 0-th cocycle of the complex $\aA(A,A')$, for all $A,A'\in\aA$. Namely, the category of right dg $\aA$-modules $\cC\aA$ has all right dg $\aA$-modules as objects and all closed (= differential zero = cocycle) graded natural transformations of degree $0$ between them as morphisms.
The {\it homotopy category of right dg $\aA$-modules} $\hH\aA := H^0(\cC_\dg\aA)$. 
A morphism $f:X\to X'$ of right dg $\aA$-modules is a {\it quasi-isomorphism} if it induces quasi-isomorphisms $f_A:XA\to X'A$ for all $A\in\aA$, or equivalently, isomorphisms $H^n(f_A):H^n(XA)\to H^n(X'A)$ for all $A\in\aA$ and $n\in\mathbb{Z}$. The {\it derived category} $\dD\aA$ of a small dg category $\aA$ is the localization of the category $\cC\aA$ with respect to the class of quasi-isomorphisms.

Let $\aA,\bB$ be small dg categories. A {\it dg $\aA\-\bB$-bimodule} is a dg functor from $\bB^\op\otimes\aA$ to $\cC_\dg k$. A dg $\aA\-\bB$-bimodule coincides with a left dg $\aA\otimes\bB^\op$-module or a right dg $\aA^\op\otimes\bB$-module. The {\it tensor product} of dg $\aA\-\bB$-bimodule $X$ and dg $\bB\-\cC$-bimodule $Y$ over dg category $\bB$ is the dg $\aA\-\cC$-bimodule $X\otimes_\bB Y$ given by 
{\small $$(X\otimes_\bB Y)(C,A)	:= {\rm Coker}\left(\bigoplus\limits_{B,B'\in\bB}X(B',A)\otimes\bB(B,B')\otimes Y(C,B) \xrightarrow{\nu} \bigoplus\limits_{B\in\bB}X(B,A)\otimes Y(C,B)\right)$$}

\noindent where the morphism of complexes $\nu$ is defined by $\nu(x\otimes b\otimes y):=xb\otimes y-x\otimes by=X(b,A)(x)\otimes y-x\otimes Y(C,b)(y)$, for all $A\in\aA$ and $C\in\cC$. For a small dg category $\aA$, the {\it identity dg $\aA$-bimodule} $I_\aA$ is given by $I_\aA(A,A'):=\aA(A,A'), I_\aA(a,a'):=\aA(a,a')=\aA(A,a')\aA(a,A')=\aA(a,A''')\aA(A'',a')$, for all $A, A', A'', A'''\in\aA, a\in\aA(A,A''), a'\in\aA(A',A''')$. 
$$\begin{tikzcd}
	\aA(A'',A') \arrow[r,"{\aA(a,A')}"] \arrow[d,"{\aA(A'',a')}"'] & \aA(A,A') \arrow[d,"{\aA(A,a')}"] \\ \aA(A'',A''') \arrow[r,"{\aA(a,A''')}"] & \aA(A,A''')
\end{tikzcd}$$
Then $(\cC(\aA^\op\otimes\aA),\otimes_\aA,I_\aA)$ is a monoidal category.

\medspace

\noindent{\bf Model structures on the category of dg modules.} (Ref. \cite[3.2]{Keller06}) Let $\aA$ be a small dg category. A dg $\aA$-module $P$ is {\it cofibrant} if for any surjective quasi-isomorphism $f:X\to Y$ and morphism $g:P \to Y$ in $\cC\aA$, there is a morphism $h:P \to X$ in $\cC\aA$ such that $g=fh$.
A dg $\aA$-module $I$ is {\it fibrant} if for any injective quasi-isomorphism $f:X \rightarrow Y$ and morphism $g:X \rightarrow I$ in $\cC\aA$, there is a morphism $h:Y \to I$ in $\cC\aA$ such that $g=hf$. 
The category $\cC\aA$ admits two structures of Quillen model category whose weak equivalences are just all quasi-isomorphisms:
(1) The {\it projective model structure}, whose fibrations are just all epimorphisms. For this model structure, each object is fibrant, and an object is cofibrant if and only if it is a cofibrant dg module.
(2) The {\it injective model structure}, whose cofibrations are just all monomorphisms. For this model structure, each object is cofibrant, and an object is fibrant if and only if it is a fibrant dg module.
For both model structures, two morphisms are homotopic if and only if they are equal in the homotopy category $\hH\aA$ of dg $\aA$-modules.

\subsection{$B_\infty$-algebras and Gerstenhaber algebras} In this subsection, we introduce the category of $B_\infty$-algebras and the category of Gerstenhaber algebras. $B_\infty$-algebras must be $A_\infty$-algebras.

\medspace

\noindent{\bf $A_\infty$-algebras.} $A_\infty$-algebra was introduced by Stasheff in \cite{Stasheff}. We refer to \cite{Keller01} for a nice introduction on $A_\infty$-algebras. 
An {\it $A_\infty$-algebra} $(A,\{m_n\}_{n\in\mathbb{N}})$ is a graded vector space $A$ endowed with
graded linear maps $m_n: A^{\otimes n}\to A$ of degree $2-n$ for all $n\ge 1$ satisfying
the {\it Stasheff identities}
$$\sum\limits_{i+j+k=n, \atop i,k\ge 0, j\ge 1} (-1)^{i+jk}\ m_{i+1+k}(1^{\otimes i}\otimes m_j\otimes 1^{\otimes k})=0,\quad \mbox{for all}\ n\ge 1.$$
Note that, when $n=1$, we have $m_1^2=0$. Thus $(A,m_1)$ is a complex of vector spaces.
An {\it $A_\infty$-morphism} $f=\{f_n\}_{n\ge 1}$ from an $A_\infty$-algebra $(A,\{m_n\}_{n\in\mathbb{N}})$ to an $A_\infty$-algebra $(A',\{m'_n\}_{n\in\mathbb{N}})$ is a set of graded linear maps $f_n:A^{\otimes n}\to A'$ of degree $1-n$ for all $n\ge 1$ such that the following identities hold:
$$\sum\limits_{i+j+k=n, \atop i,k\ge 0,\ j\ge 1} (-1)^{i+jk}\ f_{i+1+k}(1^{\otimes i}\otimes m_j\otimes 1^{\otimes k})=\sum\limits_{l\ge 1, \atop n_1+\cdots+n_l=n} (-1)^\epsilon\ m'_l(f_{n_1}\otimes\cdots\otimes f_{n_l}),\ \mbox{for all}\ n\ge 1,$$
where $\epsilon:=(l-1)(n_1-1)+(l-2)(n_2-1)+\cdots+2(n_{l-2}-1)+(n_{l-1}-1)$ for all $l\ge 2$, and $\epsilon:=0$ for $l=1$.
Note that, when $n=1$, we have $f_1m_1=m'_1f_1$. Thus $f_1:(A,m_1)\to(A',m'_1)$ is a morphism of complexes.
The {\it composition} $f'\circ f$ of two $A_\infty$-morphisms $f=\{f_n\}_{n\ge 1}:A\to A'$ and $f'=\{f'_n\}_{n\ge 1}:A'\to A''$ is given by 
$$(f'\circ f)_n:=\sum\limits_{l\ge 1, \atop n_1+\cdots+n_l=n} (-1)^\epsilon\ f'_l(f_{n_1}\otimes\cdots\otimes f_{n_l}),\ \mbox{for all}\ n\ge 1,$$
where $\epsilon$ is defined as above. 
An $A_\infty$-morphism $f:A\to A'$ is {\it strict} if $f_n=0$ for all $n\ge 2$. 
The {\it identity $A_\infty$-morphism} on an $A_\infty$-algebra $A$ is the strict morphism $1_A:A\to A$. 
All $A_\infty$-algebras and $A_\infty$-morphisms form the {\it category $A_\infty$ of $A_\infty$-algebras}.
An $A_\infty$-morphism $f:A\to A'$ is an {\it $A_\infty$-isomorphism} if there exists an $A_\infty$-morphism $f': A'\to A$ such that $f'\circ f=1_A$ and $f\circ f'=1_{A'}$.
An $A_\infty$-morphism $f:A\to A'$ is an {\it $A_\infty$-quasi-isomorphism} if the morphism of complexes $f_1:(A,m_1)\to(A',m'_1)$ is a quasi-isomorphism.

\medspace

\noindent{\bf $B_\infty$-algebras.} 
$B_\infty$-structure first appeared in \cite{Baues81}. $B_\infty$-algebra was introduced by Getzler and Jones in \cite[Subsection 5.2]{GetzlerJones94}. 
For more knowledge on $B_\infty$-algebras, we refer to \cite{ChenLiWang20}.
Let $A$ be a graded vector space. For all nonnegative integers $i,j\ge 0$, positive integer $p\ge 1$, nonnegative $p$-partitions $i=i_1+\cdots+i_p$ and $j=j_1+\cdots+j_p$ with $i_1,\cdots,i_p,j_1,\cdots,j_p\ge 0$, we define the graded linear isomorphism 
$$\tau_{i_1,\cdots,i_p; j_1,\cdots,j_p}: A^{\otimes i}\bigotimes A^{\otimes j} \to
(A^{\otimes i_1}\bigotimes A^{\otimes j_1})\otimes\cdots\otimes(A^{\otimes i_p}\bigotimes A^{\otimes j_p})$$
by sending $a_1\otimes\cdots\otimes a_i\bigotimes a'_1\otimes\cdots\otimes a'_j$ to $(-1)^\epsilon\
(a_1\otimes\cdots\otimes a_{i_1}\bigotimes a'_1\otimes\cdots\otimes a'_{j_1}) \otimes\cdots\otimes (a_{i_1+\cdots+i_{p-1}+1}\otimes\cdots\otimes a_i\bigotimes a'_{j_1+\cdots+j_{p-1}+1}\otimes\cdots\otimes a'_j)$,
where $\epsilon := \sum\limits_{l=1}^p
(|a_{i_1+\cdots+i_l+1}|+\cdots+|a_i|)(|a'_{j_1+\cdots+j_{l-1}+1}|+\cdots+|a'_{j_1+\cdots+j_l}|)$.

A {\it $B_\infty$-algebra} $(A,\{m_n\}_{n\in\mathbb{N}},\{m_{p,q}\}_{p,q\in\mathbb{N}_0})$
is a graded vector space $A$ endowed with a graded linear map
$m_n : A^{\otimes n}\to A$ of degree $2-n$ for each $n\ge 1$ and
a graded linear map $m_{p,q} : A^{\otimes p}\otimes A^{\otimes q} \to A$ of degree $1-p-q$ for each pair $p,q\ge 0$ satisfying:

\medskip

(1) {\it Unity}:
$$m_{1,0}=1_A=m_{0,1},\quad m_{p,0}=0=m_{0,p},\ \mbox{for all}\ p\ge 2.$$

(2) {\it Associativity}: For all $i,j,k\ge 1$,
$$\begin{array}{ll}
	& \sum\limits_{p=1}^{i+j}\sum\limits_{i_1+\cdots+i_p=i, \atop j_1+\cdots+j_p=j} (-1)^{\eta_1}\ m_{p,k}(((m_{i_1,j_1}\otimes\cdots\otimes m_{i_p,j_p})\tau_{i_1,\cdots,i_p; j_1,\cdots,j_p})\bigotimes 1^{\otimes k}) \\ [8mm]
	=&\sum\limits_{q=1}^{j+k} \sum\limits_{j_1+\cdots+j_q=j, \atop k_1+\cdots+k_q=k} (-1)^{\epsilon_1}\ m_{i,q}(1^{\otimes i} \bigotimes((m_{j_1,k_1}\otimes\cdots\otimes m_{j_q,k_q})\tau_{j_1,\cdots,j_q; k_1,\cdots,k_q})),
\end{array}$$
where
$$\eta_1 := \sum\limits_{l=1}^p(i_l+j_l-1)(k+p-l)+\sum\limits_{l=1}^{p-1}j_l(i_{l+1}+\cdots+i_p),$$
$$\epsilon_1 := \sum\limits_{t=1}^{q-1}(j_t+k_t-1)(q-t)+\sum\limits_{t=1}^{q-1}k_t(j_{t+1}+\cdots+j_q).$$

(3) {\it Leibniz rule}: For all $i,j\ge 1$,
$$\begin{array}{ll}
	& \sum\limits_{n=1}^{i+j} \sum\limits_{i_1+\cdots+i_n=i, \atop j_1+\cdots+j_n=j} (-1)^{\epsilon_2}\ m_n(m_{i_1,j_1}\otimes\cdots\otimes m_{i_n,j_n})\tau_{i_1,\cdots,i_n;j_1,\cdots,j_n} \\ [8mm]
	= & \sum\limits_{p+q+r=i, \atop p,r\ge 0,\ q\ge 1} (-1)^{\eta_2}\ m_{p+1+r,j}(1^{\otimes p}\otimes m_q\otimes 1^{\otimes r} \bigotimes 1^{\otimes j}) \\ [8mm]
	& + \sum\limits_{u+v+w=j, \atop u,w\ge 0,\ v\ge 1} (-1)^{\eta_3}\ m_{i,u+1+w}(1^{\otimes i} \bigotimes 1^{\otimes u} \otimes m_v\otimes 1^{\otimes w}),
\end{array}$$
where
$$\epsilon_2 := \sum\limits_{l=1}^{n-1}(i_l+j_l-1)(n-l)+\sum\limits_{l=1}^{n-1}j_l(i-i_1-\cdots-i_l),$$
$$\eta_2 := (i+j-p-q)q+p,\quad\quad \eta_3=i+u+v(j-u-v).$$

\medspace

Let $A=(A,\{m_n\}_{n\ge 1},\{m_{p,q}\}_{p,q\ge 0})$ and $A'=(A',\{m'_n\}_{n\ge 1},\{m'_{p,q}\}_{p,q\ge 0})$ be two $B_\infty$-algebras. A {\it $B_\infty$-morphism} from $A$ to $A'$ is an $A_\infty$-morphism $f=\{f_n\}_{n\ge 1}:A\to A'$ satisfying the following identities:
$$\begin{aligned}
& \sum\limits_{p,q\ge 0} \sum\limits_{i_1+\cdots+i_p=i \atop j_1+\cdots+j_q=j} (-1)^\epsilon\ m'_{p,q}(f_{i_1}\otimes\cdots\otimes f_{i_p} \bigotimes f_{j_1}\otimes\cdots\otimes f_{j_q}) \\
={} & \sum\limits_{n\ge 1} \sum\limits_{i_1+\cdots+i_n=i \atop j_1+\cdots+j_n=j} (-1)^\eta\ f_n\circ(m_{i_1,j_1}\otimes\cdots\otimes m_{i_n,j_n}) \circ \tau_{i_1,\cdots,i_n;j_1,\cdots,j_n},\ \mbox{for all}\ i,j\ge 0,
\end{aligned}$$
where $\epsilon:=\sum\limits_{k=1}^p(i_k-1)(p+q-k)+\sum\limits_{k=1}^{q-1}(j_k-1)(q-k)$ and $\eta:=\sum\limits_{k=1}^{n-1}j_k(i-i_1-\cdots-i_k)+\sum\limits_{k=1}^{n-1}(i_k+j_k-1)(n-k)$.
The {\it composition} of $B_\infty$-morphisms is just their composition as $A_\infty$-morphisms. 
A $B_\infty$-morphism is {\it strict} if it is strict as an $A_\infty$-morphism. 
The {\it identity $B_\infty$-morphism} on a $B_\infty$-algebra $A$ is the identity $A_\infty$-morphism $1_A:A\to A$.
A $B_\infty$-morphism $f:A\to A'$ is a {\it $B_\infty$-isomorphism} if there exists a $B_\infty$-morphism $f': A'\to A$ such that $f'\circ f=1_A$ and $f\circ f'=1_{A'}$.
A $B_\infty$-morphism is a {\it $B_\infty$-quasi-isomorphism} if it is an $A_\infty$-quasi-isomorphism. 

All $B_\infty$-algebras and $B_\infty$-morphisms form the {\it category $B_\infty$ of $B_\infty$-algebras}.  
The category $B_\infty$ of $B_\infty$-algebras is the category of algebras over the $B_\infty$-operad (Ref. \cite[Section 2]{Voronov00}). Since the category $\cC k$ is a symmetric monoidal category with all small limits and colimits, by \cite[Proposition 2.3.5]{Rezk96} or \cite[Proposition 3.3.2]{Fresse09}, $B_\infty$ has all limits and colimits. 
Since the $B_\infty$-operad is an asymmetric operad, by \cite[Theorem 4.1.1 and Example 4.2.5]{Hinich97}, $B_\infty$ admits a model structure for which weak equivalences are just $B_\infty$-quasi-isomorphisms and fibrations are the $B_\infty$-morphisms corresponding to surjective morphisms of complexes. We denote by ${\rm Ho}B_\infty$ the homotopy category of $B_\infty$ with respect to this model structure. 

\medspace

\noindent{\bf Brace $B_\infty$-algebras.} 
The most important $B_\infty$-algebras are brace $B_\infty$-algebras \cite{Gerstenhaber63,GetzlerJones94,GerstenhaberVoronov95}. The typical examples of brace $B_\infty$-algebras include the singular cochains complexes of simplicial sets \cite{Baues81} and the Hochschild cochain complexes of dg algebras, more general, dg categories \cite{GetzlerJones94,Keller03}.
A {\it brace $B_\infty$-algebra} is a $B_\infty$-algebra $(A,\{m_n\}_{n\ge 1},\{m_{p,q}\}_{p,q\ge 0})$ satisfying $m_n=0$ for all $n\ge 3$ and $m_{p,q}=0$ for all $p\ge 2$. 
A brace $B_\infty$-algebra is also called a {\it homotopy $G$-algebra} or a
{\it Gerstenhaber-Voronov algebra} in some literatures.
The underlying $A_\infty$-algebra structure of a brace $B_\infty$-algebra is just a dg algebra. 

\medspace

\noindent{\bf Gerstenhaber algebras.} Gerstenhaber first found that the Hochschild cohomology of an algebra admits a Gerstenhaber algebra structure \cite{Gerstenhaber63}. 

A {\it Gerstenhaber algebra} $(G,-\cup-,[-,-])$ is a graded vector space $G$ equipped with two graded linear maps: a cup product $-\cup-: G\otimes G\to G$ of degree 0, and a Gerstenhaber bracket $[-,-]: G\otimes G\to G$ of degree $-1$, 
satisfying the following three conditions:

(1) $(G,-\cup-)$ is a graded commutative associative algebra.

(2) $(G^{\bullet+1},[-,-])$ is a graded Lie algebra, i.e., satisfies the following two conditions:

(2.1) $[a,b]=-(-1)^{(|a|-1)(|b|-1)}[b,a]$, and

(2.2) {\it Graded Jacobi idenity:} $$(-1)^{(|c|-1)(|a|-1)}[[a,b],c]+ (-1)^{(|a|-1)(|b|-1)}[[b,c],a]+ (-1)^{(|b|-1)(|c|-1)}[[c,a],b]=0.$$

(3) {\it Graded Leibniz rule:} $[a,b\cup c]=[a,b]\cup c + (-1)^{(|a|-1)|b|}\ b\cup[a,c].$

Let $G=(G,-\cup-,[-,-])$ and $G'=(G',-\cup'-,[-,-]')$ be two Gerstenhaber algebras. A {\it morphism} of Gerstenhaber algebras from $G$ to $G'$ is a graded linear map $f:G\to G'$ of degree $0$ that preserves cup product and Gerstenhaber bracket, i.e., $f(a\cup b)=f(a)\cup' f(b)$ and $f([a,b])=[f(a),f(b)]'$ for all $a,b\in G$.
The {\it composition} of two morphisms of Gerstenhaber algebras is just their composition as graded linear maps. 
The {\it identity morphism} on a Gerstenhaber algebra $G$ is just the identity map $1_G$. 

All Gerstenhaber algebras and morphisms of Gerstenhaber algebras form the {\it category $\mathsf{G}$ of Gerstenhaber algebras}.
The category $\mathsf{G}$ of Gerstenhaber algebras is the category of algebras over the Gerstenhaber operad which is isomorphic to the singular homology of the little disc operad (Ref. \cite{Cohen76} or \cite[Section 13.3]{LodayVallette12}). Since the category $\gG k$ is a symmetric monoidal category with all small limits and colimits, by \cite[Proposition 2.3.5]{Rezk96} or \cite[Proposition 3.3.2]{Fresse09}, $\mathsf{G}$ has all limits and colimits. 
By \cite[Example 1.1.5]{Hovey99}, $\mathsf{G}$ admits a model structure for which weak equivalences are isomorphisms and all morphisms are both fibration and cofibration.

\medspace

\noindent{\bf From $B_\infty$-algebras to Gerstenhaber algebras.}
Let $(A,\{m_n\}_{n\ge 1},\{m_{p,q}\}_{p,q\ge 0})$ be a $B_\infty$-algebra. Then its cohomology $HA := H(A,m_1)$ is a Gerstenhaber algebra with cup product $\overline{a}\cup\overline{b} := \overline{m_2(a,b)}$ 
and Gerstenhaber bracket $[\overline{a},\overline{b}]:=(-1)^{|a|}\ \overline{m_{1,1}(a,b)}-(-1)^{(|a|-1)(|b|-1)+|b|}\ \overline{m_{1,1}(b,a)}.$

Let $f=(f_n)_{n\ge1}:(A,\{m_n\}_{n\ge 1},\{m_{p,q}\}_{p,q\ge 0}) \rightarrow (A',\{m'_n\}_{n\ge 1},\{m'_{p,q}\}_{p,q\ge 0})$ be a $B_\infty$-morphism. 
Firstly, since $f$ is an $A_\infty$-morphism, we have $f_1m_1=m_1'f_1$, i.e., $f_1: (A,m_1)\to (A',m_1')$ is a morphism of complexes.
Thus $f_1$ induces a graded linear map $Hf_1:HA\to HA'$.
Secondly, since $f$ is an $A_\infty$-morphism, we have $f_1 m_2=m_2'(f_1\otimes f_1)+f_2(1\otimes m_1+m_1\otimes 1)+m_1' f_2$.
Hence $Hf_1(\overline{a}\cup \overline{b})=Hf_1(\overline{a})\cup Hf_1(\overline{b})$ for all $\overline{a},\overline{b}\in HA$. 
Thirdly, since $f$ is a $B_\infty$-morphism, we have $m'_{1,1}(f_1\otimes f_1)=f_1 m_{1,1}-f_2(m_{0,1}\otimes m_{1,0})\tau+f_2(m_{1,0}\otimes m_{0,1})$, where $\tau : A\otimes A\to A\otimes A, a\otimes b\mapsto (-1)^{|a||b|}b\otimes a$, is the swapping map.
Due to $m_{0,1}=m_{1,0}=1$, we get $m'_{1,1}(f_1\otimes f_1)-f_2(1-\tau)=f_1 m_{1,1}.$
Thus $f_1 m_{1,1}(1+\tau) = m_{1,1}'(f_1\otimes f_1)(1+\tau)$. Hence, $Hf_1([\overline{a},\overline{b}])=[Hf_1(\overline{a}),Hf_1(\overline{b})]$ for all $\overline{a},\overline{b}\in HA$. So far we have shown that $Hf_1$ is a morphism of Gerstenhaber algebras.
Finally, let $f'=(f'_n)_{n\ge1}:(A',\{m'_n\}_{n\ge 1},\{m'_{p,q}\}_{p,q\ge 0}) \rightarrow (A'',\{m''_n\}_{n\ge 1},\{m''_{p,q}\}_{p,q\ge 0})$ be also a $B_\infty$-morphism. Due to $(f'f)_1=f'_1f_1$ and $(1_A)_1=1_A$, $H$ is compatible with compositions and units. 

We have finished the proof of the following result.

\begin{proposition} \label{Prop:B_inf-to-G}
	Taking  cohomology $H : B_\infty \rightarrow \mathsf{G}, A\mapsto HA, (f:A\to A')\mapsto (Hf_1:HA\to HA'),$ is a functor.
\end{proposition}

\subsection{Hochschild cochain complex of a dg category} In this subsection, we recall the Hochschild cochain complex of a dg category and the brace $B_\infty$-structure on it. 

\medspace

\noindent{\bf Hochschild cochain complex $C(\aA)$.} 
Let $\aA$ be a small dg category. 
The {\it Hochschild cochain complex} $C(\aA)$ of $\aA$ is the graded vector space 
\[
C(\aA)=\prod_{n\ge0}C^n(\aA)=\prod_{n\ge0}\prod_{A_0,\cdots,A_n\in\aA}
\mathcal{H}\mathrm{om}_k(s\aA(A_{n-1},A_n)\otimes \cdots\otimes s\aA(A_0,A_1),\aA(A_0,A_n)),
\]
where $s$ is the shift functor on $\cC k$,
equipped with differential $\delta_{C(\aA)}:=(\delta_{C(\aA)})_\nin+(\delta_{C(\aA)})_\ex$, where the {\it internal differential} $(\delta_{C(\aA)})_\nin$ is defined as follows: For any $\phi\in C^n(\aA)$, its image $(\delta_{C(\aA)})_\nin\phi\in C^n(\aA)$ is given by
\[
((\delta_{C(\aA)})_\nin\phi)(s{a}_{1,n})\coloneqq d_\aA\phi(s{a}_{1,n})+\sum_{i=1}^n(-1)^{|\phi|+|s{a}_{1,i-1}|}\phi(s{a}_{1,i-1}\otimes s{d_\aA a}_i\otimes s{a}_{i+1,n}),
\]
where $s$ denotes the graded linear map $X\to sX, (x\in X^n)\mapsto (x\in (sX)^{n-1})$, of degree $-1$, for all $X\in\cC k$ and $n\in \mathbb{Z}$, $sa_{1,n}:= sa_1\otimes\cdots\otimes sa_n$, and $|sa_{1,i-1}|:= \sum_{j=1}^{i-1}|sa_j|$, and
the {\it external differential} $(\delta_{C(\aA)})_\ex$ is defined as follows: For any $\phi\in C^n(\aA)$, its image $(\delta_{C(\aA)})_\ex\phi\in C^{n+1}(\aA)$ is given by
\[
\begin{aligned}
	((\delta_{C(\aA)})_\ex\phi)(s{a}_{1,n+1}):= & -(-1)^{|s{a}_1||\phi|}a_1\phi(s{a}_{2,n+1})+(-1)^{|\phi|+|s{a}_{1,n}|}\phi(s{a}_{1,n})a_{n+1} \\
	& -\sum_{i=1}^{n}(-1)^{|\phi|+|s{a}_{1,i}|}\phi(s{a}_{1,i-1}\otimes s(a_ia_{i+1})\otimes s{a}_{i+2,n+1}).
\end{aligned}
\]

The {\it Hochschild cohomology} $\HH^\bullet(\aA)$ of $\aA$ is the cohomology $H^\bullet(C(\aA))$ of $C(\aA)$.

From the two-sided bar resolution of the identity dg $\aA$-bimodule $I_\aA$, we know that the Hochschild cochain complex $C(\aA)$ of a small dg category $\aA$ is isomorphic to the derived Hom complex $\RHom_{\aA^\op\otimes\aA}(I_\aA,I_\aA)$ in the derived category $\dD k$ of $k$.	

\medspace

\noindent{\bf The brace $B_\infty$-structure on Hochschild cochain complex.}
Let $\aA$ be a small dg category. There are two operations on the Hochschild cochain complex $C(\aA)$ of $\aA$ --- cup product and brace operation.
The {\it cup product} $-\cup- : C(\aA)\otimes C(\aA) \rightarrow C(\aA), \phi\otimes\psi\mapsto\phi\cup\psi$, is defined by
$$(\phi\cup\psi)(sa_{1,p+q}):= (-1)^{|sa_{1,p}||\psi|}\phi(sa_{1,p})\psi(sa_{p+1,p+q})$$
for all $\phi\in C^p(\aA)$ and $\psi\in C^q(\aA)$.
The {\it brace operation} $-\{-,\cdots,-\}:C(\aA)\otimes C(\aA)^{\otimes k} \rightarrow C(\aA), \phi\otimes\psi_1\otimes\cdots\otimes\psi_k\mapsto\phi\{\psi_1,\cdots,\psi_k\}$, is defined by
$$\phi\{\psi_1,\cdots,\psi_k\}:=\sum_{\substack{i_0+\cdots+i_{k}=n-k,\\ i_0,\dots,i_k\ge 0}}\phi\left(1^{\otimes i_0}\otimes s\psi_1\otimes 1^{\otimes i_1}\otimes\cdots\otimes s\psi_k\otimes 1^{\otimes i_k}\right)$$
for all $k\ge 1$, $\phi\in C^n(\aA)$ and $\psi_i\in C^{n_i}(\aA), 1\le i\le k$.
By convention, the brace operation $\{\}:=1_{C(\aA)}$ when $k=0$.
It is well known that the Hochschild cochain complex $C(\aA)$ of $\aA$ equipped with these two operations is a brace $B_\infty$-algebra (Ref. \cite[5.2, Page 51]{GetzlerJones94}). 

A very useful result observed by Keller in \cite[4.3, Page 8]{Keller03} is the following: For any fully faithful dg functor $F:\aA \rightarrow \bB$, the restriction along $F$ yields a $B_\infty$-morphism $F^*:C(\bB) \rightarrow C(\aA)$.

\subsection{Actions of Hochschild cochain complexes on derived Hom complexes} In this subsection, we will observe all kinds of actions of the Hochschild cochain complex of a dg category on the derived Hom complexes of dg bimodules.

\medspace

\noindent{\bf Derived Hom complex $C(X,X')$.}
Let $\aA$ and $\bB$ be small dg categories, and $X$ and $X'$ dg $\aA$-$\bB$-bimodules. The {\it derived Hom complex} $C(X,X')$ from $X$ to $X'$ is the graded vector space
$C(X,X')=\prod_{n\ge0}C^n(X,X')=\prod_{n\ge0}\prod_{\substack{A_0,\cdots,A_l\in\aA, \\ B_0,\cdots,B_m\in\bB, \\ l+m=n}}\mathcal{H}\mathrm{om}_k(s\aA(A_{l-1},A_l)\otimes\cdots\otimes s\aA(A_0,A_1)\otimes X(B_m,A_0)\otimes s\bB(B_{m-1},B_m)\otimes\cdots\otimes s\bB(B_0,B_1),X'(B_0,A_l))$
equipped with differential $\delta_{C(X,X')}:=(\delta_{C(X,X')})_\nin+(\delta_{C(X,X')})_\ex$, where
the {\it internal differential} $(\delta_{C(X,X')})_\nin$ is defined as follows: For any $\phi\in C^n(X,X')$, its image $(\delta_{C(X,X')})_\nin\phi\in C^n(X,X')$ is given by 
\[
\begin{aligned}
    ((\delta_{C(X,X')})_\nin\phi)& (sa_{1,l} \otimes x\otimes sb_{1,m})\coloneqq d_{X'}\phi(sa_{1,l}\otimes x\otimes sb_{1,m}) \\
	& + \sum_{i=1}^l(-1)^{|\phi|+|sa_{1,i-1}|}\phi(sa_{1,i-1}\otimes sd_\aA a_i\otimes sa_{i+1,l}\otimes x\otimes sb_{1,m}) \\
	& -(-1)^{|\phi|+|sa_{1,l}|}\phi(sa_{1,l}\otimes d_X x\otimes sb_{1,m}) \\
	& +\sum_{j=1}^m(-1)^{|\phi|+|sa_{1,l}|+|x|+|sb_{1,j-1}|}\phi(sa_{1,l}\otimes x\otimes sb_{1,j-1}\otimes sd_\bB b_j\otimes sb_{j+1,m}),
\end{aligned}
\]
where $l+m=n$ and $l,m\ge 0$, and the {\it external differential} $(\delta_{C(X,X')})_\ex$ is defined as follows: For any $\phi\in C^n(X,X')$, its image $(\delta_{C(X,X')})_\ex\phi\in C^{n+1}(X,X')$ is given by
\[
\begin{aligned}
	((\delta_{C(X,X')})_\ex\phi) &(sa_{1,l} \otimes x\otimes sb_{1,m}) \coloneqq -(-1)^{|\phi||sa_1|}a_1\phi(sa_{2,l}\otimes x\otimes sb_{1,m}) \\
	& -\sum_{i=1}^{l-1}(-1)^{|\phi|+|sa_{1,i}|}\phi(sa_{1,i-1}\otimes s(a_ia_{i+1})\otimes sa_{i+2,l}\otimes x\otimes sb_{1,m}) \\
	& +(-1)^{|\phi|+|sa_{1,l}|}\phi(sa_{1,l-1}\otimes a_lx\otimes b_{1,m}) \\
	& -(-1)^{|\phi|+|sa_{1,l}|+|x|}\phi(sa_{1,l}\otimes xb_1\otimes sb_{2,m}) \\
	& -\sum_{j=1}^{m-1}(-1)^{|\phi|+|sa_{1,l}|+|x|+|sb_{1,j}|}\phi(sa_{1,l}\otimes x\otimes sb_{1,j-1}\otimes s(b_jb_{j+1})\otimes sb_{j+2,m}) \\
	& +(-1)^{|\phi|+|sa_{1,l}|+|x|+|sb_{1,m-1}|}\phi(sa_{1,l}\otimes x\otimes sb_{1,m-1})b_m,
\end{aligned}
\]
where $l+m=n+1$ and $l,m\ge 0$.

From the two-sided bar resolution of the dg $\aA$-$\bB$-bimodule $X$, we know that the derived Hom complex $C(X,X')$ is isomorphic to $\RHom_{\aA^\op\otimes\bB}(X,X')$ in $\dD k$.	

For simplicity, we denote $C(X,X)$ by $C(X)$.

\medspace

\noindent{\bf Left action $\alpha_X:C(\aA)\to C(X)$ and right action $\beta_X:C(\bB) \to C(X)$.} Let $\aA$ and $\bB$ be small dg categories, and $X$ a dg $\aA$-$\bB$-bimodule. The morphism of complexes $\alpha_X:C(\aA)\rightarrow C(X)$ is defined as follows: For any $\phi\in C^n(\aA)$, its image $\alpha_X\phi\in C^n(X)$ has the only nontrivial component $s\aA(A_{n-1},A_n)\otimes\cdots\otimes s\aA(A_0,A_1)\otimes X(B_0,A_0)\to X(B_0,A_n), sa_{1,n}\otimes x\mapsto\phi(sa_{1,n})x$. 
The morphism of complexes $\beta_X:C(\bB) \rightarrow C(X)$ is defined as follows: For any $\phi\in C^n(\bB)$, its image $\beta_X\phi\in C^n(X)$ has the only nontrivial component $X(B_n,A_0)\otimes s\bB(B_{n-1},B_n)\otimes\cdots\otimes s\bB(B_0,B_1)\to X(B_0,A_0), x\otimes sb_{1,n} \mapsto (-1)^{|\phi||x|} x\phi(sb_{1,n})$.

A very useful result observed by Keller in \cite[4.4]{Keller03} is the following: If a dg $\aA$-$\bB$-bimodule $X$ is cofibrant over $\bB$ (resp. $\aA$) and the derived tensor product functor $-\otimes_\aA^{\bf L}X:\per\aA \rightarrow \dD\bB$ (resp. $X\otimes_\bB^{\bf L}-:\per\bB^\op \rightarrow \dD\aA^\op$)
is fully faithful then $\alpha_X$ (resp. $\beta_X$) is a quasi-isomorphism. In particular, if $X$ is the identity $\aA$-bimodule $I_\aA$ then both $\alpha_{I_\aA}$ and $\beta_{I_\aA}$ are quasi-isomorphisms.

\medspace

\noindent{\bf Quasi-isomorphism $\gamma_{\aA}:C(I_\aA) \to C(\aA)$.} Let $\aA$ be a small dg category. The morphism of complexes $\gamma_\aA:C(I_\aA) \rightarrow C(\aA)$
is defined as follows: For any $\phi\in C^n(I_\aA)$, its image $\gamma_\aA\phi\in C^n(\aA)$ is given by $(\gamma_\aA\phi)(sa_{1,n}):= \sum_{i=0}^n\phi(sa_{1,i}\otimes 1\otimes sa_{i+1,n})$.

It is easy to see that $\gamma_\aA\alpha_{I_\aA}=1_{C(\aA)}=\gamma_\aA\beta_{I_\aA}$. Since $\alpha_{I_\aA}$ (or $\beta_{I_\aA}$) is a quasi-isomorphism, $\gamma_\aA$ is also a quasi-isomorphism. Furthermore, $\alpha_{I_\aA}=\beta_{I_\aA}$ in $\dD k$.

\medspace

\noindent{\bf Shifted derived Hom complex $D(X,X')$.}
Let $\aA$ and $\bB$ be small dg categories, and $X$ and $X'$ dg $\aA$-$\bB$-bimodules. The {\it shifted derived Hom complex} $D(X,X')$ from $X$ to $X'$ is the graded vector space
$D(X,X') =\prod_{n\ge0}D^n(X,X') =\prod_{n\ge0}\prod_{\substack{A_0,\cdots,A_l\in\aA \\ B_0,\cdots,B_m\in\bB \\ l+m=n}}\mathcal{H}\mathrm{om}_k(s\aA(A_{l-1},A_l)\otimes\cdots\otimes s\aA(A_0,A_1)\otimes sX(B_m,A_0) \otimes s\bB(B_{m-1},B_m)\otimes\cdots\otimes s\bB(B_0,B_1),X'(B_0,A_l))$
equipped with differential $\delta_{D(X,X')}:=(\delta_{D(X,X')})_\nin+(\delta_{D(X,X')})_\ex$, where
the {\it internal differential} $(\delta_{D(X,X')})_\nin$ is defined as follows: For any $\phi\in D^n(X,X')$, its image $(\delta_{D(X,X')})_\nin\phi\in D^n(X,X')$ is given by
\[
\begin{aligned}
	((\delta_{D(X,X')})_\nin\phi)& (sa_{1,l} \otimes sx\otimes sb_{1,m})\coloneqq d_{X'}\phi(sa_{1,l}\otimes sx\otimes sb_{1,m}) \\
	& + \sum_{i=1}^l(-1)^{|\phi|+|sa_{1,i-1}|}\phi(sa_{1,i-1}\otimes sd_\aA a_i\otimes sa_{i+1,l}\otimes sx\otimes sb_{1,m}) \\
	& +(-1)^{|\phi|+|sa_{1,l}|}\phi(sa_{1,l}\otimes sd_X x\otimes sb_{1,m}) \\
	& +\sum_{j=1}^m(-1)^{|\phi|+|sa_{1,l}|+|sx|+|sb_{1,j-1}|}\phi(sa_{1,l}\otimes sx\otimes sb_{1,j-1}\otimes sd_\bB b_j\otimes sb_{j+1,m}),
\end{aligned}
\]
where $l+m=n$ and $l,m\ge 0$, and the {\it external differential} $(\delta_{D(X,X')})_\ex$ is defined as follows: For any $\phi\in D^n(X,X')$, its image $(\delta_{D(X,X')})_\ex\phi\in D^{n+1}(X,X')$ is given by
\[
\begin{aligned}
	((\delta_{D(X,X')})_\ex\phi)& (sa_{1,l} \otimes sx\otimes sb_{1,m}) \coloneqq -(-1)^{|\phi||sa_1|}a_1\phi(sa_{2,l}\otimes sx\otimes sb_{1,m}) \\
	& -\sum_{i=1}^{l-1}(-1)^{|\phi|+|sa_{1,i}|}\phi(sa_{1,i-1}\otimes s(a_ia_{i+1})\otimes sa_{i+2,l}\otimes sx\otimes sb_{1,m}) \\
	& -(-1)^{|\phi|+|sa_{1,l}|}\phi(sa_{1,l-1}\otimes s(a_lx)\otimes b_{1,m}) \\
	& -(-1)^{|\phi|+|sa_{1,l}|+|sx|}\phi(sa_{1,l}\otimes s(xb_1)\otimes sb_{2,m}) \\
	& -\sum_{j=1}^{m-1}(-1)^{|\phi|+|sa_{1,l}|+|sx|+|sb_{1,j}|}\phi(sa_{1,l}\otimes sx\otimes sb_{1,j-1}\otimes s(b_jb_{j+1})\otimes sb_{j+2,m}) \\
	& +(-1)^{|\phi|+|sa_{1,l}|+|sx|+|sb_{1,m-1}|}\phi(sa_{1,l}\otimes sx\otimes sb_{1,m-1})b_m,
\end{aligned}
\]
where $l+m=n+1$ and $l,m\ge 0$.

For simplicity, we denote $D(X,X)$ by $D(X)$.

\medspace

\noindent{\bf Shift morphism $\eta_{X,X'}:C(X,X')\to D(X,X')$.} 
Let $\aA$ and $\bB$ be small dg categories, and $X$ and $X'$ dg $\aA$-$\bB$-bimodules. The graded linear map
$\eta_{X,X'}: C(X,X') \rightarrow D(X,X')$ of degree 1 
is defined as follows: For any $\phi\in C^n(X,X')$, its image  $\eta_{X,X'}\phi\in D^n(X,X')$ is given by
$(\eta_{X,X'}\phi)(sa_{1,l}\otimes sx\otimes sb_{1,m}) := (-1)^{|\phi|+|sa_{1,l}|}\phi(sa_{1,l}\otimes x\otimes sb_{1,m})$.
It is easy to see that $\eta_{X,X'}$ is a 1-cocycle in $\hHom_k(C(X,X'),D(X,X'))$, and induces an isomorphism $D(X,X')\cong s^{-1}C(X,X')$ in $\cC k$. 

For simplicity, we denote $\eta_{X,X}:C(X,X)\to D(X,X)$ by $\eta_X:C(X)\to D(X)$.

\medspace

\noindent{\bf Left action $\tilde{\alpha}_X:C(\aA) \rightarrow D(X)$ and right action $\tilde{\beta}_X:C(\bB) \rightarrow D(X)$.} 	
The closed morphism $\alpha_X:C(\aA) \rightarrow C(X)$ (resp. $\beta_X:C(\bB) \rightarrow C(X)$) of degree 0 induces a closed morphism
$\tilde{\alpha}_X:=\eta_{X}\alpha_X:C(\aA) \rightarrow D(X)$ (resp. $\tilde{\beta}_X:=\eta_{X}\beta_X:C(\bB) \rightarrow D(X)$) of degree 1
with the only nontrivial component given by $(\tilde{\alpha}_X\phi)(sa_{1,n}\otimes sx)=(-1)^{|\phi|+|sa_{1,n}|}\phi(s{a}_{1,n})x$ (resp. 
$(\tilde{\beta}_X\phi)(sx\otimes sb_{1,n})=(-1)^{|\phi||sx|}x\phi(s{b}_{1,n})$).

\medspace

\noindent{\bf Derived Hom complex $C(X_1|\cdots|X_m,X')$.} 
Let $\aA_0,\cdots,\aA_m$ be small dg categories, $X_i$ a dg $\aA_{i-1}$-$\aA_i$-bimodule for all $1\le i\le m$, and $X'$ a dg $\aA_0$-$\aA_m$-bimodule. The {\it derived Hom complex} $C(X_1|\cdots|X_m,X')$ is the graded vector space
$\prod_{n\ge 0}C^n(X_1|\cdots|X_m,X') := \prod_{n\ge0}\prod_{\substack{A_{0,0},\cdots,A_{0,l_0}\in\aA_0, \cdots,
		A_{m,0},\cdots,A_{m,l_m}\in\aA_m, \\ l_0+\cdots+l_m=n}}
\mathcal{H}\mathrm{om}_k(s\aA_0(A_{0,l_0-1},A_{0,l_0}) \otimes\cdots\otimes s\aA_0(A_{0,0},A_{0,1}) \otimes \linebreak X_1(A_{1,l_1},A_{0,0})\otimes\cdots\otimes s\aA_{m-1}(A_{m-1,l_{m-1}-1},A_{m-1,l_{m-1}}) \otimes\cdots\otimes s\aA_{m-1}(A_{m-1,0},A_{m-1,1}) \otimes \linebreak X_m(A_{m,l_m}, A_{m-1,0})\otimes s\aA_m(A_{m,l_m-1},A_{m,l_m})  \otimes\cdots\otimes s\aA_m(A_{m,0},A_{m,1}), X'(A_{m,0},A_{0,l_0}))$
with the differential $\delta_{C(X_1|\cdots|X_m,X')}$ defined similar to $\delta_{C(X,X')}$.

Clearly, $C(X_1|\cdots|X_m,X')\cong \RHom_{\aA_0^\op\otimes\aA_m}(X_1\otimes^{\bL}_{\aA_1}\cdots \otimes^{\bL}_{\aA_{m-1}}X_m,X')$ in $\dD k$.

For simplicity, we denote $C(X_1|\cdots|X_m,X_1\otimes_{\aA_1}\cdots\otimes_{\aA_{m-1}}X_m)$ by $C(X_1|\cdots|X_m)$.

\medspace

\noindent{\bf Shifted derived Hom complex $D(X_1|\cdots|X_m,X')$.} 
Let $\aA_0,\cdots,\aA_m$ be small dg categories, $X_i$ a dg $\aA_{i-1}$-$\aA_i$-bimodule for all $1\le i\le m$, and $X'$ a dg $\aA_0$-$\aA_m$-bimodule. The {\it shifted derived Hom complex} $D(X_1|\cdots|X_m,X')$ 
is the graded vector space \linebreak 
$\prod_{n\ge 0}D^n(X_1|\cdots|X_m,X') := \prod_{n\ge0}\prod_{\substack{A_{0,0},\cdots,A_{0,l_0}\in\aA_0,\cdots,A_{m,0},\cdots,A_{m,l_m}\in\aA_m, \\ l_0+\cdots+l_m=n}}
\mathcal{H}\mathrm{om}_k(s\aA_0(A_{0,l_0-1},A_{0,l_0}) \linebreak \otimes\cdots\otimes s\aA_0(A_{0,0},A_{0,1}) \otimes sX_1(A_{1,l_1},A_{0,0})\otimes\cdots \otimes s\aA_{m-1}(A_{m-1,l_{m-1}-1},A_{m-1,l_{m-1}}) \otimes\cdots\linebreak \otimes s\aA_{m-1}(A_{m-1,0},A_{m-1,1}) \otimes sX_m(A_{m,l_m},A_{m-1,0})\otimes s\aA_m(A_{m,l_m-1},A_{m,l_m}) \otimes\cdots\otimes \linebreak s\aA_m(A_{m,0}, A_{m,1}),X'(A_{m,0},A_{0,l_0}))$
with the differential $\delta_{D(X_1|\cdots|X_m,X')}$ defined similar to  $\delta_{D(X,X')}$.

For simplicity, we denote $D(X_1|\cdots|X_m,X_1\otimes_{\aA_1}\cdots\otimes_{\aA_{m-1}}X_m)$ by $D(X_1|\cdots|X_m)$.

\medspace

\noindent{\bf Shift morphism $\eta_{X_1|\cdots|X_m,X'}:C(X_1|\cdots|X_m,X')\to D(X_1|\cdots|X_m,X')$.} 
Let $\aA_0,\cdots, \linebreak \aA_m$ be small dg categories, $X_i$ a dg $\aA_{i-1}$-$\aA_i$-bimodule for all $1\le i\le m$, and $X'$ a dg $\aA_0$-$\aA_m$-bimodule.
The graded linear map 
$\eta_{X_1|\cdots|X_m,X'}:C(X_1|\cdots|X_m,X') \rightarrow D(X_1|\cdots|X_m, X')$ of degree $m$ is defined as follows: For any $\phi\in C^n(X_1|\cdots|X_m,X')$, its image  $\eta_{X_1|\cdots|X_m,X'}\phi\in D^n(X_1|\cdots|X_m,X')$ is given by
$(\eta_{X_1|\cdots|X_m,X'}\phi)(s(a_0)_{1,l_0}\otimes sx_1\otimes s(a_1)_{1,l_1}\otimes sx_2\otimes\cdots\otimes sx_m\otimes s(a_m)_{1,l_m}) := (-1)^{\sigma} \phi(s(a_0)_{1,l_0}\otimes x_1\otimes s(a_1)_{1,l_1}\otimes x_2\otimes\cdots\otimes x_m\otimes s(a_m)_{1,l_m})$, where $\sigma:=m|\phi|+m|s(a_0)_{1,l_0}|+(m-1)|sx_1|+(m-1)|s(a_1)_{1,l_1}|+(m-2)|sx_2|+\cdots+|sx_{m-1}|+|s(a_{m-1})_{1,l_{m-1}}|$.
It is easy to see that $\eta_{X_1|\cdots|X_m,X'}$ is an $m$-cocycle in $\hHom_k(C(X_1|\cdots|X_m,X'),D(X_1|\cdots|X_m,X'))$, and it induces an isomorphism $D(X_1|\cdots|X_m,X')\cong s^{-m}C(X_1|\cdots|X_m,X')$ in $\cC k$. 	

For simplicity, we denote $\eta_{X_1|\cdots|X_m,X_1\otimes_{\aA_1}\cdots\otimes_{\aA_{m-1}}X_m}$ by $\eta_{X_1|\cdots|X_m}$.

\medspace

\noindent{\bf Left actions $\alpha_{X_1|\cdots|X_m,X'\otimes_{\aA_{m-1}} X_m}$ and $\tilde{\alpha}_{X_1|\cdots|X_m,X'\otimes_{\aA_{m-1}} X_m}$.} 
Let $\aA_0,\cdots,\aA_m$ be small dg categories, $X_i$ a dg $\aA_{i-1}$-$\aA_i$-bimodule for all $1\le i\le m$, and $X'$ a dg $\aA_0$-$\aA_{m-1}$-bimodule. The morphism of complexes 
$$\alpha_{X_1|\cdots|X_m,X'\otimes_{\aA_{m-1}} X_m}: C(X_1|\cdots|X_{m-1},X') \rightarrow C(X_1|\cdots|X_{m-1}|X_m,X'\otimes_{\aA_{m-1}} X_m)$$
is defined as follows: For any $\phi\in C^n(X_1|\cdots|X_{m-1},X')$, its image $\alpha_{X_1|\cdots|X_m,X'\otimes_{\aA_{m-1}} X_m}\phi \in C^n(X_1|\cdots|X_m,X'\otimes_{\aA_{m-1}} X_m)$ has the only nontrivial component $s\aA_0(A_{0,l_0-1},A_{0,l_0})\otimes\cdots\otimes s\aA_0(A_{0,0},A_{0,1})\otimes X_1(A_{1,l_1},A_{0,0})\otimes\cdots\otimes s\aA_{m-1}(A_{m-1,l_{m-1}-1}, A_{m-1,l_{m-1}})\otimes\cdots\otimes s\aA_{m-1}(A_{m-1,0}, \linebreak A_{m-1,1})\otimes X_m(A_{m,0},A_{m-1,0})\to X'(A_{m-1,0},A_{0,l_0})\otimes X_m(A_{m,0}, A_{m-1,0}), s(a_0)_{1,l_0}\otimes x_1\otimes\cdots \otimes s(a_{m-1})_{1,l_{m-1}}\otimes x_m \mapsto \phi(s(a_0)_{1,l_0}\otimes x_1\otimes\cdots\otimes s(a_{m-1})_{1,l_{m-1}})\otimes x_m$.

The morphism of complexes $\alpha_{X_1|\cdots|X_m,X'\otimes_{\aA_{m-1}} X_m}$ induces a closed graded linear map 
$$\tilde{\alpha}_{X_1|\cdots|X_m,X'\otimes_{\aA_{m-1}} X_m}:D(X_1|\cdots|X_{m-1},X') \rightarrow D(X_1|\cdots|X_{m-1}|X_m,X'\otimes_{\aA_{m-1}} X_m)$$
of degree 1 such that the following diagram commutes.
\[
\begin{tikzcd}[column sep=110]
	C(X_1|\cdots|X_{m-1},X') \arrow[r, "\alpha_{X_1|\cdots|X_m,X'\otimes_{\aA_{m-1}} X_m}"] \arrow[d, "\eta_{X_1|\cdots|X_{m-1},X'}"'] & C(X_1|\cdots|X_{m-1}|X_m,X'\otimes_{\aA_{m-1}} X_m) \arrow[d, "\eta_{X_1|\cdots|X_m,X'\otimes_{\aA_{m-1}} X_m}"] \\
	D(X_1|\cdots|X_{m-1},X') \arrow[r, "\tilde{\alpha}_{X_1|\cdots|X_m,X'\otimes_{\aA_{m-1}} X_m}"] & D(X_1|\cdots|X_{m-1}|X_m,X'\otimes_{\aA_{m-1}} X_m) 
\end{tikzcd}
\]

For simplicity, we denote $\alpha_{X_1|\cdots|X_m,X_1\otimes_{\aA_1}\cdots\otimes_{\aA_{m-1}}X_m}$ and $\tilde{\alpha}_{X_1|\cdots|X_m,X_1\otimes_{\aA_1}\cdots\otimes_{\aA_{m-1}}X_m}$ by $\alpha_{X_1|\cdots|X_m}$ and $\tilde{\alpha}_{X_1|\cdots|X_m}$ respectively.

\medspace

\noindent{\bf Right actions $\beta_{X_1|\cdots|X_m,X_1\otimes_{\aA_1}X'}$ and $\tilde{\beta}_{X_1|\cdots|X_m,X_1\otimes_{\aA_1}X'}$.} 
Let $\aA_0,\cdots,\aA_m$ be small dg categories, $X_i$ a dg $\aA_{i-1}$-$\aA_i$-bimodule for all $1\le i\le m$, and $X'$ a dg $\aA_1$-$\aA_m$-bimodule. The morphism of complexes 
$$\beta_{X_1|\cdots|X_m,X_1\otimes_{\aA_1}X'}: C(X_2|\cdots|X_m,X') \rightarrow C(X_1|X_2|\cdots|X_m,X_1\otimes_{\aA_1}X')$$ 
is defined as follows: 
For any $\phi\in C^n(X_2|\cdots|X_m,X')$, its image $\beta_{X_1|\cdots|X_m,X_1\otimes_{\aA_1}X'}\phi\in C^n(X_1|\cdots |X_m,X_1\otimes_{\aA_1}X')$ has only nontrivial component 
$X_1(A_{1,l_1},A_{0,0})\otimes s\aA_1(A_{1,l_1-1},\linebreak A_{1,l_1})\otimes\cdots\otimes s\aA_1(A_{1,0},A_{1,1})\otimes\cdots\otimes X_m(A_{m,l_m},A_{m-1,0})\otimes s\aA_m(A_{m,l_m-1},A_{m,l_m}) \otimes \cdots\otimes s\aA_m(A_{m,0},A_{m,1}) 
\to X_1(A_{1,l_1},A_{0,0})\otimes X'(A_{m,0},A_{1,l_1}), 
x_1\otimes s(a_1)_{1,l_1}\otimes \cdots\otimes x_m \otimes s(a_m)_{1,l_m} 
\mapsto (-1)^{|\phi||x_1|} x_1\otimes\phi(s(a_1)_{1,l_1}\otimes \cdots\otimes x_m \otimes s(a_m)_{1,l_m})$.
It induces a closed morphism
\[
\tilde{\beta}_{X_1|\cdots|X_m,X_1\otimes_{\aA_1}X'}:D(X_2|\cdots|X_m,X') \rightarrow D(X_1|X_2|\cdots|X_m,X_1\otimes_{\aA_1}X')
\]
of degree 1 such that the following diagram commutes.
\[
\begin{tikzcd}[column sep=8pc]
	C(X_2|\cdots|X_m,X') \arrow[r, "\beta_{X_1|\cdots|X_m,X_1\otimes_{\aA_1}X'}"] \arrow[d, "\eta_{X_2|\cdots|X_m,X'}"'] & C(X_1|X_2|\cdots|X_m,X_1\otimes_{\aA_1}X') \arrow[d, "\eta_{X_1|\cdots|X_m,X_1\otimes_{\aA_1}X'}"] \\
	D(X_2|\cdots|X_m,X') \arrow[r, "\tilde{\beta}_{X_1|\cdots|X_m,X_1\otimes_{\aA_1}X'}"] & D(X_1|X_2|\cdots|X_m,X_1\otimes_{\aA_1}X') 
\end{tikzcd}
\]

For simplicity, we denote $\beta_{X_1|\cdots|X_m,X_1\otimes_{\aA_1}\cdots\otimes_{\aA_{m-1}}X_m}$ and $\tilde{\beta}_{X_1|\cdots|X_m,X_1\otimes_{\aA_1}\cdots\otimes_{\aA_{m-1}}X_m}$ by $\beta_{X_1|\cdots|X_m}$ and $\tilde{\beta}_{X_1|\cdots|X_m}$ respectively.

\medspace

\noindent{\bf Derived morphisms $\mu_{X_1|\cdots|X_m,X'}^i$ and $\tilde{\mu}_{X_1|\cdots|X_m,X'}^i$.} Let $\aA_0,\cdots,\aA_m$ be small dg categories, $X_i$ a dg $\aA_{i-1}$-$\aA_i$-bimodule for all $1\le i\le m$, and $X'$ a dg $\aA_0$-$\aA_m$-bimodule. For any $1\le i\le m-1$, the morphism of complexes 
$$\mu_{X_1|\cdots|X_m,X'}^i: C(X_1|\cdots|X_{i-1}|(X_{i}\otimes_{\aA_{i}}X_{i+1})|X_{i+2}|\cdots|X_m,X') \rightarrow C(X_1|\cdots|X_m,X')$$ 
is defined as follows: For any $\phi\in C(X_1|\cdots|X_{i-1}|(X_{i}\otimes_{\aA_{i}}X_{i+1})|X_{i+2}|
\cdots|X_m,X')$, its image $\mu_{X_1|\cdots|X_m,X'}^i\phi\in C(X_1|\cdots|X_m,X')$ has the only nontrivial component $s\aA_0(A_{0,l_0-1}, \linebreak A_{0,l_0})\otimes\cdots\otimes s\aA_0(A_{0,0},A_{0,1})\otimes X_1(A_{1,l_1}, A_{0,0})\otimes\cdots\otimes s\aA_{i-1}(A_{i-1,0}, A_{i-1,1})\otimes X_i(A_{i,0},A_{i-1,0})\otimes X_{i+1}(A_{i+1,l_{i+1}},A_{i,0})\otimes s\aA_{i+1}(A_{i+1,l_{i+1}-1}, A_{i+1,l_{i+1}})\otimes\cdots\otimes X_m(A_{m,l_m}, A_{m-1,0})\otimes\cdots\otimes \linebreak s\aA_m(A_{m,0},A_{m,1}) \to X'(A_{m,0},A_{0,l_0}), s(a_0)_{1,l_0}\otimes x_1\otimes\cdots\otimes x_i\otimes x_{i+1} \otimes\cdots\otimes s(a_m)_{1,l_m} \mapsto \phi(s(a_0)_{1,l_0}\otimes\cdots \otimes (x_{i}\otimes x_{i+1}) \otimes\cdots\otimes s(a_m)_{1,l_m})$.
It induces a closed morphism of degree 1
\[
\tilde{\mu}_{X_1|\cdots|X_m,X'}^i:D(X_1|\cdots|X_{i-1}|(X_{i}\otimes_{\aA_{i}}X_{i+1})|X_{i+2}|\cdots|X_m,X') \rightarrow D(X_1|\cdots|X_m,X')
\]
such that the following diagram commutes.
\[
\begin{tikzcd}[column sep=huge]
	C(X_1|\cdots|X_{i-1}|(X_{i}\otimes_{\aA_i}X_{i+1})|X_{i+2}|\cdots|X_m,X') \arrow[r,"\mu_{X_1|\cdots|X_m,X'}^i"] \arrow[d, "\eta_{X_1|\cdots|X_{i-1}|(X_{i}\otimes_{\aA_i}X_{i+1})|X_{i+2}|\cdots|X_m,X'}"'] & C(X_1|\cdots|X_m,X') \arrow[d,"\eta_{X_1|\cdots|X_m,X'}"] \\
	D(X_1|\cdots|X_{i-1}|(X_{i}\otimes_{\aA_{i}}X_{i+1})|X_{i+2}|\cdots|X_m,X') \arrow[r,"\tilde{\mu}_{X_1|\cdots|X_m,X'}^i"] & D(X_1|\cdots|X_m,X') 
\end{tikzcd}
\] 

For simplicity, we denote $\mu_{X_1|\cdots|X_m,X_1\otimes_{\aA_1}\cdots\otimes_{\aA_{m-1}}X_m}^i$ and $\tilde{\mu}_{X_1|\cdots|X_m,X_1\otimes_{\aA_1}\cdots\otimes_{\aA_{m-1}}X_m}^i$ by $\mu_{X_1|\cdots|X_m}^i$ and $\tilde{\mu}_{X_1|\cdots|X_m}^i$ respectively.

\medspace


\subsection{Hochschild cochain complexes of upper triangular matrix dg categories}
In this subsection, we will observe the Hochschild cochain complexes of upper triangular matrix dg categories which were introduced by Keller in \cite{Keller03} and will play crucial roles in our main results.

\medspace

\noindent{\bf Upper triangular matrix dg categories.} 
Let $\aA_0,\dots,\aA_m$ be small dg categories and $X_i$ a dg $\aA_{i-1}$-$\aA_i$-bimodule for all $1\le i\le m$. The {\it upper triangular matrix dg category} $\tT_{X_1,\cdots,X_m}$ is the small dg category whose object set $\Ob(\tT_{X_1,\cdots,X_m}):=\coprod_{i=0}^m\Ob(\aA_i)$, the disjoint union of object sets of $\aA_0,\dots,\aA_m$, and whose Hom space
\[
\tT_{X_1,\cdots,X_m}(A_i,A'_j):=
\begin{cases}
	\left(X_{j+1}\otimes_{\aA_{j+1}}X_{j+2}\otimes_{\aA_{j+2}}\cdots \otimes_{\aA_{i-1}}X_{i}\right)(A_i,A'_j), & i>j; \\
	\aA_i(A_i,A'_i), & i=j; \\
	0, & i<j.
\end{cases}
\]
for all $A_i\in\aA_i,A'_j\in\aA_j, 0\le i,j\le m$. It is formally denoted by the following upper triangular matrix.
$$\tT_{X_1,\cdots,X_m} =
\begin{pmatrix}
	\aA_0 & X_1   & X_1\otimes_{\aA_1} X_2 & \cdots & \cdots        & \cdots    & X_1\otimes_{\aA_1}\cdots\otimes_{\aA_{m-1}}X_m \\
	      & \aA_1 & X_2                    & \ddots &               &           & \vdots                                         \\
	      &       & \aA_2                  & \ddots & \ddots        &           & \vdots                                         \\
	      &       &                        & \ddots & \ddots        & \ddots    & \vdots                                         \\
	      &       &                        &        & \aA_{m-2}     & X_{m-1}   & X_{m-1}\otimes_{\aA_{m-1}} X_m                 \\
	      &       &                        &        &               & \aA_{m-1} & X_m                                            \\
	      &       &                        &        &               &           & \aA_m
\end{pmatrix}$$
The composition of two morphisms is given by matrix multiplication, and the identity morphisms coincide with those in $\aA_0,\cdots,\aA_m$.

For each $0\le i\le m$, there is a natural fully faithful functor 
$$\iota_i:\aA_i \rightarrow \tT_{X_1,\cdots,X_m}.$$ 

For any dg $\aA_{i-1}$-$\aA_i$-bimodule morphisms $f_i:X_i \rightarrow X'_i, 1\le i\le m$, there is a dg functor
\[
\tT_{f_1,\dots, f_m}:\tT_{X_1,\cdots,X_m} \rightarrow \tT_{X'_1,\cdots,X'_m}.
\]
Furthermore, if $X_i$ and $X'_i$ are cofibrant dg right $\aA_i$-modules for all $1\le i\le m-1$, or $X_i$ and $X'_i$ are cofibrant dg left $\aA_{i-1}$-modules for all $2\le i\le m$, in particular, $X_i$ and $X'_i$ are cofibrant dg $\aA_{i-1}$-$\aA_i$-bimodules for all $1\le i\le m$, and $f_i:X_i \rightarrow X'_i, 1\le i\le m$, are dg $\aA_{i-1}$-$\aA_i$-bimodule quasi-isomorphisms, then $\tT_{f_1,\dots, f_m}$ is a quasi-equivalence.

\medspace

\noindent{\bf Hochschild cochain complex of a $2\times 2$ upper triangular matrix dg category.} Let $\aA$ and $\bB$ be small dg categories, and $X$ a dg $\aA$-$\bB$-bimodule. The Hochschild cochain complex $C(\tT_X)$ of upper triangular matrix dg category $\tT_X=\begin{pmatrix}
	\aA&X\\&\bB
\end{pmatrix}$ is the graded vector space
\[
\begin{array}{rl}
	C(\tT_X) = & \prod\limits_{n\ge 0} C^n(\tT_X) \\ [3mm]
	= & \prod\limits_{n\ge 0} \prod\limits_{T_0,\cdots,T_n\in\tT_X}
	\mathcal{H}\mathrm{om}_k(s\tT_X(T_{n-1},T_n)\otimes\cdots\otimes s\tT_X(T_0,T_1),\tT_X(T_0,T_n)) \\ [3mm]
	= & \prod\limits_{n\ge 0} \prod\limits_{A_0,\cdots,A_n\in\aA}
	\mathcal{H}\mathrm{om}_k(s\aA(A_{n-1},A_n)\otimes\cdots\otimes s\aA(A_0,A_1), \aA(A_0,A_n)) \\ [3mm]
	& \times\prod\limits_{n\ge 0} \prod\limits_{B_0,\cdots,B_n\in\bB}
	\mathcal{H}\mathrm{om}_k(s\bB(B_{n-1},B_n)\otimes\cdots\otimes s\bB(B_0,B_1), \bB(B_0,B_n)) \\ [3mm]
	& \times\prod\limits_{n\ge 0} \prod\limits_{\substack{l+m=n,\ l,m\ge 0 \\ A_0,\cdots,A_l\in\aA \\ B_0,\cdots,B_m\in\bB}}
	\mathcal{H}\mathrm{om}_k(s\aA(A_{l-1},A_l)\otimes\cdots\otimes s\aA(A_0,A_1)\otimes sX(B_m,A_0)\otimes \\ [3mm]
	& \hspace{40mm} s\bB(B_{m-1},B_m)\otimes\cdots\otimes s\bB(B_0,B_1), X(B_0,A_n)) \\ [5mm]
	= & C(\aA)\oplus C(\bB)\oplus D(X)
\end{array}
\]
equipped with differential
\[
\delta_{C(\tT_X)} = 
\begin{pmatrix}
	\delta_{C(\aA)} & 0 & 0 \\
	0 & \delta_{C(\bB)} & 0 \\
	\tilde{\alpha}_X & -\tilde{\beta}_X & \delta_{D(X)}
\end{pmatrix}.
\]

\medspace

To study the Hochschild cochain complexes of upper triangular matrix dg categories further, we need the following lemma.

\begin{lemma}\label{Lem-LowerTriangularDifferential-Triangle}
	Let $V, V'$ and $V''$ be complexes of vector spaces such that $V=V'\oplus V''$ as graded vector spaces and the differential $d_V=\begin{pmatrix} d_{V'} & 0 \\ f & d_{V''} \end{pmatrix}$,
	where $f:V'\to V''$ is a graded linear map of degree 1. Then $f$ is a 1-cocycle in $\hHom_k(V',V'')$. Moreover, there is a triangle in $\dD k$
	\[
	V'' \xrightarrow{0 \choose 1_{V''}} V \xrightarrow{(1_{V'}\ 0)} V' \xrightarrow{-s\circ f} sV''.
	\]
	
\end{lemma}

\begin{proof}
	Due to $d_V^2=0$, we have $d_{V''}f+fd_{V'}=0$. Thus the first conclusion holds. The other one follows from the isomorphism $\cone(-s\circ f)\cong sV$ in $\cC k$.
\end{proof}

From Lemma \ref{Lem-LowerTriangularDifferential-Triangle}, we obtain the following lemma.

\begin{lemma} \label{Lem-Homotopy-Bicart} {\rm (Keller \cite[Section 4.5]{Keller03})}
	Let $\aA$ and $\bB$ be small dg categories, and $X$ a dg $\aA$-$\bB$-bimodule. Then there is a homotopy bicartesian diagram
	\[
	\begin{tikzcd}
		C(\tT_X) \arrow[r, "\iota_1^*"] \arrow[d, "\iota_2^*"'] & C(\aA) \arrow[d, "\alpha_X"] \\
		C(\bB) \arrow[r, "\beta_X"] & C(X),
	\end{tikzcd}
	\]
	i.e., a triangle $C(\tT_X) \xrightarrow{\iota^*_1 \choose \iota^*_2} C(\aA)\oplus C(\bB) \xrightarrow{(\alpha_X,-\beta_X)} C(X) \to$ in $\dD k$, where $\iota_1^*:C(\tT_X)\to C(\aA)$ and $\iota_2^*:C(\tT_X)\to C(\bB)$ are the $B_\infty$-morphisms induced by the fully faithful dg functors $\iota_1:\aA\hookrightarrow \tT_X$ and $\iota_2:\bB\hookrightarrow \tT_X$.
\end{lemma}

Let $\aA$ be a small dg category. Then there are two different fully faithful dg functors
\[
\iota_1,\iota_2:\aA \hookrightarrow \tT_{I_\aA}=\begin{pmatrix}
	\aA&I_\aA\\ &\aA
\end{pmatrix},
\]
which induce two $B_\infty$-morphisms $\iota_1^*:C(\tT_{I_\aA})\to C(\aA)$ and $\iota_2^*:C(\tT_{I_\aA})\to C(\aA)$. We have known that $\alpha_{I_\aA}:C(\aA)\to C(I_\aA)$ and $\beta_{I_{\aA}}:C(\aA)\to C(I_\aA)$ are quasi-isomorphisms. It follows from Lemma \ref{Lem-Homotopy-Bicart} that $\iota_1^*$ and $\iota_2^*$ are quasi-isomorphisms as well. Furthermore, $\iota_1^*=\iota_2^*$ in $\Ho B_{\infty}$ (Ref. \cite[Theorem c)]{Keller06}), which also can be obtained by the Lemma \ref{Lem-ThetaA} below.

For any small dg category $\aA$, there is an isomorphism of complexes 
$$\kappa_\aA:C(\aA) \rightarrow D(I_\aA)$$ 
given by
$(\kappa_\aA\phi)(sa_{1,l}\otimes sx\otimes sb_{1,m}) := \phi(sa_{1,l}\otimes sx\otimes sb_{1,m}).$

The following result is a preparation for the construction of a lax functor in Theorem \ref{Thm-LaxFuntor-HomotopyCat-B-Inf}.

\begin{lemma} \label{Lem-ThetaA}
	Let $\aA$ be a small dg category. Then the graded linear map
	$$\theta_\aA=(1_{C(\aA)}, 1_{C(\aA)}, \kappa_\aA)^T : C(\aA) \rightarrow C(\tT_{ I_\aA})=C(\aA)\oplus C(\aA)\oplus D(I_\aA)$$
	is a $B_\infty$-quasi-isomorphism. Moreover, the $B_\infty$-morphisms $\iota_1^*,\iota_2^*: C(\tT_{ I_\aA})\to C(\aA)$ induced by the fully faithful dg functors $\iota_1,\iota_2: \aA\hookrightarrow\tT_{ I_\aA}$ are $B_\infty$-quasi-isomorphisms in $B_\infty$ and equal in ${\rm Ho} B_\infty$.
\end{lemma}

\begin{proof}
	See Appendix \ref{Appendix-ThetaA}.
\end{proof}

\medspace

\noindent{\bf Hochschild cochain complex of a $3\times 3$ upper triangular matrix dg category.} 
Let $\aA,\bB$ and $\cC$ be small dg categories, $X$ a dg $\aA$-$\bB$-bimodule, and $Y$ a dg $\bB$-$\cC$-bimodule. The Hochschild cochain complex $C(\tT_{X,Y})$ of upper triangular matrix dg category $\tT_{X,Y}=\begin{pmatrix}
	\aA&X&X\otimes_\bB Y\\&\bB&Y\\ &&\cC
\end{pmatrix}$ is the graded vector space
\[
\begin{array}{l}
	\quad\ C(\tT_{X,Y}) \\ [2mm]
	= \prod\limits_{n\ge 0} C^n(\tT_{X,Y}) \\ [2mm]
	=\prod\limits_{n\ge 0}\ \prod\limits_{T_0,\cdots,T_n\in\tT_{X,Y}}
	\mathcal{H}\mathrm{om}_k(s\tT_{X,Y}(T_{n-1},T_n)\otimes\cdots\otimes s\tT_{X,Y}(T_0,T_1),\tT_{X,Y}(T_0,T_n)) \\ [5mm]
	=\prod\limits_{n\ge 0}(\prod\limits_{A_0,\cdots,A_n\in\aA}
	\mathcal{H}\mathrm{om}_k(s\aA(A_{n-1},A_n)\otimes\cdots\otimes s\aA(A_0,A_1),\aA(A_0,A_n))\\ [5mm]
	\hspace{20pt}\times\prod\limits_{B_0,\cdots,B_n\in\bB}
	\mathcal{H}\mathrm{om}_k(s\bB(B_{n-1},B_n)\otimes\cdots\otimes s\bB(B_0,B_1),\bB(B_0,B_n)) \\ [5mm]
	\hspace{20pt}\times\prod\limits_{C_0,\cdots,C_n\in\cC}
	\mathcal{H}\mathrm{om}_k(s\cC(C_{n-1},C_n)\otimes\cdots\otimes s\cC(C_0,C_1),\cC(C_0,C_n)) \\ [5mm]
	\hspace{20pt}\times\prod\limits_{\substack{l+m=n,\ l,m\ge 0, \\ A_0,\cdots,A_l\in\aA, \\ B_0,\cdots,B_m\in\bB}} \mathcal{H}\mathrm{om}_k(s\aA(A_{l-1},A_l)\otimes\cdots\otimes s\aA(A_0,A_1)\otimes sX(B_m,A_0) \\ 
	\hspace{125pt} \otimes s\bB(B_{m-1},B_m)\otimes\cdots\otimes s\bB(B_0,B_1),X(B_0,A_l)) \\ [6mm]
	\hspace{20pt}\times\prod\limits_{\substack{l+m=n,\ l,m\ge 0, \\ B_0,\cdots,B_l\in\bB, \\ C_0,\cdots,C_m\in\cC}}
	\mathcal{H}\mathrm{om}_k(s\bB(B_{l-1},B_l)\otimes\cdots\otimes s\bB(B_0,B_1)\otimes sY(C_m,B_0) \\ [6mm] 
	\hspace{125pt} \otimes s\cC(C_{m-1},C_m)\otimes\cdots\otimes s\cC(C_0,C_1),Y(C_0,B_l)) \\ [6mm]
	\hspace{20pt} \times\prod\limits_{\substack{l+m=n,\ l,m\ge 0, \\ A_0,\cdots,A_l\in\aA, \\ C_0,\cdots,C_m\in\cC}} \mathcal{H}\mathrm{om}_k(s\aA(A_{l-1},A_l)\otimes\cdots\otimes s\aA(A_0,A_1)\otimes s(X\otimes_\bB Y)(C_m,A_0) \\ 
	\hspace{125pt} \otimes s\cC(C_{m-1},C_m)\otimes\cdots\otimes s\cC(C_0,C_1),
	(X\otimes_\bB Y)(C_0,A_l)) \\ [4mm]
	\hspace{20pt} \times\prod\limits_{\substack{l+m+p=n,\ l,m,p\ge 0, \\ A_0,\cdots,A_l\in\aA, \\ B_0,\cdots,B_m\in\bB, \\ C_0,\cdots,C_p\in\cC}}
	\mathcal{H}\mathrm{om}_k(s\aA(A_{l-1},A_l)\otimes\cdots\otimes s\aA(A_0,A_1)\otimes sX(B_m,A_0) \\ \hspace{125pt} 
	\otimes s\bB(B_{m-1},B_m)\otimes\cdots\otimes s\bB(B_0,B_1) \otimes sY(C_p,B_0) \\ [4mm]
	\hspace{125pt} \otimes s\cC(C_{p-1},C_p)\otimes\cdots\otimes s\cC(C_0,C_1),(X\otimes_\bB Y)(C_0,A_l))) \\ [4mm]
	= C(\aA)\oplus C(\bB)\oplus C(\cC)\oplus D(X) \oplus D(Y)\oplus D(X\otimes_\bB Y)\oplus D(X|Y)
\end{array}
\]
equipped with differential
\[
\delta_{C(\tT_{X,Y})} =
\begin{pmatrix}
	\delta_{C(\aA)} & & & & & & \\
	0 & \delta_{C(\bB)} & & & & & \\
	0 & 0 & \delta_{C(\cC)} & & & & \\
	\tilde{\alpha}_X & -\tilde{\beta}_X & 0 & \delta_{D(X)} & & & \\
	0 & \tilde{\alpha}_Y & -\tilde{\beta}_Y & 0 & \delta_{D(Y)} & & \\
	\tilde{\alpha}_{X\otimes_{\bB}Y} & 0 & -\tilde{\beta}_{X\otimes_\bB Y} & 0 & 0 & \delta_{D(X\otimes_{\bB} Y)} & \\
	0 & 0 & 0 & -\tilde{\alpha}_{X|Y} & -\tilde{\beta}_{X|Y} & \tilde{\mu}_{X|Y}^1 & \delta_{D(X|Y)}  
\end{pmatrix}.
\]

Denote by $\ac(X|Y)$ the complex whose underlying graded vector space is $D(X\otimes_\bB Y)\oplus D(X|Y)$, and whose differential is
$\begin{pmatrix}
	\delta_{D(X\otimes_\bB Y)} & 0 \\
	\tilde{\mu}_{X|Y}^1 & \delta_{D(X|Y)}
\end{pmatrix}.$
By Lemma \ref{Lem-LowerTriangularDifferential-Triangle}, we have a triangle $D(X|Y) \rightarrow \ac(X|Y) \rightarrow D(X\otimes_\bB Y) \xrightarrow{-s\circ\tilde{\mu}_{X|Y}^1} sD(X|Y)$ in $\dD k$.

Observe the $3\times 3$ upper triangular matrix dg category $\tT_{X,Y}=\begin{pmatrix}
	\aA&X&X\otimes_\bB Y\\&\bB&Y\\ &&\cC
\end{pmatrix}$. Denote by $(X\ X\otimes_\bB Y)$ the dg $\aA$-$\tT_Y$-bimodule given as follows:
$$\begin{array}{c}
\begin{array}{ll}
	(X\ X\otimes_\bB Y)(B,A)\coloneqq X(B,A), & (X\ X\otimes_\bB Y)(C,A)\coloneqq \left(X\otimes_\bB Y\right)(C,A), \\
	(X\ X\otimes_\bB Y)(b\otimes a)\coloneqq X(b,a), & (X\ X\otimes_\bB Y)(c\otimes a)\coloneqq \left(X\otimes_\bB Y\right)(c,a), \\
	(X\ X\otimes_\bB Y)(y\otimes a)\coloneqq (-\otimes y)\circ X(B,a)& 
\end{array}
\\ [10mm]
\begin{tikzcd}
	X(B,A) \arrow[d,"{X(B,a)}"] \arrow[r,"{-\otimes y}"] & (X\otimes_\bB Y)(C,A) \arrow[d,"{(X\otimes_\bB Y)(C,a)}"] \\
	X(B,A') \arrow[r,"{-\otimes y}"] & (X\otimes_\bB Y)(C,A')
\end{tikzcd}
\end{array}$$
for all $A\in\aA, B\in\bB, C\in\cC$ and $a\in\aA(A,A'), b\in\bB(B,B'), c\in\cC(C,C'), y\in Y(C,B)$.
It is not difficult to see that the underlying graded vector space of $D((X\ X\otimes_\bB Y))$ is
$D(X)\oplus D(X\otimes_\bB Y)\oplus D(X|Y),$
and the differential of $D((X\ X\otimes_\bB Y))$ is
$$\delta_{D((X\ X\otimes_\bB Y))} =
\begin{pmatrix}
	\delta_{D(X)} & & \\
	0 & \delta_{D(X\otimes_\bB Y)} & \\
	-\tilde{\alpha}_{X|Y} & \tilde{\mu}^1_{X|Y} & \delta_{D(X|Y)}
\end{pmatrix}.$$

By Lemma \ref{Lem-LowerTriangularDifferential-Triangle}, we have a triangle in $\dD k$
$$\ac(X|Y) \rightarrow D((X\ X\otimes_\bB Y)) \xrightarrow{p} D(X) \rightarrow$$
where $p=(1_{D(X)},0)$ is the natural projection.

The following lemma implies that the notation $\ac(X|Y)$ is reasonable.

\begin{lemma}\label{lem:3.6}
	Let $\aA,\bB$ and $\cC$ be small dg categories, $X$ a dg $\aA$-$\bB$-bimodule, and $Y$ a dg $\bB$-$\cC$-bimodule. If $X$ or $Y$ is cofibrant as dg $\bB$-module, then $\ac(X|Y)$ is acyclic.
\end{lemma}

\begin{proof}
	Observe the $2\times 2$ upper triangular matrix dg category $\tT_Y=\begin{pmatrix}
		\bB&Y\\ &\cC
	\end{pmatrix}$. Denote by $(\bB\ Y)$ the dg $\bB$-$\tT_Y$-bimodule given as follows:
	$$\begin{array}{c}
		\begin{array}{ll}
			(\bB\ Y)(B,B')\coloneqq \bB(B,B'), & (\bB\ Y)(C,B)\coloneqq Y(C,B), \\
			(\bB\ Y)(b\otimes b')\coloneqq \bB(b,b'), & (\bB\ Y)(c\otimes b)\coloneqq Y(c,b), \\
			(\bB\ Y)(y\otimes b)\coloneqq Y(C,-)(y)\circ \bB(B'',b) & 
		\end{array}
		\\ [10mm]
		\begin{tikzcd}[column sep=50]
			\bB(B'',B) \arrow[d,"{\bB(B'',b)}"] \arrow[r,"{Y(C,-)(y)}"] & Y(C,B) \arrow[d,"{Y(C,b)}"] \\
			\bB(B'',B') \arrow[r,"{Y(C,-)(y)}"] & Y(C,B')
		\end{tikzcd}
	\end{array}$$
	for all $B,B',B'',B'''\in\bB, C\in\cC$ and $b\in\bB(B,B'),b'\in\bB(B'',B'''), c\in\cC(C,C'), y\in Y(C,B'')$.
	It is clear that $(\bB\ Y)\cong \tT_Y$ as dg $\bB$-$\tT_Y$-bimodules, $(\bB\ Y)\cong I_\bB$ as dg $\bB$-bimodules, $(X\ X\otimes_\bB Y) \cong X\otimes_\bB(\bB\ Y)$ as $\aA$-$\tT_Y$-bimodules, and $(X\ X\otimes_\bB Y) \cong X$ as $\aA$-$\bB$-bimodules.
	By the assumption that $X$ or $Y$ is cofibrant as dg $\bB$-module, we have the following isomorphisms in $\dD k$:
	\[
	\begin{aligned}
		C((X\ X\otimes_\bB Y)) & \cong \RHom_{\aA^\op\otimes\tT_Y}((X\  X\otimes_\bB Y),(X\ X\otimes_\bB Y)) \\
		& \cong \RHom_{\aA^\op\otimes\tT_Y}(X\otimes_\bB(\bB\ Y),(X\ X\otimes_\bB Y)) \\
		& \cong \RHom_{\aA^\op\otimes\tT_Y}(X\otimes_\bB^{\bL}(\bB\ Y),(X\ X\otimes_\bB Y)) \\
		& \cong \RHom_{\aA^\op\otimes\bB}(X,\RHom_{\tT_Y}((\bB\ Y),(X\ X\otimes_\bB Y))) \\
		& \cong \RHom_{\aA^\op\otimes\bB}(X,\RHom_{\tT_Y}(\tT_Y,(X\ X\otimes_\bB Y))) \\
		& \cong \RHom_{\aA^\op\otimes\bB}(X,(X\ X\otimes_\bB Y)) \\
		& \cong \RHom_{\aA^\op\otimes\bB}(X,X) \\
		& \cong C(X),
	\end{aligned}
	\]
	which is just the natural projection $p$. Moreover, we have the following commutative diagram.
	\[
	\begin{tikzcd}
		C((X\ X\otimes_\bB Y)) \arrow[r, "p"] \arrow[d, "\eta_{(X\ X\otimes_\bB Y)}"'] & C(X) \arrow[d, "\eta_X"] \\
		D((X\ X\otimes_\bB Y)) \arrow[r, "p"] & D(X)
	\end{tikzcd}
	\]
	Hence, the projection $p:D((X\ X\otimes_\bB Y)) \rightarrow D(X)$ is a quasi-isomorphism. Therefore, $\ac(X|Y)$ is acyclic.
\end{proof}

\medspace

By Lemma \ref{Lem-LowerTriangularDifferential-Triangle}, we have the following triangle in $\dD k$ 
\[
\ac(X|Y) \rightarrow C(\tT_{X,Y}) \rightarrow \bar{C}(\tT_{X,Y}) \rightarrow,
\]
where $\bar{C}(\tT_{X,Y})$ is the graded vector space
\[
C(\aA)\oplus C(\bB)\oplus C(\cC)\oplus D(X) \oplus D(Y)
\]
equipped with differential
\[
\delta_{\bar{C}(\tT_{X,Y})} =
\begin{pmatrix}
	\delta_{C(\aA)} & & & &  \\
	0 & \delta_{C(\bB)} & & &  \\
	0 & 0 & \delta_{C(\cC)} & & \\
	\tilde{\alpha}_X & -\tilde{\beta}_X & 0 & \delta_{D(X)} & \\
	0 & \tilde{\alpha}_Y & -\tilde{\beta}_Y & 0 & \delta_{D(Y)}
\end{pmatrix}.
\]

The following result is also a preparation for the construction of a lax functor in Theorem \ref{Thm-LaxFuntor-HomotopyCat-B-Inf}.

\begin{lemma}\label{Lem-CTXY-bar}
	Let $\aA,\bB$ and $\cC$ be small dg categories, $X$ a dg $\aA$-$\bB$-bimodule, and $Y$ a dg $\bB$-$\cC$-bimodule. If $X$ or $Y$ is cofibrant over $\bB$, then the natural projection
	\[
	C(\tT_{X,Y}) \rightarrow \bar{C}(\tT_{X,Y})
	\]
	is a $B_\infty$-quasi-isomorphism. Moreover, there is a pullback in the category $B_\infty$ of $B_\infty$-algebras
	\[
	\begin{tikzcd}
		\bar{C}(\tT_{X,Y}) \arrow[r, "\iota_{12}^*"] \arrow[d, "\iota_{23}^*"'] & C(\tT_X) \arrow[d,"\iota_2^*"] \\
		C(\tT_Y) \arrow[r, "\iota_1^*"] & C(\bB),
	\end{tikzcd}
	\]
	where the $B_\infty$-morphisms $\iota_{12}^*:\bar{C}(\tT_{X,Y})\to C(\tT_X)$ and $\iota_{12}^*:\bar{C}(\tT_{X,Y})\to C(\tT_X)$ are induced by the $B_\infty$-morphisms $\iota_{12}^*:C(\tT_{X,Y})\to C(\tT_X)$ and $\iota_{12}^*:C(\tT_{X,Y})\to C(\tT_X)$ which are induced by the fully faithful dg functors $\iota_{12}: \tT_X \hookrightarrow \tT_{X,Y}$ and $\iota_{23}: \tT_Y \hookrightarrow \tT_{X,Y}$.
\end{lemma}

\begin{proof}
	Due to Lemma \ref{lem:3.6}, the complex $\ac(X|Y)$ is acyclic, which shows that the natural projection
	$C(\tT_{X,Y}) \rightarrow \bar{C}(\tT_{X,Y})$ is a quasi-isomorphism in $\cC k$.
	It is clear that the $B_\infty$-structure on $C(\tT_{X,Y})$ induces a $B_\infty$-structure on $\bar{C}(\tT_{X,Y})$, 
	and the natural projection $C(\tT_{X,Y}) \rightarrow \bar{C}(\tT_{X,Y})$ is a $B_\infty$-quasi-isomorphism.
	The last statement holds since 
	$$0 \to \bar{C}(\tT_{X,Y}) \xrightarrow{\iota_{12}^* \choose \iota_{23}^*} C(\tT_X)\oplus C(\tT_Y) \xrightarrow{(\iota_2^*,\iota_1^*)} C(\bB)\to 0$$
	is a short exact sequence in $\cC k$.
\end{proof}




\medspace

Let $\aA$ be a small dg category. Then we have an isomorphism of complexes $\kappa_\aA:C(\aA) \rightarrow D(I_\aA)$. 
More generally, there is a quasi-isomorphism of complexes
\[
\kappa_\aA^m:C(\aA) \rightarrow D(\underbrace{I_\aA|\cdots|I_\aA}_{m\text{ times}})
\]
given by $\kappa_\aA^m\phi(s(a_0)_{1,l_0}\otimes sa'_1\otimes s(a_1)_{1,l_1}\otimes \cdots\otimes sa'_m\otimes s(a_m)_{1,l_m}):=\phi(s(a_0)_{1,l_0}\otimes sa'_1\otimes s(a_1)_{1,l_1}\otimes \cdots\otimes sa'_m\otimes s(a_m)_{1,l_m})$ for all $m\ge 1$.

The following lemma is another preparation for the construction of a lax functor in Theorem \ref{Thm-LaxFuntor-HomotopyCat-B-Inf}.

\begin{lemma}\label{Lem-ThetaA^2}
	Let $\aA$ be a small dg category. Then the graded linear map
	\[
	\theta_\aA^2 := (1_{C(\aA)}, 1_{C(\aA)}, 1_{C(\aA)}, \kappa_\aA, \kappa_\aA, \kappa_\aA, \kappa_\aA^2)^{\mT} : C(\aA) \rightarrow C(\tT_{ I_\aA,I_\aA})
	\]
	is a $B_\infty$-quasi-isomorphism. Moreover, the $B_\infty$-morphisms $\iota_1^*,\iota_2^*,\iota_3^*: C(\tT_{I_\aA,I_\aA})\to C(\aA)$ induced by the fully faithful dg functors $\iota_1,\iota_2,\iota_3: \aA\hookrightarrow\tT_{I_\aA,I_\aA}$ are $B_\infty$-quasi-isomorphisms in $B_\infty$ and equal in ${\rm Ho} B_\infty$.
\end{lemma}

\begin{proof}
	See Appendix \ref{Appendix-ThetaA^2}.
\end{proof}

\medspace

\noindent{\bf Hochschild cochain complex of a $4\times 4$ upper triangular matrix dg category.} 
Let $\aA,\bB,\cC$ and $\dD$ be small dg categories, $X$ a dg $\aA$-$\bB$-bimodule, $Y$ a dg $\bB$-$\cC$-bimodule, and $Z$ a dg $\cC$-$\dD$-bimodules. The Hochschild cochain complex $C(\tT_{X,Y,Z})$ of upper triangular matrix dg category $\tT_{X,Y,Z} = \begin{pmatrix}
	\aA&X&X\otimes_\bB Y&X\otimes_\bB Y\otimes_\cC Z\\ &\bB&Y&Y\otimes_\cC Z\\ &&\cC&Z\\ &&&\dD
\end{pmatrix}$ is the graded vector space
	$C(\aA)\oplus C(\bB)\oplus C(\cC)\oplus C(\dD)\oplus D(X) \oplus D(Y)\oplus D(Z)\oplus D(X\otimes_\bB Y)\oplus D(Y\otimes_\cC Z)\oplus 
	D(X\otimes_\aA Y\otimes_\bB Z)\oplus D(X|Y) \oplus D(Y|Z)\oplus D(X|(Y\otimes_\bB Z)) \oplus D((X\otimes_\bB Y)|Z) \oplus D(X|Y|Z)$
equipped with differential $\delta_{C(\tT_{X,Y,Z})}$ similar to $\delta_{C(\tT_{X,Y})}$.

\section{Constructions of bicategories}

In this section, we will construct the source and target bicategories of our objective (co)lax functors. 

\subsection{Bicategories and (co)lax functors}
Bicategory was introduced by B\'{e}nabou \cite{Benabou67}. It is a multi-object version of monoidal category.
The lax functors between bicategories are generalization of the monoidal functors between monoidal categories. We refer \cite{JohnsonYau21} for the knowledge of bicategories and (co)lax functors.

\medspace

\noindent{\bf Bicategory.} A {\it bicategory} $\mathsf{B}$ is a sextuple $(\mathsf{B},1,c,a,l,r)$ consisting of the following data:

(1) {\it Class of objects:} A class $\Ob(\mathsf{B})$, whose elements are called {\it objects} or {\it 0-cells}. If $X\in \Ob(\mathsf{B})$, we also write $X\in \mathsf{B}$.

(2) {\it Hom categories:} For each pair of objects $X,Y\in \mathsf{B}$, a category $\mathsf{B}(X,Y)$, called {\it Hom category}, whose objects $f:X\to Y$ are called {\it 1-cells}, whose morphisms $\alpha:f\to f'$ are called {\it 2-cells}, whose compositions $\alpha'\alpha$ are called {\it vertical compositions}, and whose identity morphisms $1_f$ are called {\it identity 2-cells}. Vertical compositions satisfy associativity $(\alpha''\alpha')\alpha=\alpha''(\alpha'\alpha)$ and unity $\alpha 1_f=\alpha=1_{f'}\alpha$ for all 2-cells $\alpha:f\to f', \alpha':f'\to f'', \alpha'':f''\to f'''$.

(3) {\it Identity 1-cells:} For each object $X\in \mathsf{B}$, a functor $1_X: {\bf 1} \to \mathsf{B}(X,X)$, where ${\bf 1}$ is the category with one object $*$ and one morphism $1_*$. We identify the functor $1_X$ with the 1-cell $1_X(*)\in \mathsf{B}(X,X)$, called the {\it identity 1-cell} of $X$, whose identity 2-cell is $1_{1_X}$.

(4) {\it Horizontal compositions:} For each triple of objects $X,Y, Z\in \mathsf{B}$, a functor $c_{XYZ}: \mathsf{B}(Y,Z) \times \mathsf{B}(X,Y) \to \mathsf{B}(X,Z)$, called {\it horizontal composition}. For all 1-cells $f,f'\in \mathsf{B}(X,Y)$ and $g,g'\in \mathsf{B}(Y,Z)$, and 2-cells $\alpha:f\to f'$ and $\beta:g\to g'$, we denote $c_{XYZ}(g,f)$ by $g\circ f$ or $gf$, and denote $c_{XYZ}(\beta,\alpha)$ by $\beta * \alpha$. The functor $c_{XYZ}$ preserves vertical compositions, i.e., $(\beta'\beta)*(\alpha'\alpha)=(\beta'*\alpha')(\beta*\alpha)$, called {\it middle four exchange}, and identity 2-cells, i.e., $1_g*1_f=1_{gf}$, for all 1-cells $f,f',f''\in\mathsf{B}(X,Y), g,g',g''\in\mathsf{B}(Y,Z)$ and 2-cells $\alpha:f\to f', \alpha':f'\to f'',\beta:g\to g', \beta':g'\to g''$.

(5) {\it Associators:} For each quadruple of objects $W,X,Y, Z\in \mathsf{B}$, a natural isomorphism $a_{WXYZ}: c_{WXZ}(c_{XYZ}\times \Id_{\mathsf{B}(W,X)}) \to c_{WYZ}(\Id_{\mathsf{B}(Y,Z)}\times c_{WXY})$ between two functors from $\mathsf{B}(Y,Z)\times \mathsf{B}(X,Y)\times \mathsf{B}(W,X)$ to $\mathsf{B}(W,Z)$, called {\it associator}. Its component 2-cell is isomorphism $a_{h,g,f}:(hg)f\to h(gf)$. The naturality means $a_{h',g',f'}((\gamma*\beta)*\alpha)=(\gamma*(\beta*\alpha))a_{h,g,f}$ for all 1-cells $f,f'\in\mathsf{B}(W,X), g,g'\in\mathsf{B}(X,Y), h,h'\in\mathsf{B}(Y,Z)$ and 2-cells $\alpha:f\to f', \beta:g\to g', \gamma:h\to h'$.

(6) {\it Unitors.} For each pair of objects $X,Y\in\mathsf{B}$, two natural isomorphisms $l_{XY}: c_{XYY}(1_Y\times \Id_{\mathsf{B}(X,Y)}) \to \Id_{\mathsf{B}(X,Y)}$ and $r_{XY}: c_{XXY}(\Id_{\mathsf{B}(X,Y)}\times 1_X) \to \Id_{\mathsf{B}(X,Y)}$, called {\it left unitor} and {\it right unitor}, respectively.
Their component 2-cells are isomorphisms $l_f: 1_Yf\to f$ and $r_f:f1_X\to f$. The naturality means $l_{f'}(1_{1_Y}*\alpha)=\alpha l_f$ and $r_{f'}(\alpha*1_{1_X})=\alpha r_f$ for all 1-cells $f,f'\in\mathsf{B}(X,Y)$ and 2-cell $\alpha:f\to f'$.

The above data satisfy the following two axioms:

{\it Unity Axiom:} For each pair of 1-cells $f\in \mathsf{B}(X,Y), g\in \mathsf{B}(Y,Z)$, the {\it middle unity diagram}
\[
  \begin{tikzcd}
	(g1_Y)f \arrow[rr,"a_{g,1_Y,f}"] \arrow[dr,"r_g * 1_f" '] & & g(1_Yf) \arrow[dl,"1_g * l_f"] \\
	& gf &
  \end{tikzcd}
\]
in $\mathsf{B}(X,Z)$ is commutative.

{\it Pentagon Axiom:} For each quadruple of 1-cells $f\in \mathsf{B}(V,W), g\in \mathsf{B}(W,X), h\in \mathsf{B}(X,Y), k\in \mathsf{B}(Y,Z)$, the following diagram in $\mathsf{B}(V,Z)$ is commutative.
\[
  \begin{tikzcd}
	&&(kh)(gf)\arrow[drr,"a_{k,h,gf}"] &&\\
	((kh)g)f \arrow[urr,"a_{kh,g,f}"] \arrow[dr,"a_{k,h,g}*1_f"] &&&& k(h(gf)) \\
	&(k(hg))f \arrow[rr,"a_{k,hg,f}"] &&k((hg)f) \arrow[ur,"1_k*a_{h,g,f}"] &
  \end{tikzcd}
\]

\medspace

\noindent{\bf Sub-bicategory.} A bicategory $(\mathsf{B}',1',c',a',l',r')$ is called a {\it sub-bicategory} of a bicategory $(\mathsf{B},1,c,a,l,r)$ if the following statements hold:

(1) {\it Class of objects:} $\Ob(\mathsf{B}')$ is a sub-class of $\Ob(\mathsf{B})$.

(2) {\it Hom categories:} For each pair of objects $X,Y\in \mathsf{B}'$, $\mathsf{B}'(X,Y)$ is a sub-category of $\mathsf{B}(X,Y)$.

(3) {\it Identity 1-cells:} The identity 1-cell $1'_X$ of $X$ in $\mathsf{B}'$ is equal to the identity 1-cell $1_X$ of $X$ in $\mathsf{B}$.

(4) {\it Horizontal compositions:} For each triple of objects $X,Y,Z\in \mathsf{B}'$, the following diagram is commutative.
\[
  \begin{tikzcd}
	\mathsf{B}'(Y,Z)\times \mathsf{B}'(X,Y) \ar[r,"c'_{XYZ}"] \ar[d] & \mathsf{B}'(X,Z) \ar[d] \\
	\mathsf{B}(Y,Z)\times \mathsf{B}(X,Y) \ar[r,"c_{XYZ}"] & \mathsf{B}(X,Z)
  \end{tikzcd}
\]

(5) {\it Associators and unitors:} Every component 2-cell of the associator (resp. the left unitor, the right unitor) in $\mathsf{B}'$ is equal to the corresponding component 2-cell of the associator (resp. the left unitor, the right unitor) in $\mathsf{B}$.

\medspace

\noindent{\bf Lax functor.} A {\it lax functor} $F$ from a bicategory $(\mathsf{B},1,c,a,l,r)$ to a bicategory $(\mathsf{B}',1',c', \\ a',l',r')$ is a triple $(F,F^2,F^0)$ consisting of the following data:

(1) {\it Function on classes of objects:} A function $F: \Ob(\mathsf{B}) \to \Ob(\mathsf{B}'), X\mapsto FX$.

(2) {\it Functors on Hom categories:} For each pair of objects $X,Y\in\mathsf{B}$, a functor $F=F_{XY}:\mathsf{B}(X,Y)\to\mathsf{B}'(FX,FY), f\mapsto Ff, \alpha\mapsto F\alpha$, which preserves vertical compositions, i.e., $F(\alpha'\alpha)=F(\alpha')F(\alpha)$, and preserves identity 2-cells, i.e., $F(1_f)=1_{F(f)}$, for all 1-cells $f,f',f''\in\mathsf{B}(X,Y)$ and 2-cells $\alpha:f\to f',\alpha':f'\to f''$.

(3) {\it Lax functoriality constraints:} For each triple of objects $X,Y,Z\in\mathsf{B}$, a natural transformation $F^2=F^2_{XYZ}:c'_{FX,FY,FZ}(F_{YZ}\times F_{XY}) \Rightarrow F_{XZ}c_{XYZ}$ with component 2-cell $F^2_{g,f}: Fg\circ Ff \to F(gf)$, called {\it lax functoriality constraint}.
\[
\begin{tikzcd}[row sep=large,column sep=large]
	\mathsf{B}(Y,Z) \times \mathsf{B}(X,Y) \ar[r,"c_{XYZ}"] \ar[d,"F_{YZ}\times F_{XY}"'] & \mathsf{B}(X,Z) \ar[d,"F_{XZ}"] \\
	\mathsf{B}'(FY,FZ) \times \mathsf{B}'(FX,FY) \ar[r,"c'_{FX,FY,FZ}"] \ar[ur,shorten=1cm, Rightarrow,"F^2_{XYZ}"]& \mathsf{B}'(FX,FZ) 
\end{tikzcd}
\]
The naturality of $F^2_{XYZ}$ means $F(\beta*\alpha)\circ F^2_{g,f}=F^2_{g',f'}\circ (F\beta*F\alpha)$, i.e., the following diagram is commutative 
\[
\begin{tikzcd}
	Fg\circ Ff \ar[r,"F^2_{g,f}"] \ar[d,"F\beta * F\alpha"'] & 
	F(gf) \ar[d,"F(\beta * \alpha)"] \\
	Fg'\circ Ff' \ar[r,"F^2_{g',f'}"] & F(g'f')
\end{tikzcd}
\]
for all 1-cells $f,f'\in\mathsf{B}(X,Y), g,g'\in\mathsf{B}(Y,Z)$ and 2-cells $\alpha:f\to f',\beta:g\to g'$.

(4) {\it Lax unity constraints:} For each object $X\in\mathsf{B}$, a natural transformation $F^0_X:1'_{FX} \Rightarrow F1_X$ with component 2-cell $F^0_X : 1'_{FX} \to F1_X$, called {\it lax unity constraint}.
\[
\begin{tikzcd}
	{\bf 1} \ar[r,"1_X"] \ar[dr,"1'_{FX}" '{name=U},] & \mathsf{B}(X,X) \ar[d,"F"] \arrow[Rightarrow,from=U,"F^0_X"] \\
	& \mathsf{B}'(FX,FX)
\end{tikzcd}
\]
The naturality of $F^0_X$ means $F1_{1_X}\circ F^0_X=F^0_X\circ 1_{1'_{FX}}$,  i.e., the following diagram is commutative 
\[
\begin{tikzcd}
	1'_{FX} \ar[r,"F^0_X"] \ar[d,"1_{1'_{FX}}"'] & 
	F1_X \ar[d,"F1_{1_X}"] \\
	1'_{FX} \ar[r,"F^0_X"] & F1_X
\end{tikzcd}
\]
which always holds (Ref. \cite[Explanation 4.1.5 (3)]{JohnsonYau21}). 

The above data satisfy:

{\it Lax Associativity:} For all 1-cells $f\in\mathsf{B}(W,X), g\in\mathsf{B}(X,Y), h\in\mathsf{B}(Y,Z)$,
the following diagram in $\mathsf{B}'(FW,FZ)$ is commutative.
\[
\begin{tikzcd}[row sep=large,column sep=large]
	(Fh\circ Fg)\circ Ff \arrow[r, "a'_{Fh,Fg,Ff}"] \arrow[d, "{F^2_{h,g}*1_{Ff}}"'] & Fh\circ(Fg\circ Ff) \arrow[d, "{1_{Fh}*F^2_{g,f}}"] \\
	F(hg)\circ Ff \arrow[d, "{F^2_{hg,f}}"'] & Fh\circ F(gf) \arrow[d, "{F^2_{h,gf}}"] \\
	F((hg)f) \arrow[r, "Fa_{h,g,f}"] & F(h(gf))
\end{tikzcd}
\]

{\it Lax Left Unity:} For all 1-cell $f\in\mathsf{B}(X,Y)$,
the following diagram in $\mathsf{B}'(FX,FY)$ is commutative.
\[
\begin{tikzcd}
	1'_{FY}\circ Ff \arrow[r,"l'_{Ff}"] \arrow[d, "{F_Y^0*1_{Ff}}"'] & Ff \\
	F1_Y\circ Ff \arrow[r, "{F^2_{\one_Y,f}}"] & F(1_Y\circ f) \arrow[u, "Fl_f"'] 
\end{tikzcd}
\]

{\it Lax Right Unity:} For all 1-cell $f\in\mathsf{B}(X,Y)$,
the following diagram in $\mathsf{B}'(FX,FY)$ is commutative.
\[
\begin{tikzcd}
	Ff\circ 1'_{FX} \ar[r,"r'_{Ff}"] \ar[d,"1_{Ff}*F^0_X" '] & Ff  \\
	Ff\circ F1_X \ar[r,"F^2_{f,1_X}"] & F(f \circ 1_X) \ar[u,"Fr_f" ']
\end{tikzcd}
\]

\medspace

\noindent{\bf Pseudofunctors.} A {\it pseudofunctor} is a lax functor $(F,F^2,F^0)$ in which $F^2$ and $F^0$ are natural isomorphisms.

\medspace

\noindent{\bf Strict functors.} 
A {\it strict functor} is a lax functor $(F,F^2,F^0)$ in which $F^2$ and $F^0$ are 
are identity natural transformations.

\medspace

\noindent{\bf Colax functors.} Let $\mathsf{B}$ be a bicategory. Then its co-bicategory $\mathsf{B}^{\rm co}$, which has the same objects as $\mathsf{B}$ and the opposite Hom categories (Ref. \cite[Definition 2.6.3]{JohnsonYau21}), is still a bicategory. A {\it colax functor} from a bicategory $\mathsf{B}$ to a bicategory $\mathsf{B}'$ is a lax functor from the co-bicategory $\mathsf{B}^{\rm co}$ of $\mathsf{B}$ to the co-bicategory $\mathsf{B}'^{\rm co}$ of $\mathsf{B}'$.

\medspace

\noindent{\bf 1-skeleton of a bicategory.} Every category is equivalent to its skeleton. Similarly, every bicategory is biequivalent to its 1-skeleton.

Let $\mathsf{C}$ be a category. The {\it skeleton} of $\mathsf{C}$ is the category $\overline{\mathsf{C}}$ defined as follows:
For each object $X\in\mathsf{C}$, we denote by $[X]$ its isomorphism class. For each isomorphism class $[X]$ of $X\in\mathsf{C}$ and each object $X'\in[X]$, we fix an object $\overline{X}\in[X]$ and an isomorphism $\phi_{X'}:\overline{X} \rightarrow X'$.
	\begin{enumerate}
		\item The objects of $\overline{\mathsf{C}}$ are all isomorphism classes of objects in $\mathsf{C}$, alternately, all chosen representatives $\overline{X}$ of isomorphism classes $[X]$ of objects $X$ in $\mathsf{C}$. 
		\item For any two objects $[X],[Y]\in\overline{\mathsf{C}}$, the Hom set $\overline{\mathsf{C}}([X],[Y]):=\mathsf{C}(\overline{X},\overline{Y}).$
		\item For any object $[X]\in\overline{\mathsf{C}}$, the identity morphism $1_{[X]}:=1_{\overline{X}}$.
		\item The composition in $\overline{\mathsf{C}}$ is the same as that in $\mathsf{C}$.
	\end{enumerate}
The functor $T_\mathsf{C}: \overline{\mathsf{C}} \to \mathsf{C}, [X]\mapsto \overline{X}, f\mapsto f$, is fully faithful and dense, 
so it is an equivalence with a quasi-inverse $S_\mathsf{C}: \mathsf{C} \to \overline{\mathsf{C}}, X\mapsto [X], (f:X\to Y)\mapsto (\phi_Y^{-1}f\phi_X: \overline{X}\to\overline{Y})$.
$$\begin{tikzcd}
	\overline{X} \arrow[r,"S_\mathsf{C}f"] \arrow[d,"\phi_X"] & \overline{Y} \arrow[d,"\phi_Y"] \\
	X \arrow[r,"f"] & Y
\end{tikzcd}$$

\medspace

Let $\mathsf{B}$ be a bicategory. The {\it $1$-skeleton} of $\mathsf{B}$ is the bicategory $\overline{\mathsf{B}}$ defined as follows:
	\begin{enumerate}
		\item The objects of $\overline{\mathsf{B}}$ are all objects of $\mathsf{B}$.
		
		\item For all $X,Y\in\overline{\mathsf{B}}$, the Hom category
		  $\overline{\mathsf{B}}(X,Y):=\overline{\mathsf{B}(X,Y)}$,
		  the skeleton of the category $\mathsf{B}(X,Y)$, where we fix $\overline{1_X}=1_X$ when $Y=X$.
		  
		\item For all $X\in\overline{\mathsf{B}}$, the identity 1-cell $1_X:=[1_X]$.
		
		\item For all $X,Y,Z\in\overline{\mathsf{B}}$, the horizontal composition functor 
		\[
		\bar{c}_{XYZ}: \overline{\mathsf{B}}(Y,Z)\times \overline{\mathsf{B}}(X,Y) \to \overline{\mathsf{B}}(X,Z)
		\] 
		in $\overline{\mathsf{B}}$ is induced by the horizontal composition functor $c_{XYZ}$ in $\mathsf{B}$. 
		\[
			\begin{tikzcd}
				\overline{\mathsf{B}}(Y,Z)\times \overline{\mathsf{B}}(X,Y) \arrow[r, "\bar{c}_{XYZ}"] \arrow[d, "\simeq"',"{T_{\mathsf{B}(Y,Z)}\times T_{\mathsf{B}(X,Y)}}"] & \overline{\mathsf{B}}(X,Z) \\
				\mathsf{B}(Y,Z)\times\mathsf{B}(X,Y) \arrow[r, "c_{XYZ}"] & \mathsf{B}(X,Z) \arrow[u, "\simeq","{S_{\mathsf{B}(X,Z)}}"']
			\end{tikzcd}
		\]
		More precisely, $[g]\circ[f]:=[\overline{g}\circ\overline{f}]$ and $\beta\bar{*}\alpha:=\phi_{\overline{g'}\circ\overline{f'}}^{-1}(\beta*\alpha)\phi_{\overline{g}\circ\overline{f}}$, for all 1-cells $[f]\in\overline{\mathsf{B}}(X,Y), [g]\in\overline{\mathsf{B}}(Y,Z)$ and 2-cells $\alpha:[f] \rightarrow [f'],\beta:[g] \rightarrow [g']$.
		
		\item For all $W,X,Y,Z\in\overline{\mathsf{B}}$, the associator $\bar{a}_{WXYZ}$ is given by the component 2-cell
		\[
			\bar{a}_{[h],[g],[f]} := \phi_{\overline{h}\circ\overline{\overline{g}\circ\overline{f}}}^{-1} (1_{\overline{h}}*\phi_{\overline{g}\circ\overline{f}}^{-1}) a_{\overline{h},\overline{g},\overline{f}} (\phi_{\overline{h}\circ\overline{g}}*1_{\overline{f}}) \phi_{\overline{\overline{h}\circ\overline{g}}\circ\overline{f}}
		\]
		for all 1-cells $[f]\in\overline{\mathsf{B}}(W,X), [g]\in\overline{\mathsf{B}}(X,Y), [h]\in\overline{\mathsf{B}}(Y,Z)$.
		$$\begin{tikzcd} [column sep=large]
			\overline{ \overline{\overline{h}\circ\overline{g}} \circ \overline{f} } \arrow[r,"{ \bar{a}_{[h],[g],[f]} }"] 
			\arrow[d,"{ \phi_{ \overline{\overline{h}\circ\overline{g}} \circ \overline{f}} }"] 
			& \overline{ \overline{h} \circ \overline{\overline{g}\circ\overline{f}} } 
			\arrow[d,"{ \phi_{ \overline{h} \circ \overline{\overline{g}\circ\overline{f}} } }"] 
			\\
			\overline{\overline{h}\circ\overline{g}} \circ \overline{f} 
			\arrow[d,"{ \phi_{ \overline{h}\circ\overline{g} }*1_{\overline{f}} }"] 
			& \overline{h} \circ \overline{ \overline{g} \circ \overline{f} } \arrow[d,"{ 1_{\overline{h}}*\phi_{ \overline{g}\circ\overline{f} } }"] \\
			(\overline{h} \circ \overline{g}) \circ \overline{f} \arrow[r,"{ a_{\overline{h},\overline{g},\overline{f}} }"] & \overline{h} \circ (\overline{g} \circ \overline{f})
		\end{tikzcd}$$
		
		\item For all $X,Y\in\overline{\mathsf{B}}$, the left unitor $\bar{l}_{XY}$ is the natural isomorphism
		\[
			S_{\mathsf{B}(X,Y)}l_{XY}T_{\mathsf{B}(X,Y)},
		\]
		i.e., its component 2-cell $\bar{l}_{[f]}:=l_{\overline{f}}\phi_{1_Y\circ \overline{f}}$ for all 1-cell $[f]\in\overline{\mathsf{B}}(X,Y)$.

		The right unitor $\bar{r}_{XY}$ is the natural isomorphism
		\[
			S_{\mathsf{B}(X,Y)}r_{XY}T_{\mathsf{B}(X,Y)},
		\]
		i.e., its component 2-cell $\bar{r}_{[f]}:=r_{\overline{f}}\phi_{\overline{f}\circ 1_X}$ for all 1-cell $[f]\in\overline{\mathsf{B}}(X,Y)$.
	\end{enumerate}

\begin{proposition} \label{Prop-Bicat-1-Skeleton} {\rm (Ref. \cite[Lemma 2.20]{Schommer-Pries09})}
	Every bicategory $\mathsf{B}$ is biequivalent to its 1-skeleton $\overline{\mathsf{B}}$. More precisely, the pseudofunctor $\mathscr{S}: \mathsf{B}\to\overline{\mathsf{B}}, X\mapsto X, f\mapsto [f], (\alpha:f\to g)\mapsto (\phi_g^{-1}\alpha\phi_f: \overline{f}\to\overline{g})$, is a biequivalence with inverse $\mathscr{T}: \overline{\mathsf{B}}\to\mathsf{B}, X\mapsto X, [f]\mapsto \overline{f}, \alpha\mapsto\alpha$.
\end{proposition}

\subsection{Bicategories of dg categories} Now we construct some bicategories of small dg categories which are the source bicategories of objective (co)lax functors.

\medspace 

\noindent{\bf Bicategory $\dgCAT_{\rm c,h}$.} The bicategory $\dgCAT_{\rm c,h}$ is defined as follows:

\begin{enumerate}
	\item The objects of $\dgCAT_{\rm c,h}$ are all small dg categories.
	
	\item For all $\aA,\bB\in\dgCAT_{\rm c,h}$, the Hom category $\dgCAT_{\rm c,h}(\aA,\bB)$ is the subcategory $\hH_{\rm c,h}(\aA^\op\otimes\bB)$ of the homotopy category $\hH(\aA^\op\otimes\bB)$ of dg $\aA$-$\bB$-bimodules, whoses objects are all cofibrant $\aA$-$\bB$-bimodules and whose morphisms are all homotopy equivalence classes of dg $\aA\-\bB$-bimodule quasi-isomorphisms (=homotopy equivalences since all its objects are cofibrant dg $\aA\-\bB$-bimodules). 
	
	\item For all $\aA\in\dgCAT_{\rm c,h}$, the identity 1-cell $1_\aA:=\bp I_\aA$, where $\bp$ is the cofibrant replacement functor from the derived category $\dD(\aA^\op\otimes\bB)$ of dg $\aA$-$\bB$-bimodules to the full subcategory $\hH_{\rm c}(\aA^\op\otimes\bB)$ of the homotopy category $\hH(\aA^\op\otimes\bB)$ of dg $\aA$-$\bB$-bimodules consisting of all cofibrant dg $\aA$-$\bB$-bimodules.
	
	\item For all $\aA,\bB,\cC\in\dgCAT_{\rm c,h}$, the horizontal composition functor $c_{\aA\bB\cC}$ is the tensor product functor $\otimes_\bB$.
	
	\item For all $\aA,\bB,\cC,\dD\in\dgCAT_{\rm c,h}$, the associator $a_{\aA\bB\cC\dD}$ is given by the canonical natural isomorphism
	\[
	  a_{Z,Y,X}:X\otimes_\bB(Y\otimes_\cC Z) \xrightarrow{\cong} (X\otimes_\bB Y)\otimes_\cC Z
	\]
	for all cofibrant dg bimodules $_\aA X_\bB,{}_\bB Y_\cC$, and ${}_\cC Z_\dD$.
	
	\item For all $\aA,\bB\in\dgCAT_{\rm c,h}$, the left unitor $l_{\aA\bB}$ is given by the canonical natural isomorphism
	\[
	  l_X: X\otimes_\bB \bp I_\bB \rightarrow X\otimes_\bB I_\bB \xrightarrow{\cong} X,
	\]
    and the right unitor $r_{\aA\bB}$ is given by the canonical natural isomorphism
	\[
	 r_X: \bp I_{\aA}\otimes_\aA X \rightarrow I_\aA\otimes_\aA X \xrightarrow{\cong} X
	\]
	for all cofibrant dg $\aA$-$\bB$-bimodule $X$.
\end{enumerate}

It is easy to see that the above data satisfy the Unity Axiom and the Pentagon Axiom.


\medspace

\noindent{\bf Bicategory $\dg\cC\aA\tT$.} The bicategory $\dg\cC\aA\tT$ is defined as follows:

\begin{enumerate}
	\item The objects of $\dg\cC\aA\tT$ are all small dg categories.
	
	\item For all $\aA,\bB\in\dg\cC\aA\tT$, the Hom category $\dg\cC\aA\tT(\aA,\bB)$ is a subcategory of the derived category $\dD(\aA^\op\otimes\bB)$ of dg $\aA$-$\bB$-bimodules, whose objects are all objects in $\dD(\aA^\op\otimes\bB)$, and whose morphisms are all isomorphisms in $\dD(\aA^\op\otimes\bB)$.
	
	\item { For all $\aA\in\dg\cC\aA\tT$, the identity 1-cell $1_\aA:=I_\aA$.}
	
	\item For all $\aA,\bB,\cC\in\dg\cC\aA\tT$, the horizontal composition functor 
	\[
	  c_{\aA\bB\cC}:\dg\cC\aA\tT(\bB, \\ \cC) \times\dg\cC\aA\tT(\aA,\bB) \to \dg\cC\aA\tT(\aA,\cC)
	\] 
	  is the derived tensor product functor 
	\[
	  \dD(\bB^\op\otimes\cC)\times \dD(\aA^\op\otimes\bB) \to \dD(\aA^\op\otimes\cC), (Y,X)\mapsto X\otimes^\bL_\bB Y.
	\] 
	
	\item For all $\aA,\bB,\cC,\dD\in\dg\cC\aA\tT$, the associator $a_{\aA\bB\cC\dD}$ is given by the canonical natural isomorphism
	\[
	  a_{Z,Y,X}: X\otimes_\bB^\bL(Y\otimes_\cC^\bL Z) \xrightarrow{\cong} 
	  (X\otimes_\bB^\bL Y)\otimes_\cC^\bL Z
	\]
	for all dg bimodules $_\aA X_\bB,{}_\bB Y_\cC$, and ${}_\cC Z_\dD$.
	
	\item For all $\aA,\bB\in\dg\cC\aA\tT$, the left unitor $l_{\aA\bB}$ is given by the canonical natural isomorphism
	\[
	  l_X: X\otimes_\bB^\bL I_\bB \xrightarrow{\cong} X,
	\]
	and the right unitor $r_{\aA\bB}$ is given by the canonical natural isomorphism
	\[
	  r_X: I_\aA\otimes_\aA^\bL X \xrightarrow{\cong} X
	\]
	for all dg $\aA$-$\bB$-bimodule $X$.
\end{enumerate}

Clearly, the above data satisfy the Unity Axiom and the Pentagon Axiom.

\subsection{Bicategories of $B_\infty$-algebras}

Now we construct the target bicategory of the objective (co)lax functors. In fact, we will provide a general construction of a bicategory from a model category, for which {\it spans} and {\it spans of spans} will play crucial roles.

\medspace

\noindent{\bf Bicategory $\mathsf{C}\-\Span$.}  (Ref. \cite[Example 2.1.22]{JohnsonYau21})
Let $\mathsf{C}$ be a category with pullbacks. For a pair of objects $x,y\in\mathsf{C}$, a {\it span} from $x$ to $y$ is a pair of morphisms $x\xleftarrow{f_l} f \xrightarrow{f_r} y$ in $\mathsf{C}$ with common source $f$, denoted by $\hat{f}=(x\xleftarrow{f_l} f \xrightarrow{f_r} y)$. A {\it morphism} $\hat{\hat{\alpha}}:\hat{f}\to\hat{f}'$ of spans from $x$ to $y$ is a morphism $\alpha:f \rightarrow f'$ in $\mathsf{C}$ such that $f'_l\alpha=f_l$ and $f'_r\alpha=f_r$, i.e., the following diagram is commutative. 
\[
  \begin{tikzcd}
	& f \arrow[ld, "f_l"'] \arrow[rd, "f_r"] \arrow[dd, "\alpha"] & \\
	x & & y \\
	& f' \arrow[ul, "f'_l"] \arrow[ur, "f'_r"'] &
  \end{tikzcd}
\]
The {\it composition} of morphisms $\hat{\hat{\alpha}}:\hat{f}\to\hat{f}'$ and $\hat{\hat{\alpha}}':\hat{f}'\to\hat{f}''$ of spans from $x$ to $y$ is the composition $\alpha'\alpha:f \rightarrow f''$ of $\alpha:f\to f'$ and $\alpha':f'\to f''$ in $\mathsf{C}$. 
The {\it identity morphism} $1_{\hat{f}}$ on the span $\hat{f}$ is the identity morphism $1_f$ on $f$ in $\mathsf{C}$.
All spans from $x$ to $y$ and morphisms between them form the {\it category of spans from $x$ to $y$}, denoted by $\mathsf{C}\-\Span(x,y)$. 

The bicategory $\mathsf{C}\-\Span$ is defined by the following data:

(1) The objects of $\mathsf{C}\-\Span$ are all objects of $\mathsf{C}$.

(2) For all $x,y\in\mathsf{C}$, the Hom category $\mathsf{C}\-\Span(x,y)$ is the category of spans from $x$ to $y$ defined above.

(3) For all $x\in\mathsf{C}$, the identity 1-cell $1_x$ on $x$ is the span $x \xleftarrow{1_x} x\xrightarrow{1_x} x$ written as $x=x=x$ for short sometimes.

(4) For all $x,y,z\in\mathsf{C}$, the horizontal composition ${c}_{xyz}:\mathsf{C}\-\Span(y,z)\times\mathsf{C}\-\Span(x,y)\to\mathsf{C}\-\Span(x,z), (\hat{g},\hat{f}) \mapsto {c}_{xyz}(\hat{g},\hat{f}):=(x\xleftarrow{f_l\tilde{g}_l} f\times_yg\xrightarrow{g_r\tilde{f}_r} z), (\hat{\hat{\beta}},\hat{\hat{\alpha}})\mapsto {c}_{xyz}(\hat{\hat{\beta}},\hat{\hat{\alpha}}):=(\alpha\times_y\beta:f\times_yg\to f'\times_yg')$, for all 1-cells $\hat{f}=(x\xleftarrow{f_l} f \xrightarrow{f_r} y), \hat{f}'=(x\xleftarrow{f'_l} f' \xrightarrow{f'_r} y), \hat{g}=(y\xleftarrow{g_l} g \xrightarrow{g_r} z), \hat{g}'=(y\xleftarrow{g'_l} g' \xrightarrow{g'_r} z)$ and 2-cells $\hat{\hat{\alpha}}:\hat{f}\to\hat{f}',\hat{\hat{\beta}}:\hat{g}\to\hat{g}'$.
\[
  \begin{tikzcd}[row sep=small,column sep=small]
	& & f\times_yg \arrow[ld,"\tilde{g}_l"'] \arrow[rd,"\tilde{f}_r"] \arrow[dddd,dotted,bend left=6,"\alpha\times_y\beta",near start]& & \\
	& f \arrow[ld,"f_l"'] \arrow[rd,"f_r"] \arrow[dd,"\alpha"] & & g\arrow[ld,"g_l"'] \arrow[rd,"g_r"] \arrow[dd,"\beta"]  & \\
	x & & y &   & z \\
	& f' \arrow[ul,"f'_l"] \arrow[ru,"f'_r"']  & & g'\arrow[ul,"g'_l"] \arrow[ur,"g'_r" ']  & \\
	& & f'\times_yg' \arrow[lu,"\tilde{g}'_l"] \arrow[ru,"\tilde{f}'_r" '] & &
  \end{tikzcd}
\]

(5) For all $w,x,y,z\in\mathsf{C}$, the component 2-cell of the associator $a_{wxyz}$ is given by the canonical isomorphism $a_{\hat{h},\hat{g},\hat{f}}:=a_{f,g,h}:f\times_x (g\times_y h) \rightarrow (f\times_x g)\times_y h$ for all 1-cells $\hat{f}=(w\leftarrow f\to x), \hat{g}=(x\leftarrow g\to y), \hat{h}=(y\leftarrow h\to z)$.
\[
\begin{tikzcd}
	&& (f\times_xg)\times_yh\arrow[d] \arrow[rrrdd] \arrow[rrd,dashed]
	&& f\times_x(g\times_y h)\arrow[d] \arrow[ddlll] \arrow[ll,dotted,"\cong","{a_{f,g,h}}"'] \arrow[lld,dashed] &&\\
	&& f\times_xg\arrow[ld] \arrow[rd] && g\times_zh\arrow[ld] \arrow[rd] &&\\
	& f\arrow[ld] \arrow[rd] && g\arrow[ld] \arrow[rd] && h\arrow[ld] \arrow[rd] & \\
	w && x && y && z		
\end{tikzcd} 
\]

(6) For all $x,y\in\mathsf{C}$, the component 2-cell of the left unitor $l_{xy}$ is given by the canonical isomorphism $l_{\hat{f}}:=l_f: f\times_yy\rightarrow f$, and the component 2-cell of the right unitor $r_{xy}$ is given by the canonical isomorphism $r_{\hat{f}}:=r_f: x\times_xf\rightarrow f$, for all 1-cell $\hat{f}=(x\leftarrow f\to y)$.
\[
\begin{tikzcd}[column sep=small]
	&& f\times_yy \arrow[ld,"\cong"',"l_f"] \arrow[rd,"\tilde{f}_r"'] && \\
	&f\arrow[ld,"f_l"'] \arrow[rd,"f_r"] & & y \arrow[ld,equal,"1_y"'] \arrow[rd,equal,"1_y"] &\\
	x && y && y 
\end{tikzcd} 
\qquad
\begin{tikzcd}[column sep=small]
	&& x\times_xf \arrow[rd,"\cong","r_f"'] \arrow[ld,"\tilde{f}_l"] && \\
	& x\arrow[ld,equal,"1_x"'] \arrow[rd,equal,"1_x"] & & f\arrow[ld,"f_l"'] \arrow[rd,"f_r"] &\\
	x && x && y 
\end{tikzcd} 
\]

\medspace

\noindent{\bf Bicategory $\mathsf{C}\-\Span^2$.} (Ref. \cite[Theorem 8]{Rebro15}) Let $\mathsf{C}$ be a category with pullbacks. Then for all $x,y\in\mathsf{C}$, the category $\mathsf{C}\-\Span(x,y)$ of spans in $\mathsf{C}$ from $x$ to $y$ admits pullbacks. Indeed, for all cospan $\hat{f}\to\hat{e}\leftarrow\hat{g}$ in $\mathsf{C}\-\Span(x,y)$, the pullback $f\times_eg$ of the cospan $f\to e\leftarrow g$ in $\mathsf{C}$ gives the pullback $x\leftarrow f\times_eg\to y$ of the cospan $\hat{f}\to\hat{e}\leftarrow\hat{g}$ in $\mathsf{C}\-\Span(x,y)$. 
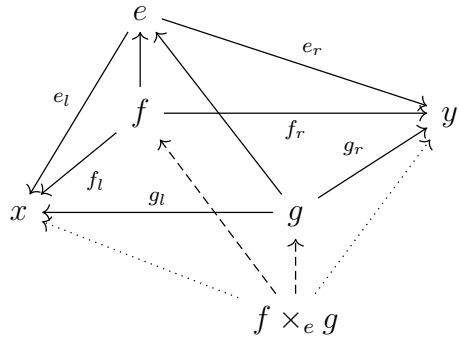
\[
  \begin{tikzcd}
	& e \arrow[ddl,"e_l"'] \arrow[drr,"e_r"] & & \\
	& f  \arrow[dl,"f_l"] \arrow[u] \arrow[rr,"f_r"'] & & y \\
	x & & g \arrow[uul] \arrow[ll,"g_l"'] \arrow[ur,"g_r"]  & \\
	& & f\times_eg \arrow[ull,dotted] \arrow[uul,dashed] \arrow[u,dashed] \arrow[uur,dotted]&
  \end{tikzcd} 
\]
Thus we have the bicategory $(\mathsf{C}\-\Span(x,y))\-\Span$. 
For all spans $\hat{f},\hat{f}'\in \mathsf{C}\-\Span(x,y)$, we have the Hom category $(\mathsf{C}\-\Span(x,y))\-\Span(\hat{f},\hat{f}')$, whose objects are spans $\hat{\hat{\alpha}}=(\hat{f}\leftarrow\hat{\alpha}\to\hat{f}')$ in $\mathsf{C}\-\Span(x,y)$, or equivalently, spans $f\xleftarrow{\alpha_l} \alpha \xrightarrow{\alpha_r} f'$ in $\mathsf{C}$ such that the following diagram is commutative, 
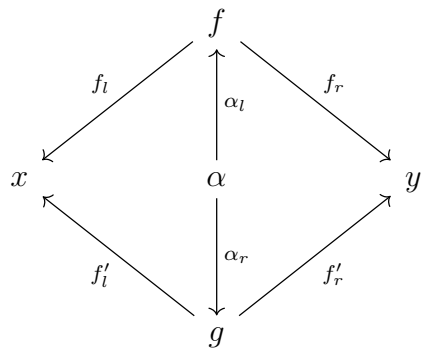
\[
  \begin{tikzcd}[row sep=huge,column sep=huge]
  	& f \arrow[ld,"f_l" '] \arrow[rd,"f_r"] & \\
  	x & \alpha \arrow[u,"\alpha_l" '] \arrow[d,"\alpha_r"] & y \\
  	& g \arrow[ul,"f'_l"] \arrow[ur,"f'_r" '] &
  \end{tikzcd}
\]
called {\it spans of spans} in $\mathsf{C}$, denoted by $\hat{\hat{\alpha}}=(\hat{f}\leftarrow\hat{\alpha}\to\hat{f}')=(f\xleftarrow{\alpha_l} \alpha \xrightarrow{\alpha_r} f')$,  and whose morphisms from span of spans $\hat{\hat{\alpha}}$ to $\hat{\hat{\beta}}$ is just a morphism $\theta:\alpha\to\beta$ in $\mathsf{C}$ such that the following diagram is commutative. 
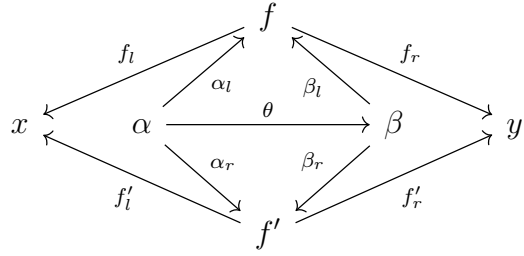
\[
  \begin{tikzcd}
	& & f \arrow[lld,"f_l" '] \arrow[rrd,"f_r"] & & \\
	x & \alpha \arrow[ur,"\alpha_l" '] \arrow[rd,"\alpha_r"] \arrow[rr, "\theta"] & & \beta \arrow[ul,"\beta_l"] \arrow[dl,"\beta_r" '] & y \\
	& & f' \arrow[ull,"f'_l"] \arrow[urr,"f'_r" '] & & 
  \end{tikzcd}
\]

The bicategory $\mathsf{C}\-\Span^2$ is defined by the following data:

(1) The objects of $\mathsf{C}\-\Span^2$ are all objects of $\mathsf{C}$.

(2) For all $x,y\in\mathsf{C}\-\Span^2$, the Hom category $\mathsf{C}\-\Span^2(x,y)$ is defined as follows:

(2.1) 1-cells are all objects $\hat{f}=(x\xleftarrow{f_l} f \xrightarrow{f_r} y)$ in $\mathsf{C}\-\Span(x,y)$.

(2.2) For all 1-cells $\hat{f},\hat{f}'\in \mathsf{C}\-\Span(x,y)$, the set of 2-cells $\mathsf{C}\-\Span^2(x,y)(\hat{f},\hat{f}'):= \linebreak \text{Iso}((\mathsf{C}\-\Span(x,y))\-\Span(\hat{f},\hat{f}'))$, the set of the isomorphism classes $[\hat{\hat{\alpha}}]$ of objects $\hat{\hat{\alpha}}=(\hat{f}\leftarrow \hat{\alpha}\to\hat{f}')$ in $(\mathsf{C}\-\Span(x,y))\-\Span(\hat{f},\hat{f}')$.

(3) For all $x\in\mathsf{C}\-\Span^2$, the identity 1-cell $1_x$ is the span $x \xleftarrow{1_x} x\xrightarrow{1_x} x$.

(4) For all $x,y,z\in\mathsf{C}\-\Span^2$, the horizontal composition $$\begin{array}{c}
	{c}_{xyz}:\mathsf{C}\-\Span^2(y,z)\times\mathsf{C}\-\Span^2(x,y) \to \mathsf{C}\-\Span^2(x,z),\\ [2mm]
	(\hat{g},\hat{f}) \mapsto {c}_{xyz}(\hat{g},\hat{f}) := (x\xleftarrow{f_l\tilde{g}_l} f\times_yg\xrightarrow{g_r\tilde{f}_r} z), \\ [2mm] ([\hat{\hat{\beta}}],[\hat{\hat{\alpha}}]) \mapsto {c}_{xyz}([\hat{\hat{\beta}}],[\hat{\hat{\alpha}}]) := [f\times_yg \xleftarrow{\alpha_l\times_y\beta_l} \alpha\times_y\beta \xrightarrow{\alpha_r\times_y\beta_r} f'\times_yg'],
\end{array}$$ for all 1-cells $\hat{f}=(x\xleftarrow{f_l} f \xrightarrow{f_r} y), \hat{f}'=(x\xleftarrow{f'_l} f' \xrightarrow{f'_r} y), \hat{g}=(y\xleftarrow{g_l} g \xrightarrow{g_r} z), \hat{g}'=(y\xleftarrow{g'_l} g' \xrightarrow{g'_r} z)$ and 2-cells $[\hat{\hat{\alpha}}] = [\hat{f}\leftarrow \hat{\alpha}\to\hat{f}'] = [f\xleftarrow{\alpha_l} \alpha \xrightarrow{\alpha_r} f'], [\hat{\hat{\beta}}] = [\hat{g}\leftarrow \hat{\beta}\to\hat{g}'] = [g\xleftarrow{\beta_l} \beta \xrightarrow{\beta_r} g']$.
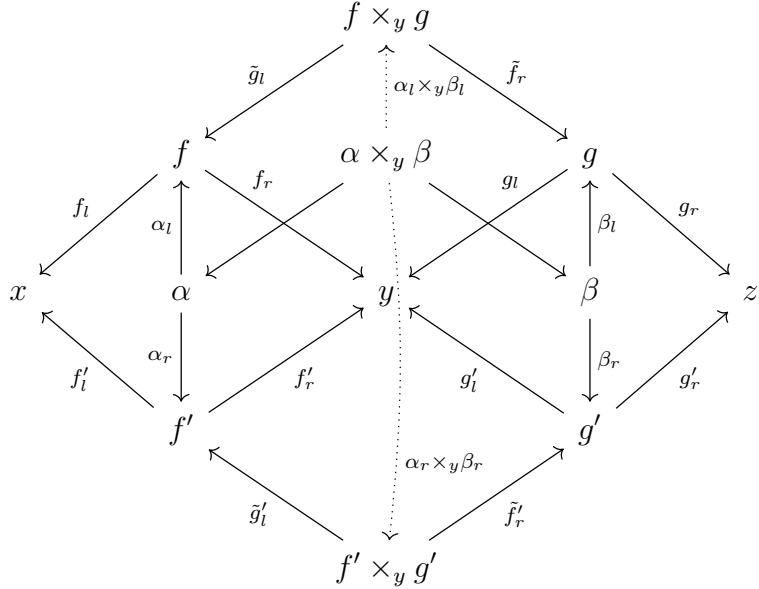
\[
\begin{tikzcd}[row sep=large,column sep=large]
	& & f\times_yg \arrow[ld,"\tilde{g}_l"'] \arrow[rd,"\tilde{f}_r"] & & \\
	& f \arrow[ld,"f_l"'] \arrow[rd,"f_r",near start] & \alpha\times_y\beta \arrow[u,dotted,"\alpha_l\times_y\beta_l"'] \arrow[dl] \arrow[dr] \arrow[ddd,dotted,bend left=6,"\alpha_r\times_y\beta_r",near end] & g\arrow[ld,"g_l"',near start] \arrow[rd,"g_r"]  & \\
	x & \alpha \arrow[u,"\alpha_l"] \arrow[d,"\alpha_r"'] & y & \beta \arrow[u,"\beta_l"'] \arrow[d,"\beta_r"] & z \\
	& f' \arrow[ul,"f'_l"] \arrow[ru,"f'_r"']  & & g'\arrow[ul,"g'_l"] \arrow[ur,"g'_r" ']  & \\
	& & f'\times_yg' \arrow[lu,"\tilde{g}'_l"] \arrow[ru,"\tilde{f}'_r" '] & &
\end{tikzcd}
\]

(5) For all $w,x,y,z\in\mathsf{C}\-\Span^2$, the component 2-cell of the associator $a_{wxyz}$ is  $a_{\hat{h},\hat{g},\hat{f}}:=[f\times_x (g\times_y h) = f\times_x (g\times_y h) \xrightarrow{a_{f,g,h}} (f\times_x g)\times_y h]$ for all 1-cells $\hat{f}=(w\leftarrow f\to x), \hat{g}=(x\leftarrow g\to y)$ and $\hat{h}=(y\leftarrow h\to z)$.

(6) For all $x\in\mathsf{C}\-\Span^2$, the component 2-cell of the left unitor $l_{xy}$ is $l_{\hat{f}}:=[f\times_yy = f\times_yy \xrightarrow{l_f} f]$, and the component 2-cell of the right unitor $r_{xy}$ is $r_{\hat{f}}:=[x\times_xf = x\times_xf\xrightarrow{r_f} f]$, for all 1-cell $\hat{f}=(x\leftarrow f\to y)$.

\begin{remark}
	In \cite[Theorem 8]{Rebro15}, Rebro assumed that the category $\mathsf{C}$ has pullbacks and a terminal object. Under this assumption, $\mathsf{C}$ has all finite limits, which were used many times in the proof of the theorem. In fact, the condition that $\mathsf{C}$ has a terminal object is unnecessary. Under the assumption that $\mathsf{C}$ has pullbacks, one can prove the conclusion of the theorem by the properties of pullbacks, analogous to the proof of Theorem \ref{Thm-Const-Bicat}. 
\end{remark}

\medspace

\noindent{\bf Bicategory $\mathsf{M}\-{\rm span}^2$.}  
The following two lemmas will be used frequently in the sequel.

\begin{lemma} \label{Lem-Pullback-Fib-AcycFib} {\rm (Ref. \cite[Corollary 1.1.11]{Hovey99})}
	In a model category, fibrations and acyclic fibrations are closed under pullbacks, i.e., for any pullback diagram
	\[
	\begin{tikzcd}
		x\times_zy \arrow[r,"\tilde{f}"]  \arrow[d,"\tilde{g}"] & y \arrow[d,"g"]  \\
		x \arrow[r,"f"] & z, 
	\end{tikzcd} 
	\]
	if $f$ is a fibration (resp. an acyclic fibration) then so is $\tilde{f}$. 
\end{lemma}

\begin{lemma} \label{Lem-Pullback-2-Fib-AcycFib} 
	In a model category, for any 2-cospan $x \xrightarrow{f} x' \xrightarrow{f'} z  \xleftarrow{g'} y' \xleftarrow{g} y$, if $f$ and $g$ are fibrations (resp. acyclic fibrations) then so is $f\times_zg$.
	\[
	\begin{tikzcd}
		x\times_zy \arrow[dd] \arrow[rr] \arrow[dr,dotted,"{f\times_zg}"] & & y \arrow[d,"g"] \\
		& x'\times_zy' \arrow[r,"{\tilde{f}'}"] \arrow[d,"{\tilde{g}'}"] & y' \arrow[d,"{g'}"]  \\
		x \arrow[r,"f"]& x' \arrow[r,"{f'}"] & z 
	\end{tikzcd} 
	\]	
\end{lemma}

\begin{proof}
	Since the composition of pullbacks is a pullback, we have a commutative diagram
	\[
		\begin{tikzcd}
			x\times_z y \arrow[rr] \arrow[d, "\tilde{g}"] \arrow[rd, "f\times_z g", dotted] & & y \arrow[d, "g"] \\
			x\times_z y' \arrow[r, "\tilde{f}"] \arrow[d] & x'\times_z y' \arrow[r,"{\tilde{f}'}"] \arrow[d,"{\tilde{g}'}"] & y' \arrow[d,"{g'}"] \\
			x \arrow[r, "f"] & x' \arrow[r,"{f'}"] & z
		\end{tikzcd}
	\]
	where the bottom two squares and the upper rectangle are pullbacks. If $f$ and $g$ are fibrations (resp. acyclic fibrations) then by Lemma \ref{Lem-Pullback-Fib-AcycFib}, $\tilde{f}$ and $\tilde{g}$ are fibrations (resp. acyclic fibrations), and further, $f\times_z g=\tilde{f}\circ\tilde{g}$ is a fibration (resp. acyclic fibration).
\end{proof}

The following theorem provides a construction of a bicategory from a model category.

\begin{theorem} \label{Thm-Const-Bicat}
	Let $\mathsf{M}$ be a model category. Then the following data define a bicategory $\mathsf{M}\-{\rm span}^2$: 
	
	{\rm (1)} The objects of $\mathsf{M}\-{\rm span}^2$ are all objects of $\mathsf{M}$.
	
	{\rm (2)} For all $x,y\in \mathsf{M}\-{\rm span}^2$, the Hom category  $\mathsf{M}\-{\rm span}^2(x,y)$ is given as follows:
	
	{\rm (2.1)} 1-cells are all spans $\hat{f}=(x\xleftarrow{f_l} f \xrightarrow{f_r} y)$ in $\mathsf{M}$ with $f_l$ and $f_r$ being fibrations.
	
	{\rm (2.2)} 2-cells from 1-cell $\hat{f}$ to 1-cell $\hat{f}'$ are all equivalent classes $[\hat{\hat{\alpha}}]$ of  spans of spans $\hat{\hat{\alpha}} = (\hat{f} \leftarrow \hat{\alpha} \to \hat{f}') = (f \xleftarrow{\alpha_l} \alpha \xrightarrow{\alpha_r} f')$ with $\alpha_l$ being an acyclic fibration under the equivalent relation $\sim$ on spans from span $\hat{f}$ to span $\hat{f}'$ generated by ``\;$\hat{\hat{\alpha}} \sim \hat{\hat{\alpha}}' \stackrel{\triangle}{\Leftrightarrow}\exists\ \mbox{\rm a morphism of spans of spans}\ \theta: \hat{\hat{\alpha}}\to\hat{\hat{\alpha}}'$\;''.		
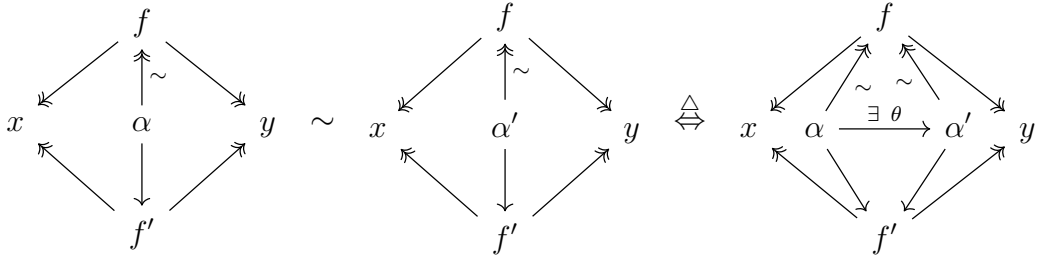
	\[
	\begin{tikzcd}
		& f \arrow[ld,two heads] \arrow[rd,two heads] & \\
		x & \alpha \arrow[u,two heads,"\sim" '] \arrow[d] & y \\
		& f' \arrow[ul,two heads] \arrow[ur,two heads] &
	\end{tikzcd} 
	\ \sim\ 
	\begin{tikzcd}
		& f \arrow[ld,two heads] \arrow[rd,two heads] & \\
		x & \alpha' \arrow[u,two heads,"\sim" '] \arrow[d] & y \\
		& f' \arrow[ul,two heads] \arrow[ur,two heads] &
	\end{tikzcd}
	\ \stackrel{\triangle}{\Leftrightarrow}\ 
	\begin{tikzcd}[column sep=tiny]
		& & f \arrow[lld,two heads] \arrow[rrd,two heads] & & \\
		x & \alpha \arrow[ur,two heads,"\sim" '] \arrow[rd] \arrow[rr, "\exists\ \theta"] & & \alpha' \arrow[ul,two heads,"\sim"] \arrow[dl] & y \\
		& & f' \arrow[ull,two heads] \arrow[urr,two heads] & & 
	\end{tikzcd}
	\]
	By 2-out-of-3 property of weak equivalences, $\theta:\alpha\rightarrow \alpha'$ must be a weak equivalence.
	
	{\rm (2.3)} The vertical composition of two 2-cells $[\hat{\hat{\alpha}}] =[\!\!\begin{tikzcd}[column sep=small] f & \alpha \arrow[l,two heads,"\alpha_l"',"\sim"] \arrow[r,"\alpha_r"] & f'\end{tikzcd}\!\!]$ and
	$[\hat{\hat{\alpha}}'] =[\!\!\begin{tikzcd}[column sep=small] f' & \alpha' \arrow[l,two heads,"\alpha'_l"',"\sim"]\end{tikzcd} \linebreak \stackrel{\alpha'_r}{\rightarrow} f'']$ is the 2-cell
	$[\hat{\hat{\alpha}}']\circ[\hat{\hat{\alpha}}] := [\hat{\hat{\alpha}}'\circ\hat{\hat{\alpha}}] = [\!\!\begin{tikzcd} f & \alpha\times_{f'}\alpha' \arrow[l,two heads,"\alpha_l\tilde{\alpha}'_l"',"\sim"]\end{tikzcd}  \xrightarrow{\alpha'_r\tilde{\alpha}_r} f'']$, where $\alpha_l\tilde{\alpha}'_l$ is an acyclic fibration due to Lemma~\ref{Lem-Pullback-2-Fib-AcycFib}.
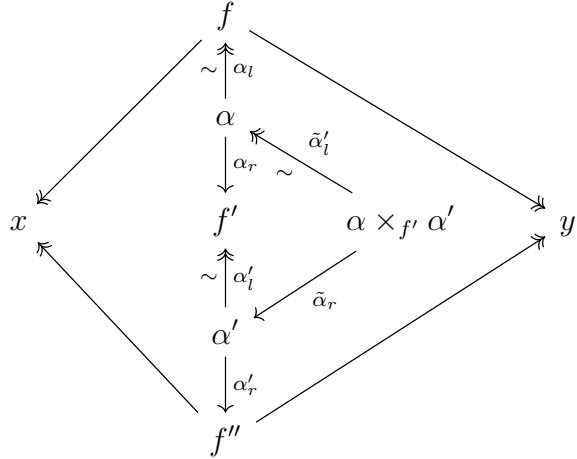
	\[
	\begin{tikzcd}
		&& f \arrow[lldd,two heads] \arrow[rrdd,two heads] && \\
		&& \alpha \arrow[u,two heads,"\sim","\alpha_l"'] \arrow[d,"\alpha_r"] && \\
		x && f' & \alpha\times_{f'}\alpha' \arrow[ul,two heads,"\sim","\tilde{\alpha}'_l"'] \arrow[dl,"\tilde{\alpha}_r"] & y \\
		&& \alpha' \arrow[u,two heads,"\sim","\alpha'_l"'] \arrow[d,"\alpha'_r"] && \\
		&& f'' \arrow[uull,two heads] \arrow[uurr,two heads] &&
	\end{tikzcd}
	\]			
	
	{\rm (2.4)} The identity 2-cell of a 1-cell $\hat{f}=(x\twoheadleftarrow f \twoheadrightarrow y)$ is $1_{\hat{f}}:=[f\xleftarrow{1_f}f\xrightarrow{1_f} f]$.	
	
	{\rm (3)} For all $x\in\mathsf{M}\-{\rm span}^2$, the identity 1-cell of $x$ is the span $1_x:=(x\xleftarrow{1_x} x\xrightarrow{1_x} x)$.
	
	{\rm (4)} For all $x,y,z\in\mathsf{M}\-{\rm span}^2$, the horizontal composition functor $c_{xyz}:\mathsf{M}\-{\rm span}^2(y,z)\times\mathsf{M}\-{\rm span}^2(x,y)\to\mathsf{M}\-{\rm span}^2(x,z) $ is given as follows:
	
	{\rm (4.1)} For all 1-cells $\hat{f}=(x \stackrel{f_l}{\twoheadleftarrow} f \stackrel{f_r}{\twoheadrightarrow} y)$ and $\hat{g}=(y \stackrel{g_l}{\twoheadleftarrow} g \stackrel{g_r}{\twoheadrightarrow} z)$, the horizontal composition $c_{xyz}(\hat{g},\hat{f})=\hat{g}\circ\hat{f}:=(x \stackrel{f_l\tilde{g}_l}{\twoheadleftarrow} f\times_y g \stackrel{g_r\tilde{f}_r}{\twoheadrightarrow} z)$, where $f_l\tilde{g}_l$ and $g_r\tilde{f}_r$ are fibrations due to Lemma~\ref{Lem-Pullback-2-Fib-AcycFib}. See the commutative diagram below.
	
	{\rm (4.2)} For all 2-cells $[\hat{\hat{\alpha}}]=[\!\!\begin{tikzcd}[column sep=small] f & \alpha \arrow[l,two heads,"\alpha_l"',"\sim"] \arrow[r,"\alpha_r"] & f'\end{tikzcd}\!\!]$ and $[\hat{\hat{\beta}}]=[\!\!\begin{tikzcd}[column sep=small] g & \beta \arrow[l,two heads,"\beta_l"',"\sim"] \arrow[r,"\beta_r"] & g'\end{tikzcd}\!\!]$, the horizontal composition $c_{xyz}([\hat{\hat{\beta}}],[\hat{\hat{\alpha}}])=[\hat{\hat{\beta}}]*[\hat{\hat{\alpha}}]:=[\hat{\hat{\beta}}*\hat{\hat{\alpha}}]=[\!\!\begin{tikzcd} f\times_y g & \alpha\times_y\beta \arrow[l,two heads,"\alpha_l\times_y\beta_l"',"\sim"] \arrow[r,"\alpha_r\times_y\beta_r"] & f'\times_y g'\end{tikzcd}\!\!]$, where $\alpha_l\times_y\beta_l$ is an acyclic fibration due to Lemma~\ref{Lem-Pullback-2-Fib-AcycFib}.
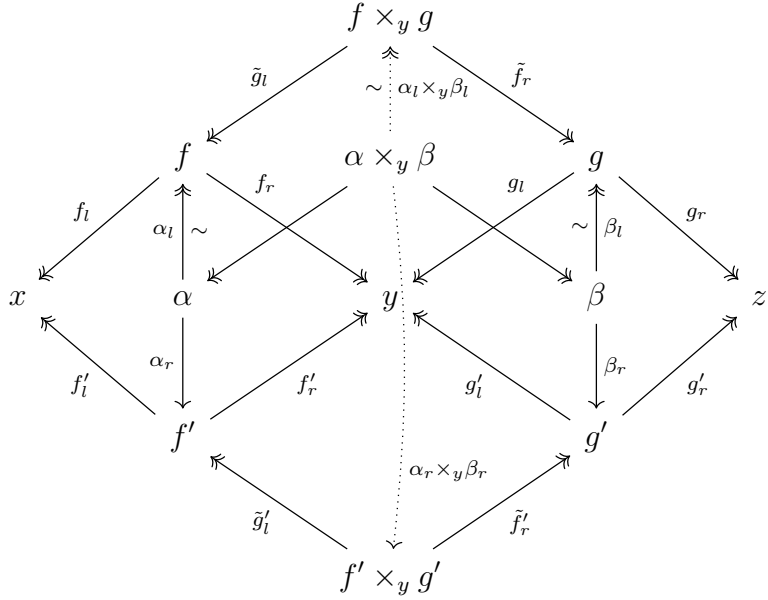
	\[
	\begin{tikzcd}[row sep=large,column sep=large]
		& & f\times_yg \arrow[ld,two heads,"\tilde{g}_l"'] \arrow[rd,two heads,"\tilde{f}_r"] & & \\
		& f \arrow[ld,two heads,"f_l"'] \arrow[rd,two heads,"f_r",near start] & \alpha\times_y\beta \arrow[u,two heads,dotted,"\alpha_l\times_y\beta_l"',"\sim"] \arrow[dl,two heads] \arrow[dr,two heads] \arrow[ddd,dotted,bend left=6,"\alpha_r\times_y\beta_r",near end] & g\arrow[ld,two heads,"g_l"',near start] \arrow[rd,two heads,"g_r"]  & \\
		x & \alpha \arrow[u,two heads,"\alpha_l","\sim"'] \arrow[d,"\alpha_r"'] & y & \beta \arrow[u,two heads,"\beta_l"',"\sim"] \arrow[d,"\beta_r"] & z \\
		& f' \arrow[ul,two heads,"f'_l"] \arrow[ru,two heads,"f'_r"']  & & g'\arrow[ul,two heads,"g'_l"] \arrow[ur,two heads,"g'_r" ']  & \\
		& & f'\times_yg' \arrow[lu,two heads,"\tilde{g}'_l"] \arrow[ru,two heads,"\tilde{f}'_r" '] & &
	\end{tikzcd}
	\]
	
	{\rm (5)} For all $w,x,y,z\in\mathsf{M}\-{\rm span}^2$, the component 2-cells of the associator $a_{wxyz}$ are $a_{\hat{h},\hat{g},\hat{f}} := [f\times_x(g\times_yh) \xleftarrow{1_{f\times_x(g\times_yh)}} f\times_x(g\times_yh) \xrightarrow{a_{f,g,h}} (f\times_x g)\times_y h] = [f\times_x(g\times_yh) \xleftarrow{a_{f,g,h}^{-1}} (f\times_xg)\times_yh \xrightarrow{1_{(f\times_xg)\times_yh}} (f\times_x g)\times_yh]$ for all 1-cells $\hat{f}=(w\twoheadleftarrow f \twoheadrightarrow x), \hat{g}=(x\twoheadleftarrow g\twoheadrightarrow y)$, and $\hat{h}=(y\twoheadleftarrow h\twoheadrightarrow z)$.
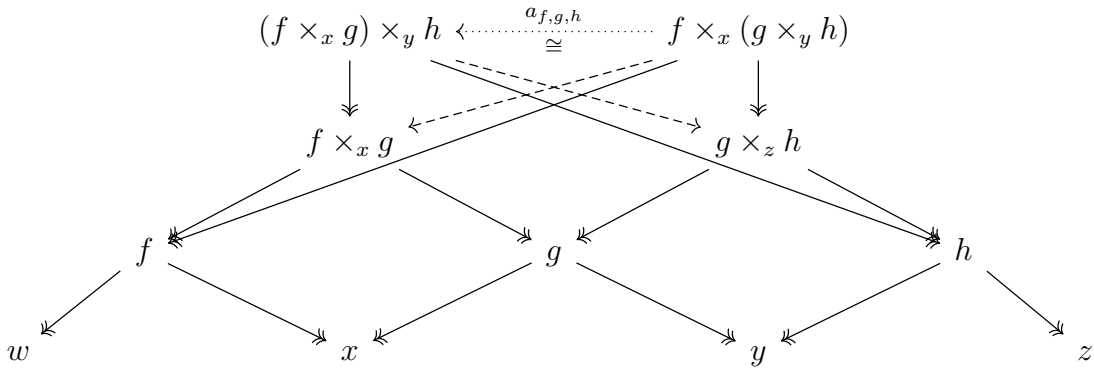
	\[
	\begin{tikzcd}
		&& (f\times_xg)\times_yh\arrow[d,two heads] \arrow[rrrdd,two heads] \arrow[rrd,dashed]
		&& f\times_x(g\times_y h)\arrow[d,two heads] \arrow[ddlll,two heads] \arrow[ll,dotted,"\cong","{a_{f,g,h}}"'] \arrow[lld,dashed] &&\\
		&& f\times_xg\arrow[ld,two heads] \arrow[rd,two heads] && g\times_zh\arrow[ld,two heads] \arrow[rd,two heads] &&\\
		& f\arrow[ld,two heads] \arrow[rd,two heads] && g\arrow[ld,two heads] \arrow[rd,two heads] && h\arrow[ld,two heads] \arrow[rd,two heads] & \\
		w && x && y && z		
	\end{tikzcd} 
	\]
	
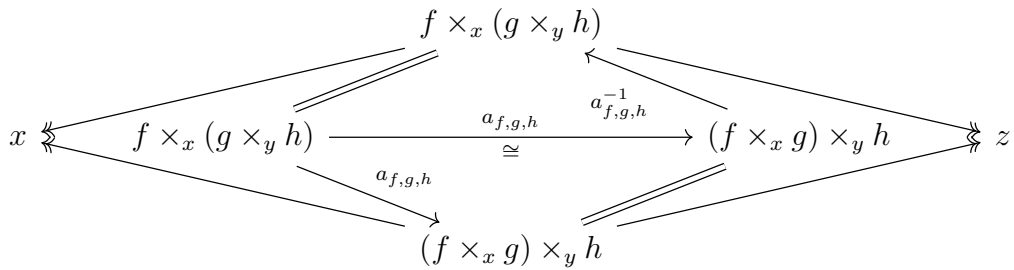
\[
	\begin{tikzcd}
		&& f\times_x(g\times_yh) \arrow[lld,two heads] \arrow[rrd,two heads] && \\
		x & f\times_x(g\times_yh) \arrow[ur,equal] \arrow[dr,"{a_{f,g,h}}"] \arrow[rr,"{a_{f,g,h}}","\cong"'] && (f\times_x g)\times_y h \arrow[ul,"{a_{f,g,h}^{-1}}"] \arrow[dl,equal] & z \\
		&& (f\times_x g)\times_yh \arrow[ull,two heads] \arrow[urr,two heads] &&
	\end{tikzcd} 
	\]
	
	{\rm (6)} For all $x,y\in\mathsf{M}\-{\rm span}^2$, the component 2-cells of the left unitor $l_{xy}$ are $l_{\hat{f}}:=[f\times_yy \xleftarrow{1_{f\times_yy}} f\times_yy\xrightarrow{l_f} f]=[f\times_yy \xleftarrow{l_f^{-1}} f\xrightarrow{1_f} f]$ for all 1-cell $\hat{f}=(x \twoheadleftarrow f\twoheadrightarrow y)$,
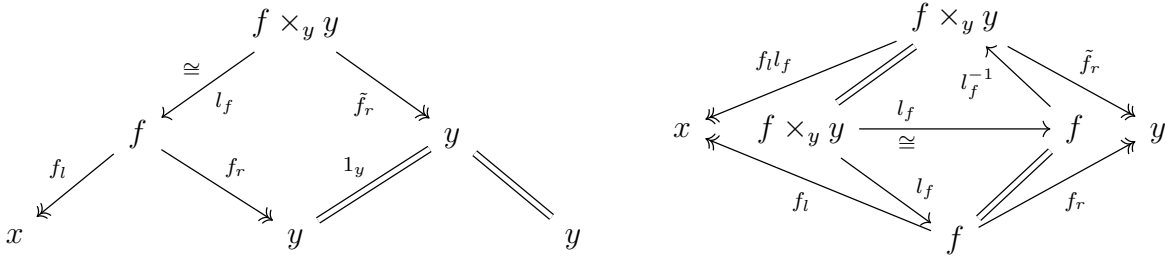
	\[
	\begin{tikzcd}
		&& f\times_yy \arrow[ld,"\cong"',"l_f"] \arrow[rd,two heads,"\tilde{f}_r"'] && \\
		&f\arrow[ld,two heads,"f_l"'] \arrow[rd,two heads,"f_r"] & & y \arrow[ld,equal,"1_y"'] \arrow[rd,equal] &\\
		x && y && y 
	\end{tikzcd} 
	\quad\quad
	\begin{tikzcd}[column sep=small]
		&& f\times_yy \arrow[lld,two heads,"f_ll_f"'] \arrow[rrd,two heads,"\tilde{f}_r"] && \\
		x & f\times_yy \arrow[ur,equal] \arrow[dr,"l_f",near end] \arrow[rr,"l_f","\cong"',near start] && f \arrow[ul,"l_f^{-1}",near end] \arrow[dl,equal] & y \\
		&& f \arrow[ull,two heads,"f_l"] \arrow[urr,two heads,"f_r"'] &&
	\end{tikzcd} 
	\]
	and the component 2-cells of the right unitor $r_x$ are $r_{\hat{f}}:=[x\times_xf \xleftarrow{1_{x\times_xf}} x\times_xf \xrightarrow{r_f} f] =[x\times_xf \xleftarrow{r_f^{-1}} f \xrightarrow{1_f} f]$ for all 1-cell $\hat{f}=(x \twoheadleftarrow f\twoheadrightarrow y)$.
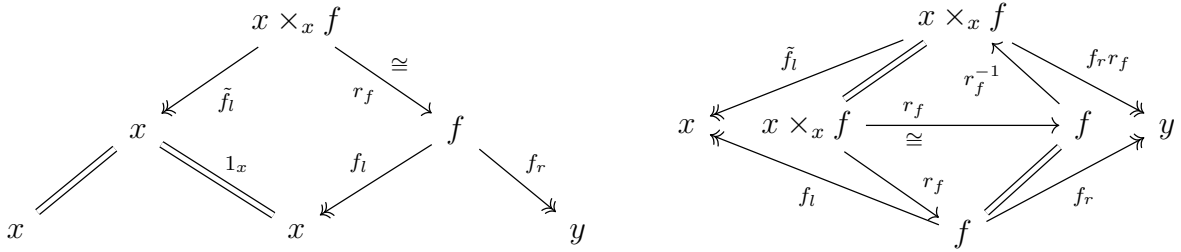
	\[
	\begin{tikzcd}
		&& x\times_xf \arrow[rd,"\cong","r_f"'] \arrow[ld,two heads,"\tilde{f}_l"] && \\
		& x\arrow[ld,equal] \arrow[rd,equal,"1_x"] & & f\arrow[ld,two heads,"f_l"'] \arrow[rd,two heads,"f_r"] &\\
		x && x && y 
	\end{tikzcd} 
	\quad\quad
	\begin{tikzcd}[column sep=small]
		&& x\times_x f \arrow[lld,two heads,"\tilde{f}_l"'] \arrow[rrd,two heads,"f_rr_f"] && \\
		x & x\times_x f \arrow[ur,equal] \arrow[dr,"r_f",near end] \arrow[rr,"r_f","\cong"',near start] && f \arrow[ul,"r_f^{-1}",near end] \arrow[dl,equal] & y \\
		&& f \arrow[ull,two heads,"f_l"] \arrow[urr,two heads,"f_r"'] &&
	\end{tikzcd} 
	\]
\end{theorem}

\begin{proof}
	(1) {\it Hom categories:} For any $x,y\in\mathsf{M}$, we need to check that $\mathsf{M}\-{\rm span}^2(x,y)$ is a category. 
	
	(1.1) The vertical composition of 2-cells is well-defined: 
	We need to show that the definition of vertical composition of 2-cells is independent of the choices of representatives of the equivalent classes of spans of spans. The problem can be reduced to the following two simple cases:
	
	(1.1.1) If $\hat{\hat{\alpha}}=(f \stackrel{\sim}{\twoheadleftarrow} \alpha \to f') \sim \hat{\hat{\beta}}=(f \stackrel{\sim}{\twoheadleftarrow} \beta \to f')$ is given by a morphism $\theta:\alpha\to\beta$ then $(f \stackrel{\sim}{\twoheadleftarrow} \alpha\times_{f'}\alpha' \to f'')\sim(f \stackrel{\sim}{\twoheadleftarrow} \beta\times_{f'}\alpha' \to f'')$ is given by the morphism $\theta\times_{f'}1_{\alpha'}: \alpha\times_{f'}\alpha'\to\beta\times_{f'}\alpha'$ for all spans of spans $\hat{\hat{\alpha}}'=(f '\stackrel{\sim}{\twoheadleftarrow} \alpha' \to f'')$.
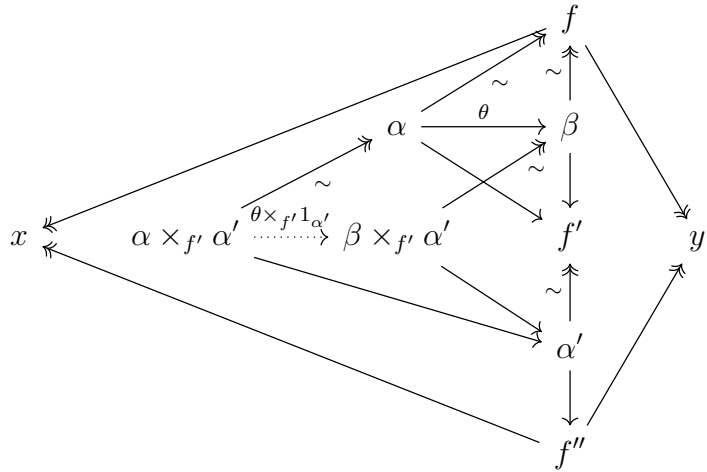
	\[
	\begin{tikzcd}
		&&& f \arrow[llldd,two heads] \arrow[rdd,two heads] & \\
		&& \alpha \arrow[ur,two heads,"\sim"'] \arrow[dr] \arrow[r,"\theta"] & \beta \arrow[u,two heads,"\sim"] \arrow[d] & \\
		x &\alpha\times_{f'}\alpha' \arrow[ur,two heads,"\sim"'] \arrow[drr] \arrow[r,dotted,"{\theta\times_{f'}1_{\alpha'}}"] & \beta\times_{f'}\alpha' \arrow[ur,two heads,"\sim"',near end] \arrow[dr] & f' & y \\
		&&& \alpha' \arrow[u,two heads,"\sim"] \arrow[d] & \\
		&&& f'' \arrow[uulll,two heads] \arrow[uur,two heads] &
	\end{tikzcd}
	\]
	
	(1.1.2) If $\hat{\hat{\alpha}}' = (f' \stackrel{\sim}{\twoheadleftarrow} \alpha' \to f'') \sim \hat{\hat{\beta}}' = (f' \stackrel{\sim}{\twoheadleftarrow} \beta' \to f'')$ is given by a morphism $\theta':\alpha'\to\beta'$ then $(f \stackrel{\sim}{\twoheadleftarrow} \alpha\times_{f'}\alpha' \to f'') \sim (f \stackrel{\sim}{\twoheadleftarrow} \alpha\times_{f'}\beta' \to f'')$ is given by the morphism $1_{\alpha}\times_{f'}\theta':\alpha\times_{f'}\alpha' \to \alpha\times_{f'}\beta'$ for all spans of spans $\hat{\hat{\alpha}} = (f\stackrel{\sim}{\twoheadleftarrow} \alpha \to f')$.
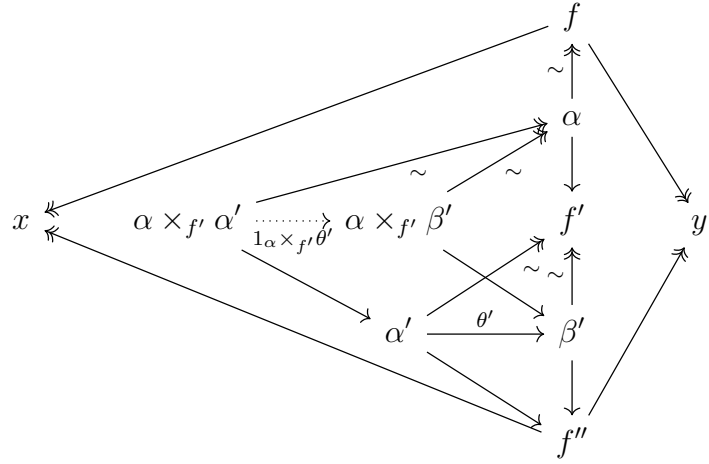
	\[
	\begin{tikzcd}
		&&& f \arrow[llldd,two heads] \arrow[rdd,two heads] & \\
		&& & \alpha \arrow[u,two heads,"\sim"] \arrow[d] & \\
		x &\alpha\times_{f'}\alpha' \arrow[urr,two heads,"\sim"'] \arrow[dr] \arrow[r,dotted,"{1_{\alpha}\times_{f'}\theta'}"'] & \alpha\times_{f'}\beta' \arrow[ur,two heads,"\sim"'] \arrow[dr] & f' & y \\
		&& \alpha' \arrow[ur,two heads,"\sim"',near end] \arrow[dr] \arrow[r,"{\theta'}"] & \beta' \arrow[u,two heads,"\sim"] \arrow[d] & \\
		&&& f'' \arrow[uulll,two heads] \arrow[uur,two heads] &
	\end{tikzcd}
	\]		
	
	(1.2) The associativity of vertical compositions:		
	For all 2-cells $[\hat{\hat{\alpha}}]=[f\stackrel{\sim}{\twoheadleftarrow} \alpha\to f'], [\hat{\hat{\alpha}}']=[f'\stackrel{\sim}{\twoheadleftarrow} \alpha'\to f'']$ and $[\hat{\hat{\alpha}}'']= [f''\stackrel{\sim}{\twoheadleftarrow} \alpha''\to f''']$, we need to prove $([\hat{\hat{\alpha}}'']\circ[\hat{\hat{\alpha}}'])\circ[\hat{\hat{\alpha}}]= [\hat{\hat{\alpha}}'']\circ([\hat{\hat{\alpha}}']\circ[\hat{\hat{\alpha}}])$.
	
	It is enough to show that the spans of spans $(\hat{\hat{\alpha}}''\circ\hat{\hat{\alpha}}')\circ\hat{\hat{\alpha}}$ and $\hat{\hat{\alpha}}''\circ(\hat{\hat{\alpha}}'\circ\hat{\hat{\alpha}})$ are equivalent, i.e., 
	$(f \stackrel{\sim}{\twoheadleftarrow} \alpha\times_{f'}(\alpha'\times_{f''}\alpha'') \to f''') \sim (f \stackrel{\sim}{\twoheadleftarrow} (\alpha\times_{f'}\alpha')\times_{f''}\alpha'' \to f''')$. It is given by the isomorphism $a_{\alpha,\alpha',\alpha''}: \alpha\times_{f'}(\alpha'\times_{f''}\alpha'') \to (\alpha\times_{f'}\alpha')\times_{f''}\alpha''$ obtained in the following commutative diagram.
	
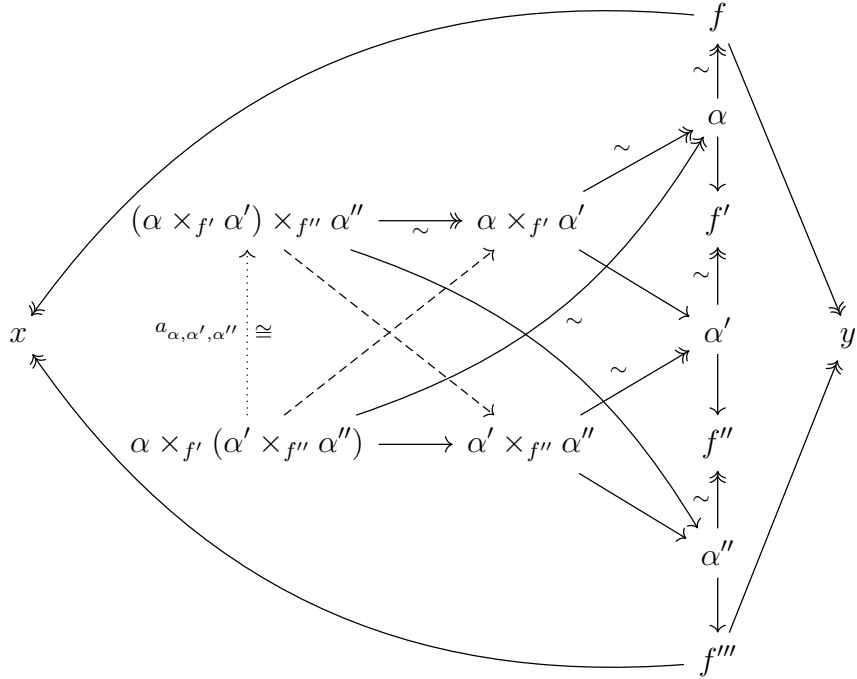
\[
	\begin{tikzcd}	
		&&& f \arrow[lllddd,two heads,bend right=30] \arrow[rddd,two heads] & \\
		&&& \alpha \arrow[u,two heads,"\sim"] \arrow[d] & \\
		& (\alpha\times_{f'}\alpha')\times_{f''}\alpha'' \arrow[r,two heads,"\sim"'] \arrow[rrddd,bend left=20] \arrow[ddr,dashed]  & \alpha\times_{f'}\alpha'\arrow[ur,two heads,"\sim"] \arrow[dr] & f' & \\
		x &&& \alpha' \arrow[u,two heads,"\sim"] \arrow[d] & y \\
		& \alpha\times_{f'}(\alpha'\times_{f''}\alpha'') \arrow[r] \arrow[uur,dashed] \arrow[rruuu,bend right=20,two heads,"\sim"'] \arrow[uu,dotted,"\cong"',"a_{\alpha,\alpha',\alpha''}"] & \alpha'\times_{f''}\alpha''\arrow[ur,two heads,"\sim"] \arrow[dr] & f'' & \\
		&&& \alpha'' \arrow[u,two heads,"\sim"] \arrow[d] & \\
		&&& f''' \arrow[uuulll,two heads,bend left=30] \arrow[uuur,two heads] & 
	\end{tikzcd}
	\]				
	
	(1.3) The unity of vertical compositions: For any 2-cell $[\!\!\begin{tikzcd}[column sep=small] f & \alpha \arrow[l,two heads,"\alpha_l"',"\sim"] \arrow[r,"\alpha_r"] & f'\end{tikzcd}\!\!]$, we need to prove $[\!\!\begin{tikzcd}[column sep=small] f & \alpha \arrow[l,two heads,"\alpha_l"',"\sim"] \arrow[r,"\alpha_r"] & f'\end{tikzcd}\!\!] \circ [f\xleftarrow{1_f}f\xrightarrow{1_f} f] = [\!\!\begin{tikzcd}[column sep=small] f & \alpha \arrow[l,two heads,"\alpha_l"',"\sim"] \arrow[r,"\alpha_r"] & f'\end{tikzcd}\!\!]$ and $[f'\xleftarrow{1_{f'}}f'\xrightarrow{1_{f'}} f'] \circ [\!\!\begin{tikzcd}[column sep=small] f & \alpha \arrow[l,two heads,"\alpha_l"',"\sim"] \arrow[r,"\alpha_r"] & f'\end{tikzcd}\!\!] = [\!\!\begin{tikzcd}[column sep=small] f & \alpha \arrow[l,two heads,"\alpha_l"',"\sim"] \arrow[r,"\alpha_r"] & f'\end{tikzcd}\!\!]$, i.e., $(\!\!\begin{tikzcd}[column sep=small] f & f\times_f\alpha \arrow[l,two heads,"\tilde{\alpha}_l"',"\sim"] \arrow[r,"\alpha_rr_\alpha"] & f'\end{tikzcd}\!\!) \sim (\!\!\begin{tikzcd}[column sep=small] f & \alpha \arrow[l,two heads,"\alpha_l"',"\sim"] \arrow[r,"\alpha_r"] & f'\end{tikzcd}\!\!)$ and $(\!\!\begin{tikzcd}[column sep=small] f & \alpha\times_{f'}f' \arrow[l,two heads,"\alpha_ll_\alpha"',"\sim"] \arrow[r,"\tilde{\alpha}_r"] & f'\end{tikzcd}\!\!) \sim (\!\!\begin{tikzcd}[column sep=small] f & \alpha \arrow[l,two heads,"\alpha_l"',"\sim"] \arrow[r,"\alpha_r"] & f'\end{tikzcd}\!\!)$. 
	
	Indeed, they are given by the isomorphisms $r_\alpha: f\times_f\alpha \to \alpha$ and $l_\alpha: \alpha\times_{f'}f' \to \alpha$.
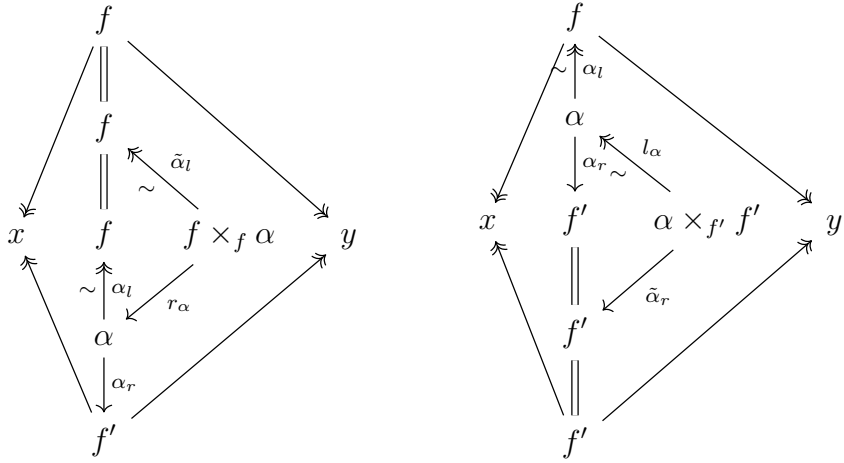
	\[
	\begin{tikzcd}[column sep=small]
		& f \arrow[ldd,two heads] \arrow[ddrr,two heads] &&  \\
		& f \arrow[u,equal] \arrow[d,equal] && \\
		x & f & f\times_f\alpha \arrow[ul,two heads,"\sim","\tilde{\alpha}_l"'] \arrow[ld,"r_\alpha"] & y \\
		& \alpha \arrow[u,two heads,"\sim","\alpha_l"'] \arrow[d,"\alpha_r"] && \\
		& f' \arrow[uul,two heads] \arrow[uurr,two heads] & & 
	\end{tikzcd}
	\quad\quad\quad
	\begin{tikzcd}[column sep=small]
		& f \arrow[ldd,two heads] \arrow[ddrr,two heads] & &  \\
		& \alpha \arrow[u,two heads,"\sim","\alpha_l"'] \arrow[d,"\alpha_r"] && \\
		x & f' & \alpha\times_{f'}f' \arrow[ul,two heads,"\sim","l_\alpha"'] \arrow[ld,"\tilde{\alpha}_r"] & y \\
		& f' \arrow[u,equal] \arrow[d,equal] &&& \\
		& f' \arrow[uul,two heads] \arrow[uurr,two heads] & &  
	\end{tikzcd}
	\]		
	
	\medspace
	
	(2) {\it Horizontal compositions:} We need to show that the horizontal composition $$c_{xyz}:\mathsf{M}\-{\rm span}^2(y,z)\times\mathsf{M}\-{\rm span}^2(x,y) \rightarrow \mathsf{M}\-{\rm span}^2(x,z)$$ is a functor.
	
	(2.1) The horizontal composition is well-defined:
	It suffices to prove that the horizontal compositions of 2-cells are independent of the choices of representatives of equivalence classes. This is similar to the proof of (1.1).
	
	(2.2) 
	The horizontal composition $c_{xyz}$ preserves the composition of morphisms: 
	For any 2-cells $[\hat{\hat{\alpha}}]=[\!\!\begin{tikzcd}[column sep=small] f & \alpha \arrow[l,two heads,"\alpha_l"',"\sim"] \arrow[r,"\alpha_r"] & f'\end{tikzcd}\!\!], [\hat{\hat{\alpha}}']=[\!\!\begin{tikzcd}[column sep=small] f' & \alpha' \arrow[l,two heads,"\alpha'_l"',"\sim"] \arrow[r,"\alpha'_r"] & f''\end{tikzcd}\!\!], [\hat{\hat{\beta}}]=[\!\!\begin{tikzcd}[column sep=small] g & \beta \arrow[l,two heads,"\beta_l"',"\sim"] \arrow[r,"\beta_r"] & g'\end{tikzcd}\!\!], [\hat{\hat{\beta}}']=[\!\!\begin{tikzcd}[column sep=small] g' & \beta' \arrow[l,two heads,"\beta'_l"',"\sim"] \arrow[r,"\beta'_r"] & g''\end{tikzcd}\!\!]$, we need to prove $([\hat{\hat{\beta}}']*[\hat{\hat{\alpha}}'])\circ([\hat{\hat{\beta}}]*[\hat{\hat{\alpha}}]) = ([\hat{\hat{\beta}}']\circ[\hat{\hat{\beta}}])*([\hat{\hat{\alpha}}']\circ[\hat{\hat{\alpha}}])$, i.e.,
	\[
	\begin{array}{ll}
		& [f'\times_yg'\stackrel{\sim}{\twoheadleftarrow} \alpha'\times_y\beta'\rightarrow f''\times_yg''] \circ [f\times_yg\stackrel{\sim}{\twoheadleftarrow} \alpha\times_y\beta\rightarrow f'\times_yg'] \\
		= &
		[g\stackrel{\sim}{\twoheadleftarrow} \beta\times_{g'}\beta'\rightarrow g''] *
		[f\stackrel{\sim}{\twoheadleftarrow} \alpha\times_{f'}\alpha'\rightarrow f''],
	\end{array}
	\]
	i.e., 
	\[
	\begin{array}{ll}
		& [f\times_yg \stackrel{\sim}{\twoheadleftarrow} (\alpha\times_y\beta)\times_{f'\times_y g'}(\alpha'\times_y\beta') \rightarrow f''\times_yg''] \\
		= & [f\times_yg\stackrel{\sim}{\twoheadleftarrow} (\alpha\times_{f'}\alpha')\times_y(\beta\times_{g'}\beta') \rightarrow f''\times_yg''],
	\end{array}
	\]
	i.e., 
	\[
	\begin{array}{ll}
		& (f\times_yg \stackrel{\sim}{\twoheadleftarrow} (\alpha\times_y\beta)\times_{f'\times_y g'}(\alpha'\times_y\beta') \rightarrow f''\times_yg'') \\
		\sim & (f\times_yg\stackrel{\sim}{\twoheadleftarrow} (\alpha\times_{f'}\alpha')\times_y(\beta\times_{g'}\beta') \rightarrow f''\times_yg'').
	\end{array}
	\]
	It is given by the isomorphism $\tau:(\alpha\times_{f'}\alpha')\times_y(\beta\times_{g'}\beta')\to(\alpha\times_y\beta)\times_{f'\times_y g'}(\alpha'\times_y\beta')$ obtained in the following commutative diagram.
	
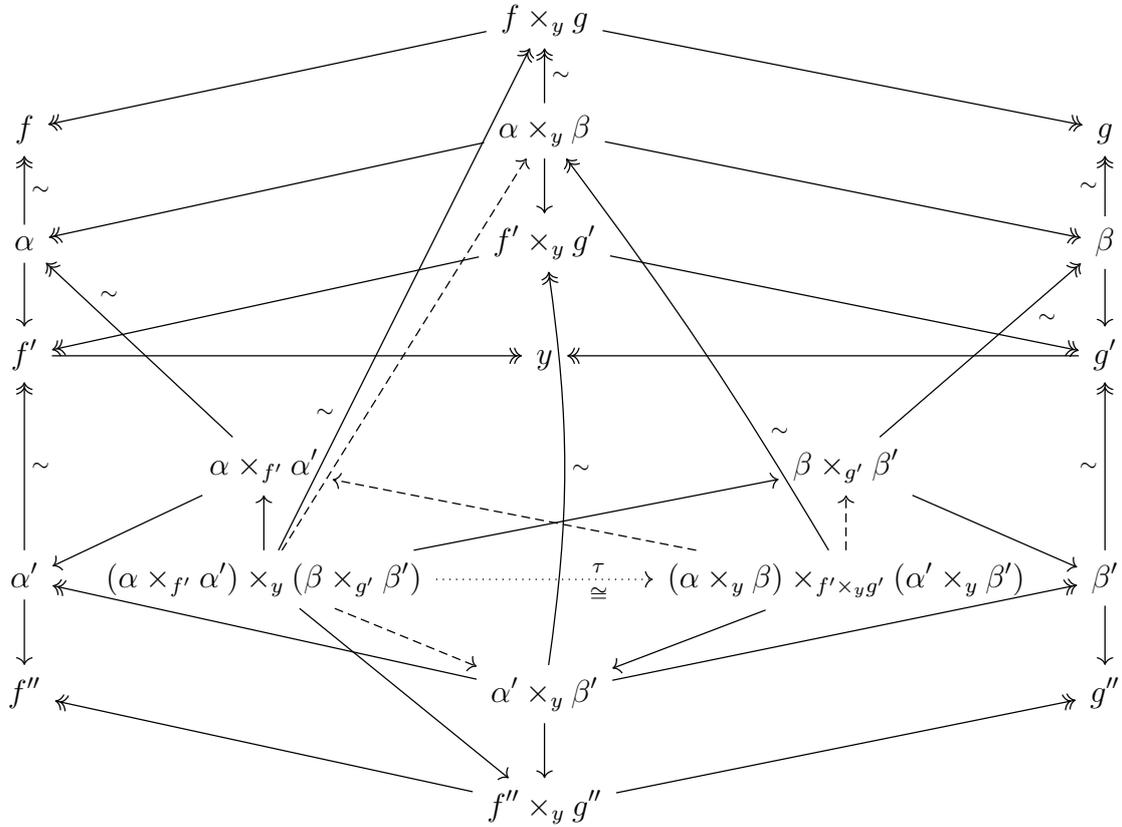
\[
	\begin{tikzcd}[column sep=small]
		&& f\times_yg \arrow[dll,two heads]\arrow[drr,two heads] && \\
		f && \alpha\times_y\beta \arrow[dll,two heads]\arrow[drr,two heads] \arrow[u,two heads,"\sim"'] \arrow[d] && g \\
		\alpha \arrow[u,two heads,"\sim"'] \arrow[d] && f'\times_yg' \arrow[dll,two heads]\arrow[drr,two heads] && \beta \arrow[u,two heads,"\sim"] \arrow[d] \\
		f' \arrow[rr,two heads] && y && g' \arrow[ll,two heads] \\
		& \alpha\times_{f'}\alpha' \arrow[uul,two heads,"\sim"',near end] \arrow[dl] && \beta\times_{g'}\beta' \arrow[uur,two heads,"\sim"',near end]\arrow[dr] & \\	
		\alpha'\arrow[uu,two heads,"\sim"'] \arrow[d] & (\alpha\times_{f'}\alpha')\times_y(\beta\times_{g'}\beta') \arrow[uuuuur,two heads,"\sim",near start] \arrow[uuuur,dashed] \arrow[dr,dashed] \arrow[ddr] \arrow[rr,dotted,"\tau","\cong"',near end] \arrow[u] \arrow[urr] && (\alpha\times_y\beta)\times_{f'\times_y g'}(\alpha'\times_y\beta') \arrow[uuuul,two heads,bend right=3,"\sim"',near start] \arrow[ull,dashed] \arrow[dl] \arrow[u,dashed] & \beta' \arrow[uu,two heads,"\sim"] \arrow[d] \\
		f'' && \alpha'\times_y\beta' \arrow[ull,two heads] \arrow[urr,two heads] \arrow[uuuu,two heads,bend right=8,"\sim"'] \arrow[d] && g''\\
		&& f''\times_yg'' \arrow[ull,two heads]\arrow[urr,two heads] &&
	\end{tikzcd}
	\]		
	
	(2.3) The horizontal composition $c_{xyz}$ preserves identity: For all 1-cells $\hat{f}=(x \leftarrow f \rightarrow y)$ and $\hat{g}=(y \leftarrow g \rightarrow z)$, $c_{xyz}(1_{\hat{g}},1_{\hat{f}}) = [f\times_yg \xleftarrow{1_f\times_y1_g} f\times_yg \xrightarrow{1_f\times_y1_g} f\times_yg] = [f\times_yg \xleftarrow{1_{f\times_yg}} f\times_yg \xrightarrow{1_{f\times_yg}} f\times_yg] = 1_{c_{xyz}(\hat{g},\hat{f})}$.
	
	(3) {\it Associators:} For all $w,x,y,z\in\mathsf{M}\-{\rm span}^2$, we need to show that $a_{wxyz}$ is a natural isomorphism. For all 1-cells $\hat{f}_i=(w \leftarrow f_i \rightarrow x), \hat{g}_i=(x\leftarrow g_i \rightarrow y)$ and $\hat{h}_i=(y \leftarrow h_i \rightarrow z)$, $i=1,2$, and 2-cells $[\hat{\hat{\alpha}}]=[\hat{f}_1\leftarrow \hat{\alpha}\rightarrow \hat{f}_2], [\hat{\hat{\beta}}]=[\hat{g}_1\leftarrow \hat{\beta}\rightarrow \hat{g}_2]$ and $[\hat{\hat{\gamma}}]=[\hat{h}_1\leftarrow\hat{\gamma}\rightarrow \hat{h}_2]$, we need to prove the following diagram is commutative.
	\[
	\begin{tikzcd}
		(\hat{h}_1\circ\hat{g}_1)\circ\hat{f}_1 \arrow[r,"a_{\hat{h}_1,\hat{g}_1,\hat{f}_1}"] \arrow[d,"{([\hat{\hat{\gamma}}]*[\hat{\hat{\beta}}])*[\hat{\hat{\alpha}}]}"'] & \hat{h}_1\circ(\hat{g}_1\circ\hat{f}_1) \arrow[d,"{[\hat{\hat{\gamma}}]*([\hat{\hat{\beta}}]*[\hat{\hat{\alpha}}])}"] \\
		(\hat{h}_2\circ\hat{g}_2)\circ\hat{f}_2 \arrow[r,"a_{\hat{h}_2,\hat{g}_2,\hat{f}_2}"] & \hat{h}_2\circ(\hat{g}_2\circ\hat{f}_2)
	\end{tikzcd}
	\]	
	There is an isomorphism $a_{\alpha,\beta,\gamma} : \alpha\times_x(\beta\times_y\gamma)\to(\alpha\times_x\beta)\times_y\gamma$ such that the following diagram is commutative.	
	{\footnotesize \[
	\begin{tikzcd}[row sep=30]
		&& (f_1\times_xg_1)\times_yh_1 \arrow[d,two heads] \arrow[ddrrr,two heads] && f_1\times_x(g_1\times_yh_1) \arrow[ll,"a_{f_1,g_1,h_1}"',"{\cong}"] \arrow[d,two heads] \arrow[ddlll,two heads] && \\	
		&& f_1\times_xg_1 \arrow[dl,two heads] \arrow[dr,two heads] && g_1\times_yh_1 \arrow[dl,two heads] \arrow[dr,two heads] && \\	
		& f_1 \arrow[dl,two heads] \arrow[dr,two heads] & \alpha\times_x\beta \arrow[u,two heads,near start,"\sim","\alpha_l\times_x\beta_l"'] \arrow[ddd,bend left=50,near start,"\alpha_r\times_x\beta_r"description] \arrow[dl,two heads] \arrow[dr,two heads] & g_1 \arrow[dl,two heads] \arrow[dr,two heads] & \beta\times_y\gamma \arrow[u,two heads,near start,"\sim","\beta_l\times_y\gamma_l"'] \arrow[ddd,bend right=50,near start,"{\beta_r\times_y\gamma_r}"description] \arrow[dl,two heads] \arrow[dr,two heads] & h_1 \arrow[dl,two heads] \arrow[dr,two heads] & \\	
		w & \alpha \arrow[u,two heads,"\sim","\alpha_l"'] \arrow[d,"\alpha_r"] & x & \beta \arrow[u,two heads,"\sim","\beta_l"'] \arrow[d,"\beta_r"] & y & \gamma \arrow[u,two heads,"\sim","\gamma_l"'] \arrow[d,"\gamma_r"] & z \\
		& f_2 \arrow[ul,two heads] \arrow[ur,two heads] & (\alpha\times_x\beta)\times_y\gamma \arrow[uuuu,two heads,bend left=50,"\sim"',"(\alpha_l\times_x\beta_l)\times_y\gamma_l",near end] \arrow[dd,bend right=50,"(\alpha_r\times_x\beta_r)\times_y\gamma_r"'] \arrow[uu,bend left=50] \arrow[urrr] \arrow[ul,dashed] \arrow[uurr,dashed,bend right=8] & g_2 \arrow[ul,two heads] \arrow[ur,two heads] & \alpha\times_x(\beta\times_y\gamma) \arrow[ll,dotted,bend left=30,"a_{\alpha,\beta,\gamma}"',"{\cong}"] \arrow[uuuu,two heads,bend right=50,"\sim","\alpha_l\times_x(\beta_l\times_y\gamma_l)"',near end] \arrow[dd,bend left=50,"\alpha_r\times_x(\beta_r\times_y\gamma_r)"] \arrow[ulll] \arrow[uu,bend right=50] \arrow[ur,dashed] \arrow[uull,dashed,bend left=6] & h_2 \arrow[ul,two heads] \arrow[ur,two heads] & \\
		&& f_2\times_xg_2 \arrow[ul,two heads] \arrow[ur,two heads] && g_2\times_yh_2 \arrow[ul,two heads] \arrow[ur,two heads] && \\
		&& (f_2\times_xg_2)\times_yh_2 \arrow[u,two heads]  \arrow[uurrr,two heads] && f_2\times_x(g_2\times_yh_2) \arrow[u,two heads] \arrow[uulll,two heads] \arrow[ll,"a_{f_2,g_2,h_2}"',"{\cong}"] &&
	\end{tikzcd}
	\]}
	Thus
	\[
	\begin{array}{ll}
		& ([\hat{\hat{\gamma}}]*([\hat{\hat{\beta}}]*[\hat{\hat{\alpha}}])) \circ a_{\hat{h}_1,\hat{g}_1,\hat{f}_1} \\ [2mm]
		= & [(f_1\times_xg_1)\times_yh_1 \xleftarrow{(\alpha_l\times_x\beta_l)\times_y\gamma_l} (\alpha\times_x\beta)\times_y\gamma \xrightarrow{(\alpha_r\times_x\beta_r)\times_y\gamma_r} (f_2\times_xg_2)\times_yh_2] \\ [2mm]
		& \circ [f_1\times_x(g_1\times_yh_1) \xleftarrow{a_{f_1,g_1,h_1}^{-1}} (f_1\times_xg_1)\times_yh_1 =(f_1\times_xg_1)\times_yh_1] \\ [2mm]
		= & [f_1\times_x(g_1\times_yh_1) \xleftarrow{a_{f_1,g_1,h_1}^{-1}((\alpha_l\times_x\beta_l)\times_y\gamma_l)} (\alpha\times_x\beta)\times_y\gamma \xrightarrow{(\alpha_r\times_x\beta_r)\times_y\gamma_r} (f_2\times_xg_2)\times_yh_2] \\ [2mm]
		= & [f_1\times_x(g_1\times_yh_1) \xleftarrow{(\alpha_l\times_x(\beta_l\times_y\gamma_l)) a_{\alpha,\beta,\gamma}^{-1}} (\alpha\times_x\beta)\times_y\gamma \xrightarrow{(\alpha_r\times_x\beta_r)\times_y\gamma_r} (f_2\times_xg_2)\times_yh_2] \\ [2mm]
		= & [f_1\times_x(g_1\times_yh_1) \xleftarrow{\alpha_l\times_x(\beta_l\times_y\gamma_l) } \alpha\times_x(\beta\times_y\gamma) \xrightarrow{((\alpha_r\times_x\beta_r)\times_y\gamma_r)a_{\alpha,\beta,\gamma}} (f_2\times_xg_2)\times_yh_2] \\ [2mm]
		= & [f_1\times_x(g_1\times_yh_1) \xleftarrow{\alpha_l\times_x(\beta_l\times_y\gamma_l)} \alpha\times_x(\beta\times_y\gamma) \xrightarrow{a_{f_2,g_2,h_2}(\alpha_r\times_x(\beta_r\times_y\gamma_r))} (f_2\times_xg_2)\times_yh_2] \\ [2mm]
		= & [f_2\times_x(g_2\times_yh_2) = f_2\times_x(g_2\times_yh_2) \xrightarrow{a_{f_2,g_2,h_2}} (f_2\times_xg_2)\times_yh_2] \\ [2mm]
		& \circ [f_1\times_x(g_1\times_yh_1) \xleftarrow{\alpha_l\times_x(\beta_l\times_y\gamma_l)} \alpha\times_x(\beta\times_y\gamma) \xrightarrow{\alpha_r\times_x(\beta_r\times_y\gamma_r)} f_2\times_x(g_2\times_yh_2)] \\ [2mm]
		= & a_{\hat{h}_2,\hat{g}_2,\hat{f}_2}\circ ((\hat{\hat{\gamma}}*\hat{\hat{\beta}})*\hat{\hat{\alpha}}).
	\end{array}
	\]			
	
	(4) {\it Unitors:} For all $x,y\in\mathsf{M}\-{\rm span}^2$, we need to prove that the left unitor $l_{xy}$ and the right unitor $r_{xy}$ are natural isomorphisms. We just show that $l_{xy}$ is a natural isomorphism. The proof for $r_{xy}$ to be a natural isomorphism	is similar.
	
	For all 1-cells $\hat{f}_i=(x\twoheadleftarrow f_i\twoheadrightarrow y)$, $i=1,2$, and 2-cell $[\hat{\hat{\alpha}}]=[\hat{f}_1\leftarrow \hat{\alpha}\rightarrow \hat{f}_2] =[f_1\stackrel{\sim}{\twoheadleftarrow} \alpha\rightarrow f_2]$, we need to show the following diagram is commutative.
	\[
	\begin{tikzcd}
		1_y\circ \hat{f}_1 \arrow[r,"l_{\hat{f}_1}"] \arrow[d,"{1_{1_y}*[\hat{\hat{\alpha}}]}"'] & \hat{f}_1 \arrow[d,"{[\hat{\hat{\alpha}}]}"] \\
		1_y\circ \hat{f}_2 \arrow[r,"l_{\hat{f}_2}"] & \hat{f}_2  
	\end{tikzcd}
	\]
	We have the following commutative diagram.
	\[
	\begin{tikzcd}[column sep=huge]
		&& f_1\times_yy \arrow[dl,"{\cong}","l_{f_1}"'] \arrow[dr,two heads] &&\\		
		& f_1 \arrow[dl,two heads] \arrow[dr,two heads] & \alpha\times_yy \arrow[u,two heads,"\sim","\alpha_l\times_y1_y"'] \arrow[ddd,bend left=8,"\alpha_r\times_y1_y",near end] \arrow[dl,"{\cong}","l_\alpha"',near start] \arrow[dr,two heads] & y \arrow[dl,equal] \arrow[dr,equal] &\\	
		x & \alpha \arrow[u,two heads,"\sim","\alpha_l"'] \arrow[d,"\alpha_r"] & y & y\arrow[u,equal] \arrow[d,equal] & y \\
		& f_2 \arrow[ul,two heads] \arrow[ur,two heads] & & y \arrow[ul,equal] \arrow[ur,equal]&\\	
		&& f_2\times_yy \arrow[ul,"{\cong}"',"l_{f_2}"] \arrow[ur,two heads] &&
	\end{tikzcd}
	\]	
	Thus	
	\[
	\begin{array}{ll}
	& [\hat{\hat{\alpha}}]\circ l_{\hat{f}_1} \\ [2mm]
	= & [f_1 \xleftarrow{\alpha_l} \alpha \xrightarrow{\alpha_r} f_2] \circ [f_1\times_yy \xleftarrow{l_{f_1}^{-1}} f_1=f_1] 
	= [f_1\times_yy \xleftarrow{l_{f_1}^{-1}\alpha_l} \alpha \xrightarrow{\alpha_r} f_2] \\ [2mm]
	= & [f_1\times_yy \xleftarrow{\alpha_l\times_y1_y} \alpha\times_yy \xrightarrow{\alpha_rl_\alpha} f_2] 
	=  [f_1\times_yy \xleftarrow{\alpha_l\times_y1_y} \alpha\times_yy \xrightarrow{l_{f_2}(\alpha_r\times_y1_y)} f_2] \\ [2mm]
	= & [f_2\times_yy = f_2\times_yy \xrightarrow{l_{f_2}} f_2] \circ [f_1\times_yy \xleftarrow{\alpha_l\times_y1_y} \alpha\times_yy \xrightarrow{\alpha_r\times_y1_y} f_2\times_yy] \\ [2mm]
	= & l_{\hat{f}_2}\circ(1_{1_y}*[\hat{\hat{\alpha}}])
	\end{array}
	\]
	
	(5) {\it Unity Axiom:} For all 1-cells $\hat{f}=(x\twoheadleftarrow f\twoheadrightarrow y)$ and $\hat{g}=(y\twoheadleftarrow g \twoheadrightarrow z)$, we need to prove the following diagram is commutative.
	\[
	\begin{tikzcd}
		(\hat{g}\circ 1_y)\circ \hat{f} \arrow[rr,"a_{\hat{g},1_y,\hat{f}}"] \arrow[rd,"r_{\hat{g}}*1_{\hat{f}}"'] & & \hat{g}\circ(1_y\circ \hat{f}) \arrow[ld,"1_{\hat{g}}*l_{\hat{f}}"] \\
		& \hat{g}\circ \hat{f} &
	\end{tikzcd}
	\]		
	There is an isomorphism $a_{f,1_y,g}: f\times_y(y\times_yg) \to (f\times_yy)\times_yg $ such that the following diagram is commutative.
	\[
	\begin{tikzcd}
		&& (f\times_yy)\times_yg \arrow[d,two heads] \arrow[rrrdd,two heads] \arrow[rrd,dashed]  \arrow[dddddr,,"{\cong}","l_f\times_y1_g"',near end]
		&& f\times_y(y\times_y g) \arrow[d,two heads] \arrow[ddlll,two heads] \arrow[ll,dotted,"\cong","{a_{f,1_y,g}}"'] \arrow[lld,dashed] \arrow[dddddl,"{\cong}"',"1_f\times_yr_g",near end] &&\\
		&& f\times_yy \arrow[ld,"{\cong}","l_f"',near start] \arrow[rd,two heads] && y\times_yg \arrow[ld,two heads] \arrow[rd,,"{\cong}"',"r_g",near start] &&\\
		& f \arrow[ld,two heads] \arrow[rd,two heads] && y \arrow[ld,equal] \arrow[rd,equal] && g \arrow[ld,two heads] \arrow[rd,two heads] & \\
		x & f \arrow[u,equal] \arrow[d,equal] & y && y & g \arrow[u,equal] \arrow[d,equal] & z \\
		& f \arrow[lu,two heads] \arrow[ru,two heads] & && & g \arrow[lu,two heads] \arrow[ru,two heads] & \\	
		&&& f\times_yg \arrow[llu,two heads] \arrow[rru,two heads] &&&
	\end{tikzcd} 
	\]
	Thus
	\[
	\begin{array}{ll}
	& (1_{\hat{g}}*l_{\hat{f}})\circ a_{\hat{g},1_y,\hat{f}} \\ [2mm]
	= & [(f\times_yy)\times_yg = (f\times_yy)\times_yg \xrightarrow{l_f\times_y1_g} f\times_yg] \\ [2mm]
	& \circ [f\times_y(y\times_y g) =f\times_y(y\times_y g) \xrightarrow{a_{f,1_y,g}} (f\times_yy)\times_yg] \\ [2mm]
	= & [f\times_y(y\times_y g) = f\times_y(y\times_y g) \xrightarrow{(l_f\times_y1_g)a_{f,1_y,g}} f\times_yg] \\ [2mm]
	= & [f\times_y(y\times_y g) = f\times_y(y\times_y g) \xrightarrow{1_f\times_yr_g} f\times_yg] \\ [2mm]
	= & r_{\hat{g}}*1_{\hat{f}}.
	\end{array}
	\] 
	
	(6) {\it Pentagon Axiom:} For all 1-cells $\hat{f}=(v \twoheadleftarrow f \twoheadrightarrow w), \hat{g}=(w\twoheadleftarrow g\twoheadrightarrow x), \hat{h}=(x\twoheadleftarrow h\twoheadrightarrow y)$ and $\hat{k}=(y\twoheadleftarrow k \twoheadrightarrow z)$, we need to prove the following diagram is commutative.
	
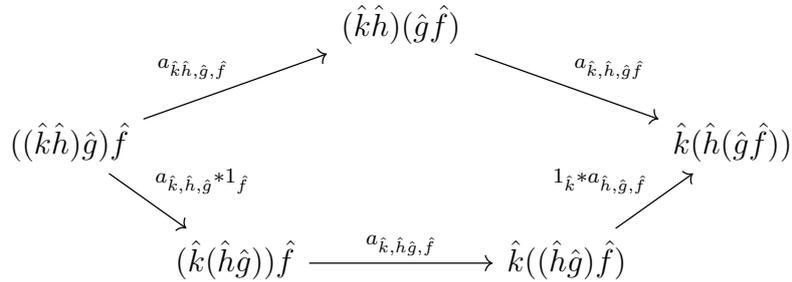
\[
	\begin{tikzcd}[column sep=tiny]
		& & (\hat{k}\hat{h})(\hat{g}\hat{f}) \arrow[rrd,"a_{\hat{k},\hat{h},\hat{g}\hat{f}}"] & & \\
		((\hat{k}\hat{h})\hat{g})\hat{f} \arrow[urr,"a_{\hat{k}\hat{h},\hat{g},\hat{f}}"] \arrow[rd,"a_{\hat{k},\hat{h},\hat{g}}*1_{\hat{f}}"] & & & & \hat{k}(\hat{h}(\hat{g}\hat{f})) \\
		& (\hat{k}(\hat{h}\hat{g}))\hat{f} \arrow[rr,"a_{\hat{k},\hat{h}\hat{g},\hat{f}}"] & & \hat{k}((\hat{h}\hat{g})\hat{f}) \arrow[ru,"1_{\hat{k}}*a_{\hat{h},\hat{g},\hat{f}}"] &
	\end{tikzcd}
	\]		
	There are isomorphisms 
	\[
	\begin{array}{l}
		a_{f,g,hk}:	 f\times_w(g\times_x(h\times_yk)) \to (f\times_wg)\times_x(h\times_yk), \\
		a_{fg,h,k}:(f\times_wg)\times_x(h\times_yk) \to ((f\times_wg)\times_xh)\times_yk, \\
		a_{f,gh,k}:f\times_w((g\times_xh)\times_yk) \to (f\times_w(g\times_xh))\times_yk, \\
		a_{f,g,h}: f\times_w(g\times_xh) \to (f\times_wg)\times_xh, \\
		a_{g,h,k}: g\times_x(h\times_yk) \to (g\times_xh)\times_yk,
	\end{array}
	\]
	such that the following diagram is commutative.
	Note that in the following commutative diagram, we omit all pullback notations $\times_?$ for short. 	
	\[
	\begin{tikzcd}[row sep=large,column sep=10]
		&&&& (fg)(hk) \arrow[dlll,dotted,"{\cong}","{a_{fg,h,k}}"'] \arrow[llddd,two heads,bend left=25] \arrow[rrddd,two heads,bend right=25] &&&&\\
		& ((fg)h)k \arrow[rrrrrrddd,two heads,bend left=5] \arrow[d,two heads] && (f(gh))k \arrow[ll,dotted,"{\cong}"',"{a_{f,g,h}\times_y1_k}"] \arrow[d,two heads] \arrow[rrrrddd,two heads] && f((gh)k)  \arrow[ll,dotted,"{\cong}"',"{a_{f,gh,k}}"] \arrow[llllddd,two heads] \arrow[d,two heads] && f(g(hk)) \arrow[ll,dotted,"{\cong}"',"{1_f\times_wa_{g,h,k}}"]  \arrow[lllu,dotted,"{\cong}","{a_{f,g,hk}}"'] \arrow[llllllddd,two heads,bend right=5] \arrow[d,two heads] & \\	
		& (fg)h \arrow[rrrrdd,two heads] \arrow[rd,two heads] && f(gh) \arrow[ll,dotted,"a_{f,g,h}"',"{\cong}"] \arrow[lldd,two heads,bend right=10] \arrow[rd,two heads] && (gh)k \arrow[ld,two heads] \arrow[rrdd,two heads,bend left=10] && g(hk) \arrow[ll,dotted,"a_{g,h,k}"',"{\cong}"] \arrow[ld,two heads] \arrow[lllldd,two heads] & \\		
		&& fg \arrow[ld,two heads] \arrow[rd,two heads] && gh \arrow[ld,two heads] \arrow[rd,two heads] && hk \arrow[ld,two heads] \arrow[rd,two heads] && \\	
		& f \arrow[ld,two heads] \arrow[rd,two heads] && g \arrow[ld,two heads] \arrow[rd,two heads] && h \arrow[ld,two heads] \arrow[rd,two heads] && k \arrow[ld,two heads] \arrow[rd,two heads] & \\	
		v && w && x && y && z
	\end{tikzcd}
	\]	
	Thus {\footnotesize 
	\[
	\begin{array}{ll}
		&(1_{\hat{k}}*a_{\hat{h},\hat{g},\hat{f}})\circ a_{\hat{k},\hat{h}\hat{g},\hat{f}}\circ (a_{\hat{k},\hat{h},\hat{g}}*1_{\hat{f}}) \\ [2mm]
		= & [(f\times_w(g\times_xh))\times_yk = (f\times_w(g\times_xh))\times_yk \xrightarrow{a_{f,g,h}\times_y1_k} ((f\times_wg)\times_xh)\times_yk] \\ [2mm]
		& \circ [f\times_w((g\times_xh)\times_yk) = f\times_w((g\times_xh)\times_yk) \xrightarrow{a_{f,gh,k}} (f\times_w(g\times_xh))\times_yk] \\ [2mm]
		& \circ [f\times_w(g\times_x(h\times_yk)) = f\times_w(g\times_x(h\times_yk)) \xrightarrow{1_f\times_wa_{g,h,k}} f\times_w((g\times_xh)\times_yk)] \\ [2mm]
		= & [f\times_w(g\times_x(h\times_yk)) = f\times_w(g\times_x(h\times_yk)) \xrightarrow{(a_{f,g,h}\times_y1_k)a_{f,gh,k}(1_f\times_wa_{g,h,k})} (f\times_w(g\times_xh))\times_yk] \\ [2mm]
		= & [f\times_w(g\times_x(h\times_yk)) = f\times_w(g\times_x(h\times_yk)) \xrightarrow{a_{fg,h,k}a_{f,g,hk}} (f\times_w(g\times_xh))\times_yk] \\ [2mm]
		= & [(f\times_wg)\times_x(h\times_yk)) = (f\times_wg)\times_x(h\times_yk)) \xrightarrow{a_{fg,h,k}} (f\times_w(g\times_xh))\times_yk] \\ [2mm]
		& \circ [f\times_w(g\times_x(h\times_yk)) = f\times_w(g\times_x(h\times_yk)) \xrightarrow{a_{f,g,hk}} (f\times_wg)\times_x(h\times_yk))] \\ [2mm]
		= & a_{\hat{k},\hat{h},\hat{g}\hat{f}} \circ a_{\hat{k}\hat{h},\hat{g},\hat{f}}.
	\end{array}
	\]}
	
	Now we have finished the proof of the theorem.
\end{proof}

For passing from $B_\infty$-algebras to Gerstenhaber algebras, we need do some preparations.

\begin{lemma}\label{Lem-Functor-Pullback}
	Let $\mathsf{M}$ be a model category, $\mathsf{C}$ a category with pullbacks, and $F:\mathsf{M}\rightarrow\mathsf{C}$ a functor sending each weak equivalence in $\mathsf{M}$ to an isomorphism in $\mathsf{C}$. Then for each cospan $x \xrightarrow{f} z \xleftarrow{g} y$ in $\mathsf{M}$ with $g$ being an acyclic fibration, the canonical morphism
	$\theta_{x,y}:F(x\times_zy) \to Fx\times_{Fz}Fy$ 
	given by the following pullback commutative diagrams is an isomorphism in $\mathsf{C}$.
	\[
	\begin{tikzcd}
		& x\times_zy \arrow[r,"\tilde{f}"] \arrow[d,two heads,"\sim"',"\tilde{g}"] & y \arrow[d,two heads,"\sim"',"g"] \\
		& x \arrow[r,"f"] & z 
	\end{tikzcd}
	\quad\quad
	\begin{tikzcd}
		F(x\times_zy) \arrow[rrd,"F\tilde{f}"',bend left=20] \arrow[rdd,"\cong"',"F\tilde{g}",bend right=20] \arrow[rd, dotted,"\theta_{x,y}"] & & \\
		& Fx\times_{Fz}Fy \arrow[r] \arrow[d,"\cong"] & Fy \arrow[d,"\cong","Fg"'] \\
		& Fx \arrow[r,"Ff"] & Fz 
	\end{tikzcd}
	\]	
\end{lemma}

\begin{proof}
	It follows from Lemma \ref{Lem-Pullback-2-Fib-AcycFib} that $\tilde{g}$ is also an acyclic fibration. By the assumption that $F$ sends weak equivalences to isomorphisms, $Fg$ and $F\tilde{g}$ are isomorphisms. Since pullbacks preserve isomorphisms, the morphism $Fx\times_{Fz}Fy\to Fx$ is an isomorphism. So is the morphism $\theta_{x,y}: F(x\times_zy)\to Fx\times_{Fz}Fy$.
\end{proof}

\begin{proposition}\label{Prop-ColaxFunctor-Mspan-CSpan}
	Let $\mathsf{M}$ be a model category, $\mathsf{C}$ a category with pullbacks, and $F:\mathsf{M}\rightarrow\mathsf{C}$ a functor sending each weak equivalence to an isomorphism. Then the functor $F$ induces a colax functor
	$F:\mathsf{M}\-{\rm span}^2 \rightarrow \mathsf{C}\-\Span^2$ given by the following data:
	
	{\rm (1)} The function $F:\Ob(\mathsf{M}\-{\rm span}^2) \rightarrow \Ob(\mathsf{C}\-\Span^2), x \mapsto Fx.$
	
	{\rm (2)} For all $x,y\in\mathsf{M}\-{\rm span}^2$, the functor $F_{xy}:\mathsf{M}\-{\rm span}^2(x,y) \rightarrow \mathsf{C}\-\Span^2(Fx,Fy), \hat{f} =(x\stackrel{f_l}{\twoheadleftarrow} f \stackrel{f_r}{\twoheadrightarrow} y) \mapsto F_{xy}(\hat{f}) := (Fx \xleftarrow{Ff_l} Ff \xrightarrow{Ff_r} Fy), [\hat{\hat{\alpha}}]=[\!\!\begin{tikzcd} f & \alpha \arrow[l,two heads,"\alpha_l"',"\sim"] \end{tikzcd}\!\! \xrightarrow{\alpha_r} f'] \mapsto F_{xy}([\hat{\hat{\alpha}}]) := [\!\!\begin{tikzcd} Ff & F\alpha  \arrow[l,"F\alpha_l"',"{\cong}"] \end{tikzcd}\!\!\xrightarrow{F\alpha_r} Ff']$.
	
	{\rm (3)} For all objects $x,y,z$ and 1-cells $\hat{f}=(x\twoheadleftarrow f\twoheadrightarrow y)$ and $\hat{f}'=(y\twoheadleftarrow g\twoheadrightarrow z)$ in $\mathsf{M}\-{\rm span}^2$, the colax functoriality constraint $F^2_{\hat{f}',\hat{f}}: F(\hat{f}'\circ\hat{f})\to F(\hat{f}')\circ F(\hat{f})$ is the 2-cell $[F(f\times_y g) = F(f\times_y g) \xrightarrow{\theta_{f,g}} Ff\times_{Fy}Fg]$.
	\[
	\begin{tikzcd}
		& F(f\times_y g) \arrow[ld] \arrow[rd] & \\
		Fx & F(f\times_y g) \arrow[u, equal] \arrow[d,"\theta_{f,g}"] & Fz \\
		& Ff\times_{Fy}Fg \arrow[lu] \arrow[ru] &
	\end{tikzcd}
	\]
	
	{\rm (4)} For all $x\in\mathsf{M}\-{\rm span}^2$, the colax unity constraint $F_x^0:=[Fx=Fx=Fx]$.
	\[
	\begin{tikzcd}
		& Fx \arrow[ld,"F1_x"'] \arrow[rd,"F1_x"] & \\
		Fx & Fx \arrow[u, equal] \arrow[d, equal] & Fx \\
		& Fx \arrow[ru,"1_{Fx}"'] \arrow[lu,"1_{Fx}"]
	\end{tikzcd}
	\]
	
\end{proposition}

\begin{proof}	
	(1) For all $x,y\in\mathsf{M}\-{\rm span}^2$, $F_{xy}$ is a functor: 
	
	(1.1) $F_{xy}$ is well-defined: If $\hat{\hat{\alpha}}=(\!\!\begin{tikzcd} f & \alpha \arrow[l,two heads,"\alpha_l"',"\sim"] \end{tikzcd}\!\! \xrightarrow{\alpha_r} f') \sim \hat{\hat{\alpha}}'=(\!\!\begin{tikzcd} f & \alpha' \arrow[l,two heads,"\alpha'_l"',"\sim"] \end{tikzcd}\!\! \xrightarrow{\alpha'_r} f')$ is given by a morphism $\theta:\alpha\to\alpha'$ then $(\!\!\begin{tikzcd} Ff & F\alpha \arrow[l,"F\alpha_l"',"{\cong}"] \end{tikzcd}\!\!\xrightarrow{F\alpha_r} Ff') \sim (\!\!\begin{tikzcd} Ff & F\alpha' \arrow[l,"F\alpha'_l"',"{\cong}"] \end{tikzcd}\!\! \xrightarrow{F\alpha'_r} Ff')$ is given by the isomorphism $F\theta:F\alpha\to F\alpha'$.
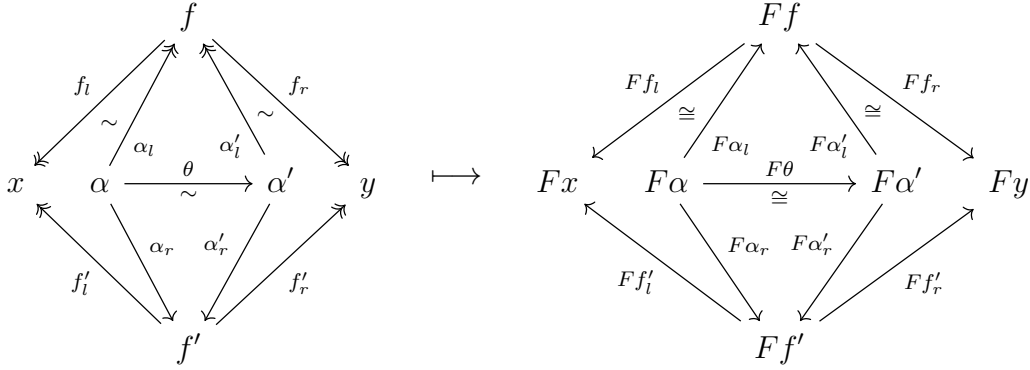
	\[
	\begin{tikzcd}[row sep=huge,column sep=small]
		&& f \arrow[lld,two heads,"f_l"'] \arrow[rrd,two heads,"f_r"] && \\
		x & \alpha \arrow[ur,two heads,near start,"\sim","\alpha_l"'] \arrow[dr,"\alpha_r"] \arrow[rr,"\sim"',"\theta"] &&  \alpha' \arrow[ul,two heads,near start,"\sim"',"\alpha'_l"] \arrow[dl,"\alpha'_r"'] & y \\
		&& f' \arrow[ull,two heads,"f'_l"] \arrow[urr,two heads,"f'_r"'] &&
	\end{tikzcd} \quad \longmapsto \quad
	\begin{tikzcd}[row sep=huge,column sep=small]
		&& Ff \arrow[lld,"Ff_l"'] \arrow[rrd,"Ff_r"] && \\
		Fx & F\alpha \arrow[ur,near start,"\cong","F\alpha_l"'] \arrow[dr,"F\alpha_r"] \arrow[rr,"F\theta","{\cong}"'] && F\alpha' \arrow[ul,near start,"\cong"',"F\alpha'_l"] \arrow[dl,"F\alpha'_r"'] & Fy \\
		&& Ff' \arrow[ull,"Ff'_l"] \arrow[urr,"Ff'_r"'] &&
	\end{tikzcd}
	\]
	
	(1.2) $F_{xy}$ preserves vertical compositions of 2-cells: By Lemma \ref{Lem-Functor-Pullback}, we have $F_{xy}([\hat{\hat{\alpha}}']\circ[\hat{\hat{\alpha}}]) = [Ff \leftarrow F(\alpha\times_{f'}\alpha') \rightarrow Ff''] = [Ff \leftarrow F\alpha\times_{Ff'}F\alpha' \rightarrow Ff''] =  F_{xy}([\hat{\hat{\alpha}}'])\circ F_{xy}([\hat{\hat{\alpha}}])$.
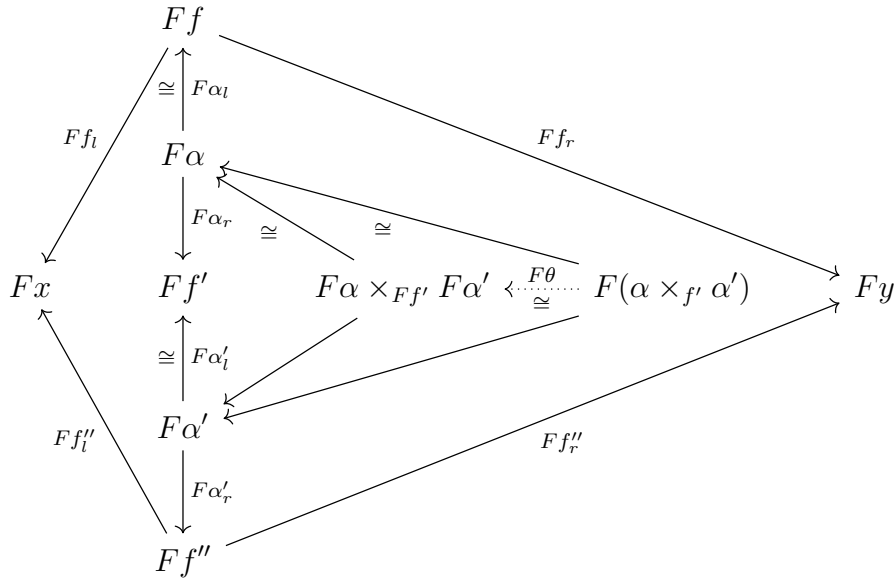
	\[
	\begin{tikzcd}[row sep=large]
		& Ff \arrow[ldd,"Ff_l"'] \arrow[rrrdd,"Ff_r"] &&& \\
		& F\alpha \arrow[u,"\cong","F\alpha_l"'] \arrow[d,"F\alpha_r"] &&& \\
		Fx & Ff' & F\alpha\times_{Ff'}F\alpha' \arrow[ul,"{\cong}"] \arrow[dl] & F(\alpha\times_{f'}\alpha') \arrow[l,dotted,"\cong","F\theta"'] \arrow[ull,"{\cong}"] \arrow[dll] & Fy\\
		& F\alpha' \arrow[u,"\cong","F\alpha'_l"'] \arrow[d,"F\alpha'_r"] &&& \\
		& Ff'' \arrow[uul,"{Ff''_l}"] \arrow[uurrr,"{Ff''_r}"'] &&&
	\end{tikzcd}
	\]
	
	(1.3) $F_{xy}$ preserves identity 2-cells: Obviously, $F1_{\hat{f}}=1_{F\hat{f}}$. 
	
	Thus $F_{xy}$ is a functor.
	
	(2) For all objects $x,y,z\in\mathsf{M}\-{\rm span}^2$,  $F^2_{xyz}$ is a natural transformation: We need to show that for all 1-cells
	$\hat{f}=(x \twoheadleftarrow f \twoheadrightarrow y), \hat{f}'=(x \twoheadleftarrow f' \twoheadrightarrow y), \hat{g}=(y \twoheadleftarrow g \twoheadrightarrow z), \hat{g}'=(y \twoheadleftarrow g' \twoheadrightarrow z)$
	and 2-cells
	$[\hat{\hat{\alpha}}]=[f \stackrel{\sim}{\twoheadleftarrow} \alpha \rightarrow f'], 
	[\hat{\hat{\beta}}]=[g \stackrel{\sim}{\twoheadleftarrow} \beta \rightarrow g']$ in $\mathsf{M}\-{\rm span}^2$, the following diagram is commutative. 
	\[
	\begin{tikzcd}
		F(\hat{g}\circ\hat{f}) \arrow[r,"F^2_{\hat{g},\hat{f}}"] \arrow[d,"{F([\hat{\hat{\beta}}]*[\hat{\hat{\alpha}}])}"'] & F(\hat{g})\circ F(\hat{f}) \arrow[d,"{F([\hat{\hat{\beta}}])*F([\hat{\hat{\alpha}}])}"] \\
		F(\hat{g}'\circ\hat{f}') \arrow[r,"F^2_{\hat{g}',\hat{f}'}"] & F(\hat{g}')\circ F(\hat{f}') 
	\end{tikzcd}
	\]
	
	Using the pullbacks $Ff\times_{Fy}Fg$ and $Ff'\times_{Fy}Fg'$, we can obtain the following commutative diagram, where $F(\alpha_l\times_y\beta_l)$ is an isomorphism due to Lemma~\ref{Lem-Pullback-2-Fib-AcycFib}.
	\[
	\begin{tikzcd}
		& Ff\times_{Fy}Fg \arrow[dl] \arrow[drrr] && F(f\times_yg)\arrow[ll,dotted,near end,"\theta_{f,g}"'] \arrow[dr] \arrow[dlll] &\\
		Ff \arrow[rr] && Fy && Fg \arrow[ll] \\
		& F\alpha\times_{Fy}F\beta \arrow[uu,"{\cong}","F\alpha_l\times_{Fy}F\beta_l"',near start] \arrow[dd,"F\alpha_r\times_{Fy}F\beta_r",near end] \arrow[dl] \arrow[drrr] && F(\alpha\times_y\beta) \arrow[ll,dotted,near end,"\theta_{\alpha,\beta}"'] \arrow[uu,"{\cong}"',"F(\alpha_l\times_y\beta_l)",near start] \arrow[dd,"F(\alpha_r\times_y\beta_r)"',near end] \arrow[dr] \arrow[dlll] &\\
		F\alpha \arrow[rr] \arrow[uu,"{\cong}"] \arrow[dd] && Fy \arrow[uu,equal] \arrow[dd,equal] && F\beta \arrow[ll] \arrow[uu,"{\cong}"] \arrow[dd] \\
		& Ff'\times_{Fy}Fg' \arrow[dl] \arrow[drrr] && F(f'\times_yg') \arrow[ll,dotted,near end,"\theta_{f',g'}"'] \arrow[dr] \arrow[dlll] &\\
		Ff' \arrow[rr] && Fy && Fg' \arrow[ll] 
	\end{tikzcd}		
	\]
	Thus
	\[
	\begin{array}{ll}
		& (F([\hat{\hat{\beta}}])*F([\hat{\hat{\alpha}}]))\circ F^2_{\hat{g},\hat{f}} \\ [2mm]
		= & [Ff\times_{Fy}Fg \xleftarrow{F\alpha_l\times_{Fy}F\beta_l} F\alpha\times_{Fy}F\beta \xrightarrow{F\alpha_r\times_{Fy}F\beta_r} Ff'\times_{Fy}Fg'] \\ [2mm]
		& \circ [F(f\times_yg) = F(f\times_yg) \xrightarrow{\theta_{f,g}} Ff\times_{Fy}Fg] \\ [2mm]
		= & [F(f\times_yg) =  F(f\times_yg) \xrightarrow{(F\alpha_r\times_{Fy}F\beta_r)\circ(F\alpha_l\times_{Fy}F\beta_l)^{-1}\circ\theta_{f,g}}  Ff'\times_{Fy}Fg'] \\ [2mm]
		= & [F(f\times_yg) =  F(f\times_yg) \xrightarrow{\theta_{f',g'}\circ F(\alpha_r\times_y\beta_r) \circ F(\alpha_l\times_y\beta_l)^{-1}}  Ff'\times_{Fy}Fg'] \\ [2mm]
		= & [F(f'\times_yg') = F(f'\times_yg') \xrightarrow{\theta_{f',g'}} Ff'\times_{Fy} Fg'] \\ [2mm]
		&  \circ [F(f\times_yg) \xleftarrow{F(\alpha_l\times_y\beta_l)} F(\alpha\times_y\beta) \xrightarrow{F(\alpha_r\times_y\beta_r)} F(f'\times_yg')] \\ [2mm]
		= & F^2_{\hat{g}',\hat{f}'}\circ F([\hat{\hat{\beta}}]*[\hat{\hat{\alpha}}]).
	\end{array}
	\]
	
	(3) {\it Colax associativity:} For all objects $w,x,y,z$ and 1-cells
	$\hat{f}=(w \twoheadleftarrow f \twoheadrightarrow x),\  \hat{g}=(x \twoheadleftarrow g \twoheadrightarrow y)$ and $\hat{h}=(y \twoheadleftarrow h \twoheadrightarrow z)$ in $\mathsf{M}\-{\rm span}^2$, we need to show that the following diagram is commutative.
	\[
	\begin{tikzcd}[row sep=large,column sep=large]
		(F\hat{h}\circ F\hat{g})\circ F\hat{f} \arrow[r,"{a'_{F\hat{h},F\hat{g},F\hat{f}}}"] & F\hat{h}\circ(F\hat{g}\circ F\hat{f}) \\
		F(\hat{h}\circ\hat{g})\circ F\hat{f} \arrow[u,"{F^2_{\hat{h},\hat{g}}*1_{F\hat{f}}}"] & F\hat{h}\circ F(\hat{g}\circ\hat{f}) \arrow[u,"{1_{F\hat{h}}*F^2_{\hat{g},\hat{f}}}"'] \\
		F((\hat{h}\circ\hat{g})\circ\hat{f}) \arrow[r,"{Fa_{\hat{h},\hat{g},\hat{f}}}"] \arrow[u,"F^2_{\hat{h}\circ\hat{g},\hat{f}}"] & F(\hat{h}\circ(\hat{g}\circ\hat{f})) \arrow[u,"F^2_{\hat{h},\hat{g}\circ\hat{f}}"']
	\end{tikzcd}		
	\]
	
	Using the pullback $(Ff\times_{Fx}Fg)\times_{Fy}Fh$, we can obtain the following commutative diagram in which we omit all pullback notations $\times_?$ for short.
	\[
	\begin{tikzcd}[column sep=25]
		&& F((fg)h) \arrow[d,"\theta_{fg,h}"',near end] \arrow[dddll] \arrow[ddddrrr,bend right=23] && F(f(gh)) \arrow[ll,"Fa_{f,g,h}"']\arrow[d,"\theta_{f,gh}",near end] \arrow[dddrr] \arrow[ddddlll,bend left=23] && \\
		&& F(fg)Fh \arrow[d,"\theta_{f,g}\times_{Fy}1_{Fh}"',near end] \arrow[ddll,bend right=10] \arrow[dddrrr,bend right=15] && FfF(gh) \arrow[d,"1_{Ff}\times_{Fx}\theta_{g,h}",near end] \arrow[ddrr,bend left=10] \arrow[dddlll,bend left=15] && \\
		&& (FfFg)Fh \arrow[d] \arrow[ddrrr,bend right=5] && Ff(FgFh) \arrow[ll,"{a'_{Ff,Fg,Fh}}"'] \arrow[d] \arrow[ddlll,bend left=5] && \\
		F(fg) \arrow[rr,"\theta_{f,g}"] \arrow[dr] \arrow[drrr] && FfFg \arrow[dl] \arrow[dr] && FgFh \arrow[dl] \arrow[dr] && F(gh) \arrow[ll,"\theta_{f,g}"'] \arrow[dl] \arrow[dlll] \\	
		& Ff \arrow[dl] \arrow[dr] && Fg \arrow[dl] \arrow[dr] && Fh \arrow[dl] \arrow[dr] & \\
		Fw && Fx && Fy && Fz 
	\end{tikzcd}
	\]
	Thus {\small		
	\[
	\begin{array}{ll}
		& (1_{F\hat{h}}*F^2_{\hat{g},\hat{f}})\circ F^2_{\hat{h},\hat{g}\circ\hat{f}} \circ Fa_{\hat{h},\hat{g},\hat{f}} \\ [2mm]
		= & [F(f\times_xg)\times_{Fy}Fh = F(f\times_xg)\times_{Fy}Fh \xrightarrow{\theta_{f,g}\times_{Fy}1_{Fh}} (Ff\times_{Fx}Fg)\times_{Fy}Fh] \\ [2mm]
		& \circ [F((f\times_xg)\times_yh) = F((f\times_xg)\times_yh) \xrightarrow{\theta_{fg,h}} F(f\times_xg)\times_{Fy}Fh] \\ [2mm]
		& \circ [F(f\times_x(g\times_yh)) = F(f\times_x(g\times_yh)) \xrightarrow{Fa_{f,g,h}} F((f\times_xg)\times_yh)] \\ [2mm]
		= & [F(f\times_x(g\times_yh)) = F(f\times_x(g\times_yh)) \xrightarrow{(\theta_{f,g}\times_{Fy}1_{Fh})\circ\theta_{fg,h}\circ Fa_{f,g,h}} (Ff\times_{Fx}Fg)\times_{Fy}Fh] \\ [2mm]
		= & [F(f\times_x(g\times_yh)) = F(f\times_x(g\times_yh)) \xrightarrow{a'_{Fh,Fg,Ff}(1_{Ff}\times_{Fx}\theta_{g,h})\theta_{f,gh}} (Ff\times_{Fx}Fg)\times_{Fy}Fh] \\ [2mm]
		= & [Ff\times_{Fx}(Fg\times_{Fy}Fh) = Ff\times_{Fx}(Fg\times_{Fy}Fh) \xrightarrow{a'_{Fh,Fg,Ff}} (Ff\times_{Fx}Fg)\times_{Fy}Fh] \\ [2mm]
		& \circ [Ff\times_{Fx}F(g\times_yh) = Ff\times_{Fx}F(g\times_yh) \xrightarrow{1_{Ff}\times_{Fx}\theta_{g,h}} Ff\times_{Fx}(Fg\times_{Fy}Fh)] \\ [2mm]
		& \circ [F(f\times_x(g\times_yh)) = F(f\times_x(g\times_yh)) \xrightarrow{\theta_{f,gh}} Ff\times_{Fx}F(g\times_yh)] \\ [2mm]
		= & a'_{F\hat{h},F\hat{g},F\hat{f}} \circ (F^2_{\hat{h},\hat{g}}*1_{F\hat{f}}) \circ F^2_{\hat{h}\circ\hat{g},\hat{f}}.	
	\end{array}
	\]}
	
	(4) {\it Colax left unity:} For any 1-cell $\hat{f}=(x \twoheadleftarrow f \twoheadrightarrow y)$, we need to show that the following diagram is commutative.
	\[
	\begin{tikzcd}
		1_{Fy}\circ F\hat{f} \arrow[r,"l'_{F\hat{f}}"] & F\hat{f} \\
		F1_y\circ F\hat{f} \arrow[u,"F^0_y*1_{F\hat{f}}"] & F(1_y\circ \hat{f}) \arrow[u,"Fl_{\hat{f}}"'] \arrow[l,"F^2_{1_y,\hat{f}}"']
	\end{tikzcd}
	\]
	
	We have the following commutative diagram. 
	\[
	\begin{tikzcd}
		& Ff\times_{Fy}Fy \arrow[d,"l'_{Ff}"] \arrow[drr] && F(f\times_y y) \arrow[ll,"\theta_{f,y}"'] \arrow[d] \arrow[dll,"Fl_f",near end] & \\
		& Ff \arrow[dl] \arrow[dr] && Fy \arrow[dl,equal] \arrow[dr,equal] & \\
		Fx && Fy & & Fy
	\end{tikzcd}
	\]	
	Thus	
	\[
	\begin{array}{ll}
		& l'_{F\hat{f}} \circ (F^0_y*1_{F\hat{f}}) \circ F^2_{1_y,\hat{f}} \\ [2mm]
		= & [Ff\times_{Fy}Fy = Ff\times_{Fy}Fy \xrightarrow{l'_{Ff}} Ff] \circ [Ff\times_{Fy}Fy = Ff\times_{Fy}Fy = Ff\times_{Fy}Fy] \\ [2mm]
		& \circ [F(f\times_yy) = F(f\times_yy) \xrightarrow{\theta_{f,y}} Ff\times_{Fy}Fy] \\ [2mm]
		= & [F(f\times_yy) = F(f\times_yy) \xrightarrow{l'_{Ff}\circ\theta_{f,y}} Ff] \\ [2mm]
		= & [F(f\times_yy) = F(f\times_yy) \xrightarrow{Fl_f} Ff] \\ [2mm]
		= & Fl_{\hat{f}}.
	\end{array}
	\]
	
	(5) {\it Colax right unity:} For any 1-cell $\hat{f}=(x \twoheadleftarrow f \twoheadrightarrow y)$, we need to show that the following diagram is commutative.
	\[
	\begin{tikzcd}
		F\hat{f}\circ 1_{Fx} \arrow[r,"r'_{F\hat{f}}"] & F\hat{f} \\
		F\hat{f}\circ F1_x \arrow[u,"1_{F\hat{f}}*F^0_x"] & F(\hat{f}\circ 1_x) \arrow[u,"Fr_{\hat{f}}"'] \arrow[l,"F^2_{\hat{f},1_x}"']
	\end{tikzcd}
	\]
	
	The proof is similar to that of the colax left unity.	
\end{proof}

As an application, we can obtain a colax functor from some bicategory of $B_\infty$-algebras to some bicategory of Gerstenhaber algebras.

\begin{corollary}\label{Cor-Colax functor-Bspan2-GSpan2}
The cohomology functor $H:B_\infty\to\mathsf{G}, A\mapsto HA, (f:A\to A')\mapsto (Hf_1:HA\to HA')$, induces a colax functor $\mathscr{H}:B_\infty\-{\rm span}^2 \rightarrow \mathsf{G}\-\Span^2$.
\end{corollary}

\begin{proof}
	It follows from Proposition \ref{Prop-ColaxFunctor-Mspan-CSpan}.
\end{proof}

To obtain lax functors, we consider the restrictions of above colax functors on sub-bicategories.
Denote by $\mathsf{M}\-{\rm span}^2_{\rm raf}$ the sub-bicategory of $\mathsf{M}\-{\rm span}^2$ whose objects are all objects of $\mathsf{M}\-{\rm span}^2$ or $\mathsf{M}$, whose 1-cells are all 1-cells $x \stackrel{f_l}{\twoheadleftarrow} f\stackrel{f_r}{\twoheadrightarrow} y$ in $\mathsf{M}\-{\rm span}^2$ with $f_r$ being an acyclic fibration, and $\mathsf{M}\-{\rm span}^2_{\rm raf}(x,y)(\hat{f},\hat{f}'):=\mathsf{M}\-{\rm span}^2(x,y)(\hat{f},\hat{f}')$ for all objects $x,y$ and 1-cells $\hat{f}=(\!\!\begin{tikzcd}[column sep=small] x & f \arrow[l,two heads,"f_l"'] \arrow[r,two heads,"f_r","\sim"'] & y \end{tikzcd}\!\!)$ and $\hat{f}'=(\!\!\begin{tikzcd}[column sep=small] x & f' \arrow[l,two heads,"{f'_l}"'] \arrow[r,two heads,"{f'_r}","\sim"'] & y \end{tikzcd}\!\!)$ in $\mathsf{M}\-{\rm span}^2_{\rm raf}$.

\begin{proposition}\label{Prop-LaxFunctor-Mspanaf-CSpan}
	Let $\mathsf{M}$ be a model category, $\mathsf{C}$ a category with pullbacks, and $F:\mathsf{M}\rightarrow\mathsf{C}$ a functor sending each weak equivalence to an isomorphism. Then the functor $F$ induces a lax functor
	\[
	F':{\mathsf{M}\-{\rm span}^2_{\rm raf}} \rightarrow \mathsf{C}\-\Span^2
	\]
	given by the following data:
	
	{\rm (1)} $F':\Ob(\mathsf{M}\-{\rm span}^2_{\rm raf})\to\Ob(\mathsf{C}\-\Span^2), x\mapsto F'x:=Fx$.
	
	{\rm (2)} For all $x,y\in\mathsf{M}\-{\rm span}^2_{\rm raf}$, the functor $F'_{xy}: \mathsf{M}\-{\rm span}^2_{\rm raf} \rightarrow \mathsf{C}\-\Span^2, \hat{f}=(\!\!\begin{tikzcd}[column sep=small] x & f \arrow[l,two heads,"f_l"'] \arrow[r,two heads,"f_r","\sim"'] & y \end{tikzcd}\!\!) \linebreak \mapsto F'_{xy}\hat{f}:=(\!\!\begin{tikzcd}[column sep=small] Fx & Ff \arrow[l,"Ff_l"'] \arrow[r,"Ff_r","\cong"'] & Fy
	\end{tikzcd}\!\!), [\hat{\hat{\alpha}}]=[\hat{f}\leftarrow \hat{\alpha}\rightarrow \hat{f}']=[\!\!\begin{tikzcd}[column sep=small] f & \alpha \arrow[l,two heads,"\alpha_l"',"\sim"] \end{tikzcd}\!\!\xrightarrow{\alpha_r} f'] \mapsto F'_{xy}[\hat{\hat{\alpha}}]:=[\!\!\begin{tikzcd}[column sep=small] Ff & F\alpha \arrow[l,"F\alpha_l"',"\cong"] \end{tikzcd}\!\!\xrightarrow{F\alpha_r} Ff']$.
	
	{\rm (3)} For all objects $x,y,z$ and 1-cells $\hat{f}=(x {\twoheadleftarrow} f \stackrel{\sim}{\twoheadrightarrow} y)$ and $\hat{g}=(y {\twoheadleftarrow} g \stackrel{\sim}{\twoheadrightarrow} z)$ in $\mathsf{M}\-{\rm span}^2_{\rm raf}$, the lax functoriality constraint $F'^2_{\hat{g},\hat{f}}:=[Ff\times_{Fy} Fg \xleftarrow[\cong]{\theta_{f,g}} F(f\times_yg)=F(f\times_yg)]$, where $\theta_{f,g}$ is an isomorphism due to Lemma~\ref{Lem-Functor-Pullback}.
	\[
	\begin{tikzcd}
		& Ff\times_{Fy} Fg \arrow[ld] \arrow[rd] & \\
		Fx & F(f\times_y g) \arrow[u,"{\cong}","\theta_{f,g}"'] \arrow[d,equal] & Fz \\
		& F(f\times_yg) \arrow[lu] \arrow[ru] &
	\end{tikzcd}
	\]
		
	{\rm (4)} For all $x\in\mathsf{M}\-{\rm span}^2_{\rm raf}$, the lax unity constraint $F'^0_x:=[Fx=Fx=Fx]$.
\end{proposition}

\begin{proof}
(1) For all objects $x,y,z$ in $\mathsf{M}\-{\rm span}^2_{\rm raf}$, $F'^2_{xyz}$ is a natural transformation: We need to prove that for any 1-cells
{$\hat{f}=(x {\twoheadleftarrow} f \stackrel{\sim}{\twoheadrightarrow} y), \hat{f}'=(x {\twoheadleftarrow} f' \stackrel{\sim}{\twoheadrightarrow} y), \hat{g}=(y {\twoheadleftarrow} g \stackrel{\sim}{\twoheadrightarrow} z), \hat{g}'=(y {\twoheadleftarrow} g' \stackrel{\sim}{\twoheadrightarrow} z)$}
and 2-cells $[\hat{\hat{\alpha}}]=[f \stackrel{\sim}{\twoheadleftarrow} \alpha \rightarrow f']$ and $[\hat{\hat{\beta}}]=[g \stackrel{\sim}{\twoheadleftarrow} \beta \rightarrow g']$ in $\mathsf{M}\-{\rm span}^2_{\rm raf}$, the following diagram is commutative.
\[
  \begin{tikzcd}
			F'\hat{g}\circ F'\hat{f} \arrow[r,"F'^2_{\hat{g},\hat{f}}"] \arrow[d,"{F'[\hat{\hat{\beta}}]*F'[\hat{\hat{\alpha}}]}"'] & F'(\hat{g}\circ\hat{f}) \arrow[d,"{F'([\hat{\hat{\beta}}]*[\hat{\hat{\alpha}}])}"] \\
			F'\hat{g}'\circ F'\hat{f}' \arrow[r,"{F'^2_{\hat{g}',\hat{f}'}}"] & F'(\hat{g}'\circ\hat{f}')
  \end{tikzcd}
\]

As in the proof (2) of Proposition \ref{Prop-ColaxFunctor-Mspan-CSpan}, we have the following commutative diagram.
\[
  \begin{tikzcd}[row sep=large,column sep=huge]
  	& F'f\times_{F'y}F'g \arrow[ld] \arrow[drr,bend left=40] & F'(f\times_yg) \arrow[l,"\theta_{f,g}"',"{\cong}"] \arrow[dr] & \\
  	F'x & F'\alpha\times_{F'y}F'\beta \arrow[u,"{F'\alpha_l\times_{F'y}F'\beta_l}"',"{\cong}"] \arrow[d,"{F'\alpha_r\times_{F'y}F'\beta_r}"] & F'(\alpha\times_y\beta)  \arrow[l,"\theta_{\alpha,\beta}"',"{\cong}"] \arrow[u,"F'(\alpha_l\times_y\beta_l)","{\cong}"'] \arrow[d,"F'(\alpha_r\times_y\beta_r)"'] & F'z \\
  	& F'f'\times_{F'g}F'g' \arrow[lu] & F'(f'\times_yg')  \arrow[llu,bend left=40] \arrow[l,"\theta_{f',g'}"',"{\cong}"] \arrow[ru] &  
  \end{tikzcd}
\]
Thus 
\[
\begin{array}{ll}
& (F'([\hat{\hat{\beta}}]*[\hat{\hat{\alpha}}])) \circ F'^2_{\hat{g},\hat{f}} \\ [2mm]
= & [F'(f\times_yg) \xleftarrow{F'(\alpha_l\times_y\beta_l)} F'(\alpha\times_y\beta) \xrightarrow{F'(\alpha_r\times_y\beta_r)} F'(f'\times_yg')] \\[2mm]
& \circ [F'f\times_{F'y}F'g \xleftarrow{\theta_{f,g}} F'(f\times_yg) = F'(f\times_yg)] \\[2mm]
= & [F'f\times_{F'y}F'g \xleftarrow{\theta_{f,g}\circ F'(\alpha_l\times_y\beta_l)} F'(\alpha\times_y\beta) \xrightarrow{F'(\alpha_r\times_y\beta_r)} F'(f'\times_yg')] \\[2mm]
= & [F'f\times_{F'y}F'g \xleftarrow{(F'\alpha_l\times_{F'y}F'\beta_l)\circ \theta_{\alpha,\beta}} F'(\alpha\times_y\beta) \xrightarrow{F'(\alpha_r\times_y\beta_r)} F'(f'\times_yg')] \\[2mm]
= & [F'f\times_{F'y}F'g \xleftarrow{F'\alpha_l\times_{F'y}F'\beta_l} F'\alpha\times_{F'y}F'\beta \xrightarrow{\theta_{f',g'}^{-1}(F'\alpha_r\times_{F'y}F'\beta_r)}  F'(f'\times_yg')] \\[2mm]
= & [F'f'\times_{F'g}F'g' = F'f'\times_{F'g}F'g' \xrightarrow{\theta_{f',g'}^{-1}} F'(f'\times_yg')] \\[2mm]
& \circ [F'f\times_{F'y}F'g \xleftarrow{F'\alpha_l\times_{F'y}F'\beta_l} F'\alpha\times_{F'y}F'\beta \xrightarrow{F'\alpha_r\times_{F'y}F'\beta_r} F'f'\times_{F'g}F'g'] \\[2mm]
= & F'^2_{\hat{g}',\hat{f}'} \circ (F'[\hat{\hat{\beta}}]*F'[\hat{\hat{\alpha}}]).
\end{array}
\]	
		
(2) {\it Lax associativity:} For any 1-cells {$\hat{f}=(w {\twoheadleftarrow} f \stackrel{\sim}{\twoheadrightarrow} x), \hat{g}=(x {\twoheadleftarrow} g \stackrel{\sim}{\twoheadrightarrow} y)$ and $\hat{h}=(y {\twoheadleftarrow} c \stackrel{\sim}{\twoheadrightarrow} z),$}
we need to show that the following diagram is commutative.
\[
  \begin{tikzcd}[row sep=large,column sep=huge]
	(F'\hat{h}\circ F'\hat{g})\circ F'\hat{f} \arrow[r,"a_{F'\hat{h},F'\hat{g},F'\hat{f}}"] \arrow[d,"F'^2_{\hat{h},\hat{g}}*1_{F'\hat{f}}"'] & F'\hat{h}\circ(F'\hat{g}\circ F'\hat{f}) \arrow[d,"1_{F'\hat{h}}*F'^2_{\hat{g},\hat{f}}"] \\
	F'(\hat{h}\circ\hat{g})\circ F'\hat{f} \arrow[d,"F'^2_{\hat{h}\circ\hat{g},\hat{f}}"'] & F'\hat{h}\circ F'(\hat{g}\circ \hat{f}) \arrow[d,"F'^2_{\hat{h},\hat{g}\circ\hat{f}}"] \\
	F'((\hat{h}\circ\hat{g})\circ\hat{f}) \arrow[r,"{F'a_{\hat{h},\hat{g},\hat{f}}}"]  & F'(\hat{h}\circ(\hat{g}\circ\hat{f})) 
  \end{tikzcd}
\]

As in the proof (3) of Proposition \ref{Prop-ColaxFunctor-Mspan-CSpan}, we have the following commutative diagram.
\[
  \begin{tikzcd}[row sep=large,column sep=large]
	& F'f\times_{F'x}(F'g\times_{F'y}F'h) \arrow[ld] \arrow[rrd, bend left=40] & (F'f\times_{F'x}F'g)\times_{F'y}F'h \arrow[l,"{a_{F'f,F'g,F'h}}"'] & \\
	F'w & F'f\times_{F'x}F'(g\times_yh) \arrow[u,"1_{F'f}\times_{F'x}\theta_{g,h}"'] & F'(f\times_xg)\times_{F'y}F'h \arrow[u,"\theta_{f,g}\times_{F'y}1_{F'h}"] & F'z \\
	& F'(f\times_x(g\times_yh)) \arrow[u,"\theta_{f,g\times_yh}"'] & F'((f\times_xg)\times_yh) \arrow[l,"F'a_{f,g,h}"'] \arrow[u,"\theta_{f\times_xg,h}"] \arrow[ur] \arrow[llu,bend left=40] & 
  \end{tikzcd}
\]
Thus {\footnotesize
\[
\begin{array}{ll}
	& F'^2_{\hat{h},\hat{g}\circ\hat{f}} \circ (1_{F'\hat{h}}*F'^2_{F'\hat{g},F'\hat{f}}) \circ a_{F'\hat{h},F'\hat{g},F'\hat{f}} \\[2mm]
	= & [F'(f\times_xg)\times_{F'y}F'h \xleftarrow{\theta_{f\times_xg,h}} F'((f\times_xg)\times_yh) = F'((f\times_xg)\times_yh)] \\[2mm]
	& \circ [(F'f\times_{F'x}F'g)\times_{F'y}F'h \xleftarrow{\theta_{f,g}\times_{F'y}1_{F'h}} F'(f\times_xg)\times_{F'y}F'h = F'(f\times_xg)\times_{F'y}F'h] \\[2mm]
	& \circ [F'f\times_{F'x}(F'g\times_{F'y}F'h) \xleftarrow{a_{F'f,F'g,F'h}} (F'f\times_{F'x}F'g)\times_{F'y}F'h = (F'f\times_{F'x}F'g)\times_{F'y}F'h] \\[2mm]
	= & [F'f\times_{F'x}(F'g\times_{F'y}F'h) \xleftarrow{a_{F'f,F'g,F'h} (\theta_{f,g}\times_{F'y}1_{F'h})\theta_{f\times_xg,h}} F'((f\times_xg)\times_yh) = F'((f\times_xg)\times_yh)] \\[2mm]
	= & [F'f\times_{F'x}(F'g\times_{F'y}F'h) \xleftarrow{(1_{F'f}\times_{F'x}\theta_{g,h}) \circ \theta_{f,g\times_yh} \circ F'a_{f,g,h}} F'((f\times_xg)\times_yh) = F'((f\times_xg)\times_yh)] \\[2mm]
	= & [F'(f\times_x(g\times_yh)) \xleftarrow{F'a_{f,g,h}} F'((f\times_xg)\times_yh) = F'((f\times_xg)\times_yh)] \\[2mm]
	& \circ [F'f\times_{F'x}F'(g\times_yh) \xleftarrow{\theta_{f,g\times_yh}} F'(f\times_x(g\times_yh)) = F'(f\times_x(g\times_yh))] \\[2mm]
	& \circ [F'f\times_{F'x}(F'g\times_{F'y}F'h) \xleftarrow{1_{F'f}\times_{F'x}\theta_{g,h}} F'f\times_{F'x}F'(g\times_yh) = F'f\times_{F'x}F'(g\times_yh)] \\[2mm]
	= & F'a_{\hat{h},\hat{g},\hat{f}} \circ F'^2_{\hat{h}\circ\hat{g},\hat{f}} \circ (F'^2_{F'\hat{h},F'\hat{g}}*1_{F'\hat{f}}). 
\end{array}
\] }
		
(3) {\it Lax left unity.} For any 1-cell {$\hat{f}=(x {\twoheadleftarrow} f \stackrel{\sim}{\twoheadrightarrow} y)$} in $\mathsf{M}\-{\rm span}^2_{\rm raf}$, we need to show that the following diagram is commutative.
\[
  \begin{tikzcd}
	1_{F'y}\circ F'\hat{f} \arrow[r,"l'_{F'\hat{f}}"] \arrow[d,"F'^0_y*1_{F'\hat{f}}"'] & F'\hat{f} \\
	F'1_y\circ F'\hat{f} \arrow[r,"F'^2_{1_y,\hat{f}}"] & F'(1_y\circ\hat{f}) \arrow[u,"F'l_{\hat{f}}"'] 
  \end{tikzcd}
\]

As in the proof (4) of Proposition \ref{Prop-ColaxFunctor-Mspan-CSpan}, we have the following commutative diagram.
\[
  \begin{tikzcd}
			& F'f\times_{F'y}F'y \arrow[dl] \arrow[ddr,"l'_{F'f}","\cong"'] && F'(f\times_{y}y) \arrow[dr] \arrow[ddl,"F'l_f"',"{\cong}"] \arrow[ll,"\theta_{f,y}"',"{\cong}"] & \\
			F'x &&&& F'y \\
			&& F'f \arrow[ull] \arrow[urr] && 
  \end{tikzcd}
\]
Thus
\[
\begin{array}{ll}
	& F'l_{\hat{f}} \circ F'^2_{1_y,\hat{f}} \circ (F'^0_y*1_{F'\hat{f}}) \\[2mm]
	= & [F'(f\times_{y}y) =F'(f\times_{y}y) \xrightarrow{F'l_f} F'f]\circ [F'f\times_{F'y}F'y \xleftarrow{\theta_{f,y}} F'(f\times_{y}y) = F'(f\times_{y}y)] \\[2mm]	
	& \circ [F'f\times_{F'y}F'y = F'f\times_{F'y}F'y = F'f\times_{F'y}F'y] \\[2mm]
	= & [F'f\times_{F'y}F'y \xleftarrow{\theta_{f,y}} F'(f\times_{y}y) \xrightarrow{F'l_f} F'f] \\[2mm]
	= & [F'f\times_{F'y}F'y = F'f\times_{F'y}F'y \xrightarrow{l'_{F'f}} F'f] \\[2mm]
	= & l'_{F'\hat{f}}.
\end{array}
\]

(4) {\it Lax right unity:} For any 1-cell {$\hat{f}=(x {\twoheadleftarrow} f \stackrel{\sim}{\twoheadrightarrow} y)$} in $\mathsf{M}\-{\rm span}^2_{\rm raf}$, we need to show that the following diagram is commutative.
\[
\begin{tikzcd}
	F'\hat{f}\circ 1_{F'x} \arrow[r,"r'_{F'\hat{f}}"] \arrow[d,"1_{F'\hat{f}}*F'^0_x"'] & F'\hat{f} \\
	F'\hat{f}\circ F'1_x \arrow[r,"F'^2_{\hat{f},1_x}"] & F'(\hat{f}\circ 1_x) \arrow[u,"F'r_{\hat{f}}"'] 
\end{tikzcd}
\]

The proof is similar to that of the lax left unity.
\end{proof}

As an application, we can obtain a lax functor from some bicategory of $B_\infty$-algebras to some bicategory of Gerstenhaber algebras.

\begin{corollary} \label{Cor-LaxFunctor-B8span2-GSpan2}
	The cohomology functor $H:B_\infty\to\mathsf{G}, A\mapsto HA, (f:A\to A')\mapsto (Hf_1:HA\to HA')$, induces a lax functor
	$\mathscr{H} : B_\infty\-{\rm span}^2_{\rm raf} \rightarrow \mathsf{G}\-\Span^2.$
\end{corollary}

\begin{proof}
It follows from Proposition \ref{Prop-LaxFunctor-Mspanaf-CSpan}.
\end{proof}

\section{Lax functoriality of Hochschild cochain complex}

In this section, we show that there are lax functors from some bicategories of dg categories to some bicategories of $B_\infty$-algebras sending a dg category to its Hochschild cochain complex. 

\subsection{Main lax functors}

To construct a lax functor from some bicategory of dg categories to some bicategory of $B_\infty$-algebras
sending a dg category to its Hochschild cochain complex, we need do some preparations.

\medspace

\noindent{\bf Restricted dg bimodule $X_F$.} Let $\aA$ and $\bB$ be small dg categories and $F:\aA \rightarrow \bB$ a dg functor. The {\it restricted dg $\aA$-$\bB$-bimodule} $X_F$ is given by $X_F(B,A):=\bB(B,FA)$ and $X_F(b,a):=\bB(b,Fa)$ for all $A,A'\in\aA$, $B,B'\in\bB$, $a\in\aA(A,A')$ and $b\in\bB(B,B')$. 
It defines an upper triangular matrix dg category $\tT_{X_F}=
\begin{pmatrix}
	\aA & X_F \\
	& \bB
\end{pmatrix}$.

It is shown by Keller in \cite[4.6, Page 11]{Keller06} that if $F$ is a quasi-isomorphism, i.e., it induces quasi-isomorphisms
$F_{AA'}:\aA(A,A') \rightarrow \bB(FA,FA')$ for all $A,A'\in\aA$,
then the projection $\iota^*_2:C(\tT_{X_F}) \twoheadrightarrow C(\bB)$ induced by the fully faithful dg funtor $\iota_2:\bB\hookrightarrow\tT_{X_F}$ is a quasi-isomorphism in $B_\infty$. Furthermore, if $F$ is a quasi-equivalence then the projection $\iota^*_1:C(\tT_{X_F}) \twoheadrightarrow C(\aA)$ induced by the fully faithful dg funtor $\iota_1:\aA\hookrightarrow\tT_{X_F}$ is also a quasi-isomorphism in $B_\infty$. 

\medspace

\noindent{\bf Hochschild cochain complex $C(F)$ of a dg functor $F$.} Let $\aA$ and $\bB$ be small dg categories, and $F:\aA \rightarrow \bB$ a dg functor. The {\it Hochschild cochain complex $C(F)$ of $F$} is the Hochschild cochain complex $C(\tT_{X_F})$ of the upper triangular matrix dg category $\tT_{X_F}=
\begin{pmatrix}
	\aA & X_F \\
	& \bB
\end{pmatrix}$ (cf. \cite[3.3, Page 768]{Borisov07}).

If $f:X\to X'$ is a quasi-isomorphism of cofibrant dg $\aA$-$\bB$-bimodules, then the dg functor $\tT_f: \tT_X\to\tT_{X'}$ is a quasi-equivalence. 
Therefore, the projections $\iota^*_1:C(\tT_f)=C(\tT_{X_{\tT_f}})=C\left(\begin{pmatrix}
	\tT_X&X_{\tT_f} \\ &\tT_{X'}
\end{pmatrix}\right) \twoheadrightarrow C(\tT_X)$ and $\iota^*_2:C(\tT_f) \twoheadrightarrow C(\tT_{X'})$ are quasi-isomorphisms in $B_\infty$.

\medspace

\noindent{\bf Lax functor from $\dgCAT_{\rm c,h}$ to $B_\infty\-{\rm span}^2$.} Now we can construct a lax functor 
from bicategory $\dgCAT_{\rm c,h}$ to bicategory $B_\infty\-{\rm span}^2$ sending a dg category to its Hochschild cochain complex.

\begin{theorem} \label{Thm-LaxFuntor-HomotopyCat-B-Inf}
There is a lax functor $\ccC=(\ccC,\ccC^2, \ccC^0):\dgCAT_{\rm c,h} \rightarrow B_\infty\-{\rm span}^2$ given by the following data:

{\rm (1)} The function $\ccC:\Ob(\dgCAT_{\rm c,h}) \rightarrow \Ob(B_\infty\-{\rm span}^2), \aA\mapsto C(\aA)$.

{\rm (2)} For any objects $\aA,\bB\in\dgCAT_{\rm c,h}$, the functor
$\ccC_{\aA\bB}:\dgCAT_{\rm c,h}(\aA,\bB) \rightarrow B_\infty\-{\rm span}^2 \linebreak (C(\aA),C(\bB))$ is defined as follows:

{\rm (2.1)} For any 1-cell $X\in\dgCAT_{\rm c,h}(\aA,\bB)$, $\ccC_{\aA\bB}(X):= (C(\aA)\stackrel{\iota_1^*}{\twoheadleftarrow} C(\tT_X)\stackrel{\iota_2^*}{\twoheadrightarrow} C(\bB))$. 

{\rm (2.2)} For any 2-cell $[f]\in\dgCAT_{\rm c,h}(\aA,\bB)(X,X')$, 
$\ccC_{\aA\bB}([f]):=[C(\tT_X) \stackrel{\sim}{\twoheadleftarrow} \hat{C}(\tT_f) \stackrel{\sim}{\twoheadrightarrow} C(\tT_{X'})]$, where $\hat{C}(\tT_f)$ is the following pullback in $B_\infty$ (Ref. Lemma~\ref{Lem-ThetaA}),
\[
\begin{tikzcd}
	& \hat{C}(\tT_f) \arrow[r] \arrow[d,"\sim"'] & C(\aA)\times C(\bB) \arrow[d,"{(\theta_\aA,\theta_\bB)}","\sim"'] \\
	& C(\tT_f) \arrow[r,two heads,"{\begin{pmatrix}
		\iota^*_\aA \\ \iota^*_\bB	
		\end{pmatrix}}"] \arrow[dl,two heads,"\sim"'] \arrow[dr,two heads,"\sim"] & C(\tT_{I_\aA})\times C(\tT_{I_\bB})\\
	C(\tT_X) && C(\tT_{X'})
\end{tikzcd}
\]
$C(\tT_f)$ is the Hochschild cochain complex of the quasi-equivalence $\tT_f:\tT_X\to\tT_{X'}$, and $\iota_\aA:\tT_{I_\aA}=\begin{pmatrix}
	\aA&I_\aA \\ &\aA
\end{pmatrix} \hookrightarrow \tT_{X_{\tT_f}} = \begin{pmatrix}
	\tT_X&X_{\tT_f} \\ &\tT_{X'}
\end{pmatrix} = 
\begin{pmatrix}
\aA&X&I_\aA&X' \\ &\bB&&I_\bB\\ &&\aA&X'\\ &&&\bB
\end{pmatrix}$ 
and $\iota_\bB:\tT_{I_\bB} \hookrightarrow \tT_{X_{\tT_f}}$ are natural fully faithful dg functors.

{\rm (3)} For any $\aA,\bB,\cC\in\dgCAT_{\rm c,h}$, the natural transformation $\ccC^2_{\aA\bB\cC}:c_{\ccC(\aA)\ccC(\bB)\ccC(\cC)}\circ(\ccC_{\bB\cC}\times\ccC_{\aA\bB}) \Rightarrow \ccC_{\aA\cC}\circ c_{\aA\bB\cC}$ is defined as follows:
For any 1-cells $X\in\dgCAT_{\rm c,h}(\aA,\bB)$ and $Y\in\dgCAT_{\rm c,h}(\bB,\cC)$, the lax functoriality constraint $\ccC^2_{Y,X}:\ccC(Y)\circ\ccC(X) \rightarrow \ccC(Y\circ X)$ is the 2-cell $[\bar{C}(\tT_{X,Y}) \stackrel{\sim}{\twoheadleftarrow} C(\tT_{X,Y}) \twoheadrightarrow C(\tT_{X\otimes_\bB Y})]$ (Ref. Lemma~\ref{Lem-CTXY-bar}).
\[
\begin{tikzcd}
	&& \bar{C}(\tT_{X,Y}) \arrow[ld,two heads] \arrow[rd,two heads] && \\
	& C(\tT_{X}) \arrow[ld,two heads] \arrow[rd,two heads] & C(\tT_{X,Y}) \arrow[u,two heads,"\sim"'] \arrow[ddd,two heads,bend left=30] & C(\tT_{Y}) \arrow[ld,two heads] \arrow[rd,two heads] &\\
	C(\aA) && C(\bB) && C(\cC) \\
	&&&&\\
	&& C(\tT_{X\otimes_\bB Y}) \arrow[rruu,two heads] \arrow[lluu,two heads] && 
\end{tikzcd}
\]  

{\rm (4)} For any object $\aA\in\dgCAT_{\rm c,h}$, the lax unity constraint $\ccC^0_\aA: 1_{\ccC(\aA)}=(C(\aA)=C(\aA)=C(\aA)) \to \ccC(\bp I_\aA)=(C(\aA) \stackrel{\sim}{\twoheadleftarrow} C(\tT_{\bp I_\aA}) \twoheadrightarrow C(\aA))$ is the 2-cell $[C(\aA) \stackrel{\sim}{\twoheadleftarrow} C(\aA)\times_{C(\tT_{I_\aA})}\hat{C}(\tT_{p_{I_\aA}}) \rightarrow C(\tT_{\bp I_\aA})]$ given by the composition of spans $C(\aA) = C(\aA) \stackrel{\theta_\aA}{\rightarrow} C(\tT_{I_\aA})$ and $C(\tT_{I_\aA}) \stackrel{\sim}{\twoheadleftarrow} \hat{C}(\tT_{p_{I_\aA}}) \stackrel{\sim}{\twoheadrightarrow} C(\tT_{\bp I_\aA})$, where $p_{I_\aA}:\bp I_\aA \rightarrow I_\aA$ is the natural quasi-isomorphism from the cofibrant replacement $\bp I_\aA$ of dg $\aA$-bimodule $I_\aA$ to $I_\aA$. 
\[
\begin{tikzcd}[column sep=large]
	& C(\aA) \arrow[dl,equal] \arrow[d,"\theta_\aA","\sim"'] \arrow[dr,equal] & C(\aA)\times_{C(\tT_{I_\aA})}\hat{C}(\tT_{p_{I_\aA}}) \arrow[l,two heads,"\sim" '] \arrow[ddl] \\
	C(\aA) &  C(\tT_{I_\aA}) \arrow[l,two heads,"\iota^*_1"'] \arrow[r,two heads,near start,"\iota^*_2"] & C(\aA) \\
	& \hat{C}(\tT_{p_{I_\aA}}) \arrow[u,two heads,"\sim"] \arrow[d,two heads,"\sim"'] & \\
	& C(\tT_{\bp I_\aA}) \arrow[uul,two heads] \arrow[uur,two heads] &
\end{tikzcd}	
\]	
\end{theorem}
 
\begin{proof}
	
(1) {\it For any $\aA,\bB\in\dgCAT_{\rm c,h}$, $\ccC_{\aA\bB}:\dgCAT_{\rm c,h}(\aA,\bB) \rightarrow B_\infty\-{\rm span}^2(C(\aA),C(\bB))$ is a functor:}

(1.1) {\it $\ccC_{\aA\bB}$ is well-defined:} We need to show that if  $[f]=[g]$ in $\dgCAT_{\rm c,h}(\aA,\bB)(X,X')$, i.e., two quasi-isomorphisms $f,g:X \rightarrow X'$
are homotopy equivalent, then $\ccC_{\aA\bB}([f]) = \ccC_{\aA\bB}([g])$, i.e.,  the spans $C(\tT_X) \stackrel{\sim}{\twoheadleftarrow} \hat{C}(\tT_f) \stackrel{\sim}{\twoheadrightarrow} C(\tT_{X'})$ and $C(\tT_X) \stackrel{\sim}{\twoheadleftarrow} \hat{C}(\tT_g) \stackrel{\sim}{\twoheadrightarrow} C(\tT_{X'})$ are equivalent.
	
Since the quasi-isomorphisms $f,g:X \rightarrow X'$ are homotopy equivalent, there exists a quasi-isomorphism $h:X'\to X''$ such that $hf=hg$. By the (1.2) below, which is independent of (1.1),
we have 
\[
\begin{array}{ll}
	& (C(\tT_{X'}) \stackrel{\sim}{\twoheadleftarrow} \hat{C}(\tT_h) \stackrel{\sim}{\twoheadrightarrow} C(\tT_{X''})) \circ (C(\tT_X) \stackrel{\sim}{\twoheadleftarrow} \hat{C}(\tT_g) \stackrel{\sim}{\twoheadrightarrow} C(\tT_{X'})) \\[2mm]
	\stackrel{(1.2)}{\sim} & (C(\tT_X) \stackrel{\sim}{\twoheadleftarrow} \hat{C}(\tT_{hg}) \stackrel{\sim}{\twoheadrightarrow} C(\tT_{X''})) \\[2mm]
	= & (C(\tT_X) \stackrel{\sim}{\twoheadleftarrow} \hat{C}(\tT_{hf}) \stackrel{\sim}{\twoheadrightarrow} C(\tT_{X''})) \\[2mm]
	\stackrel{(1.2)}{\sim} & (C(\tT_{X'}) \stackrel{\sim}{\twoheadleftarrow} \hat{C}(\tT_h) \stackrel{\sim}{\twoheadrightarrow} C(\tT_{X''}))\circ(C(\tT_X) \stackrel{\sim}{\twoheadleftarrow} \hat{C}(\tT_f) \stackrel{\sim}{\twoheadrightarrow} C(\tT_{X'})).
\end{array}
\]
Left compose with the span of spans $C(\tT_{X''}) \stackrel{\sim}{\twoheadleftarrow} \hat{C}(\tT_h) \stackrel{\sim}{\twoheadrightarrow} C(\tT_{X'})$ and note that
$(C(\tT_{X''}) \stackrel{\sim}{\twoheadleftarrow} \hat{C}(\tT_h) \stackrel{\sim}{\twoheadrightarrow} C(\tT_{X'})) \circ (C(\tT_{X'}) \stackrel{\sim}{\twoheadleftarrow} \hat{C}(\tT_h) \stackrel{\sim}{\twoheadrightarrow} C(\tT_{X''})) \sim (C(\tT_{X'}) = C(\tT_{X'}) = C(\tT_{X'}))$, 
we obtain $(C(\tT_X) \stackrel{\sim}{\twoheadleftarrow} \hat{C}(\tT_f) \stackrel{\sim}{\twoheadrightarrow} C(\tT_{X'})) \sim (C(\tT_X) \stackrel{\sim}{\twoheadleftarrow} \hat{C}(\tT_g) \stackrel{\sim}{\twoheadrightarrow} C(\tT_{X'}))$.
	
(1.2) {\it $\ccC_{\aA\bB}$ preserves composition:} We need to show for all $[f]\in\dgCAT_{\rm c,h}(\aA,\bB)(X,X')$ and $[f']\in\dgCAT_{\rm c,h}(\aA,\bB)(X',X'')$, the equality $\ccC_{\aA\bB}([f']\circ[f])=\ccC_{\aA\bB}([f'])\circ\ccC_{\aA\bB}([f])$ holds, or equivalently, for any quasi-isomorphisms $f:X \rightarrow X'$ and $f':X' \rightarrow X''$ of cofibrant dg $\aA\-\bB$-bimodules, the spans of spans 
$(C(\tT_{X'}) \stackrel{\sim}{\twoheadleftarrow} \hat{C}(\tT_{f'}) \stackrel{\sim}{\twoheadrightarrow} C(\tT_{X''})) \circ 
(C(\tT_X) \stackrel{\sim}{\twoheadleftarrow} \hat{C}(\tT_f) \stackrel{\sim}{\twoheadrightarrow} C(\tT_{X'}))$ 
and $C(\tT_X) \stackrel{\sim}{\twoheadleftarrow} \hat{C}(\tT_{f'f}) \stackrel{\sim}{\twoheadrightarrow} C(\tT_{X''})$ are equivalent.
	
The quasi-isomorphisms $X \xrightarrow{f} X'$ and $X'\xrightarrow{f'} X''$ induce quasi-equivalences $\tT_X \xrightarrow{\tT_f} \tT_{X'}$ and $\tT_{X'} \xrightarrow{\tT_{f'}} \tT_{X''}$ respectively. Observe the upper triangular matrix dg category
\[
  \tT_{X_{\tT_f},X_{\tT_{f'}}} 
  =
  \begin{pmatrix}
  	\tT_X & X_{\tT_f} & X_{\tT_{f'f}} \\
  	& \tT_{X'} & X_{\tT_{f'}} \\
  	&& \tT_{X''}
  \end{pmatrix}
  = 
  \begin{pmatrix}
			\aA & X & I_\aA & X' & I_\aA & X'' \\
			& \bB & & I_\bB & & I_\bB \\
			& & \aA & X' & I_\aA & X'' \\
			& & & \bB & & I_\bB \\
			& & & & \aA & X'' \\
			& & & & & \bB 
  \end{pmatrix}.
\]
There are the following natural fully faithful dg functors.
\[
  \begin{array}{ll}
  	\iota_{12}:\tT_{X_{\tT_f}}\hookrightarrow\tT_{X_{\tT_f},X_{\tT_{f'}}} & \\ [2mm] \iota_{23}:\tT_{X_{\tT_{f'}}}\hookrightarrow\tT_{X_{\tT_f},X_{\tT_{f'}}} & \\ [2mm] 
  	\iota_{13}:\tT_{X_{\tT_{f'f}}}\hookrightarrow\tT_{X_{\tT_f},X_{\tT_{f'}}} & \\ [2mm]
  	\iota_\aA:\tT_{I_\aA,I_\aA}\hookrightarrow\tT_{X_{\tT_f},X_{\tT_{f'}}} & \iota_\bB:\tT_{I_\bB,I_\bB}\hookrightarrow\tT_{X_{\tT_f},X_{\tT_{f'}}} \\[2mm]
	\iota_{\aA12}:\tT_{I_\aA}\hookrightarrow\tT_{X_{\tT_f}}\hookrightarrow \tT_{X_{\tT_f},X_{\tT_{f'}}} & 
	\iota_{\bB12}:\tT_{I_\bB}\hookrightarrow\tT_{X_{\tT_f}}\hookrightarrow \tT_{X_{\tT_f},X_{\tT_{f'}}} \\[2mm]
	\iota_{\aA23}:\tT_{I_\aA}\hookrightarrow\tT_{X_{\tT_{f'}}}\hookrightarrow \tT_{X_{\tT_f},X_{\tT_{f'}}} & 
	\iota_{\bB23}:\tT_{I_\bB}\hookrightarrow\tT_{X_{\tT_{f'}}}\hookrightarrow \tT_{X_{\tT_f},X_{\tT_{f'}}} \\[2mm]
	\iota_{\aA13}:\tT_{I_\aA}\hookrightarrow\tT_{X_{\tT_{f'f}}}\hookrightarrow \tT_{X_{\tT_f},X_{\tT_{f'}}} & 
	\iota_{\bB13}:\tT_{I_\bB}\hookrightarrow\tT_{X_{\tT_{f'f}}}\hookrightarrow \tT_{X_{\tT_f},X_{\tT_{f'}}} 
  \end{array}
\]
Let $\hat{C}(\tT_{X_{\tT_f},X_{\tT_{f'}}})$ be the following pullbacks in $B_\infty$ (Ref. Lemma~\ref{Lem-ThetaA^2}).
\[
	\begin{tikzcd}[row sep=large]
		\hat{C}(\tT_{X_{\tT_f},X_{\tT_{f'}}}) \arrow[r,two heads] \arrow[d,"\sim"] & C(\aA)\times C(\bB) \arrow[d,"{\theta_\aA^2 \times \theta_\bB^2}","\sim"'] \\
		C(\tT_{X_{\tT_f},X_{\tT_{f'}}}) \arrow[r,two heads,"{\begin{pmatrix}
				\iota^*_\aA \\ \iota^*_\bB
		\end{pmatrix}}"] & C(\tT_{I_\aA,I_\aA})\times C(\tT_{I_\bB,I_\bB})
	\end{tikzcd}
\]
Since Hochschild cochain complex is functorial on fully faithful dg functors (Ref. \cite[4.3, Page 8]{Keller03}), the commutative diagram of fully faithful dg functors
\[
  \begin{tikzcd}
	& \tT_{X_{\tT_f},X_{\tT_{f'}}} & \\
	\tT_{X_{\tT_f}} \arrow[ru,hook] & \tT_{X_{\tT_{f'}}} \arrow[u,hook] & \tT_{X_{\tT_{f'f}}} \arrow[lu,hook] \\ 
	\tT_{X'} \arrow[u,hook] \arrow[ru,hook] & \tT_{X} \arrow[lu,hook] \arrow[ru,hook] & \tT_{X''} \arrow[u,hook] \arrow[lu,hook]
  \end{tikzcd}
\]
induces the following commutative diagram $\textcircled{\tiny 1}$ of $B_\infty$-algebras (Ref. Lemma~\ref{Lem-CTXY-bar})
\[
  \begin{tikzcd}
		& C(\tT_{X_{\tT_f},X_{\tT_{f'}}}) \arrow[ld,two heads,"\sim"] \arrow[d,two heads,"\sim"] \arrow[rd,two heads,"\sim"'] & \\
		C(\tT_f) \arrow[d,two heads,"\sim"] \arrow[rd,two heads,"\sim",near start] & C(\tT_{f'}) \arrow[ld,two heads,"\sim"',near start] \arrow[rd,two heads,"\sim",near start] & C(\tT_{f'f}) \arrow[ld,two heads,"\sim"',near start] \arrow[d,two heads,"\sim"] \\ 
		C(\tT_{X'}) & C(\tT_{X}) & C(\tT_{X''})
  \end{tikzcd} 
\]
and the commutative diagram of fully faithful dg functors
\[
\begin{tikzcd}[column sep=large]
	&&& \tT_{X_{\tT_{f'f}}} \arrow[ld,hook] & \\
	\tT_{X_{\tT_{f}}} \arrow[rr,hook] && \tT_{X_{\tT_f},X_{\tT_{f'}}} && \tT_{X_{\tT_{f'}}} \arrow[ll,hook] \\
	&&& \tT_{I_\aA}\times \tT_{I_\bB} \arrow[uu,hook]\arrow[ld,hook,"{\iota_{\aA13} \times \iota_{\bB13}}"'] & \\
	\tT_{I_\aA}\times\tT_{I_\bB} \arrow[uu,hook] \arrow[rr,hook,"{\iota_{\aA12} \times \iota_{\bB12}}"] && \tT_{I_\aA,I_\aA}\times\tT_{I_\bB,I_\bB} \arrow[uu,hook] && \tT_{I_\aA}\times \tT_{I_\bB} \arrow[uu,hook] \arrow[ll,hook,"{\iota_{\aA23} \times \iota_{\bB23}}"']
\end{tikzcd}
\] 
induces the following commutative diagram $\textcircled{\tiny 2}$ of $B_\infty$-algebras.
\[
\begin{tikzcd}[column sep=small]
	&&& C(\tT_{f'f}) \arrow[dd,two heads] & \\
	C(\tT_f) \arrow[dd,two heads] && C(\tT_{X_{\tT_f},X_{\tT_{f'}}}) \arrow[ll,two heads,"\sim"] \arrow[ur,two heads] \arrow[rr,two heads,"\sim"',near end] \arrow[dd,two heads] && C(\tT_{f'}) \arrow[dd,two heads] \\
	&&& C(\tT_{I_\aA})\times C(\tT_{I_\bB}) & \\
	C(\tT_{I_\aA})\times C(\tT_{I_\bB}) && C(\tT_{I_\aA,I_\aA})\times C(\tT_{I_\bB,I_\bB}) \arrow[ll,two heads, "{\iota_{\aA12}^* \times \iota_{\bB12}^*}"'] \arrow[ur,two heads,"{\iota_{\aA13}^* \times \iota_{\bB13}^*}"]  \arrow[rr,two heads,"{\iota_{\aA23}^* \times \iota_{\bB23}^*}"] && C(\tT_{I_\aA})\times C(\tT_{I_\bB}) 
\end{tikzcd}
\] 
Moreover, the following diagram $\textcircled{\tiny 3}$ is commutative (Ref. Lemma~\ref{Lem-ThetaA} and Lemma~\ref{Lem-ThetaA^2}). 
\[
\begin{tikzcd}
	 & C(\aA)\times C(\bB) \arrow[ld,"{\theta^2_\aA \times \theta^2_\bB}"',"\sim"] \arrow[rd,"{\theta_\aA \times \theta_\bB}","\sim"'] &  \\
	C(\tT_{I_\aA,I_\aA})\times C(\tT_{I_\bB,I_\bB})\arrow[rr,two heads,"\sim"', "{\iota_{\aA12}^* \times \iota_{\bB12}^*}"] \arrow[rr,two heads,bend right=20,"\sim"', "{\iota_{\aA23}^* \times \iota_{\bB23}^*}"] \arrow[rr,two heads,bend right=40,"\sim"', "{\iota_{\aA13}^* \times \iota_{\bB13}^*}"] & & C(\tT_{I_\aA})\times C(\tT_{I_\bB}) 
\end{tikzcd}
\] 
Combining the commutative diagrams $\textcircled{\tiny 1}$, $\textcircled{\tiny 2}$ and $\textcircled{\tiny 3}$, we obtain the following diagram $\textcircled{\tiny 4}$ which is commutative when we consider only solid arrows.
\[
\begin{tikzcd}[column sep=small]
	&& \hat{C}(\tT_{X_{\tT_f},X_{\tT_{f'}}}) \arrow[ld,two heads,dotted,"\sim"'] \arrow[d,"\sim"] \arrow[dd,two heads,bend left=60] \arrow[rd,two heads,dotted,"\sim"] && \\
	& \hat{C}(\tT_f) \arrow[ld,"\sim"'] \arrow[rd,two heads] & C(\tT_{X_{\tT_f},X_{\tT_{f'}}}) \arrow[lld,two heads,"\sim"] \arrow[rrd,two heads,"\sim"'] \arrow[dd,two heads,bend right=60] & \hat{C}(\tT_{f'}) \arrow[ld,two heads] \arrow[rd,"\sim"] & \\
	C(\tT_f) \arrow[ddd,two heads,"\sim"] \arrow[rd,two heads] \arrow[rrddd,two heads,"\sim",bend right=22] & & C(\aA)\times C(\bB) \arrow[ld,"{\theta_\aA \times \theta_\bB}"',"\sim"] \arrow[d,"{\theta^2_\aA \times \theta^2_\bB}","\sim"'] \arrow[rd,"{\theta_\aA \times \theta_\bB}","\sim"'] & & C(\tT_{f'}) \arrow[ld,two heads] \arrow[ddd,two heads,"\sim"] \arrow[llddd,two heads,"\sim"',bend left=22] \\
	& C(\tT_{I_\aA})\times C(\tT_{I_\bB}) & C(\tT_{I_\aA,I_\aA})\times C(\tT_{I_\bB,I_\bB}) \arrow[l,two heads,"\stackrel{\sim}{\iota_{\aA12}^* \times \iota_{\bB12}^*}"] \arrow[r,two heads,"\stackrel{\sim}{\iota_{\aA23}^* \times \iota_{\bB23}^*}"'] & C(\tT_{I_\aA})\times C(\tT_{I_\bB}) & \\
	&&&&\\
	C(\tT_X) \arrow[d,two heads] \arrow[drrrr,two heads] && C(\tT_{X'}) \arrow[dll,two heads] \arrow[drr,two heads] && C(\tT_{X''}) \arrow[d,two heads] \arrow[dllll,two heads] \\
	C(\aA) &&&& C(\bB)
\end{tikzcd}
\]
Since $\hat{C}(\tT_f)$ and $\hat{C}(\tT_{f'})$ are pullbacks, there are morphisms $\hat{C}(\tT_{X_{\tT_f},X_{\tT_{f'}}})\to \hat{C}(\tT_f)$ and $\hat{C}(\tT_{X_{\tT_f},X_{\tT_{f'}}})\to \hat{C}(\tT_{f'})$, which are quasi-isomorphisms by 2-out-of-3 property of quasi-isomorphisms, such that the diagrams
\[
\begin{tikzcd}
	\hat{C}(\tT_{X_{\tT_f},X_{\tT_{f'}}}) \arrow[r,two heads,"\sim"] \arrow[d,"\sim"'] & C(\tT_{X_{\tT_f},X_{\tT_{f'}}}) \arrow[d,two heads,"\sim"] \\
	\hat{C}(\tT_f) \arrow[r,"\sim"] & C(\tT_f)
\end{tikzcd}	
\quad\quad
\begin{tikzcd}
	\hat{C}(\tT_{X_{\tT_f},X_{\tT_{f'}}}) \arrow[r,two heads,"\sim"] \arrow[d,"\sim"'] & C(\tT_{X_{\tT_f},X_{\tT_{f'}}}) \arrow[d,two heads,"\sim"] \\
	\hat{C}(\tT_{f'}) \arrow[r,"\sim"] & C(\tT_{f'})
\end{tikzcd}	
\]
are commuative. Since the two compositions below have the same projections on $C(\aA)$ and $C(\bB)$ respectively, we obtain the following commutative diagram $\textcircled{\tiny 5}$
\[
\begin{tikzcd}
	\hat{C}(\tT_{X_{\tT_f},X_{\tT_{f'}}}) \arrow[r,two heads,"\sim"] \arrow[d,two heads,"\sim"'] & \hat{C}(\tT_{f'}) \arrow[d,two heads] \\
	\hat{C}(\tT_f) \arrow[r,two heads] & C(\aA)\times C(\bB)
\end{tikzcd}	
\]
which implies that the whole diagram $\textcircled{\tiny 4}$ above is commutative.
By the commutative diagram 
\[
\begin{tikzcd}
	\hat{C}(\tT_{X_{\tT_f},X_{\tT_{f'}}}) \arrow[r,two heads,"\sim"] \arrow[d,two heads,"\sim"'] & \hat{C}(\tT_{f'}) \arrow[d,two heads,"\sim"'] \\
	\hat{C}(\tT_f) \arrow[r,two heads,"\sim"'] & C(\tT_{X'})
\end{tikzcd}	
\]
and the pullback
\[
\begin{tikzcd}
	\hat{C}(\tT_f)\times_{C(\tT_{X'})}\hat{C}(\tT_{f'}) \arrow[r,two heads,"\sim"] \arrow[d,two heads,"\sim"'] & \hat{C}(\tT_{f'}) \arrow[d,two heads,"\sim"'] \\
	\hat{C}(\tT_f) \arrow[r,two heads,"\sim"'] & C(\tT_{X'}),
\end{tikzcd}	
\]
there is a morphism $\hat{C}(\tT_{X_{\tT_f},X_{\tT_{f'}}}) \to \hat{C}(\tT_f)\times_{C(\tT_{X'})}\hat{C}(\tT_{f'})$ such that the following diagram $\textcircled{\tiny 6}$ is commutative when we consider only solid arrows.
\[
\begin{tikzcd}[column sep=small]
	& C(\tT_X) \arrow[ldd,two heads] \arrow[rrrrdd,two heads] & & & & \\
	& \hat{C}(\tT_f) \arrow[u,two heads,"\sim"'] \arrow[d,two heads,"\sim"] & &  & & \\
	C(\aA) & C(\tT_{X'}) & \hat{C}(\tT_f)\times_{C(\tT_{X'})}\hat{C}(\tT_{f'}) \arrow[lu,two heads,"\sim"] \arrow[ld,two heads,"\sim"'] & \hat{C}(\tT_{X_{\tT_f},X_{\tT_{f'}}}) \arrow[l,"\sim"'] \arrow[llu,two heads,"\sim"] \arrow[lld,two heads,"\sim"'] \arrow[r,dotted] & \hat{C}(\tT_{f'f}) \arrow[llluu,two heads,"\sim"] \arrow[llldd,two heads,"\sim"'] & C(\bB) \\
	& \hat{C}(\tT_{f'}) \arrow[u,two heads,"\sim"'] \arrow[d,two heads,"\sim"] & & & & \\
	& C(\tT_{X''}) \arrow[luu,two heads] \arrow[rrrruu,two heads] & & & &
\end{tikzcd}
\]
Similarly, combining the commutative diagrams $\textcircled{\tiny 2}$ and $\textcircled{\tiny 3}$, the following diagram $\textcircled{\tiny 7}$ is commutative when we consider only solid arrows.
\[
\begin{tikzcd}[column sep=large]
	\hat{C}(\tT_{X_{\tT_f},X_{\tT_{f'}}}) \arrow[d,"\sim"] \arrow[dd,two heads,bend left=60] \arrow[rd,dotted] & & \\
	C(\tT_{X_{\tT_f},X_{\tT_{f'}}}) \arrow[rrd,two heads,"\sim"'] \arrow[dd,two heads,bend right=60] & \hat{C}(\tT_{f'f}) \arrow[ld,two heads] \arrow[rd,"\sim"] & \\
	C(\aA)\times C(\bB) \arrow[d,"{\theta^2_\aA \times \theta^2_\bB}","\sim"'] \arrow[rd,"{\theta_\aA \times \theta_\bB}","\sim"'] && C(\tT_{f'f}) \arrow[ld,two heads]\\
	C(\tT_{I_\aA,I_\aA})\times C(\tT_{I_\bB,I_\bB}) \arrow[r,two heads,"\sim", "{\iota_{\aA13}^* \times \iota_{\bB13}^*}"'] & C(\tT_{I_\aA})\times C(\tT_{I_\bB}) &
\end{tikzcd}
\] 
Since $\hat{C}(\tT_{f'f})$ is a pullback, 
there is a morphism $\hat{C}(\tT_{X_{\tT_f},X_{\tT_{f'}}}) \rightarrow \hat{C}(\tT_{f'f})$ such that the whole diagram $\textcircled{\tiny 7}$ above is commutative. Combining the commutative diagrams $\textcircled{\tiny 1}$, $\textcircled{\tiny 4}$ and $\textcircled{\tiny 7}$, we can get the following  commutative diagrams $\textcircled{\tiny 8}$ and $\textcircled{\tiny 9}$
	\[
	\begin{tikzcd}
		\hat{C}(\tT_{X_{\tT_f},X_{\tT_{f'}}}) \arrow[r,two heads,"\sim"] \arrow[d] & \hat{C}(\tT_f) \arrow[d,two heads,"\sim"] \\
		\hat{C}(\tT_{f'f}) \arrow[r,two heads,"\sim"] & C(\tT_X)
	\end{tikzcd}
	\quad\quad
	\begin{tikzcd}
		\hat{C}(\tT_{X_{\tT_f},X_{\tT_{f'}}}) \arrow[r,two heads,"\sim"] \arrow[d] & \hat{C}(\tT_{f'}) \arrow[d,two heads,"\sim"] \\
		\hat{C}(\tT_{f'f}) \arrow[r,two heads,"\sim"] & C(\tT_{X''})
	\end{tikzcd}
	\]
which implies that the whole diagram $\textcircled{\tiny 6}$ above is commutative. From the commutative diagram $\textcircled{\tiny 6}$, we obtain
\[
\begin{array}{ll}
	& \ccC_{\aA\bB}([f'])\circ\ccC_{\aA\bB}([f]) \\[2mm]
	= & [C(\tT_{X'}) \stackrel{\sim}{\twoheadleftarrow} \hat{C}(\tT_{f'}) \stackrel{\sim}{\twoheadrightarrow} C(\tT_{X''})] \circ 
	[C(\tT_X) \stackrel{\sim}{\twoheadleftarrow} \hat{C}(\tT_f) \stackrel{\sim}{\twoheadrightarrow} C(\tT_{X'})] \\[2mm]
	= & [C(\tT_X) \stackrel{\sim}{\twoheadleftarrow} \hat{C}(\tT_f)\times_{C(\tT_{X'})}\hat{C}(\tT_{f'}) \stackrel{\sim}{\twoheadrightarrow} C(\tT_{X''})] \\[2mm]
	= & [C(\tT_X) \stackrel{\sim}{\twoheadleftarrow} \hat{C}(\tT_{X_{\tT_f},X_{\tT_{f'}}}) \stackrel{\sim}{\twoheadrightarrow} C(\tT_{X''})] \\
	= & [C(\tT_X) \stackrel{\sim}{\twoheadleftarrow} \hat{C}(\tT_{f'f}) \stackrel{\sim}{\twoheadrightarrow} C(\tT_{X''})] \\[2mm]
	= & \ccC_{\aA\bB}([f'f]).
\end{array}
\] 
	
(1.3) {\it $\ccC_{\aA\bB}$ preserves identity morphism:} We need to show that for any 1-cell $X\in\dgCAT_{\rm c,h}(\aA,\bB)$, $\ccC_{\aA\bB}([1_X]) = 1_{\ccC_{\aA\bB}(X)}$, i.e., $[C(\tT_X) \stackrel{\sim}{\twoheadleftarrow} \hat{C}(\tT_{1_X}) \stackrel{\sim}{\twoheadrightarrow} C(\tT_X)] = [C(\tT_X) \stackrel{1_{C(\tT_X)}}{\leftarrow} C(\tT_X) \stackrel{1_{C(\tT_X)}}{\rightarrow} C(\tT_X)]$. For this, it suffices to prove the following more general result, which will also be used in the proof (3) of Theorem \ref{Thm-LaxFuntor-HomotopyCat-B-Inf}.

\medspace

{\bf Claim.} \label{Claim-BimodIso-B-Inf-Span}
Let $[f]\in \dgCAT_{\rm c,h}(\aA,\bB)(X,X')$ be a 2-cell such that $f: X \rightarrow X'$ is a dg $\aA\-\bB$-bimodule isomorphism. Then $\ccC_{\aA\bB}([f])
:= [C(\tT_X) \stackrel{\sim}{\twoheadleftarrow} \hat{C}(\tT_f) \stackrel{\sim}{\twoheadrightarrow} C(\tT_{X'})]
= [C(\tT_X) \xleftarrow[\cong]{\tT_f^*} C(\tT_{X'}) \xrightarrow{1_{C(\tT_{X'})}} C(\tT_{X'})] 
= [C(\tT_X) \xleftarrow{1_{C(\tT_{X})}} C(\tT_{X}) \xrightarrow[\cong]{\tT^*_{f^{-1}}} C(\tT_{X'})]$.

\medspace

{\it Proof of Claim.}
Obviously, the $B_\infty$-isomorphism $\tT_f^*:C(\tT_{X'})\to C(\tT_{X})$ gives a morphism of spans of spans from $C(\tT_X) \xleftarrow[\cong]{\tT_f^*} C(\tT_{X'}) \xrightarrow{1_{C(\tT_{X'})}} C(\tT_{X'})$ 
to $C(\tT_X) \xleftarrow{1_{C(\tT_{X})}} C(\tT_{X}) \xrightarrow[\cong]{\tT^*_{f^{-1}}} C(\tT_{X'})$. 
\[
  \begin{tikzcd} [row sep=large]
	&& C(\tT_{X}) \arrow[lld,two heads] \arrow[rrd,two heads] && \\
	C(\aA) & C(\tT_{X'}) \arrow[ur,"\tT_f^*"',"\cong"] \arrow[dr,equal]  \arrow[rr,"\tT_f^*","\cong"'] && C(\tT_{X}) \arrow[ul,equal] \arrow[dl,"\tT_{f^{-1}}^*"',"\cong"] & C(\bB) \\
	&& C(\tT_{X'}) \arrow[llu,two heads] \arrow[rru,two heads] &&
  \end{tikzcd}
\]
So $(C(\tT_X) \xleftarrow[\cong]{\tT_f^*} C(\tT_{X'}) \xrightarrow{1_{C(\tT_{X'})}} C(\tT_{X'})) \sim (C(\tT_X) \xleftarrow{1_{C(\tT_{X})}} C(\tT_{X}) \xrightarrow[\cong]{\tT^*_{f^{-1}}} C(\tT_{X'}))$. Thus $[C(\tT_X) \xleftarrow[\cong]{\tT_f^*} C(\tT_{X'}) \xrightarrow{1_{C(\tT_{X'})}} C(\tT_{X'})] 
= [C(\tT_X) \xleftarrow{1_{C(\tT_{X})}} C(\tT_{X}) \xrightarrow[\cong]{\tT^*_{f^{-1}}} C(\tT_{X'})]$.

Let $\theta_f=(1_{C(\tT_X)},\tT_{f^{-1}}^*,s\tT_{f^{-1}}^*)^T:C(\tT_X)\to C(\tT_f)=C(\tT_X)\oplus C(\tT_{X'})\oplus D(X_{\tT_f})$. Similar to Lemma~\ref{Lem-ThetaA}, we can show that $\theta_f$ is a $B_\infty$-quasi-morphism.
Moreover, the following diagram is commutative.
\[
\begin{tikzcd}[row sep=large]
 C(\tT_X) \arrow[d,"\theta_f"',"\sim"] \arrow[r,two heads,"{\begin{pmatrix} \iota^*_1 \\ \iota^*_2 \end{pmatrix}}"] & C(\aA)\times C(\bB) \arrow[d,"{(\theta_\aA,\theta_\bB)}","\sim"'] \\
 C(\tT_f) \arrow[r,two heads,"{\begin{pmatrix} \iota^*_\aA \\ \iota^*_\bB	
 \end{pmatrix}}"] & C(\tT_{I_\aA})\times C(\tT_{I_\bB})
\end{tikzcd}
\]	
Since $\hat{C}(\tT_f)$ is a pullback, there is a morphism $C(\tT_X)\to \hat{C}(\tT_f)$ such that the following diagram is commutative.
\[
\begin{tikzcd}[row sep=large]
	& C(\tT_X) \arrow[dd,bend right=50,"\theta_f"',"\sim",near end] \arrow[dr,two heads,bend left=20] \arrow[d,dotted] & \\
	& \hat{C}(\tT_f) \arrow[r,two heads] \arrow[d,"\sim"'] & C(\aA)\times C(\bB) \arrow[d,"{(\theta_\aA,\theta_\bB)}","\sim"'] \\
	& C(\tT_f) \arrow[r,two heads,"{\begin{pmatrix} \iota^*_\aA \\ \iota^*_\bB \end{pmatrix}}"] 
	\arrow[dl,two heads,"\iota_1^*"] \arrow[dr,two heads,"\iota_2^*"'] & C(\tT_{I_\aA})\times C(\tT_{I_\bB}) \\
	C(\tT_X) && C(\tT_{X'})
\end{tikzcd}
\]	
Since $\iota_1^*\circ\theta_f=1_{C(\tT_{X})}$ and $\iota_2^*\circ\theta_f=\tT^*_{f^{-1}}$, this morphism $C(\tT_X)\to \hat{C}(\tT_f)$ gives a morphism of spans of spans from $C(\tT_X) \xleftarrow{1_{C(\tT_{X})}} C(\tT_{X}) \xrightarrow[\cong]{\tT^*_{f^{-1}}} C(\tT_{X'})$ to $C(\tT_X) \stackrel{\sim}{\twoheadleftarrow} \hat{C}(\tT_f) \stackrel{\sim}{\twoheadrightarrow} C(\tT_{X'})$. Thus $\ccC_{\aA\bB}([f])
:= [C(\tT_X) \stackrel{\sim}{\twoheadleftarrow} \hat{C}(\tT_f) \stackrel{\sim}{\twoheadrightarrow} C(\tT_{X'})]
= [C(\tT_X) \xleftarrow{1_{C(\tT_{X})}} C(\tT_{X}) \xrightarrow[\cong]{\tT^*_{f^{-1}}} C(\tT_{X'})]$.

\medspace

(2) {\it For any $\aA,\bB,\cC\in\dgCAT_{\rm c,h}$,  $\ccC^2_{\aA\bB\cC}:c_{\ccC(\aA)\ccC(\bB)\ccC(\cC)}\circ(\ccC_{\bB\cC}\times\ccC_{\aA\bB}) \Rightarrow \ccC_{\aA\cC}\circ c_{\aA\bB\cC}$
is a natural transformation:} We need to show
for any 1-cells $X,X'\in\dgCAT_{\rm c,h}(\aA,\bB),Y,Y'\in\dgCAT_{\rm c,h}(\bB,\cC)$ and 2-cells $[f]\in\dgCAT_{\rm c,h}(\aA,\bB)(X,X'), [g]\in\dgCAT_{\rm c,h}(\bB,\cC)(Y,Y')$, the following diagram is commutative,
\[
	\begin{tikzcd}
		\ccC(Y)\circ\ccC(X) \arrow[r,"{\ccC^2_{Y,X}}"] \arrow[d,"{\ccC([g])*\ccC([f])}"'] & \ccC(X\otimes_\bB Y) \arrow[d,"{\ccC([f\otimes g])}"] \\
		\ccC(Y')\circ\ccC(X') \arrow[r,"{\ccC^2_{Y',X'}}"] & \ccC(X'\otimes_\bB Y') 
	\end{tikzcd}
\]
i.e., $\ccC([f\otimes g])\circ\ccC^2_{Y,X} = \ccC^2_{Y',X'}\circ(\ccC([g])*\ccC([f]))$. 

Firstly, the morphisms $f$ and $g$ induce a dg functor
\[
	\tT_{f,g}:\tT_{X,Y} \rightarrow \tT_{X',Y'}.
\]
Since $f$ and $g$ are quasi-isomorphisms, the dg functor $\tT_{f,g}$ is a quasi-equivalence. Thus the natural projections
\[
	\iota_1^*:C(\tT_{f,g}) \twoheadrightarrow C(\tT_{X,Y}) \quad\mathrm{and}\quad \iota_2^*:C(\tT_{f,g}) \twoheadrightarrow C(\tT_{X',Y'})
\]
are quasi-isomorphisms. Observe the upper trianglar matrix dg category 
\[
\tT_{X_{\tT_{f,g}}}=
\begin{pmatrix}
	\tT_{X,Y}&X_{\tT_{f,g}}\\
	&\tT_{X',Y'}
\end{pmatrix}
=
  \begin{pmatrix}
	\aA & X & X\otimes_{\bB}Y & I_\aA & X' & X'\otimes_{\bB}Y' \\
	& \bB & Y & & I_\bB & Y' \\
	& & \cC & & & I_\cC \\
	& & & \aA & X' & X'\otimes_{\bB} Y' \\
	& & & & \bB & Y' \\
	& & & & & \cC
  \end{pmatrix}.
\]
Let $\hat{C}(\tT_{f,g})$ be the following pullback in $B_\infty$ (Ref. Lemma~\ref{Lem-ThetaA}).
\[
  \begin{tikzcd}[column sep=large]
	\hat{C}(\tT_{f,g}) \arrow[r,two heads] \arrow[d, "\sim"'] & C(\aA)\times C(\bB)\times C(\cC) \arrow[d, "{\theta_\aA\times\theta_\bB\times\theta_\cC}","\sim"'] \\
	C(\tT_{f,g}) \arrow[r,two heads,"{(\iota^*_\aA,\iota^*_\bB,\iota^*_\cC)}"] & C(\tT_{ I_\aA})\times C(\tT_{ I_\bB})\times C(\tT_{ I_\cC})
  \end{tikzcd}
\]	
The commutative diagram of fully faithful dg functors
\[
	\begin{tikzcd}
		\tT_{X,Y} \arrow[r,hook] & \tT_{X_{\tT_{f,g}}} & \tT_{X',Y'} \arrow[l,hook] \\
		\tT_{X\otimes_\bB Y} \arrow[r,hook] \arrow[u,hook] & \tT_{X_{\tT_{f\otimes g}}} \arrow[u,hook] & \tT_{X'\otimes_\bB Y'} \arrow[l,hook] \arrow[u,hook]
	\end{tikzcd}
\]
induces the following commutative diagram $\textcircled{\tiny 10}$ of $B_\infty$-algebras
\[
	\begin{tikzcd}
		& \hat{C}(\tT_{f,g}) \arrow[d,"\sim"] & \\
		C(\tT_{X,Y}) \arrow[d,two heads] & C(\tT_{f,g}) \arrow[l,two heads,"\sim"'] \arrow[d,two heads] \arrow[r,two heads,"\sim"]  & C(\tT_{X',Y'}) \arrow[d,two heads] \\
		C(\tT_{X'\otimes Y'}) & C(\tT_{f\otimes g}) \arrow[l,two heads,"\sim"'] \arrow[r,two heads,"\sim"] & C(\tT_{X'\otimes Y'})
	\end{tikzcd}
\]
The following diagram $\textcircled{\tiny 11}$ is commutative when we consider only solid arrows.
\[
	\begin{tikzcd}
		\hat{C}(\tT_{f,g}) \arrow[rr,two heads] \arrow[dd,"\sim"] \arrow[rd,dotted] & & C(\aA)\times C(\bB)\times C(\cC) \arrow[dd,"\sim",near start,"{\theta_\aA\times\theta_\bB\times\theta_\cC}"',near start] \arrow[rd,two heads] & \\
		& \hat{C}(\tT_f) \arrow[rr,two heads] \arrow[dd,"\sim",near start] & & C(\aA)\times C(\bB) \arrow[dd,"\sim","{\theta_\aA\times\theta_\bB}"'] \\
		C(\tT_{f,g}) \arrow[rd,two heads] \arrow[rr,two heads,near start,"{(\iota^*_\aA,\iota^*_\bB,\iota^*_\cC)}"] & & C(\tT_{I_\aA})\times C(\tT_{I_\bB})\times C(\tT_{I_\cC}) \arrow[rd,two heads] \\
		& C(\tT_f) \arrow[rr,two heads,"{(\iota^*_\aA,\iota^*_\bB)}"] & & C(\tT_{I_\aA})\times C(\tT_{I_\bB}) 
	\end{tikzcd}
\]
By the universal property of pullback $\hat{C}(\tT_{f})$, there is a unique morphism $\hat{C}(\tT_{f,g}) \rightarrow \hat{C}(\tT_{f})$ such that the whole diagram  $\textcircled{\tiny 11}$ above is commutative. In particular, we have the following commutative diagram $\textcircled{\tiny 12}$.
\[
\begin{tikzcd}
	\hat{C}(\tT_{f,g}) \arrow[d,"\sim"] \arrow[r,two heads] & \hat{C}(\tT_f) \arrow[d,"\sim"] \\
	C(\tT_{f,g}) \arrow[r,two heads] & C(\tT_f) 
\end{tikzcd}
\]
Similarly, there are morphisms 
	\[
	\hat{C}(\tT_{f,g}) \rightarrow \hat{C}(\tT_g)\quad\text{and}\quad  \hat{C}(\tT_{f,g}) \rightarrow \hat{C}(\tT_{f\otimes g})
	\]
such that the following diagrams $\textcircled{\tiny 13}$ and $\textcircled{\tiny 14}$ are commutative. 
\[
  \begin{tikzcd}
	\hat{C}(\tT_{f,g}) \arrow[d,"\sim"] \arrow[r,two heads] & \hat{C}(\tT_g) \arrow[d,"\sim"] \\
	C(\tT_{f,g}) \arrow[r,two heads] & C(\tT_g) 
  \end{tikzcd}
  \quad\quad 
  \begin{tikzcd}
	\hat{C}(\tT_{f,g}) \arrow[d,"\sim"] \arrow[r,two heads] & \hat{C}(\tT_{f\otimes g}) \arrow[d,"\sim"] \\
	C(\tT_{f,g}) \arrow[r,two heads] & C(\tT_{f\otimes g}) 
  \end{tikzcd}
\]
Combining the commutative diagrams $\textcircled{\tiny 14}$ and $\textcircled{\tiny 10}$, we get the following commutative diagram $\textcircled{\tiny 15}$.
\[
	\begin{tikzcd}
		\hat{C}(\tT_{f,g}) \arrow[r,"\sim"] \arrow[d,two heads] & C(\tT_{f,g}) \arrow[r,two heads,"\sim"] \arrow[d,two heads] & C(\tT_{X,Y}) \arrow[d,two heads] \\
		\hat{C}(\tT_{f\otimes g}) \arrow[r,"\sim"] & C(\tT_{f\otimes g}) \arrow[r,two heads,"\sim"] & C(\tT_{X\otimes_\bB Y})
	\end{tikzcd}
\]

Secondly, by a similar proof, we can obtain the following commutative diagram $\textcircled{\tiny 16}$.
\[
\begin{tikzcd}
	\hat{C}(\tT_{f,g}) \arrow[r,"\sim"] \arrow[d,two heads] & C(\tT_{f,g}) \arrow[r,two heads,"\sim"] \arrow[d,two heads] & C(\tT_{X',Y'}) \arrow[d,two heads] \\
	\hat{C}(\tT_{f\otimes g}) \arrow[r,"\sim"] & C(\tT_{f\otimes g}) \arrow[r,two heads,"\sim"] & C(\tT_{X'\otimes_\bB Y'})
\end{tikzcd}
\]	

Thirdly, combining the commutative diagrams $\textcircled{\tiny 12}$ and $\textcircled{\tiny 13}$, we can get the following diagram $\textcircled{\tiny 17}$ which is commutative when we consider only those solid arrows.
\[
\begin{tikzcd}[row sep=large]
	& \hat{C}(\tT_{f,g}) \arrow[ld,two heads] \arrow[d,"\sim"] \arrow[rr,dotted] \arrow[rrrd,two heads] && \hat{C}(\tT_f)\times_{C(\bB)}\hat{C}(\tT_g) \arrow[llld,] \arrow[d,"\sim"] \arrow[rd] & \\
	\hat{C}(\tT_f) \arrow[d,"\sim"] & 
	C(\tT_{f,g}) \arrow[ld,two heads] \arrow[d,two heads,"\sim"] \arrow[rrrd,two heads] && C(\tT_f)\times_{C(\bB)}C(\tT_g) \arrow[llld] \arrow[d] \arrow[rd] & \hat{C}(\tT_g) \arrow[d,"\sim"]\\
	C(\tT_f) \arrow[rd,two heads,"\sim"] & C(\tT_{X,Y}) \arrow[d,two heads] \arrow[rr,two heads,"\sim"'] \arrow[rrd,two heads] && \bar{C}(\tT_{X,Y}) \arrow[lld,two heads] \arrow[d,two heads] & C(\tT_g) \arrow[ld,two heads,"\sim"] \\		 
	& C(\tT_{X}) \arrow[rd,two heads] && C(\tT_{Y}) \arrow[ld,two heads] &  \\
	& & C(\bB)	& &
\end{tikzcd}
\]	
By the universal property of pullback $\hat{C}(\tT_f)\times_{C(\bB)}\hat{C}(\tT_g)$, there is a unique morphism $\hat{C}(\tT_{f,g}) \rightarrow \hat{C}(\tT_f)\times_{C(\bB)}\hat{C}(\tT_g)$ such that the whole diagram $\textcircled{\tiny 17}$ above is commutative. Indeed, by the universal property of pullback $\bar{C}(\tT_{X,Y})$, the following diagram $\textcircled{\tiny 18}$ is commutative.
\[
\begin{tikzcd}
	\hat{C}(\tT_{f,g}) \arrow[r,"\sim"] \arrow[d] & C(\tT_{f,g}) \arrow[r,two heads,"\sim"] & C(\tT_{X,Y}) \arrow[d,two heads,"\sim"'] \\
	\hat{C}(\tT_f)\times_{C(\bB)}\hat{C}(\tT_g) \arrow[r] & C(\tT_f)\times_{C(\bB)} C(\tT_g) \arrow[r] & \bar{C}(\tT_{X,Y})
\end{tikzcd}	
\]

Fourthly, by a similar proof, we can obtain the following commutative diagram $\textcircled{\tiny 19}$.
\[
	\begin{tikzcd}
		\hat{C}(\tT_{f,g}) \arrow[r, "\sim"] \arrow[d] & C(\tT_{f,g}) \arrow[r,two heads,"\sim"] & C(\tT_{X',Y'}) \arrow[d,two heads,"\sim"'] \\
		\hat{C}(\tT_f)\times_{C(\bB)}\hat{C}(\tT_g) \arrow[r] & C(\tT_f)\times_{C(\bB)} C(\tT_g) \arrow[r] & \bar{C}(\tT_{X',Y'})
	\end{tikzcd}
\]
Note that the composition $	\hat{C}(\tT_{f,g})\to\hat{C}(\tT_f)\times_{C(\bB)}\hat{C}(\tT_g)\to C(\tT_f)\times_{C(\bB)} C(\tT_g) \to\bar{C}(\tT_{X',Y'})$ above is a surjective quasi-isomorphism.

Fifthly, combining the four commutative diagrams $\textcircled{\tiny 15}$,  $\textcircled{\tiny 16}$,  $\textcircled{\tiny 18}$ and $\textcircled{\tiny 19}$ obtained above, we get the following commutative diagram.
\[
	\begin{tikzcd}[column sep=tiny]
		& & \bar{C}(\tT_{X,Y})  & & \\
		& C(\tT_{X,Y}) \arrow[ru,two heads,"\sim"] \arrow[ld,two heads] & & \hat{C}(\tT_f)\times_{C(\bB)}\hat{C}(\tT_g) \arrow[lu,two heads,"\sim"'] \arrow[rd] & \\
		C(\tT_{X\otimes_\bB Y}) & & \hat{C}(\tT_{f,g}) \arrow[lu,two heads,"\sim"] \arrow[ld,two heads]  \arrow[ru] \arrow[rd,two heads,"\sim"'] & & \bar{C}(\tT_{X',Y'}) \\
		& \hat{C}(\tT_{f\otimes g}) \arrow[lu,two heads,"\sim"] \arrow[rd,two heads,"\sim"'] & & C(\tT_{X',Y'}) \arrow[ru,two heads,"\sim"'] \arrow[ld,two heads] & \\
		& & C(\tT_{X'\otimes_\bB Y'}) & & 
	\end{tikzcd}
\]

Finally, by the properties of pullbacks $C':=C(\tT_{X,Y})\times_{C(\tT_{X\otimes_\bB Y})} \hat{C}(\tT_{f\otimes g})$ and $C'':=(\hat{C}(\tT_f)\times_{C(\bB)}\hat{C}(\tT_g))\times_{\bar{C}(\tT_{X',Y'})} C(\tT_{X',Y'})$, there exist morphisms $\hat{C}(\tT_{f,g}) \to C'$ and $\hat{C}(\tT_{f,g}) \to C''$ such that the following diagram is commutative.
\[
	\begin{tikzcd}[column sep=tiny]
		& & \bar{C}(\tT_{X,Y})  & & \\
		& C(\tT_{X,Y}) \arrow[ru,two heads,"\sim"] \arrow[ld,two heads] & & \hat{C}(\tT_f)\times_{C(\bB)}\hat{C}(\tT_g) \arrow[lu,two heads,"\sim"'] \arrow[rd] & \\
	    C(\tT_{X\otimes_\bB Y}) & C' \arrow[u,two heads,"\sim"] \arrow[d,two heads] & \hat{C}(\tT_{f,g}) \arrow[l,dotted] \arrow[r,dotted] \arrow[lu,two heads,"\sim"] \arrow[ld,two heads] \arrow[ru] \arrow[rd,two heads,"\sim"'] & C'' \arrow[u,two heads,"\sim"] \arrow[d,two heads] & \bar{C}(\tT_{X',Y'}) \\
		& \hat{C}(\tT_{f\otimes g}) \arrow[lu,two heads,"\sim"] \arrow[rd,two heads,"\sim"'] & & C(\tT_{X',Y'}) \arrow[ru,two heads,"\sim"'] \arrow[ld,two heads] & \\
		& & C(\tT_{X'\otimes_\bB Y'}) & & 
	\end{tikzcd}
\]
Thus {\small 
\[
\begin{array}{ll}
	& \ccC([f\otimes g])\circ\ccC^2_{Y,X} \\[2mm]
	= & [C(\tT_{X\otimes_\bB Y}) \stackrel{\sim}{\twoheadleftarrow} \hat{C}(\tT_{f\otimes g}) \rightarrow C(\tT_{X'\otimes_\bB Y'})] \circ [\bar{C}(\tT_{X,Y}) \stackrel{\sim}{\twoheadleftarrow} C(\tT_{X,Y}) \rightarrow C(\tT_{X\otimes_\bB Y})] \\ [2mm]
	= & [\bar{C}(\tT_{X,Y}) \stackrel{\sim}{\twoheadleftarrow} C' \rightarrow C(\tT_{X'\otimes_\bB Y'})] \\[2mm]
	= & [\bar{C}(\tT_{X,Y}) \stackrel{\sim}{\twoheadleftarrow} \hat{C}(\tT_{f,g}) \rightarrow C(\tT_{X'\otimes_\bB Y'})] \\[2mm]
	= & [\bar{C}(\tT_{X,Y}) \stackrel{\sim}{\twoheadleftarrow} C'' \rightarrow C(\tT_{X'\otimes_\bB Y'})] \\[2mm]
	= & [\bar{C}(\tT_{X',Y'}) \stackrel{\sim}{\twoheadleftarrow} C(\tT_{X',Y'}) \rightarrow C(\tT_{X'\otimes_\bB Y'})] \circ [\bar{C}(\tT_{X,Y}) \stackrel{\sim}{\twoheadleftarrow} \hat{C}(\tT_f)\times_{C(\bB)}\hat{C}(\tT_g) \rightarrow \bar{C}(\tT_{X',Y'})] \\[2mm]
	= & \ccC^2_{Y',X'}\circ(\ccC([g])*\ccC([f])).
\end{array}
\]}

\medspace

(3) {\it $\ccC$ satisfies the lax associativity:} We need to show that
for any objects $\aA,\bB,\cC,\dD\in\dgCAT_{\rm c,h}$ and 1-cells $X\in\dgCAT_{\rm c,h}(\aA,\bB), Y\in\dgCAT_{\rm c,h}(\bB,\cC), Z\in\dgCAT_{\rm c,h}(\cC,\dD)$, the following diagram is commutative
\[
\begin{tikzcd}[row sep=large,column sep=huge]
	\left(\ccC(Z)\circ\ccC(Y)\right)\circ\ccC(X) \arrow[r,"a_{\ccC(Z),\ccC(Y),\ccC(X)}"] \arrow[d,"\ccC^2_{Z,Y}*1_{\ccC(X)}"'] & \ccC(Z)\circ(\ccC(Y)\circ\ccC(X)) \arrow[d,"1_{\ccC(Z)}*\ccC^2_{Y,X}"] \\
	\ccC(Y\otimes_\cC Z)\circ\ccC(X) \arrow[d,"\ccC^2_{Y\otimes_\cC Z,X}"'] & \ccC(Z)\circ\ccC(X\otimes_\bB Y) \arrow[d,"\ccC^2_{Z,X\otimes_\bB Y}"] \\
	\ccC(X\otimes_\bB(Y\otimes_\cC Z)) \arrow[r, "\ccC(a_{Z,Y,X})"] & \ccC((X\otimes_\bB Y)\otimes_\cC Z), 
\end{tikzcd}
\]
i.e., $\ccC^2_{Z,X\otimes_\bB Y}\circ(1_{\ccC(Z)}*\ccC^2_{Y,X}) \circ a_{\ccC(Z),\ccC(Y),\ccC(X)} = \ccC(a_{Z,Y,X})\circ \ccC^2_{Y\otimes_\cC Z,X} \circ (\ccC^2_{Z,Y}*1_{\ccC(X)})$.

Firstly, we observe the upper triangular matrix dg category
\[
\tT_{X,Y,Z}=\begin{pmatrix}
	\aA & X & X\otimes_\bB Y & X\otimes_\bB Y\otimes_\cC Z \\
	& \bB & Y & Y\otimes_\cC Z \\
	& & \cC & Z \\
	& & & \dD
\end{pmatrix}.
\]
The commutative diagram of fully faithful dg functors
\[
\begin{tikzcd}
	\tT_{X,Y,Z} & \tT_{X,Y\otimes_\cC Z} \arrow[l,hook] & \tT_{X\otimes_\bB(Y\otimes_\cC Z)} \arrow[l,hook] \arrow[dl,"\cong"',"\tT_{a_{Z,Y,X}}"] \\
	\tT_{X\otimes_\bB Y,Z} \arrow[u, hook] & \tT_{(X\otimes_\bB Y)\otimes_\cC Z} \arrow[l, hook]
\end{tikzcd}
\]
induces the following commutative diagram $\textcircled{\tiny 20}$ of $B_\infty$-algebras.
\[
\begin{tikzcd}
	C(\tT_{X,Y,Z}) \arrow[r,two heads] \arrow[d,two heads] & C(\tT_{X,Y\otimes_\cC Z}) \arrow[r,two heads] & C(\tT_{X\otimes_\bB(Y\otimes_\cC Z)}) \\
	C(\tT_{X\otimes_\bB Y}, Z) \arrow[r,two heads] & C(\tT_{(X\otimes_\bB Y)\otimes_\cC Z}) \arrow[ur,"\cong","\tT_{a_{Z,Y,X}}^*"'] & 
\end{tikzcd}
\]

Secondly, the dg category $\tT_{X,Y,Z}=\tT_{X, (Y\ Y\otimes_\cC Z)}$ and $\tT_{Y,Z}=\tT_{(Y\ Y\otimes_\cC Z)}$, where $(Y\ Y\otimes_\cC Z)$ is the natural $\bB$-$\tT_Z$-bimodule. Thus
\[
C(\tT_X)\times_{C(\bB)}C(\tT_{Y,Z})=C(\tT_X)\times_{C(\bB)}C(\tT_{(Y\ Y\otimes_\cC Z)})=\bar{C}(\tT_{X,(Y\ Y\otimes_\cC Z)}).
\]
By Lemma \ref{Lem-CTXY-bar}, we have
\[
C(\tT_{X,Y,Z})=C(\tT_{X,(Y\ Y\otimes_\cC Z)}) \stackrel{\sim}{\twoheadrightarrow} \bar{C}(\tT_{X,(Y\ Y\otimes_\cC Z)})=C(\tT_X)\times_{C(\bB)}C(\tT_{Y,Z}).
\]
Similarly, the dg category $\tT_{X,Y,Z}=\tT_{{X\otimes_\bB Y\choose Y},Z}$ and $\tT_{X,Y}=\tT_{X\otimes_\bB Y\choose Y}$, where {\tiny ${X\otimes_\bB Y\choose Y}$} is the natural $\tT_X$-$\cC$-bimodule. Thus
\[
C(\tT_{X,Y})\times_{C(\cC)}C(\tT_{Z}) = C(\tT_{X\otimes_\bB Y\choose Y}) \times_{C(\cC)} C(\tT_Z) = \bar{C}(\tT_{{X\otimes_\bB Y\choose Y}, Z}).
\]
By Lemma \ref{Lem-CTXY-bar}, we have
\[
C(\tT_{X,Y,Z})=C(\tT_{{X\otimes_\bB Y\choose Y}, Z}) \stackrel{\sim}{\twoheadrightarrow} \bar{C}(\tT_{{X\otimes_\bB Y\choose Y}, Z})=C(\tT_{X,Y})\times_{C(\cC)}C(\tT_{Z}).
\]
Then we get the following commutative diagram $\textcircled{\tiny 21}$ 
\[
\begin{tikzcd}[column sep=tiny]
	&& C(\tT_{X,Y,Z}) \arrow[ld,two heads,"\sim"] \arrow[rd,two heads,"\sim"'] && \\
	& C(\tT_{X,Y})\times_{C(\cC)}C(\tT_Z) \arrow[ld,two heads] \arrow[d,two heads,"\sim"] && C(\tT_X)\times_{C(\bB)}C(\tT_{Y,Z}) \arrow[d,two heads,"\sim"'] \arrow[rd,two heads] & \\
	C(\tT_{X,Y}) \arrow[dd,two heads] \arrow[rd,two heads,"\sim"'] & \bar{C}(\tT_{X,Y})\times_{C(\cC)}C(\tT_Z) \arrow[d,two heads] \arrow[ddrrr,two heads] && C(\tT_X)\times_{C(\bB)}\bar{C}(\tT_{Y,Z}) \arrow[ll,"a_{\ccC(Z),\ccC(Y),\ccC(X)}"',"\cong"] \arrow[ddlll,two heads]\arrow[d,two heads] & C(\tT_{Y,Z}) \arrow[ld,two heads,"\sim"] \arrow[dd,two heads] \\
	& \bar{C}(\tT_{X,Y}) \arrow[ld,two heads] \arrow[rd,two heads] && \bar{C}(\tT_{Y,Z}) \arrow[ld,two heads] \arrow[rd,two heads] & \\
	C(\tT_X) \arrow[d,two heads] \arrow[rd,two heads] && C(\tT_Y) \arrow[ld,two heads] \arrow[rd,two heads] && C(\tT_Z) \arrow[ld,two heads] \arrow[d,two heads] \\
	C(\aA) & C(\bB) && C(\cC) & C(\dD) 
\end{tikzcd}
\]
where the commutativity of the top pentagon can be proved by using the universality of the pullback $C(\tT_X)\times_{C(\bB)}\bar{C}(\tT_{Y,Z})$. 

Thirdly, by the universal property of the pullback $\bar{C}(\tT_{X,Y\otimes_\cC Z})$, we can obtain the following commutative diagram $\textcircled{\tiny 22}$.
\[ 
\begin{tikzcd}
	C(\tT_{X,Y,Z}) \arrow[r,two heads] \arrow[d,two heads] & C(\tT_{X,Y\otimes_\cC Z}) \arrow[d,two heads,"\sim"] & \\
	C(\tT_X)\times_{C(\bB)}C(\tT_{Y,Z}) \arrow[r,dotted] \arrow[d,two heads] & \bar{C}(\tT_{X,Y\otimes_\cC Z}) \arrow[r,two heads] \arrow[d,two heads] & C(\tT_X) \arrow[d,two heads] \\
	C(\tT_{Y,Z}) \arrow[r,two heads] & C(\tT_{Y\otimes_\cC Z}) \arrow[r,two heads] & C(\bB)
\end{tikzcd}
\]

Fourthly, by the universal property of the pullback $\bar{C}(\tT_{X\otimes_\bB Y,Z})$, we can obtain the following commutative diagram $\textcircled{\tiny 23}$.
\[
\begin{tikzcd}
	C(\tT_{X,Y,Z}) \arrow[r,two heads] \arrow[d,two heads] & C(\tT_{X\otimes_\bB Y,Z}) \arrow[d,two heads,"\sim"] & \\
	C(\tT_{X,Y})\times_{C(\cC)}C(\tT_{Z}) \arrow[r,dotted] \arrow[d,two heads] & \bar{C}(\tT_{X\otimes_\bB Y, Z}) \arrow[r,two heads] \arrow[d,two heads] & C(\tT_Z) \arrow[d,two heads] \\
	C(\tT_{X,Y}) \arrow[r,two heads] & C(\tT_{X\otimes_\bB Y}) \arrow[r,two heads] & C(\cC)
\end{tikzcd}
\]

Fifthly, by the Claim in the proof (1.3) of Theorem \ref{Thm-LaxFuntor-HomotopyCat-B-Inf}, the 2-cell $\ccC(a_{Z,Y,X})=$ \linebreak $[C(\tT_{X\otimes_\bB(Y\otimes_\cC Z)}) \xleftarrow{\tT_{a_{Z,Y,X}}^*} C(\tT_{(X\otimes_\bB Y)\otimes_\cC Z}) = C(\tT_{(X\otimes_\bB Y)\otimes_\cC Z})]$.
\[
\begin{tikzcd}
	& C(\tT_{X\otimes_\bB(Y\otimes_\cC Z)}) \arrow[ld,two heads] \arrow[rd,two heads] & \\
	C(\aA) & C(\tT_{(X\otimes_\bB Y)\otimes_\cC Z}) \arrow[d,equal] \arrow[u,"{\tT_{a_{Z,Y,X}}^*}"',"\cong"] & C(\dD) \\
	& C(\tT_{(X\otimes_\bB Y)\otimes_\cC Z}) \arrow[lu,two heads] \arrow[ru,two heads] & 
\end{tikzcd}
\]

Finally, from the commutative diagrams $\textcircled{\tiny 20}$, $\textcircled{\tiny 21}$, $\textcircled{\tiny 22}$ and $\textcircled{\tiny 23}$, we obtain the following commutative diagram.
\[
\begin{tikzcd}[column sep=tiny]
	C(\tT_X)\times_{C(\bB)}\bar{C}(\tT_{Y,Z}) \arrow[rr,"a_{\ccC(Z),\ccC(Y),\ccC(X)}","\cong"'] & & \bar{C}(\tT_{X,Y})\times_{C(\cC)}C(\tT_{Z}) \\
	C(\tT_X)\times_{C(\bB)}C(\tT_{Y,Z}) \arrow[u,two heads,"\sim"'] \arrow[d] & & C(\tT_{X,Y})\times_{C(\cC)}C(\tT_{Z}) \arrow[u,two heads,"\sim"'] \arrow[d] \\
	\bar{C}(\tT_{X,Y\otimes_\cC Z}) & C(\tT_{X,Y,Z}) \arrow[ul,two heads,"\sim"'] \arrow[dl]  \arrow[ur,two heads,"\sim"] \arrow[dr] & \bar{C}(\tT_{X\otimes_\bB Y,Z}) \\
	C(\tT_{X,Y\otimes_\cC Z}) \arrow[u,two heads,"\sim"'] \arrow[d] & & C(\tT_{X\otimes_\bB Y,Z}) \arrow[u,two heads,"\sim"'] \arrow[d] \\
	C(\tT_{X\otimes_\bB(Y\otimes_\cC Z)}) & &  C(\tT_{(X\otimes_\bB Y)\otimes_\cC Z}) \arrow[ll,"\tT_{a_{Z,Y,X}}^*"',"\cong"]
\end{tikzcd}
\]
Thus 
\[
\begin{array}{ll}
	& \ccC^2_{Z,X\otimes_\bB Y}\circ(1_{\ccC(Z)}*\ccC^2_{Y,X}) \circ a_{\ccC(Z),\ccC(Y),\ccC(X)} \\[2mm]
	= & [\bar{C}(\tT_{X\otimes_\bB Y,Z}) \stackrel{\sim}{\twoheadleftarrow} C(\tT_{X\otimes_\bB Y,Z}) \rightarrow C(\tT_{(X\otimes_\bB Y)\otimes_\cC Z})] \\[2mm] 
	& \circ [\bar{C}(\tT_{X,Y})\times_{C(\cC)}C(\tT_{Z}) \stackrel{\sim}{\twoheadleftarrow} C(\tT_{X,Y})\times_{C(\cC)}C(\tT_{Z}) \rightarrow \bar{C}(\tT_{X\otimes_\bB Y,Z})] \\[2mm]
	& \circ [C(\tT_X)\times_{C(\bB)}\bar{C}(\tT_{Y,Z}) = C(\tT_X)\times_{C(\bB)}\bar{C}(\tT_{Y,Z}) \xrightarrow{a_{\ccC(Z),\ccC(Y),\ccC(X)}} \bar{C}(\tT_{X,Y})\times_{C(\cC)}C(\tT_{Z})] \\[2mm]
	= & [C(\tT_X)\times_{C(\bB)}\bar{C}(\tT_{Y,Z}) \stackrel{\sim}{\twoheadleftarrow} C(\tT_{X,Y,Z}) \rightarrow C(\tT_{(X\otimes_\bB Y)\otimes_\cC Z})] \\[2mm]
	= & [C(\tT_{X\otimes_\bB(Y\otimes_\cC Z)}) \xleftarrow{\tT_{a_{Z,Y,X}}^*} C(\tT_{(X\otimes_\bB Y)\otimes_\cC Z}) = C(\tT_{(X\otimes_\bB Y)\otimes_\cC Z})] \\[2mm]
	& \circ [\bar{C}(\tT_{X,Y\otimes_\cC Z}) \stackrel{\sim}{\twoheadleftarrow} C(\tT_{X,Y\otimes_\cC Z}) \rightarrow C(\tT_{X\otimes_\bB(Y\otimes_\cC Z)})] \\[2mm]
	& \circ [C(\tT_X)\times_{C(\bB)}\bar{C}(\tT_{Y,Z}) \stackrel{\sim}{\twoheadleftarrow} C(\tT_X)\times_{C(\bB)}C(\tT_{Y,Z}) \rightarrow \bar{C}(\tT_{X,Y\otimes_\cC Z})] \\[2mm]
	= & \ccC(a_{Z,Y,X})\circ \ccC^2_{Y\otimes_\cC Z,X} \circ (\ccC^2_{Z,Y}*1_{\ccC(X)}).
\end{array}
\]

\medspace

(4) {\it $\ccC$ satisfies the lax left unity:} We need to show that for any objects $\aA,\bB\in\dgCAT_{\rm c,h}$ and 1-cell $X\in\dgCAT_{\rm c,h}(\aA,\bB)$, the following diagram is commutative.
\[
\begin{tikzcd}
	1_{C(\bB)}\circ\ccC(X) \arrow[r,"l_{\ccC(X)}"] \arrow[d,"\ccC^0_{\bB}*1_{\ccC(X)}"'] & \ccC(X) \\
	\ccC(\bp I_\bB)\circ\ccC(X) \arrow[r,"\ccC^2_{\bp I_\bB,X}"] & \ccC(X\otimes_\bB \bp I_\bB) \arrow[u,"\ccC(l_X)"']
\end{tikzcd}
\]
i.e., $\ccC(l_X) \circ \ccC^2_{\bp I_\bB,X} \circ (\ccC^0_{\bB}*1_{\ccC(X)}) = l_{\ccC(X)}$.

Firstly, the 2-cell $\ccC^0_{\bB}*1_{\ccC(X)} = [C(\tT_X)\times_{C(\bB)}C(\bB) \stackrel{\sim}{\twoheadleftarrow} C(\tT_X)\times_{C(\tT_{I_\bB})}\hat{C}(\tT_{p_{I_\bB}}) \stackrel{\sim}{\rightarrow} \bar{C}(\tT_{X,\bp I_\bB})]$ is given by the following commutative diagram $\textcircled{\tiny 24}$, where the dashed and dotted arrow are obtained by the university of pullbacks.
\[
\begin{tikzcd}[column sep=small]
	&& C(\tT_X)\times_{C(\bB)}C(\bB) \arrow[dl,"\cong"',"{l_{C(\tT_X)}}"] \arrow[dr,two heads] && \\
	& C(\tT_X) \arrow[dl,two heads] \arrow[d,equal] \arrow[dr,two heads] && C(\bB) \arrow[dl,equal] \arrow[d,"\sim","\theta_\bB"'] \arrow[dr,equal] & \\
	C(\aA) & C(\tT_X) \arrow[d,equal] \arrow[l,two heads] \arrow[r,two heads]& C(\bB) & C(\tT_{I_\bB}) \arrow[l,two heads,"\sim"']  \arrow[r,two heads,"\sim",near end] & C(\bB) \\
	& C(\tT_X) \arrow[ul,two heads] \arrow[ur,two heads] & \bar{C}(\tT_{X,I_\bB}) \arrow[ul,two heads,"\sim"',near end] \arrow[ur,two heads] \arrow[uuu,two heads,dotted,bend right=21,"\sim",near end] & \hat{C}(\tT_{p_{I_\bB}}) \arrow[u,two heads,"\sim"'] \arrow[d,two heads,"\sim",near start] & C(\bB)\times_{C(\tT_{I_\bB})}\hat{C}(\tT_{p_{I_\bB}}) \arrow[uul,two heads,"\sim",near end] \arrow[l,"\sim"']  \\
	&& C(\tT_X)\times_{C(\tT_{I_\bB})}\hat{C}(\tT_{p_{I_\bB}}) \arrow[uul,two heads,"\sim"',near start] \arrow[u,dotted] \arrow[urr,two heads,dashed] \arrow[ur] \arrow[d,dotted] & C(\tT_{\bp I_\bB}) \arrow[uul,two heads,"\sim"'] \arrow[uur,two heads,bend left=8,"\sim"',near end] & \\
	&& \bar{C}(\tT_{X,\bp I_\bB}) \arrow[uul,two heads,"\sim",bend left=10] \arrow[ur,two heads] & & 
\end{tikzcd}	
\]	

Secondly, by the result obtained in the first step of the proof of (2), i.e., the commuatative diagram $\textcircled{\tiny 15}$, we have the following commutative diagram $\textcircled{\tiny 25}$.
\[
\begin{tikzcd}
	\hat{C}(\tT_{1_X,p_{I_\bB}}) \arrow[r] \arrow[d] & C(\tT_{1_X,p_{I_\bB}}) \arrow[r] \arrow[d] & C(\tT_{X,\bp I_\bB}) \arrow[d] \\
	\hat{C}(\tT_{1_X\otimes p_{I_\bB}}) \arrow[r] & C(\tT_{1_X\otimes p_{I_\bB}}) \arrow[r] & C(\tT_{X\otimes_\bB \bp I_\bB})
\end{tikzcd}
\]
	
Thirdly, three morphisms of dg $\aA$-$\bB$-bimodules 
$$l_X: X\otimes_\bB \bp I_\bB \xrightarrow[\sim]{1_X\otimes p_{I_\bB}} X\otimes_\bB I_\bB \xrightarrow[\cong]{m} X,$$ 
induce three dg functors	
$$\tT_{l_X}: \tT_{X\otimes_\bB \bp I_\bB} \xrightarrow[\simeq]{\tT_{1_X\otimes p_{I_\bB}}} \tT_{X\otimes_\bB I_\bB} \xrightarrow[\cong]{\tT_m} \tT_X,$$ 
and further three $B_\infty$-quasi-isomorphisms  
$$\tT_{l_X}^*: C(\tT_X) \xrightarrow[\cong]{\tT_m^*} C(\tT_{X\otimes_\bB I_\bB}) \xrightarrow[\sim]{\tT_{1_X\otimes p_{I_\bB}}^*} C(\tT_{X\otimes_\bB \bp I_\bB}).$$ 
By definition of $\hat{C}(\tT_f)$, we have the following diagram $\textcircled{\tiny 26}$, which is commutative when we consider only solid arrows.
$$\begin{tikzcd}[column sep=small]
	  C(\tT_{X\otimes_\bB \bp I_\bB}) & C(\tT_{1_X\otimes p_{I_\bB}}) \arrow[l,two heads,"\sim"'] \arrow[dl,two heads,"\sim"'] \arrow[dr,two heads] && \hat{C}(\tT_{1_X\otimes p_{I_\bB}}) \arrow[dr,two heads] \arrow[ll,"\sim"'] & \\
      C(\tT_{X\otimes_\bB I_\bB})	&& C(\tT_{I_\aA})\times C(\tT_{I_\bB}) && C(\aA)\times C(\bB) \arrow[ll,near start,"\sim","{(\theta_\aA,\theta_\bB)}"'] \\
	  C(\tT_X) \arrow[u,"\cong","\tT_m^*"'] & C(\tT_{l_X}) \arrow[l,two heads,"\sim"'] \arrow[uu,dashed,"\cong"'] \arrow[uul,two heads,"\sim"'] \arrow[ur,two heads] && \hat{C}(\tT_{l_X}) \arrow[uu,dotted,near start,"\cong"'] \arrow[ur,two heads] \arrow[ll,"\sim"'] &
\end{tikzcd}$$
The dg isomorphism functor $\tT_m: \tT_{X\otimes_\bB I_\bB} \to \tT_X$ induces a dg isomorphism functor 
$$\tT_{X_{\tT_{1_X\otimes p_{I_\bB}}}} = 
\begin{pmatrix}
	\tT_{X\otimes_\bB \bp I_\bB} & X_{\tT_{1_X\otimes p_{I_\bB}}} \\ & \tT_{X\otimes_\bB I_\bB}
\end{pmatrix} 
\to \tT_{X_{\tT_{l_X}}} = 
\begin{pmatrix}
\tT_{X\otimes_\bB \bp I_\bB} & X_{\tT_{l_X}} \\ & \tT_X
\end{pmatrix},$$
and further a $B_\infty$-isomorphism $C(\tT_{l_X}) \to C(\tT_{1_X\otimes p_{I_\bB}})$ such that the diagram $\textcircled{\tiny 26}$ is commutative when we consider only solid and dashed arrows. By the universality of pullbacks, there is a  $B_\infty$-isomorphism $\hat{C}(\tT_{l_X}) \to \hat{C}(\tT_{1_X\otimes p_{I_\bB}})$ such that the whole diagram $\textcircled{\tiny 26}$ is commutative.

Finally, from the commutative diagrams $\textcircled{\tiny 24}$, $\textcircled{\tiny 25}$ and $\textcircled{\tiny 26}$, we obtain the following commutative diagram.
\[
\begin{tikzcd}[column sep=small]
	&& C(\tT_X)\times_{C(\bB)}C(\bB) \arrow[dl,"\cong"',"{l_{C(\tT_X)}}"] \arrow[dr,two heads] && \\
	& C(\tT_X) \arrow[dl,two heads] \arrow[d,equal] \arrow[dr,two heads] && C(\bB) \arrow[dl,equal] \arrow[d,"\sim","\theta_\bB"'] \arrow[dr,equal] & \\
	C(\aA) & C(\tT_X) \arrow[d,equal] \arrow[l,two heads] \arrow[r,two heads]& C(\bB) & C(\tT_{I_\bB}) \arrow[l,two heads,"\sim"']  \arrow[r,two heads,"\sim",near end] & C(\bB) \\
	& C(\tT_X) \arrow[ul,two heads] \arrow[ur,two heads] & \bar{C}(\tT_{X,I_\bB}) \arrow[ul,two heads,"\sim"',near end] \arrow[ur,two heads] \arrow[uuu,two heads,bend right=21,"\sim",near end] & \hat{C}(\tT_{p_{I_\bB}}) \arrow[u,two heads,"\sim"'] \arrow[d,two heads,"\sim",near start] & C(\bB)\times_{C(\tT_{I_\bB})}\hat{C}(\tT_{p_{I_\bB}}) \arrow[uul,two heads,"\sim",near end] \arrow[l,"\sim"']  \\
	&& C(\tT_X)\times_{C(\tT_{I_\bB})}\hat{C}(\tT_{p_{I_\bB}}) \arrow[uul,two heads,"\sim"',near start] \arrow[u] \arrow[urr,two heads] \arrow[ur] \arrow[d] & C(\tT_{\bp I_\bB}) \arrow[uul,two heads,"\sim"'] \arrow[uur,two heads,bend left=8,"\sim"',near end] & \\
	&& \bar{C}(\tT_{X,\bp I_\bB}) \arrow[uul,two heads,"\sim",bend left=10] \arrow[ur,two heads] & C(\tT_{X,I_\bB}) \arrow[uul,two heads,"\sim"',near end] \arrow[dddl,two heads] & C(\tT_X) \arrow[l,"\theta_X"',"\sim"]\\
	&& C(\tT_{X,\bp I_\bB}) \arrow[u,two heads,"\sim"'] \arrow[d,two heads] & C(\tT_{1_X,p_{I_\bB}}) \arrow[u,two heads,"\sim"'] \arrow[l,two heads,"\sim"'] \arrow[d,two heads] & \hat{C}(\tT_{1_X,p_{I_\bB}}) \arrow[l,"\sim"'] \arrow[d,two heads] \\
	&& C(\tT_{X\otimes_\bB \bp I_\bB}) & C(\tT_{1_X\otimes p_{I_\bB}}) \arrow[l,two heads,"\sim"'] \arrow[dl,two heads,near start,"\sim"'] & \hat{C}(\tT_{1_X\otimes p_{I_\bB}}) \arrow[l,"\sim"'] \\
	&& C(\tT_{X\otimes_\bB I_\bB}) & C(\tT_{l_X}) \arrow[u,"\cong"] \arrow[ul,two heads,near end,"\sim"'] \arrow[dl,two heads,"\sim"'] &  \hat{C}(\tT_{l_X}) \arrow[u,"\cong"] \arrow[l,"\sim"'] \\
	&& C(\tT_X) \arrow[u,"\cong","\tT_m^*"']&&
\end{tikzcd}	
\]
Thus
\[
\begin{array}{ll}
	& \ccC(l_X) \circ \ccC^2_{\bp I_\bB,X} \circ (\ccC^0_{\bB}*1_{\ccC(X)}) \\[2mm]
	= & [C(\tT_{X\otimes_\bB \bp I_\bB}) \stackrel{\sim}{\twoheadleftarrow} \hat{C}(\tT_{l_X}) \rightarrow C(\tT_X)] \quad \mbox{(See the lowest three rows)} \\[2mm]
	& \circ [\bar{C}(\tT_{X,\bp I_\bB}) \stackrel{\sim}{\twoheadleftarrow} C(\tT_{X, \bp I_\bB}) \rightarrow C(\tT_{X\otimes_\bB \bp I_\bB})] \quad \mbox{(See the middle column)} \\[2mm]
	& \circ [C(\tT_X)\times_{C(\bB)}C(\bB) \stackrel{\sim}{\twoheadleftarrow} C(\tT_X)\times_{C(\tT_{I_\bB})}\hat{C}(\tT_{p_{I_\bB}}) \stackrel{\sim}{\rightarrow} \bar{C}(\tT_{X,\bp I_\bB})] \\[2mm]
	= & [\bar{C}(\tT_{X,\bp I_\bB}) \stackrel{\sim}{\twoheadleftarrow} \hat{C}(\tT_{1_X,p_{I_\bB}}) \rightarrow C(\tT_X)] \quad \mbox{(See the lowest five rows)} \\[2mm]
	& \circ [C(\tT_X)\times_{C(\bB)}C(\bB) \stackrel{\sim}{\twoheadleftarrow} C(\tT_X)\times_{C(\tT_{I_\bB})}\hat{C}(\tT_{p_{I_\bB}}) \stackrel{\sim}{\rightarrow} \bar{C}(\tT_{X,\bp I_\bB})] \\[2mm]
	= & [\bar{C}(\tT_{X,\bp I_\bB}) \stackrel{\sim}{\twoheadleftarrow} C(\tT_{1_X,p_{I_\bB}}) \rightarrow C(\tT_X)] \\[2mm]
	& \circ [C(\tT_X)\times_{C(\bB)}C(\bB) \stackrel{\sim}{\twoheadleftarrow} C(\tT_X)\times_{C(\tT_{I_\bB})}\hat{C}(\tT_{p_{I_\bB}}) \stackrel{\sim}{\rightarrow} \bar{C}(\tT_{X,\bp I_\bB})] \\[2mm]
	= & [C(\tT_X)\times_{C(\bB)}C(\bB) \stackrel{\sim}{\twoheadleftarrow} C(\tT_{1_X,p_{I_\bB}}) \rightarrow C(\tT_X)] \\[2mm]
	= & [C(\tT_X)\times_{C(\bB)}C(\bB) \stackrel{\sim}{\twoheadleftarrow} C(\tT_{X,I_\bB}) \rightarrow C(\tT_X)] \\[2mm]
	= & [C(\tT_X)\times_{C(\bB)}C(\bB) \stackrel{\sim}{\twoheadleftarrow} C(\tT_X) \rightarrow C(\tT_X)] \quad \mbox{(Upmost-rightmost-lowermost)} \\[2mm]
	= & [C(\tT_X)\times_{C(\bB)}C(\bB) \xleftarrow{l_{C(\tT_X)}^{-1}} C(\tT_X) = C(\tT_X)] \\[2mm]
	= & l_{\ccC(X)}
\end{array}
\]
where the second equality from the bottom holds due to $[C(\tT_X) \stackrel{\iota_{12}^*}{\twoheadleftarrow} C(\tT_{X,I_\bB}) \stackrel{\iota_{13}^*}{\twoheadrightarrow} C(\tT_X)] = [C(\tT_X)=C(\tT_X)=C(\tT_X)]$ (Ref. Lemma \ref{Lem-ThetaX-BInfQis} in Appendix~\ref{App-ThetaX}). 
\[
  \begin{tikzcd}
	& C(\tT_X) & \\
    C(\tT_X) \arrow[ur,equal] \arrow[rr,"\theta_X","\sim"'] \arrow[dr,equal] && C(\tT_{X,I_\bB}) \arrow[ul,two heads,"{\iota_{12}^*}"',"\sim"] \arrow[dl,two heads,"{\iota_{13}^*}","\sim"'] \\	 
    & C(\tT_X) &
  \end{tikzcd}
\]

\medspace

(5) {\it $\ccC$ satisfies the lax right unity:} We need to show that for any objects $\aA,\bB\in\dgCAT_{\rm c,h}$ and 1-cell $X\in\dgCAT_{\rm c,h}(\aA,\bB)$, the following diagram is commutative.
\[
  \begin{tikzcd}
  \ccC(X)\circ 1_{C(\aA)}\arrow[r,"r_{\ccC(X)}"] \arrow[d,"1_{\ccC(X)}*\ccC^0_{\aA}"'] & \ccC(X) \\
  \ccC(X)\circ\ccC(\bp I_\aA) \arrow[r,"\ccC^2_{X,\bp I_\aA}"] & \ccC(\bp I_\aA \otimes_\aA X) \arrow[u,"\ccC(r_X)"']
  \end{tikzcd}
\]
 i.e., $\ccC(r_X) \circ \ccC^2_{X,\bp I_\aA} \circ (1_{\ccC(X)}*\ccC^0_{\aA}) = r_{\ccC(X)}$.

The proof of (5) is similar to that of (4), so we omit it here.

Now we have finished the proof of the theorem.
\end{proof}

\begin{remark}
	(1) According to \cite[Formula (4.43) and (4.44)]{DavydovKongRunkel11} and \cite[Formula (1.2) and (1.3)]{DavydovKongRunkel15}, we should define a lax functor sending a 1-cell $X$ to the cospan $C(\aA) \xrightarrow{\alpha_X} C(X) \xleftarrow{\beta_X} C(\bB)$ in the homotopy bicartesian diagram in Lemma \ref{Lem-Homotopy-Bicart}. However, $C(X)$ does not admit a nice $B_\infty$-algebra structure while $C(\tT_X)$ does, so we prefer to define a lax functor sending a 1-cell $X$ to the span $C(\aA) \xleftarrow{\iota_1^*} C(\tT_X) \xrightarrow{\iota_2^*} C(\bB)$ of $B_\infty$-algebras.
	
	(2) Let $\aA$ and $\bB$ be small dg categories, and $f:X\to X'$ a quasi-isomorphism of cofibrant dg $\aA\-\bB$-bimodules. One might wonder why we take $[C(\tT_X) \stackrel{\sim}{\twoheadleftarrow} \hat{C}(\tT_f) \rightarrow C(\tT_{X'})]$ but not $[C(\tT_X) \stackrel{\sim}{\twoheadleftarrow} C(\tT_f) \rightarrow C(\tT_{X'})]$ as $\mathscr{C}_{\aA\bB}([f])$. The reason is that the following natural diagram is not commutative, i.e., $[C(\tT_X) \stackrel{\sim}{\twoheadleftarrow} C(\tT_f) \rightarrow C(\tT_{X'})]$ is not a well-defined 2-cell.
	\[
	\begin{tikzcd}
		& C(\tT_X) \arrow[ld,two heads] \arrow[rd,two heads] & \\
		C(\aA) & C(\tT_f) \arrow[u,two heads,"\sim"'] \arrow[d,two heads,"\sim"] & C(\bB) \\
		& C(\tT_{X'}) \arrow[lu,two heads] \arrow[ru,two heads] &
	\end{tikzcd}
	\]	   
\end{remark}

\noindent{\bf Lax functor from $\dg\cC\aA\tT$ to $B_\infty\-{\rm span}^2$.} Now we want to construct a lax functor from bicategory $\dg\cC\aA\tT$ to bicategory $B_\infty\-{\rm span}^2$. We need the following biequivalences (Ref. \cite[Definition 6.2.8]{JohnsonYau21}).

\begin{proposition} \label{Prop-Biequ-HomotopyCat-DerCat}
	The bicategories $\dg\cC\aA\tT$ and $\dgCAT_{\rm c,h}$ are biequivalent.
\end{proposition}

\begin{proof}
	By Whitehead theorem for bicategories (Ref. \cite[Theorem 7.4.1]{JohnsonYau21}),
	it is enough to show that there exists a pseudofunctor $\mathscr{P} : \dg\cC\aA\tT \to \dgCAT_{\rm c,h}$ which is essentially surjective on objects, essentially full on 1-cells, and fully faithful on 2-cells.
	
	Firstly, we define a pseudofunctor $\mathscr{P} : \dg\cC\aA\tT \to \dgCAT_{\rm c,h}$ by the following data:
	
	(1) The function $\mathscr{P}:\Ob(\dg\cC\aA\tT)\to\Ob(\dgCAT_{\rm c,h}), \aA\mapsto\aA$.
	
	(2) For all $\aA,\bB\in\dg\cC\aA\tT$, the functor $\mathscr{P}_{\aA\bB}: \dg\cC\aA\tT(\aA,\bB) \to \dgCAT_{\rm c,h}(\aA,\bB)=\hH_{\rm c,h}(\aA^\op\otimes\bB), X\mapsto {\bf p}X, (f:X\to X')\mapsto[{\bf p}f:{\bf p}X\to {\bf p}X']$, is induced by a cofibrant replacement functor ${\bf p}: \dD(\aA^\op\otimes\bB) \to \hH_c(\aA^\op\otimes\bB)$ where $\hH_c(\aA^\op\otimes\bB)$ is the full subcategory of $\hH(\aA^\op\otimes\bB)$ consisting of all cofibrant dg $\aA$-$\bB$-bimodules. It is clear that $\mathscr{P}_{\aA\bB}$ is an equivalence with quasi-inverse $\mathscr{I}_{\aA\bB} : \dgCAT_{\rm c,h}(\aA,\bB) \to \dg\cC\aA\tT(\aA,\bB), X\mapsto X, [f]\mapsto f$.
	
	(3) For all objects $\aA,\bB,\cC\in\dg\cC\aA\tT$, the natural isomorphism $\mathscr{P}^2_{\aA\bB\cC} : c'_{\mathscr{P}\aA,\mathscr{P}\bB,\mathscr{P}\cC}\circ (\mathscr{P}_{\bB\cC}\times\mathscr{P}_{\aA\bB}) \to \mathscr{P}_{\aA\cC}\circ c_{\aA\bB\cC} : \dg\cC\aA\tT(\bB,\cC)\times\dg\cC\aA\tT(\aA,\bB) \to \hH_{\rm c,h}(\aA^\op\otimes\cC)$ is given as follows:
	Observe the diagram
	\[
	\begin{tikzcd}
		\dg\cC\aA\tT(\bB,\cC)\times\dg\cC\aA\tT(\aA,\bB) \arrow[r,"-\otimes^{\bf L}_{\bB}-"] 
		\arrow[d,shift right=2,"\mathscr{P}_{\bB\cC}\times\mathscr{P}_{\aA\bB}"'] & \dg\cC\aA\tT(\aA,\cC) \arrow[d,shift right=2,"\mathscr{P}_{\aA\cC}"'] \\
		\hH_{\rm c,h}(\bB^\op\otimes\cC)\times\hH_{\rm c,h}(\aA^\op\otimes\bB) \arrow[r,"-\otimes_{\bB}-"] \arrow[u,shift right=2,"\mathscr{I}_{\bB\cC}\times\mathscr{I}_{\aA\bB}"'] & \hH_{\rm c,h}(\aA^\op\otimes\cC). \arrow[u,shift right=2,"\mathscr{I}_{\aA\cC}"'] 
	\end{tikzcd}
	\]
	We have $(-\otimes^{\bf L}_{\bB}-)(\mathscr{I}_{\bB\cC}\times\mathscr{I}_{\aA\bB})=\mathscr{I}_{\aA\cC}(-\otimes_{\bB}-)$ and the natural isomorphisms $\xi_{\aA\bB}:\Id\xrightarrow{\cong} \mathscr{P}_{\aA\bB}\mathscr{I}_{\aA\bB}$ and $\zeta_{\aA\bB}:\mathscr{I}_{\aA\bB}\mathscr{P}_{\aA\bB}\xrightarrow{\cong} \Id$. The natural isomorphisms $\mathscr{P}^2_{\aA\bB\cC}$ is the composition $(-\otimes_{\bB}-)(\mathscr{P}_{\bB\cC}\times\mathscr{P}_{\aA\bB}) \xrightarrow[\cong]{\xi_{\aA\cC}(-\otimes_{\bB}-)(\mathscr{P}_{\bB\cC}\times\mathscr{P}_{\aA\bB})} \mathscr{P}_{\aA\cC}\mathscr{I}_{\aA\cC}(-\otimes_{\bB}-)(\mathscr{P}_{\bB\cC}\times\mathscr{P}_{\aA\bB}) 
	= \mathscr{P}_{\aA\cC}(-\otimes^{\bf L}_{\bB}-) (\mathscr{I}_{\bB\cC}\times \mathscr{I}_{\aA\bB}) (\mathscr{P}_{\bB\cC}\times\mathscr{P}_{\aA\bB})
	\xrightarrow[\cong]{\mathscr{P}_{\aA\cC}(-\otimes^{\bf L}_{\bB}-)(\zeta_{\bB\cC}\times\zeta_{\aA\bB})}
	\mathscr{P}_{\aA\cC}(-\otimes^{\bf L}_{\bB}-)$. 
		
	(4) For all $\aA\in\dg\cC\aA\tT$, the natural isomorphism $\mathscr{P}^0_\aA$ is given by the lax unity constraint $\mathscr{P}^0_\aA:=1_{{\bf p}I_\aA}: 1_{\mathscr{P}\aA}={\bf p}I_\aA \to \mathscr{P}_{\aA\aA}(1_\aA)={\bf p}I_\aA$.
	
	It is easy to see that $(\mathscr{P},\mathscr{P}^2,\mathscr{P}^0)$ satisfies lax associativity, lax left unity and lax right unity. Moreover, $\mathscr{P}^2_{\aA\bB\cC}$ and $\mathscr{P}^0_\aA$ are natural isomorphisms for all $\aA,\bB,\cC\in\dg\cC\aA\tT$. Thus $\mathscr{P}$ is a pseudofunctor.
	
	Secondly, $\mathscr{P}$ is essentially surjective on objects (Ref. \cite[Definition 7.0.1]{JohnsonYau21}): For any adjoint equivalence class $[\aA]'$ of an object $\aA$ in the bicategory $\dgCAT_{\rm c,h}$, we choose any object $\bB$ in the adjoint equivalence class $[\aA]$ of the object $\aA$ in the bicategory $\dg\cC\aA\tT$. Then there are 1-cells $X\in\dg\cC\aA\tT(\aA,\bB)$ and $Y\in\dg\cC\aA\tT(\bB,\aA)$ such that $X\otimes^{\bf L}_\bB Y\cong I_\aA$ and $Y\otimes^{\bf L}_\aA X\cong I_\bB$. Thus there are ${\bf p}X\in\hH_{\rm c,h}(\aA^\op\otimes\bB)$ and ${\bf p}Y\in\hH_{\rm c,h}(\bB^\op\otimes\aA)$ such that ${\bf p}X\otimes_\bB {\bf p}Y\cong {\bf p}I_\aA$ in $\hH_{\rm c,h}(\aA^\op\otimes\aA)$ and ${\bf p}Y\otimes_\aA {\bf p}X\cong {\bf p}I_\bB$ in $\hH_{\rm c,h}(\bB^\op\otimes\bB)$. Hence $\bB$ and $\aA$ are adjoint equivalent in $\dgCAT_{\rm c,h}$.	
	Therefore, $\mathscr{P}$ sends the adjoint equivalence class $[\aA]$ to the adjoint equivalence class $[\aA]'$.
	
	Finally, $\mathscr{P}$ is essentially full on 1-cells (Ref. \cite[Definition 7.0.1]{JohnsonYau21}) and fully faithful on 2-cells: Indeed, for all $\aA,\bB\in\dg\cC\aA\tT$, the functor $\mathscr{P}_{\aA\bB} : \dg\cC\aA\tT(\aA,\bB) \to \dgCAT_{\rm c,h}(\aA,\bB)$ is an equivalence. 
	\end{proof}

Combining Proposition~\ref{Prop-Biequ-HomotopyCat-DerCat} and Theorem~\ref{Thm-LaxFuntor-HomotopyCat-B-Inf}, we obtain the following theorem.

\begin{theorem} \label{Thm-LaxFuntor-DerCat-B-Inf}
There is a lax functor $\ccC=(\ccC,\ccC^2, \ccC^0):\dg\cC\aA\tT\rightarrow B_\infty\-{\rm span}^2$ given as follows:
		
{\rm (1)} The function $\ccC:\Ob(\dg\cC\aA\tT) \rightarrow \Ob(B_\infty\-{\rm span}^2), \aA\mapsto C(\aA)$.
		
{\rm (2)} For all $\aA,\bB\in\dg\cC\aA\tT$, the functor
		$\ccC_{\aA\bB}:\dg\cC\aA\tT(\aA,\bB) \rightarrow B_\infty\-{\rm span}^2(C(\aA),C(\bB))$ is defined as follows:
		
{\rm (2.1)} For any 1-cell $X\in\dg\cC\aA\tT(\aA,\bB)$, $\ccC_{\aA\bB}(X):= (C(\aA)\stackrel{\iota_1^*}{\twoheadleftarrow} C(\tT_{{\bf p}X})\stackrel{\iota_2^*}{\twoheadrightarrow} C(\bB))$. 
		
{\rm (2.2)} For any 2-cell $f\in\dg\cC\aA\tT(\aA,\bB)(X,X')$, 
$\ccC_{\aA\bB}(f):=[C(\tT_{{\bf p}X}) \stackrel{\sim}{\twoheadleftarrow} \hat{C}(\tT_{{\bf p}f}) \stackrel{\sim}{\twoheadrightarrow} C(\tT_{{\bf p}X'})]$, where $\hat{C}(\tT_{{\bf p}f})$ is the following pullback in $B_\infty$:
\[
  \begin{tikzcd}[row sep=large]
	&\hat{C}(\tT_{{\bf p}f}) \arrow[r] \arrow[d,"\sim"'] & C(\aA)\times C(\bB) \arrow[d,"{(\theta_\aA,\theta_\bB)}","\sim"'] \\
	&C(\tT_{{\bf p}f}) \arrow[r,two heads,"{\begin{pmatrix}
	\iota^*_\aA \\ \iota^*_\bB	
	\end{pmatrix}}"] \arrow[dl,two heads,"\sim"'] \arrow[dr,two heads,"\sim"] & C(\tT_{I_\aA})\times C(\tT_{I_\bB})\\
			C(\tT_{{\bf p}X})&&C(\tT_{{\bf p}X'})
  \end{tikzcd}
\]
		
{\rm (3)} For any $\aA,\bB,\cC\in\dg\cC\aA\tT$, the natural transformation $\ccC^2_{\aA\bB\cC}:c_{\ccC\aA,\ccC\bB,\ccC\cC}\circ(\ccC_{\bB\cC}\times\ccC_{\aA\bB}) \Rightarrow \ccC_{\aA\cC}\circ c_{\aA\bB\cC}$ is defined as follows:
For any 1-cells $X\in\dg\cC\aA\tT(\aA,\bB)$ and $Y\in\dg\cC\aA\tT(\bB,\cC)$, the lax functoriality constraint $\ccC^2_{Y,X}:\ccC(Y)\circ\ccC(X) \rightarrow \ccC(Y\circ X)$ is the equivalence class $[\bar{C}(\tT_{{\bf p}X,{\bf p}Y}) \stackrel{\sim}{\twoheadleftarrow} C(\tT_{{\bf p}X,{\bf p}Y})\times_{C(\tT_{{\bf p}X\otimes_\bB {\bf p}Y})}\hat{C}(\tT_{\mathscr{P}^2_{Y,X}}) \to C(\tT_{{\bf p}(X\otimes^{\bf L}_\bB Y)})]$ of the vertical composition of two spans of spans $(\bar{C}(\tT_{{\bf p}X,{\bf p}Y}) \stackrel{\sim}{\twoheadleftarrow} C(\tT_{{\bf p}X,{\bf p}Y}) \to C(\tT_{{\bf p}X\otimes_\bB {\bf p}Y}))$ and $(C(\tT_{{\bf p}X\otimes_\bB {\bf p}Y}) \stackrel{\sim}{\twoheadleftarrow} \hat{C}(\tT_{\mathscr{P}^2_{Y,X}}) \to C(\tT_{{\bf p}(X\otimes^{\bf L}_\bB Y)}))$.
\[
  \begin{tikzcd}[row sep=small]
	&& \bar{C}(\tT_{{\bf p}X,{\bf p}Y}) \arrow[ld,two heads] \arrow[rd,two heads] && \\
	& C(\tT_{{\bf p}X}) \arrow[ld,two heads] \arrow[rd,two heads] & C(\tT_{{\bf p}X,{\bf p}Y}) \arrow[u,two heads,"\sim"'] \arrow[ddd,bend left=35] & C(\tT_{{\bf p}Y}) \arrow[ld,two heads] \arrow[rd,two heads] &\\
	C(\aA) && C(\bB) && C(\cC) \\
	&&&&\\
	&& C(\tT_{{\bf p}X\otimes_\bB {\bf p}Y}) \arrow[rruu,two heads] \arrow[lluu,two heads] &&\\
	&& \hat{C}(\tT_{\mathscr{P}^2_{Y,X}})\arrow[u,two heads,"\sim"'] \arrow[d] &&\\
	&& C(\tT_{{\bf p}(X\otimes^{\bf L}_\bB Y)}) \arrow[rruuuu,two heads] \arrow[lluuuu,two heads] && 
  \end{tikzcd}
\]  
where $\mathscr{P}^2_{Y,X}:{\bf p}X\otimes_\bB {\bf p}Y \to {\bf p}(X\otimes^{\bf L}_\bB Y)$ is the lax functoriality constraint of the pseudofunctor $\mathscr{P}$ defined in the proof of  Proposition~\ref{Prop-Biequ-HomotopyCat-DerCat}.
		
{\rm (4)} For any object $\aA\in\dg\cC\aA\tT$, the lax unity constraint $\ccC^0_\aA: 1_{\ccC(\aA)}=(C(\aA) = C(\aA) = C(\aA)) \to \ccC(I_\aA)=(C(\aA) \stackrel{\sim}{\twoheadleftarrow} C(\tT_{{\bf p}I_\aA}) \twoheadrightarrow C(\aA))$ is the equivalence class $[C(\aA) \stackrel{\sim}{\twoheadleftarrow} C(\aA)\times_{C(\tT_{I_\aA})} \hat{C}(\tT_{pI_\aA})\to C(\tT_{{\bf p}I_\aA})]$ of the vertical composition of two spans of spans $(C(\aA) = C(\aA) \stackrel{\theta_\aA}{\rightarrow} C(\tT_{I_\aA}))$ and $(C(\tT_{I_\aA}) \stackrel{\sim}{\twoheadleftarrow} \hat{C}(\tT_{pI_\aA}) \to  C(\tT_{{\bf p}I_\aA}))$. 
\[
\begin{tikzcd}[column sep=large]
	& C(\aA) \arrow[dl,equal] \arrow[d,"\theta_\aA","\sim"'] \arrow[dr,equal] & \\
	C(\aA) &  C(\tT_{I_\aA}) \arrow[l,two heads,"\iota^*_1"',"\sim"] \arrow[r,two heads,near start,"\iota^*_2","\sim"'] & C(\aA) \\
	& \hat{C}(\tT_{p_{I_\aA}}) \arrow[u,two heads,"\sim"] \arrow[d,two heads,"\sim"'] & \\
	& C(\tT_{\bp I_\aA}) \arrow[uul,two heads] \arrow[uur,two heads] &
\end{tikzcd}	
\]	
\end{theorem}

\begin{proof}
	By Theorem \ref{Thm-LaxFuntor-HomotopyCat-B-Inf}, we have a lax functor $\mathscr{C} : \dgCAT_{\rm c,h} \to B_\infty\-{\rm span}^2$.
	By Proposition \ref{Prop-Biequ-HomotopyCat-DerCat}, we have a biequivalence $\mathscr{P}: \dg\cC\aA\tT \to\dgCAT_{\rm c,h}$. 	
	Since the composition of lax functors is also a lax functor (Ref. \cite[Lemma 4.1.29]{JohnsonYau21}), the composition $\mathscr{C}\mathscr{P}: \dg\cC\aA\tT \to B_\infty\-{\rm span}^2$ is the desired lax functor.	
\end{proof}

\begin{remark}{\rm 
	One might hope to construct a larger lax functor $\mathscr{C}$ starting from a larger source bicategory $\dg\cC\aA\tT$ with whole derived category $\dD(\aA^\op\otimes\bB)$ as Hom category $\dg\cC\aA\tT(\aA,\bB)$ for each pair of small dg categories $\aA$ and $\bB$, and sending a small dg category $\aA$ to its Hochschild cochain complex $C(\aA)$. We feel that it is feasible and conjecture that the tricategory $B_\infty\-\Span^3$ will be involved in this situation.
}\end{remark}

\subsection{More lax functors}

In this subsection, we will construct more lax functors and functors. The overview diagram $(\bigstar)$ below provides a global perspective.
$$\begin{tikzcd}[column sep=small]
		& {\rm dgcat}_{\rm ff} \arrow[r] \arrow[d] \arrow[rrrrddd,dotted] & {\rm Hqe}_{\rm ff} \arrow[r] \arrow[d] & {\rm Hmo}_{\rm ff} \arrow[d] \arrow[rddd,dotted,"K"'] \arrow[ddr] & &  \\
		{\rm alg} \arrow[r] \arrow[d] & {\rm dgcat} \arrow[r] \arrow[d] & {\rm Hqe} \arrow[r] & {\rm Hmo} \arrow[dl] & & \\	
		\aA\lL\gG \arrow[r] & \dg\cC\aA\tT \arrow[dl,shift left=1,"\mathscr{P}"] \arrow[r,shift left=1,"\mathscr{S}"] \arrow[d,"\ref{Thm-LaxFuntor-DerCat-B-Inf}","\mathscr{C}"'] & \overline{\dg\cC\aA\tT} \arrow[l,shift left=1,"\mathscr{T}"] \arrow[d,"\ref{Cor-overline-ccC}","\overline{\mathscr{C}}"'] & \overline{\dg\cC\aA\tT_{\rm raf}} \arrow[l] \arrow[d,"\overline{\mathscr{C}}_{\rm raf}"] &  \overline{\dg\cC\aA\tT_{\rm raf}}^{\le 1} \arrow[l] \arrow[d,near end,"\tilde{K}"] & \\
		\dgCAT_{\rm c,h} \arrow[r,"\ref{Thm-LaxFuntor-HomotopyCat-B-Inf}"'] \arrow[ur,shift left=1,"\mathscr{I}"] & B_\infty\-{\rm span}^2 \arrow[r,shift left=1,"\mathscr{S}"] & \overline{B_\infty\-{\rm span}^2} \arrow[l,shift left=1,"\mathscr{T}"] & \overline{B_\infty\-{\rm span}^2_{\rm raf}} \arrow[l] \arrow[r,"\mathscr{H}","\ref{Prop-BInfOspan2-HBInf}"'] \arrow[d,"\ref{Cor-LaxFunctor-B8span2-GSpan2}"] & ({\rm Ho}B_\infty)^\op \arrow[d] & B_\infty^\op \arrow[l] \\
		& & & \mathsf{G}\-{\rm Span}^2 & \mathsf{G}^\op &
	\end{tikzcd}
	$$

\medspace

In the following, we will provide a detailed explanation of the overview diagram $(\bigstar)$ above. 
Note that two dotted arrows in the diagram are the functors ${\rm dgcat}_{\rm ff} \to B_\infty^\op$ and $K: {\rm Hmo}_{\rm ff} \to ({\rm Ho}B_\infty)^\op$ discovered by Keller in \cite[4.3, Paragraph 1]{Keller03} and \cite[5.4, Page 180]{Keller06} respectively.  

\medspace

\noindent{\bf 1-skeletons $\overline{\dg\cC\aA\tT}$ and $\overline{B_\infty\-{\rm span}^2}$ of bicategories $\dg\cC\aA\tT$ and $B_\infty\text{-}{\rm span}^2$.}
By Proposition~\ref{Prop-Bicat-1-Skeleton}, we have biequivalences $\mathscr{T}: \overline{\dg\cC\aA\tT} \to \dg\cC\aA\tT$
and $\mathscr{S}: B_\infty\text{-}{\rm span}^2 \rightarrow \overline{B_\infty\-{\rm span}^2}.$

From Theorem \ref{Thm-LaxFuntor-DerCat-B-Inf}, we immediately obtain the following corollary.

\begin{corollary} \label{Cor-overline-ccC}
	There exists a lax functor $\overline{\ccC} : \overline{\dg\cC\aA\tT} \rightarrow \overline{B_\infty\-{\rm span}^2}$ given by the following commutative diagram  
	\[
		\begin{tikzcd}
			\dg\cC\aA\tT \arrow[d, "\ccC"] & \overline{\dg\cC\aA\tT} \arrow[l, "\mathscr{T}"] \arrow[d, "\overline{\ccC}"'] \\ 
			B_\infty\text{-}{\rm span}^2 \arrow[r,"\mathscr{S}"] & \overline{B_\infty\-{\rm span}^2}
		\end{tikzcd}
	\]
	where the lax functor $\ccC : \dg\cC\aA\tT \to B_\infty\text{-}{\rm span}^2$ was given in Theorem \ref{Thm-LaxFuntor-DerCat-B-Inf}.
\end{corollary}

\medspace

\noindent{\bf 1-skeletons $\overline{B_\infty\-{\rm span}^2_{\rm raf}}$ and $\overline{\dg\cC\aA\tT_{\rm raf}}$ of bicategores $B_\infty\text{-}{\rm span}^2_{\rm raf}$ and $\dg\cC\aA\tT_{\rm raf}$.}  
Keller gave an example in \cite[3.4]{Keller03} which implies that the compositions of fully faithful derived tensor product functors need not be fully faithful any more. Thus we cannot define the bicategory $\dg\cC\aA\tT_{\rm ff}$ in the natural way. However, we can define the bicategory  $\dg\cC\aA\tT_{\rm raf}$ to be the sub-bicategory of $\dg\cC\aA\tT$ with the same objects as $\dg\cC\aA\tT$, 1-cells in the Hom category $\dg\cC\aA\tT_{\rm raf}(\aA,\bB)$ being the dg $\aA$-$\bB$-bimodules $X$ such that the projection $C(\tT_X) \twoheadrightarrow C(\bB)$ is an acyclic fibration in the category $B_\infty$ of $B_\infty$-algebras, and the set of 2-cells $\dg\cC\aA\tT_{\rm raf}(\aA,\bB)(X,X'):=\dg\cC\aA\tT(\aA,\bB)(X,X')$, for all $\aA,\bB\in \dg\cC\aA\tT_{\rm raf}$ and $X,X'\in \dg\cC\aA\tT_{\rm raf}(\aA,\bB)$.

During defining the 2-cells of $\overline{B_\infty\-{\rm span}^2}$ and $\overline{B_\infty\-{\rm span}^2_{\rm raf}}$, we may choose the same representative for the isomorphism classes in $B_\infty\text{-}{\rm span}^2_{\rm raf}$ and $B_\infty\text{-}{\rm span}^2$ for each 1-cell in $B_\infty\text{-}{\rm span}^2_{\rm raf}$, and the same isomorphism between the representative of its isomorphism class and each 1-cell in its isomorphism class in $B_\infty\text{-}{\rm span}^2_{\rm raf}$ for each 1-cell in $B_\infty\text{-}{\rm span}^2_{\rm raf}$.
Then $\overline{B_\infty\-{\rm span}^2_{\rm raf}}$ is a sub-bicategory of $\overline{B_\infty\-{\rm span}^2}$. 
After similar operations, $\overline{\dg\cC\aA\tT_{\rm raf}}$ is a sub-bicategory of $\overline{\dg\cC\aA\tT}$.

By the definitions of bicategories $\overline{\dg\cC\aA\tT_{\rm raf}}$ and $\overline{B_\infty\-{\rm span}^2_{\rm raf}}$, the lax functor $\overline{\mathscr{C}}$ restricts to a lax functor $\overline{\mathscr{C}}_{\rm raf} : \overline{\dg\cC\aA\tT_{\rm raf}} \to \overline{B_\infty\-{\rm span}^2_{\rm raf}}$ such that the following diagram is commutative.  
\[  
	\begin{tikzcd} 
		\overline{\dg\cC\aA\tT} \arrow[d,"\overline{\mathscr{C}}"] & \overline{\dg\cC\aA\tT_{\rm raf}} \arrow[d,"\overline{\mathscr{C}}_{\rm raf}"] \arrow[l] \\  
		\overline{B_\infty\-{\rm span}^2} & \overline{B_\infty\-{\rm span}^2_{\rm raf}} \arrow[l] 
	\end{tikzcd}  
\]  

To get a functor extending Keller's functor $K$, we may define the category $\overline{\dg\cC\aA\tT_{\rm raf}}^{\le 1}$ whose objects are all small dg categories, whose morphisms are the isomorphism classes $[X]$ of all dg $\aA$-$\bB$-bimodules $X$ such that the projection $C(\tT_X) \twoheadrightarrow C(\bB)$ is an acyclic fibration in the category $B_\infty$ of $B_\infty$-algebras, whose compositions $[Y]\circ [X]:=[X\otimes^\bL_\bB Y]$, and whose identity morphisms are identity dg bimodules $I_\aA$. Viewed as a bicategory with discrete Hom categories, $\overline{\dg\cC\aA\tT_{\rm raf}}^{\le 1}$ is a sub-bicategory of $\overline{\dg\cC\aA\tT_{\rm raf}}$.

\medspace

\noindent{\bf Homotopy category of $B_\infty$-algebras.}
According to Proposition \ref{Prop:B_inf-to-G}, there exists a cohomology functor $H:B_\infty \rightarrow \mathsf{G}$ from the category $B_\infty$ of $B_\infty$-algebras to the category $\mathsf{G}$ of Gerstenhaber algebras. It sends $B_\infty$-quasi-isomorphisms to Gerstenhaber algebra isomorphisms. Thus $H$ factors through the localization functor $B_\infty \rightarrow {\rm Ho}B_\infty$. 
$$\begin{tikzcd}
	B_\infty \arrow[r] \ar[dr,"H"'] & {\rm Ho}B_\infty \arrow[d]\\ & \mathsf{G}
\end{tikzcd}$$

The following result will be used to extend Keller's functor $K$.

\begin{proposition} \label{Prop-BInfOspan2-HBInf}
    There exists a strict functor $\mathscr{H} : \overline{B_\infty\-{\rm span}^2_{\rm raf}} \rightarrow ({\rm Ho}B_\infty)^\op$ defined as follows:
	
	{\rm (1)} The function $\mathscr{H} : \Ob(\overline{B_\infty\-{\rm span}^2_{\rm raf}}) \to \Ob(({\rm Ho}B_\infty)^\op), A\mapsto A$.

	{\rm (2)} For all $A,B \in \overline{B_\infty\-{\rm span}^2_{\rm raf}}$, the functor
		$\mathscr{H}_{AB} : \overline{B_\infty\-{\rm span}^2_{\rm raf}}(A,B) \to {\rm Ho}B_\infty(B,A)$
		is defined as the following:
		
	{\rm (2.1)} For any 1-cell $[\hat{f}] = [\!\!\begin{tikzcd}[column sep=small] A & f \arrow[l,two heads,"f_l"'] \arrow[r,two heads,"f_r","\sim"'] & B \end{tikzcd}\!\!] \in \overline{B_\infty\-{\rm span}^2_{\rm raf}}(A,B)$, 
			$\mathscr{H}_{AB}([\hat{f}]) := f_lf_r^{-1}.$
			
	{\rm (2.2)} For any 2-cell $[\hat{\hat{\alpha}}] = [\hat{f} \leftarrow \hat{\alpha} \rightarrow \hat{g}] = [\!\!\begin{tikzcd}[column sep=small] f & \alpha \arrow[l,two heads,"\alpha_l"',"\sim"] \arrow[r,"\alpha_r"]  & g \end{tikzcd}\!\!] \in \overline{B_\infty\-{\rm span}^2_{\rm raf}}(A,B)([\hat{f}],[\hat{g}])$, 
			$\mathscr{H}_{AB}([\hat{\hat{\alpha}}]) := 1_{f_lf_r^{-1}}.$  
	
	{\rm (3)} For all objects $A,B,C\in \overline{B_\infty\-{\rm span}^2_{\rm raf}}$ and 1-cells $[\hat{f}] \in \overline{B_\infty\-{\rm span}^2_{\rm raf}}(A,B)$ and $[\hat{g}] \in \overline{B_\infty\-{\rm span}^2_{\rm raf}}(B,C)$, 
		the lax associativity constraint 
		$\mathscr{H}^2_{[\hat{g}],[\hat{f}]}: \mathscr{H}([\hat{g}])\circ\mathscr{H}([\hat{f}]) \to \mathscr{H}([\hat{g}]\circ[\hat{f}])$ 		
		is the identity 2-cell.  
			
	{\rm (4)} For all $A \in \overline{B_\infty\-{\rm span}^2_{\rm raf}}$, 
		the lax unity constraint 
		$\mathscr{H}^0_A : 1_{\mathscr{H}A} \to \mathscr{H}(1_A)$ 
		is the identity 2-cell.
\end{proposition}

\begin{proof}
	It is easy to see that the definition of $\mathscr{H}_{AB}([\hat{f}])$ is independent of the choice of representative of isomorphism class $[\hat{f}]$, and the definition of $\mathscr{H}_{AB}([\hat{\hat{\alpha}}])$ is independent of the choice of representative of equivalence class $[\hat{\hat{\alpha}}]$.
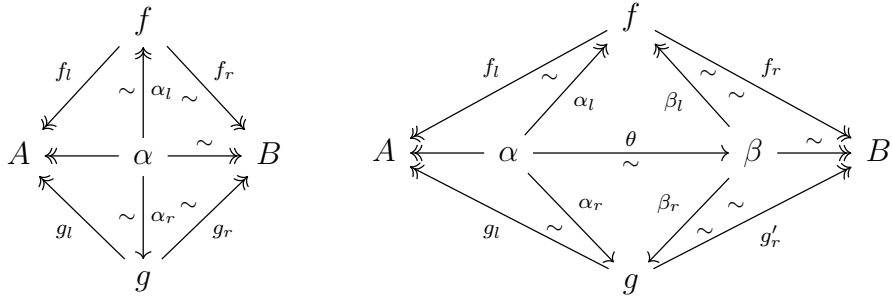
		\[
	\begin{tikzcd}[row sep=large]
		& f \arrow[ld,two heads,"f_l"'] \arrow[rd,two heads,"f_r","\sim"'] & \\
		A & \alpha \arrow[u,two heads,"\alpha_l"',"\sim"] \arrow[d,"\alpha_r","\sim"'] \arrow[l,two heads] \arrow[r,two heads,"\sim"] & B \\
		& g \arrow[lu,two heads,"g_l"] \arrow[ru,two heads,"g_r"',"\sim"] &
	\end{tikzcd} \qquad
	\begin{tikzcd}[row sep=large]
		&& f \arrow[lld,two heads,"f_l"'] \arrow[rrd,two heads,"f_r","\sim"'] && \\
		A & \alpha \arrow[ur,two heads,"\alpha_l"',"\sim"] \arrow[dr,"\alpha_r","\sim"'] \arrow[l,two heads] \arrow[rr,"\theta","\sim"'] && \beta \arrow[ul,two heads,"\beta_l","\sim"'] \arrow[dl,"\beta_r"',"\sim"] \arrow[r,two heads,"\sim"] & B \\
		&& g \arrow[llu,two heads,"g_l"] \arrow[rru,two heads,"g'_r"',"\sim"] &&
	\end{tikzcd}
	\]
	Thus the functor
	$\mathscr{H}_{AB} : \overline{B_\infty\-{\rm span}^2_{\rm raf}}(A,B) \to {\rm Ho}B_\infty(B,A)$ is well-defined.
	Note that $\mathscr{H}([\hat{g}])\circ\mathscr{H}([\hat{f}]) = (f_lf_r^{-1})(g_lg_r^{-1}) = (f_l\tilde{g}_l)(g_r\tilde{f}_r)^{-1} = \mathscr{H}([\hat{g}]\circ[\hat{f}])$ and $1_{\mathscr{H}A} = 1_A = 1_A1_A^{-1} = \mathscr{H}_{AA}([A\xleftarrow{1_A}A\xrightarrow{1_A}A])$.
	So $\mathscr{H}^2_{[\hat{g}],[\hat{f}]}$ and $\mathscr{H}^0_A$ are well-defined.	
	
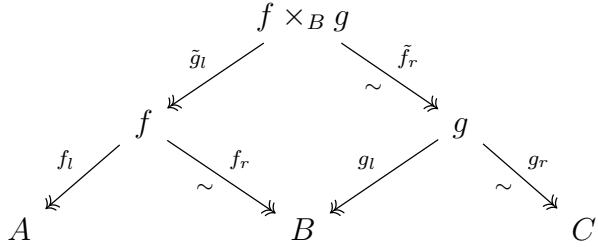
\[
	\begin{tikzcd}
		&& f\times_B g \arrow[ld,two heads,"\tilde{g}_l"'] \arrow[rd,two heads, "\tilde{f}_r","\sim"'] && \\
		& f \arrow[ld,two heads,"f_l"'] \arrow[rd,two heads,"f_r","\sim"'] & & g \arrow[ld,two heads,"g_l"'] \arrow[rd,two heads, "g_r","\sim"'] & \\
		A & & B & & C
	\end{tikzcd}
	\]
	Furthermore, it is clear that the lax associativity and lax left and right unity hold.
\end{proof}

The following commutative diagram of lax functors
\[  
\begin{tikzcd} 
	\overline{\dg\cC\aA\tT_{\rm raf}} \arrow[d,"\overline{\mathscr{C}}_{\rm raf}"] & \overline{\dg\cC\aA\tT_{\rm raf}}^{\le 1} \arrow[l] \arrow[d,"\tilde{K}"] \\  
	\overline{B_\infty\-{\rm span}^2_{\rm raf}} \arrow[r,"\mathscr{H}"] & ({\rm Ho}B_\infty)^\op 
\end{tikzcd}  
\]  
defines a strict functor $\tilde{K}: \overline{\dg\cC\aA\tT_{\rm raf}}^{\le 1} \to ({\rm Ho}B_\infty)^\op$ which corresponds to a functor $\tilde{K}:\overline{\dg\cC\aA\tT_{\rm raf}}^{\le 1} \to ({\rm Ho}B_\infty)^\op$. 
Later we will see that this functor $\tilde{K}$ extends Keller's functor $K$.

\medspace

\noindent{\bf Bicategories of algebras.}
Denote by $\alg$ the category whose objects are all $k$-algebras and whose morphisms are all homomorphisms of $k$-algebras. 
Denote by $\aA\lL\gG$ the bicategory for which objects are all $k$-algebras, the 1-cells in the Hom category $\aA\lL\gG(A,B)$ are all complexes of $A$-$B$-bimodules, the 2-cells in $\aA\lL\gG(A,B)(X,Y)$ are all isomorphisms $f:X\to Y$ in the derived category $\dD(A^\op\otimes B)$ of the category of $A$-$B$-bimodules, the horizontal composition functors are derived tensor product functors $-\otimes^{\bf L}_B-$, and the idenity 1-cells are identity bimodules $_AA_A$. Since the category ${\rm alg}$ can be viewed as a bicategory with discrete Hom categories, and each homomorphism of algebras $f:A\to B$ corresponds to an $A$-$B$-bimodule $_AB_B$ which is projective as a right $B$-module, we have a pseudofunctor
${\rm alg} \to \aA\lL\gG, A\mapsto A,\ (f:A\to B)\mapsto\ _AB_B, 1_f\mapsto 1_{_AB_B}$,
for which the lax associativity constraints are the natural isomorphisms $_AB\otimes^{\bf L}_BC_C \cong\ _AC_C$, and the lax unity contraints are identity morphisms $1_{_AA_A}$, for all algebras $A,B,C\in{\rm alg}$ and homomorphisms of algebras $f:A\to B, g:B\to C$ in ${\rm alg}$.

Every algebra can be regarded as a dg category with only one object, so we have a natural functor $\alg\rightarrow {\rm dgcat}$, and a natural pseudofunctor $\aA\lL\gG \to \dg\cC\aA\tT$. 

The category {\rm dgcat} of small dg categories can be viewed as a bicategory with discrete Hom categories. Then we have a pseudofunctor ${\rm dgcat} \rightarrow \dg\cC\aA\tT, \aA\mapsto\aA, F \mapsto X_F, 1_F \mapsto 1_{X_F}$, for which the lax associativity constraints are the natural isomorphisms $X_F\otimes^{\bf L}_{\bB}X_G \cong X_{GF}$, and the lax unity constraints are the identity morphisms $1_{I_{\aA}}$ on $I_{\aA}$. 

Furthermore, the following diagram of lax functors commutes.
\[
	\begin{tikzcd}
		\alg \arrow[r] \arrow[d] & \dgcat \arrow[d] \\	
		\aA\lL\gG \arrow[r] & \dg\cC\aA\tT
	\end{tikzcd}
\]	

\begin{remark}
	(1) Denote by ${\rm ALG}$ the bicategory for which objects are all $k$-algebras, 1-cells are all bimodules, 2-cells are all isomorphisms of bimodules, horizontal compositions are tensor product functors $-\otimes_B-$, and identity 1-cells are identity bimodules $_AA_A$. We have a natural pseudofunctor $\alg \to {\rm ALG}, A\mapsto A, (f:A \rightarrow B) \mapsto\ _AB_B, 1_f\mapsto 1_{_AB_B}$, for which the lax associativity constraints are the natural isomorphisms of bimodules $_AB\otimes_BC_C\cong\ _AC_C$, and the lax unity constraints are identity morphisms $1_{_AA_A}$. However, ${\rm ALG}\to \aA\lL\gG, A\mapsto A, X\mapsto X, (f:X\to X')\mapsto (f:X\to X')$, can not be defined as a lax functor in the natural way since the natural candidates of lax associativity constraints $X\otimes^{\bf L}_BY \to X\otimes_BY$ are not 2-cells in $\aA\lL\gG$.
	
	(2) Denote by $\dgCAT$ the bicategory for which objects are all small dg categories, 1-cells are all dg bimodules, 2-cells are all isomorphisms of dg bimodules, horizontal compositions are tensor product functors $-\otimes_\bB-$, and identity 1-cells are identity dg bimodules $I_\aA$.
	We have a natural pseudofunctor $\dgcat \to {\rm dgCAT}, \aA\mapsto \aA, (F:\aA \rightarrow \bB) \mapsto\ X_F, 1_F\mapsto 1_{X_F}$, for which the lax associativity constraints are the natural isomorphisms of dg bimodules $X_G\otimes_\bB X_F\cong X_{GF}$, and the lax unity constraints are identity morphisms $1_{I_\aA}$, for all $\aA,\bB,\cC\in{\rm dgcat}, F\in{\rm dgcat}(\aA,\bB), G\in{\rm dgcat}(\bB,\cC)$. However, ${\rm dgCAT}\to \dg\cC\aA\tT, \aA\mapsto \aA, X\mapsto X, (f:X\to X')\mapsto (f:X\to X')$, can not be defined as a lax functor in the natural way since the natural candidates of lax associativity constraints $X\otimes^{\bf L}_BY \to X\otimes_BY$ are not 2-cells in $\aA\lL\gG$.
\end{remark}

\medspace

\noindent{\bf Homotopy categories {\rm Hqe} and {\rm Hmo}.} 
The category ${\rm dgcat}$ of small dg categories and dg functors carries a cofibrantly generated Quillen model structure for which weak equivalences are quasi-equivalences. Denote by {\rm Hqe} the homotopy category of {\rm dgcat} with respect to this model structure. According to \cite[Corollary 4.8]{Toen07}, there is a functorial bijection identifying the set ${\rm Hqe}(\aA,\bB)$ with the set ${\rm Iso}(\rep(\aA,\bB)_{\rm rqr})$, where $\rep(\aA, \bB)_{\rm rqr}$ is the full subcategory of $\dD(\aA^\op\otimes\bB)$ consisting of all dg $\aA\-\bB$-bimodules $X$ such that the derived tensor product functor $-\otimes^{\bf L}_{\aA} X:\dD\aA\to \dD\bB$ sends representable right dg $\aA$-modules to objects isomorphic to representable right dg $\bB$-modules, and ${\rm Iso}(-)$ denotes the set of isomorphism classes. Denote by ${\rm Hqe}_{\rm ff}$ the subcategory of ${\rm Hqe}$ with the same objects as Hqe and morphisms being the isomorphism classes $[X]$ of dg $\aA$-$\bB$-bimodules $X\in \rep(\aA,\bB)_{\rm rqr}$ such that the functor 
$-\otimes_{\aA}^\bL X:\per\aA \rightarrow \per\bB$ is fully faithful.
  
A dg functor $F : \aA \to \bB$ is called a {\it Morita equivalence} if it induces a triangle equivalence ${\bf L}F^* : \dD\bB \to \dD\aA$. Evidently, all quasi-equivalences are Morita equivalences. The category {\rm dgcat} possesses a cofibrantly generated Quillen model structure for which weak equivalences are Morita equivalences (Ref. \cite{Tabuada05}).
Denote by {\rm Hmo} the homotopy category of {\rm dgcat} with respect to this model structure. According to \cite[1.6.3]{Tabuada15}, there exists a bijection between the set ${\rm Hmo}(\aA, \bB)$ and the set ${\rm Iso}(\rep(\aA, \bB)_{\rm rp})$, where $\rep(\aA, \bB)_{\rm rp}$ is the full subcategory of $\dD(\aA^\op \otimes \bB)$ consisting of all dg $\aA\-\bB$-bimodules $X$ that is perfect as a right dg $\bB$-module.
Denote by ${\rm Hmo}_{\rm ff}$ the subcategory of {\rm Hmo} with the same objects as {\rm Hmo} and morphisms being the isomorphism classes $[X]$ of dg $\aA$-$\bB$-bimodules $X\in\rep(\aA, \bB)_{\rm rp}$ such that the functor
$-\otimes_{\aA}^\bL X:\per\aA \rightarrow \per\bB$ is fully faithful. 
The category {\rm Hmo} can be viewed as a bicategory with discrete Hom categories. Then {\rm Hmo} is a sub-bicategory of the bicategory $\overline{\dg\cC\aA\tT}$. Moreover, the following diagram of lax functors commute.
\[
  \begin{tikzcd}
		\dgcat \arrow[r] \arrow[d] & {\rm Hqe} \arrow[r] & {\rm Hmo} \arrow[dl] & {\rm Hmo}_{\rm ff} \arrow[l] \arrow[d] \\
		\dg\cC\aA\tT \arrow[r,"\mathscr{S}"] & \overline{\dg\cC\aA\tT} & \overline{\dg\cC\aA\tT_{\rm raf}} \arrow[l] & \overline{\dg\cC\aA\tT_{\rm raf}}^{\le 1} \arrow[l]
  \end{tikzcd}
\]

\medspace

Denote by $\dgcat_{\rm ff}$ the subcategory of $\dgcat$ with the same objects as $\dgcat$ and morphisms being all fully faithful dg functors. For any fully faithful dg functor $F:\aA \rightarrow \bB$,  it follows from \cite[4.2 Lemma]{Keller94} that the functor
$-\otimes_{\aA}^\bL X_F:\per\aA \rightarrow \per\bB$ is fully faithful. Then we have a functor
$\dgcat_{\rm ff} \rightarrow {\rm Hqe}_{\rm ff}, \aA\mapsto\aA, F\mapsto [X_F]$.
Moreover, the following diagram of functors is commutative.
\[
  \begin{tikzcd}
		\dgcat_{\rm ff} \arrow[r] \arrow[d] & {\rm Hqe}_{\rm ff} \arrow[r] \arrow[d] & {\rm Hmo}_{\rm ff} \arrow[d] \\
		\dgcat	\arrow[r] & {\rm Hqe} \arrow[r] & {\rm Hmo}
  \end{tikzcd}
\]

Due to \cite[4.6 Theorem c)]{Keller03}, the following diagram of functors is also commutative.
\[
  \begin{tikzcd}
			\dgcat_{\rm ff} \arrow[r] \arrow[d] & {\rm Hqe}_{\rm ff} \arrow[r] & {\rm Hmo}_{\rm ff} \arrow[d,"K"] \\
			B^\op_\infty \arrow[rr] && ({\rm Ho}B_\infty)^\op
  \end{tikzcd}
\]

Finally, it is clear that the following diagram of functors
\[  
\begin{tikzcd} 
	{\rm Hmo}_{\rm ff} \arrow[r] \arrow[dr,"K"'] & \overline{\dg\cC\aA\tT_{\rm raf}}^{\le 1} \arrow[d,"\tilde{K}"]\\  
	& ({\rm Ho}B_\infty)^\op
\end{tikzcd}  
\]  
is commutative, which implies that the functor $\tilde{K}$ extends Keller's functor $K$. 

Now we have completed the exposition of the entire diagram $(\bigstar)$.

\appendix

\section{$\theta_\aA$ is a $B_\infty$-quasi-isomorphism}\label{Appendix-ThetaA}

In this appendix, we will give a proof of Lemma~\ref{Lem-ThetaA}.

Let $\aA$ and $\bB$ be small dg categories and $X$ a dg $\aA$-$\bB$-bimodule. 
Then the Hochschild cochain complex $C(\tT_X)$ of upper triangular matrix dg category $\tT_X=\begin{pmatrix}
	\aA&X\\&\bB
\end{pmatrix}$ is a brace $B_\infty$-algebra whose underlying graded vector space is $C(\aA)\oplus C(\bB)\oplus D(X)$.

First, we analyze the cup product on $C(\tT_X)$. For all
$\phi=(\phi_1,\phi_2,\phi_3)^T\in C^p(\tT_X)=C^p(\aA)\oplus C^p(\bB)\oplus D^{p-1}(X),
\psi=(\psi_1,\psi_2,\psi_3)^T\in C^q(\tT_X)=C^q(\aA)\oplus C^q(\bB)\oplus D^{q-1}(X)$, and 
$sa_{1,p+q}\in s\aA(A_{p+q-1},A_{p+q})\otimes\cdots\otimes s\aA(A_0,A_1)$,
we have
\[
\begin{aligned}
	(\phi\cup\psi)(sa_{1,p+q}) & =(-1)^{|\psi||sa_{1,p}|}\ \phi(sa_{1,p})\psi(sa_{p+1,p+q})\\
	& =(-1)^{|\psi_1||sa_{1,p}|}\ \phi_1(sa_{1,p})\psi_1(sa_{p+1,p+q}) \\
	& =(\phi_1\cup\psi_1)(sa_{1,p+q}).
\end{aligned}
\]
Thus $(\phi\cup\psi)_1=\phi_1\cup\psi_1$. 
Similarly,  $(\phi\cup\psi)_2=\phi_2\cup\psi_2$. 
For all $sa_{1,l}\otimes sx\otimes sb_{1,m}\in s\aA(A_{l-1},A_l)\otimes\cdots\otimes s\aA(A_0,A_1)\otimes sX(B_m,A_0)\otimes s\bB(B_{m-1},B_m)\otimes\cdots\otimes s\bB(B_0,B_1)$ with $l+m=p+q-1, l,m\ge 0$, we have
\[
\begin{array}{ll}
	& (\phi\cup\psi)(sa_{1,l}\otimes sx\otimes sb_{1,m}) \\ [2mm]
	={}& \begin{cases}
		(-1)^{|\psi||sa_{1,p}|}\ \phi(sa_{1,p})\psi(sa_{p+1,l}\otimes sx\otimes sb_{1,m}), & l\ge p; \\
		(-1)^{|\psi|(|sa_{1,l}|+|sx|+|sb_{1,p-l-1}|)}\ \phi(sa_{1,l}\otimes sx\otimes sb_{1,p-l-1})\psi(sb_{p-l,m}), & l<p.
	\end{cases} \\ [6mm]
	={}& \begin{cases}
		(-1)^{|\psi_3||sa_{1,p}|}\ \phi_1(sa_{1,p})\psi_3(sa_{p+1,l}\otimes sx\otimes sb_{1,m}), & l\ge p; \\
		(-1)^{|\psi_2|(|sa_{1,l}|+|sx|+|sb_{1,p-l-1}|)}\ \phi_3(sa_{1,l}\otimes sx\otimes sb_{1,p-l-1})\psi_2(sb_{p-l,m}), & l<p.
	\end{cases} \\ [6mm]
	={}& \begin{cases}
		(\phi_1\cup\psi_3)(sa_{1,l}\otimes sx\otimes sb_{1,m}), & l\ge p; \\
		(\phi_3\cup\psi_2)(sa_{1,l}\otimes sx\otimes sb_{1,m}), & l<p.
	\end{cases}
\end{array}
\]
Hence $(\phi\cup\psi)_3=\phi_1\cup\psi_3+\phi_3\cup\psi_2$. 

Now we have finished the proof of the following lemma.

\begin{lemma}\label{Lem-CupProd-2x2}
	Let $\aA$ and $\bB$ be small dg categories, and $X$ a dg $\aA$-$\bB$-bimodule. Then for all
	$\phi=\begin{pmatrix}
		\phi_1 \\
		\phi_2 \\
		\phi_3
	\end{pmatrix}$ and 
	$\psi=\begin{pmatrix}
		\psi_1 \\
		\psi_2 \\
		\psi_3
	\end{pmatrix}\in C(\tT_X)=C(\aA)\oplus C(\bB) \oplus D(X),$
	the cup product
	$$\phi\cup\psi=\begin{pmatrix}
		\phi_1\cup\psi_1 \\
		\phi_2\cup\psi_2 \\
		\phi_1\cup\psi_3+\phi_3\cup\psi_2
	\end{pmatrix}.$$
\end{lemma}

\medspace

Next we analyze the brace operation on $C(\tT_X)$. For all 
$\phi=(\phi_1,\phi_2,\phi_3)^T\in C^n(\tT_X), \psi_i=(\psi_{i,1},\psi_{i,2},\psi_{i,3})^T\in C^{n_i}(\tT_X), 1\le i\le k$,
by definition of brace operation, we have
\[
\phi\{\psi_1,\cdots,\psi_k\}= \sum_{\substack{i_0+\cdots+i_{k}=n-k,\\ i_0,\dots,i_k\ge 0}}\phi\left(1^{\otimes i_0}\otimes s\psi_1\otimes 1^{\otimes i_1}\otimes\cdots\otimes s\psi_k\otimes 1^{\otimes i_k}\right).
\]
It is clear that $(\phi\{\psi_1,\cdots,\psi_k\})_j =\phi_j\{\psi_{1,j},\cdots,\psi_{k,j}\}$ for $j=1,2$, and $(\phi\{\psi_1,\cdots,\psi_k\})_3 = \sum\limits_{i=1}^k \phi_3\{\psi_{1,1},\cdots,\psi_{i-1,1},\psi_{i,3},\psi_{i+1,2},\cdots,\psi_{k,2}\}$. So we get the following lemma.

\begin{lemma}\label{Lem-Brace-2x2}
	Let $\aA$ and $\bB$ be small dg categories, and $X$ a dg $\aA$-$\bB$-bimodule. Then for all 
	$\phi=\begin{pmatrix}
		\phi_1 \\
		\phi_2 \\
		\phi_3
	\end{pmatrix}\in C^n(\tT_X), \psi_i=\begin{pmatrix}
		\psi_{i,1} \\
		\psi_{i,2} \\
		\psi_{i,3}
	\end{pmatrix}\in C^{n_i}(\tT_X), 1\le i\le k,$
	the brace operation
	$$\phi\{\psi_1,\cdots,\psi_k\}=\begin{pmatrix}
		\phi_1\{\psi_{1,1},\psi_{2,1},\cdots,\psi_{k,1}\} \\
		\phi_2\{\psi_{1,2},\psi_{2,2},\cdots,\psi_{k,2}\} \\
		\sum\limits_{i=1}^k \phi_3\{\psi_{1,1},\cdots,\psi_{i-1,1},\psi_{i,3},\psi_{i+1,2},\cdots,\psi_{k,2}\}
	\end{pmatrix}.$$ 
\end{lemma}

\medspace

Now we can give a proof of Lemma~\ref{Lem-ThetaA}.

\begin{lemma} \label{Lem-ThetaA-BInfQis} {\rm (=Lemma~\ref{Lem-ThetaA})}
	Let $\aA$ be a small dg category. Then the graded linear map
	$\theta_\aA=(1_{C(\aA)}, 1_{C(\aA)}, \kappa_\aA)^T : C(\aA) \rightarrow C(\tT_{ I_\aA})$
	is a $B_\infty$-quasi-isomorphism. Moreover, the $B_\infty$-morphisms $\iota_1^*,\iota_2^*: C(\tT_{ I_\aA})\to C(\aA)$ induced by the fully faithful dg functors $\iota_1,\iota_2: \aA\hookrightarrow\tT_{ I_\aA}$ are $B_\infty$-quasi-isomorphisms in $B_\infty$ and equal in ${\rm Ho} B_\infty$.
\end{lemma}

\begin{proof}
	(1) $\theta_\aA$ is a morphism of dg algebras:
	
	(1.1) $\theta_\aA$ is a morphism of complexes: 
	We need to prove $\delta_{C(\tT_{I_\aA})}\theta_\aA=\theta_\aA\delta_{C(\aA)}$.
	Since
	$$\delta_{C(\tT_{I_\aA})}\theta_\aA=
	\begin{pmatrix}
		\delta_{C(\aA)} & & \\
		0 & \delta_{C(\aA)} & \\
		\tilde{\alpha}_{I_\aA} & -\tilde{\beta}_{I_\aA} & \delta_{D(I_\aA)}
	\end{pmatrix}
	\begin{pmatrix}
		1_{C(\aA)} \\
		1_{C(\aA)} \\
		\kappa_\aA
	\end{pmatrix} =
	\begin{pmatrix}
		\delta_{C(\aA)} \\
		\delta_{C(\aA)} \\
		\tilde{\alpha}_{I_\aA}-\tilde{\beta}_{I_\aA}+\delta_{D(I_\aA)}\kappa_\aA
	\end{pmatrix}$$
	and
	$$\theta_\aA\delta_{C(\aA)}=
	\begin{pmatrix}
		\delta_{C(\aA)} \\
		\delta_{C(\aA)} \\
		\kappa_\aA\delta_{C(\aA)}
	\end{pmatrix},$$
	it suffices to show $\tilde{\alpha}_{I_\aA}-\tilde{\beta}_{I_\aA}+\delta_{D(I_\aA)}\kappa_\aA=\kappa_\aA\delta_{C(\aA)}$. 
	Indeed, it is clear that $(\delta_{D(I_\aA)})_{\nin}\kappa_\aA=\kappa_\aA(\delta_{C(\aA)})_{\nin}$. 
	Moreover, it is easy to check
	$\tilde{\alpha}_{I_\aA}-\tilde{\beta}_{I_\aA}+(\delta_{D(I_\aA)})_{\ex}\kappa_\aA=\kappa_\aA(\delta_{C(\aA)})_{\ex}.$
		
	(1.2) $\theta_\aA$ is compatible with cup products: For all $\phi\in C^p(\aA)$ and $\psi\in C^q(\aA)$, we need to prove $\theta_\aA(\phi\cup\psi)=\theta_\aA(\phi)\cup\theta_\aA(\psi)$. By Lemma \ref{Lem-CupProd-2x2}, we have
	$$\theta_\aA(\phi)\cup\theta_\aA(\psi)=\begin{pmatrix}
		\phi \\
		\phi \\
		\kappa_\aA\phi
	\end{pmatrix}\cup\begin{pmatrix}
		\psi \\
		\psi \\
		\kappa_\aA\psi
	\end{pmatrix}=\begin{pmatrix}
		\phi\cup\psi \\
		\phi\cup\psi \\
		\kappa_\aA\phi\cup\psi+\phi\cup\kappa_\aA\psi
	\end{pmatrix}$$
	and 
	$$\theta_\aA(\phi\cup\psi)=
	\begin{pmatrix}
		\phi\cup\psi \\
		\phi\cup\psi \\
		\kappa_\aA(\phi\cup\psi)
	\end{pmatrix}.$$
	So it suffices to show
	$\kappa_\aA\phi\cup\psi+\phi\cup\kappa_\aA\psi=\kappa_\aA(\phi\cup\psi).$
	Indeed, for all $sa_{1,l}\otimes sa'\otimes sa''_{1,m}\in s\aA(A_{l-1},A_l)\otimes\cdots\otimes s\aA(A_0,A_1)\otimes s\aA(A'_m,A_0)\otimes s\aA(A'_{m-1},A'_m)\otimes\cdots\otimes s\aA(A'_0,A'_1)$ with $l+m=p+q-1, l,m\ge 0$, we have
	\[
	\kappa_\aA(\phi\cup\psi)(sa_{1,l}\otimes sa'\otimes sa''_{1,m})=(\phi\cup\psi)(sa_{1,l}\otimes sa'\otimes sa''_{1,m}),
	\]
	and
	\[
	\begin{aligned}
		(\kappa_\aA\phi\cup\psi)(sa_{1,l}\otimes sa'\otimes sa''_{1,m}) & =\begin{cases}
			0, & l\ge p; \\
			(\phi\cup\psi)(sa_{1,l}\otimes sa'\otimes sa''_{1,m}), & l<p.
		\end{cases} \\
		(\phi\cup\kappa_\aA\psi)(sa_{1,l}\otimes sa'\otimes sa''_{1,m}) & =\begin{cases}
			(\phi\cup\psi)(sa_{1,l}\otimes sa'\otimes sa''_{1,m}), & l\ge p; \\
			0, & l<p.
		\end{cases}		
	\end{aligned}
	\]
	Thus $\kappa_\aA\phi\cup\psi+\phi\cup\kappa_\aA\psi=\kappa_\aA(\phi\cup\psi)$.
	
	\medspace
	
	(2) $\theta_\aA$ is compatible with brace operations: 
	
	For all $\phi\in C^n(\aA), \psi_i\in C^{n_i}(\aA), 1\le i\le k$, by Lemma \ref{Lem-Brace-2x2}, we have
	\[
	\theta_\aA\phi\{\theta_\aA\psi_1,\cdots,\theta_\aA\psi_k\} =\begin{pmatrix}
		\phi\{\psi_1,\cdots,\psi_k\} \\
		\phi\{\psi_1,\cdots,\psi_k\} \\
		\sum\limits_{i=1}^k (\kappa_\aA\phi)\{\psi_1,\cdots,\psi_{i-1},\kappa_\aA\psi_i,\psi_{i+1},\cdots,\psi_k\}
	\end{pmatrix}
	\]
	and
	\[
	\theta_\aA(\phi\{\psi_1,\cdots,\psi_k\}) =\begin{pmatrix}
		\phi\{\psi_1,\cdots,\psi_k\} \\
		\phi\{\psi_1,\cdots,\psi_k\} \\
		\kappa_\aA(\phi\{\psi_1,\cdots,\psi_k\})
	\end{pmatrix}.
	\]
    So it is enough to show
	\[
	\sum\limits_{i=1}^k (\kappa_\aA\phi)\{\psi_1,\cdots,\psi_{i-1},\kappa_\aA\psi_i,\psi_{i+1},\cdots,\psi_k\} = \kappa_\aA(\phi\{\psi_1,\cdots,\psi_k\}).
	\]
	This is clear, since as maps in $D(I_\aA)$, both sides have the same actions on each element.
	
	(3) $\theta_\aA$ is a quasi-isomorphism and $\iota_1^*=\iota_2^*$ in ${\rm Ho}B_\infty$:
	It is clear that the following diagram 
	\[
	\begin{tikzcd}[column sep=large]
		& C(\aA) \arrow[dl,equal] \arrow[d,"\theta_\aA"] \arrow[dr,equal] & \\
		C(\aA) & C(\tT_{I_\aA}) \arrow[l,two heads,"\iota_1^*"'] \arrow[r,two heads,"\iota_2^*"] & C(\aA) \\
	\end{tikzcd}	
	\]	
    is commutative, i.e., $\iota_1^*\theta_\aA=1_{C(\aA)}=\iota_2^*\theta_\aA$. 
    Since $\iota_1^*$ (or $\iota_2^*$) is a quasi-isomorphism, $\theta_\aA$ is a quasi-isomorphism as well. Furthermore, $\iota_1^*=\iota_2^*$ in ${\rm Ho} B_\infty$.
\end{proof}

\section{$\theta_\aA^2$ is a $B_\infty$-quasi-isomorphism}\label{Appendix-ThetaA^2}

In this appendix, we will give a proof of Lemma~\ref{Lem-ThetaA^2}.

Let $\aA,\bB$ and $\cC$ be small dg categories, $X$ a dg $\aA$-$\bB$-bimodule, and $Y$ a dg $\bB$-$\cC$-bimodule. 
The underlying graded vector space of the Hochschild cochain complex $C(\tT_{X,Y})$ of upper triangular matrix dg category $\tT_{X,Y}=\begin{pmatrix}
	\aA&X&X\otimes_\bB Y\\&\bB&Y\\ &&\cC
\end{pmatrix}$ is
$C(\aA)\oplus C(\bB)\oplus C(\cC) \oplus D(X)\oplus D(Y)\oplus D(X\otimes_\bB Y) \oplus D(X|Y)$. 

First, we analyze the cup product on $C(\tT_{X,Y})$. 
Similar to Lemma \ref{Lem-CupProd-2x2}, we have the following lemma.

\begin{lemma}\label{Lem-CupProd-3x3}
	Let $\aA,\bB$ and $\cC$ be small dg categories, $X$ a dg $\aA$-$\bB$-bimodule, and $Y$ a dg $\bB$-$\cC$-bimodule. Then for all $\phi=(\phi_1,\cdots,\phi_7)^\mT\in C^n(\tT_{X,Y}),\psi=(\psi_1,\cdots,\psi_7)^\mT\in C^{n_i}(\tT_{X,Y})$, the cup product
	\[
	\phi\cup\psi=\begin{pmatrix}
		\phi_1\cup\psi_1 \\
		\phi_2\cup\psi_2 \\
		\phi_3\cup\psi_3 \\
		\phi_1\cup\psi_4+\phi_4\cup\psi_2 \\
		\phi_2\cup\psi_5+\phi_5\cup\psi_3 \\
		\phi_1\cup\psi_6+\phi_6\cup\psi_3 \\
		\phi_1\cup\psi_7+\phi_4\cup\psi_5+\phi_7\cup\psi_3
	\end{pmatrix}.
	\]
\end{lemma}

Next, we analyze the brace operation on $C(\tT_{X,Y})$. Similar to Lemma \ref{Lem-Brace-2x2}, we have the following lemma.

\begin{lemma}\label{Lem-Brace-3x3}
	Let $\aA,\bB$ and $\cC$ be small dg categories, $X$ a dg $\aA$-$\bB$-bimodule, and $Y$ a dg $\bB$-$\cC$-bimodule. Then for all $\phi=(\phi_1,\cdots,\phi_7)^\mT\in C^n(\tT_{X,Y}),\psi_i=(\psi_{i,1},\cdots,\psi_{i,7})^\mT\in C^{n_i}(\tT_{X,Y}), 1\le i\le k$, the brace operation
	\[
	\phi\{\psi_1,\cdots,\psi_k\}=\begin{pmatrix}
		\phi_1\{\psi_{1,1},\psi_{2,1},\cdots,\psi_{k,1}\} \\
		\phi_2\{\psi_{1,2},\psi_{2,2},\cdots,\psi_{k,2}\} \\
		\phi_3\{\psi_{1,3},\psi_{2,3},\cdots,\psi_{k,3}\} \\
		\sum\limits_{i=1}^k \phi_4\{\psi_{1,1},\cdots,\psi_{i-1,1},\psi_{i,4}, \psi_{i+1,2}, \cdots,\psi_{k,2}\} \\
		\sum\limits_{i=1}^k \phi_5\{\psi_{1,2},\cdots,\psi_{i-1,2},\psi_{i,5}, \psi_{i+1,3}, \cdots,\psi_{k,3}\} \\
		\sum\limits_{i=1}^k \phi_6\{\psi_{1,1},\cdots,\psi_{i-1,1},\psi_{i,6}, \psi_{i+1,3}, \cdots,\psi_{k,3}\} \\
		(\phi\{\psi_1,\cdots,\psi_k\})_7
	\end{pmatrix},
	\]
	where $(\phi\{\psi_1,\cdots,\psi_k\})_7=\sum\limits_{1\le i<j\le k} \phi_7\{\psi_{1,1},\cdots,\psi_{i-1,1},\psi_{i,4}, \psi_{i+1,2},\cdots,\psi_{j-1,2},\psi_{j,5}, \psi_{j+1,3},\linebreak \cdots, \psi_{k,3}\} + \sum\limits_{i=1}^k \phi_7\{\psi_{1,1},\cdots,\psi_{i-1,1},\psi_{i,7}, \psi_{i+1,3}, \cdots,\psi_{k,3}\}$.
\end{lemma}

Now we can give a proof of Lemma~\ref{Lem-ThetaA^2}.

\begin{lemma} \label{Lem-ThetaA^2-BInfQis} {\rm(=Lemma~\ref{Lem-ThetaA^2})}
	Let $\aA$ be a small dg category. Then the graded linear map
	\[
	\theta_\aA^2=(1_{C(\aA)}, 1_{C(\aA)}, 1_{C(\aA)}, \kappa_\aA, \kappa_\aA, \kappa_\aA, \kappa_\aA^2)^{\mT}:C(\aA) \rightarrow C(\tT_{ I_\aA,I_\aA})
	\]
	is a $B_\infty$-quasi-isomorphism. Moreover, the $B_\infty$-morphisms $\iota_1^*,\iota_2^*,\iota_3^*: C(\tT_{I_\aA,I_\aA})\to C(\aA)$ induced by the fully faithful dg functors $\iota_1,\iota_2,\iota_3: \aA\hookrightarrow\tT_{I_\aA,I_\aA}$ are $B_\infty$-quasi-isomorphisms in $B_\infty$ and equal in ${\rm Ho} B_\infty$.
\end{lemma}

\begin{proof}
	(1) $\theta_\aA^2$ is a morphism of dg algebras:
	
	(1.1) $\theta_\aA^2$ is a morphism of complexes: We need to prove $\delta_{C(\tT_{ I_\aA,I_\aA})}\theta_\aA^2=\theta^2_\aA\delta_{C(\aA)}$.
	Clearly, 	
	{\small \[
	\delta_{C(\tT_{ I_\aA,I_\aA})}\theta_\aA^2=\begin{pmatrix}
		\delta_{C(\aA)} \\
		\delta_{C(\aA)} \\
		\delta_{C(\aA)} \\
		\tilde{\alpha}_{I_\aA}-\tilde{\beta}_{I_\aA}+\delta_{D(I_\aA)}\kappa_\aA \\
		\tilde{\alpha}_{I_\aA}-\tilde{\beta}_{I_\aA}+\delta_{D(I_\aA)}\kappa_\aA \\
		\tilde{\alpha}_{I_\aA}-\tilde{\beta}_{I_\aA}+\delta_{D(I_\aA)}\kappa_\aA \\
		-\tilde{\alpha}_{I_\aA|I_\aA}\kappa_\aA-\tilde{\beta}_{I_\aA|I_\aA}\kappa_\aA+\tilde{\mu}_{I_\aA|I_\aA}^1\kappa_\aA+\delta_{D(I_\aA|I_\aA)}\kappa_\aA^2
	\end{pmatrix},\
	\theta^2_\aA\delta_{C(\aA)}=\begin{pmatrix}
		\delta_{C(\aA)} \\
		\delta_{C(\aA)} \\
		\delta_{C(\aA)} \\
		\kappa_\aA\delta_{C(\aA)} \\
		\kappa_\aA\delta_{C(\aA)} \\
		\kappa_\aA\delta_{C(\aA)} \\
		\kappa_\aA^2\delta_{C(\aA)}
	\end{pmatrix}.
	\]}
	 By the proof (1.1) of Lemma \ref{Lem-ThetaA-BInfQis}, we have  $\kappa_\aA\delta_{C(\aA)}=\tilde{\alpha}_{I_\aA}-\tilde{\beta}_{I_\aA}+\delta_{D(I_\aA)}\kappa_\aA$. So it suffices to show
	\[
	-\tilde{\alpha}_{I_\aA|I_\aA}\kappa_\aA-\tilde{\beta}_{I_\aA|I_\aA}\kappa_\aA+\tilde{\mu}_{I_\aA|I_\aA}^1\kappa_\aA+\delta_{D(I_\aA|I_\aA)}\kappa_\aA^2=\kappa_\aA^2\delta_{C(\aA)}.
	\]
    Indeed, it is clear that $(\delta_{D(I_\aA|I_\aA)})_{\nin}\kappa_\aA^2=\kappa_\aA^2(\delta_{C(\aA)})_{\nin}$. Moreover, it is easy to check
	\[
	-\tilde{\alpha}_{I_\aA|I_\aA}\kappa_\aA-\tilde{\beta}_{I_\aA|I_\aA}\kappa_\aA+\tilde{\mu}_{I_\aA|I_\aA}^1\kappa_\aA+(\delta_{D(I_\aA|I_\aA)})_{\ex}\kappa_\aA^2=\kappa_\aA^2(\delta_{C(\aA)})_{\ex}.
	\]
	
	(1.2) $\theta^2_\aA$ is compatible with cup products: For all $\phi\in C^p(\aA),\psi\in C^q(\aA)$, by Lemma \ref{Lem-CupProd-3x3}, we have
	\[
	\theta^2_\aA(\phi)\cup\theta^2_\aA(\psi)=\begin{pmatrix}
		\phi \\
		\phi \\
		\phi \\
		\kappa_\aA\phi \\
		\kappa_\aA\phi \\
		\kappa_\aA\phi \\
		\kappa_\aA^2\phi
	\end{pmatrix}\cup\begin{pmatrix}
		\psi \\
		\psi \\
		\psi \\
		\kappa_\aA\psi \\
		\kappa_\aA\psi \\
		\kappa_\aA\psi \\
		\kappa_\aA^2\psi
	\end{pmatrix}=\begin{pmatrix}
		\phi\cup\psi \\
		\phi\cup\psi \\
		\phi\cup\psi \\
		\kappa_\aA\phi\cup\psi+\phi\cup\kappa_\aA\psi \\
		\kappa_\aA\phi\cup\psi+\phi\cup\kappa_\aA\psi \\
		\kappa_\aA\phi\cup\psi+\phi\cup\kappa_\aA\psi \\
		\phi\cup\kappa_\aA^2\psi+\kappa_\aA\phi\cup\kappa_\aA\psi+\kappa_\aA^2\phi\cup\psi
	\end{pmatrix}
	\]
	and 
	\[
	\theta^2_\aA(\phi\cup\psi)=\begin{pmatrix}
		\phi\cup\psi \\
		\phi\cup\psi \\
		\phi\cup\psi \\
		\kappa_\aA(\phi\cup\psi) \\
		\kappa_\aA(\phi\cup\psi) \\
		\kappa_\aA(\phi\cup\psi) \\
		\kappa^2_\aA(\phi\cup\psi) \\
	\end{pmatrix}.
	\]
	By the proof (1.2) of Lemma \ref{Lem-ThetaA-BInfQis}, we have
	\[
	\kappa_\aA\phi\cup\psi+\phi\cup\kappa_\aA\psi =\kappa_\aA(\phi\cup\psi).
	\]
	Similarly, we have
	\[
	\phi\cup\kappa_\aA^2\psi+\kappa_\aA\phi\cup\kappa_\aA\psi+\kappa_\aA^2\phi\cup\psi=\kappa^2_\aA(\phi\cup\psi).
	\]
	
	(2) $\theta^2_\aA$ is compatible with brace operations: For all $\phi\in C^n(\aA), \psi_i\in C^{n_i}(\aA), 1\le i\le k$, by Lemma \ref{Lem-Brace-3x3}, we have 
	\[
	\theta^2_\aA\phi\{\theta^2_\aA\psi_1,\cdots,\theta^2_\aA\psi_k\} =
	\begin{pmatrix}
		\phi\{\psi_1,\cdots,\psi_k\} \\
		\phi\{\psi_1,\cdots,\psi_k\} \\
		\phi\{\psi_1,\cdots,\psi_k\} \\
		\sum\limits_{i=1}^k \kappa_\aA\phi\{\psi_1,\cdots,\psi_{i-1},\kappa_\aA\psi_i, \psi_{i+1},\cdots,\psi_k\} \\
		\sum\limits_{i=1}^k \kappa_\aA\phi\{\psi_1,\cdots,\psi_{i-1},\kappa_\aA\psi_i, \psi_{i+1},\cdots,\psi_k\} \\
		\sum\limits_{i=1}^k \kappa_\aA\phi\{\psi_1,\cdots,\psi_{i-1},\kappa_\aA\psi_i, \psi_{i+1},\cdots,\psi_k\} \\
		(\theta^2_\aA\phi\{\theta^2_\aA\psi_1,\cdots,\theta^2_\aA\psi_k\})_7
	\end{pmatrix},
	\]
	where $(\theta^2_\aA\phi\{\theta^2_\aA\psi_1,\cdots,\theta^2_\aA\psi_k\})_7 = \sum\limits_{1\le i<j\le k} \kappa^2_\aA\phi\{\psi_1,\cdots,\psi_{i-1}, \kappa_\aA\psi_i, \psi_{i+1},\cdots,\psi_{j-1},\kappa_\aA\psi_j, \linebreak \psi_{j+1},\cdots,\psi_{k,3}\} + \sum\limits_{i=1}^k \kappa^2_\aA\phi\{\psi_1,\cdots,\psi_{i-1},\kappa^2_\aA\psi_i, \psi_{i+1}, \cdots,\psi_k\}$,
	and
	\[
	\theta^2_\aA(\phi\{\psi_1,\cdots,\psi_k\}) =\begin{pmatrix}
		\phi\{\psi_1,\cdots,\psi_k\} \\
		\phi\{\psi_1,\cdots,\psi_k\} \\
		\phi\{\psi_1,\cdots,\psi_k\} \\
		\kappa_\aA(\phi\{\psi_1,\cdots,\psi_k\}) \\
		\kappa_\aA(\phi\{\psi_1,\cdots,\psi_k\}) \\
		\kappa_\aA(\phi\{\psi_1,\cdots,\psi_k\}) \\
		\kappa^2_\aA(\phi\{\psi_1,\cdots,\psi_k\})
	\end{pmatrix}.
	\]
By the proof (2) of Lemma \ref{Lem-ThetaA-BInfQis}, we have
	\[
	\sum\limits_{i=1}^k \kappa_\aA\phi\{\psi_1,\cdots,\psi_{i-1},\kappa_\aA\psi_i, \psi_{i+1},\cdots,\psi_k\} = \kappa_\aA\left(\phi\{\psi_1,\cdots,\psi_k\}\right).
	\]	
	So it is enough to prove
	\[
	\begin{aligned}
		&\ \kappa^2_\aA(\phi\{\psi_1,\cdots,\psi_k\}) \\
		= & \sum\limits_{1\le i<j\le k} \kappa^2_\aA\phi\{\psi_1,\cdots,\psi_{i-1}, \kappa_\aA\psi_i, \psi_{i+1},\cdots,\psi_{j-1},\kappa_\aA\psi_j,  \psi_{j+1},\cdots,\psi_{k,3}\} \\ 
		&+ \sum\limits_{i=1}^k \kappa^2_\aA\phi\{\psi_1,\cdots,\psi_{i-1},\kappa^2_\aA\psi_i, \psi_{i+1}, \cdots,\psi_k\}.
	\end{aligned}	
	\]
	This is clear, since as maps in $D(I_\aA|I_\aA)$, both sides have the same action on each element.	
	
	(3) $\theta^2_\aA$ is a $B_\infty$-quasi-isomorphism: Three different fully faithful dg functors 
	$$\iota_1,\iota_2,\iota_3: \aA\hookrightarrow\tT_{I_\aA,I_\aA}=\begin{pmatrix}
		\aA&I_\aA&I_\aA\otimes_\aA I_\aA\\&\aA&I_\aA\\ &&\aA
	\end{pmatrix},$$
	induce three $B_\infty$-morphisms 
	$$\iota^*_1,\iota^*_2,\iota^*_3: C(\tT_{I_\aA,I_\aA})\to C(\aA).$$ 
	The dg $\aA$-$\tT_{I_\aA}$-bimodule $X:=(I_\aA\ I_\aA\otimes_\aA I_\aA) \cong\ _\aA(\tT_{I_\aA})_{\tT_{I_\aA}}$, which is cofibrant over $\aA$, induces a fully faithful functor 
	$X\otimes^{\bf L}_{\tT_{I_\aA}} -: \per\tT_{I_\aA}^\op\to\per\aA^\op$, 
	so the right action $\beta_X: C(\tT_{I_\aA})\to C(X)$ is a quasi-isomorphism.
	By Lemma~\ref{Lem-Homotopy-Bicart}, we have the following homotopy bicartesian diagram
	\[
	\begin{tikzcd}[column sep=60]
		C(\tT_{I_\aA,I_\aA}) \arrow[r, "\iota_1^*"] \arrow[d, "\iota_{23}^*"'] & C(\aA) \arrow[d, "\alpha_X"] \\
		C(\tT_{I_\aA}) \arrow[r, "\beta_X"] & C(X).
	\end{tikzcd}
	\]
	Thus $\iota_1^*$ is a $B_\infty$-quasi-isomorphism. 
	It is clear that the following diagram 
	\[
	\begin{tikzcd}[column sep=large]
		& C(\aA) \arrow[dl,equal] \arrow[d,"\theta^2_\aA"] \\
		C(\aA) & C(\tT_{I_\aA}) \arrow[l,two heads,"\iota_i^*"'] \\
	\end{tikzcd}	
	\]	
	is commutative, i.e., $\iota_i^*\theta^2_\aA=1_{C(\aA)}$, for $i=1,2,3$. 
	Hence $\theta^2_\aA$ is a $B_\infty$-quasi-isomorphism. Furthermore, $\iota_2^*$ and $\iota_3^*$ are $B_\infty$-quasi-isomorphisms, and $\iota_1^*=\iota_2^*=\iota_3^*$ in ${\rm Ho} B_\infty$.
\end{proof}

\section{$\theta_X$ is a $B_\infty$-quasi-isomorphism} \label{App-ThetaX}

In this appendix, we will give a proof of a fact used in the proof (4) of Theorem~\ref{Thm-LaxFuntor-HomotopyCat-B-Inf}.

Let $\aA$ and $\bB$ be small dg categories and $X$ a dg $\aA$-$\bB$-bimodule. Define a quasi-isomorphism of complexes  
$$\kappa_X^l:D(X) \rightarrow D(X|I_\bB)$$
by
$(\kappa_X^l\phi)(sa_{1,l}\otimes sx\otimes sb_{1,m}\otimes sb'\otimes sb''_{1,n}) := \phi(sa_{1,l}\otimes sx\otimes sb_{1,m}\otimes sb'\otimes sb''_{1,n})$,
and a quasi-isomorphism of complexes 
$$\kappa_X^r: D(X) \rightarrow D(I_\aA|X)$$
by
$(\kappa_X^r\phi)(sa_{1,l}\otimes sa'\otimes sa''_{1,m}\otimes sx\otimes sb_{1,n}):= \phi(sa_{1,l}\otimes sa'\otimes sa''_{1,m}\otimes sx\otimes sb_{1,n}).$

\begin{lemma} \label{Lem-ThetaX-BInfQis}
	Let $\aA$ and $\bB$ be small dg categories, and $X$ a dg $\aA$-$\bB$-bimodule. Then the graded linear map
	\[
	\theta_X:=\begin{pmatrix}
		1_{C(\aA)} & 0 & 0 \\
		0 & 1_{C(\bB)} & 0 \\
		0 & 1_{C(\bB)} & 0 \\
		0 & 0 & 1_{D(X)} \\
		0 & \kappa_\bB & 0 \\
		0 & 0 & 1_{D(X)} \\
		0 & 0 & \kappa_X^l
	\end{pmatrix}:C(\tT_X) \rightarrow C(\tT_{X,I_\bB})
	\]
	is a $B_\infty$-quasi-isomorphism. Moreover, the $B_\infty$-morphisms $\iota_{12}^*$ and $\iota_{13}^*$ induced by the fully faithful dg functors $\iota_{12},\iota_{13}: \tT_{X}=\begin{pmatrix}
		\aA&X\\&\bB
	\end{pmatrix}\hookrightarrow\tT_{X,I_\bB}=\begin{pmatrix}
		\aA&X&X\otimes_\bB I_\bB\\&\bB&I_\bB\\ &&\bB
	\end{pmatrix}$ are $B_\infty$-quasi-isomorphisms in $B_\infty$ and equal in ${\rm Ho}B_\infty$.
\end{lemma}

\begin{proof}
	(1) $\theta_X$ is a morphism of dg algebras:
	
	(1.1) $\theta_X$ is a morphism of complexes: We need to prove $\delta_{C(\tT_{X,I_\bB})}\theta_X=\theta_X\delta_{C(\tT_X)}$. Clearly,
	\[
	\delta_{C(\tT_{X,I_\bB})}\theta_X=\begin{pmatrix}
		\delta_{C(\aA)} & 0 & 0 \\
		0 &\delta_{C(\bB)} & 0 \\
		0 &\delta_{C(\bB)} & 0 \\
		\tilde{\alpha}_X & -\tilde{\beta}_X & \delta_{D(X)} \\
		0 & \tilde{\alpha}_{I_\bB}-\tilde{\beta}_{I_\bB}+\delta_{D(I_\bB)}\kappa_\bB & 0 \\
		\tilde{\alpha}_{X} & -\tilde{\beta}_{X} & \delta_{D(X)} \\
		0 & -\tilde{\beta}_{X|I_\bB}\kappa_\bB & -\tilde{\alpha}_{X|I_\bB}+\tilde{\mu}_{X|I_\bB}^1+\delta_{D(X|I_\bB)}\kappa_X^l
	\end{pmatrix}
	\]
	and
	\[
	\theta_X\delta_{C(\tT_X)}=\begin{pmatrix}
		\delta_{C(\aA)} & 0 & 0 \\
		0 &\delta_{C(\bB)} & 0 \\
		0 &\delta_{C(\bB)} & 0 \\
		\tilde{\alpha}_X & -\tilde{\beta}_X & \delta_{D(X)} \\
		0 & \kappa_\bB\delta_{C(\bB)} & 0 \\
		\tilde{\alpha}_{X} & -\tilde{\beta}_{X} & \delta_{D(X)} \\
		\kappa_X^l\tilde{\alpha}_X & -\kappa_X^l\tilde{\beta}_X & \kappa_X^l\delta_{D(X)}
	\end{pmatrix}.
	\]
	By the proof (1.1) of Lemma \ref{Lem-ThetaA-BInfQis}, we have $\kappa_\bB\delta_{C(\bB)}=\tilde{\alpha}_{I_\bB}-\tilde{\beta}_{I_\bB}+\delta_{D(I_\bB)}\kappa_\bB$. From the definitions of $D(X|I_\bB)$ and $\tilde{\alpha}_X$, we obtain $\kappa_X^l\tilde{\alpha}_X=0$. Moreover, the following diagram is commutative.
	\[
	\begin{tikzcd}
		C(\bB) \arrow[r, "\tilde{\beta}_X"] \arrow[d, "\kappa_\bB"'] & D(X) \arrow[d, "\kappa_X^l"] \\
		D(I_\bB) \arrow[r, "\tilde{\beta}_{X|I_\bB}"] & D(X|I_\bB)
	\end{tikzcd}
	\]
	Thus it suffices to show 
	\[
	-\tilde{\alpha}_{X|I_\bB}+\tilde{\mu}_{X|I_\bB}^1+\delta_{D(X|I_\bB)}\kappa_X^l=\kappa_X^l\delta_{D(X)}.
	\]
	Indeed, it is clear that $(\delta_{D(X|I_\bB)})_{\nin}\kappa_X^l=\kappa_X^l(\delta_{D(X)})_{\nin}$, and it is easy to check 
	\[
	-\tilde{\alpha}_{X|I_\bB}+\tilde{\mu}_{X|I_\bB}^1+(\delta_{D(X|I_\bB)})_{\ex}\kappa_X^l=\kappa_X^l(\delta_{D(X)})_{\ex}.
	\]
	
	(1.2) $\theta_X$ is compatible with cup products: For all $\phi=(\phi_1,\phi_2,\phi_3)\in C^p(\tT_X)$ and $\psi=(\psi_1,\psi_2,\psi_3)\in C^q(\tT_X)$, by Lemma \ref{Lem-CupProd-3x3}, we have
	\[
	\theta_X(\phi)\cup\theta_X(\psi)=\begin{pmatrix}
		\phi_1 \\
		\phi_2 \\
		\phi_2 \\
		\phi_3\\
		\kappa_\bB\phi_2 \\
		\phi_3 \\
		\kappa_X^l\phi_3
	\end{pmatrix}\cup\begin{pmatrix}
		\psi_1 \\
		\psi_2 \\
		\psi_2 \\
		\psi_3\\
		\kappa_\bB\psi_2 \\
		\psi_3 \\
		\kappa_X^l\psi_3
	\end{pmatrix}=\begin{pmatrix}
		\phi_1\cup\psi_1 \\
		\phi_2\cup\psi_2 \\
		\phi_2\cup\psi_2 \\
		\phi_1\cup\psi_3+\phi_3\cup\psi_2 \\
		\phi_2\cup\kappa_\bB\psi_2+\kappa_\bB\phi_2\cup\psi_2 \\
		\phi_1\cup\psi_3+\phi_3\cup\psi_2 \\
		\phi_1\cup\kappa_X^l\psi_3+\phi_3\cup\kappa_\bB\psi_2+\kappa_X^l\phi_3\cup\psi_2
	\end{pmatrix}
	\]
	and 
	\[
	\theta_\aA(\phi\cup\psi)=\begin{pmatrix}
		\phi_1\cup\psi_1 \\
		\phi_2\cup\psi_2 \\
		\phi_2\cup\psi_2 \\
		\phi_1\cup\psi_3+\phi_3\cup\psi_2 \\
		\kappa_\bB(\phi_2\cup\psi_2) \\
		\phi_1\cup\psi_3+\phi_3\cup\psi_2 \\
		\kappa_X^l(\phi_1\cup\psi_3+\phi_3\cup\psi_2)
	\end{pmatrix}.
	\]
	By the proof (1.2) of Lemma \ref{Lem-ThetaA-BInfQis}, we have
	\[
	\phi_2\cup\kappa_\bB\psi_2+\kappa_\bB\phi_2\cup\psi_2 = \kappa_\bB(\phi_2\cup\psi_2).
	\]
	Similarly, we can obtain
	\[
	\phi_1\cup\kappa_X^l\psi_3+\phi_3\cup\kappa_\bB\psi_2+\kappa_X^l\phi_3\cup\psi_2=\kappa_X^l(\phi_1\cup\psi_3+\phi_3\cup\psi_2).
	\]
	
	(2) $\theta_X$ is compatible with brace operations: Similar to the proof (2) of Lemma \ref{Lem-ThetaA-BInfQis}.	
	
	(3) $\theta^2_\aA$ is a $B_\infty$-quasi-isomorphism: Two different fully faithful dg functors 
	$$\iota_{12},\iota_{13}: \tT_{X}=\begin{pmatrix}
		\aA&X\\&\bB
	\end{pmatrix}\hookrightarrow\tT_{X,I_\bB}=\begin{pmatrix}
		\aA&X&X\otimes_\bB I_\bB\\&\bB&I_\bB\\ &&\bB
	\end{pmatrix}$$
	induce two $B_\infty$-morphisms 
	$$\iota^*_{12},\iota^*_{13}: C(\tT_{X,I_\bB})\to C(\tT_{X}).$$ 
	Note that $\tT_{X} \cong I_\bB$ as dg $\bB$-bimodules. The dg $\tT_{X}$-$\bB$-bimodule $Y:=\begin{pmatrix}
		X\otimes_\bB I_\bB\\ I_\bB
	\end{pmatrix} \cong\ _{\tT_{X}}(\tT_{X})_\bB$, which is cofibrant over $\tT_{X}$, induces a fully faithful dg functor 
	$Y\otimes^{\bf L}_\bB-: \per\bB^\op\to\per\tT_{X}^\op$, 
	so the right action $\beta_Y: C(\bB)\to C(Y)$ is a quasi-isomorphism.
	By Lemma~\ref{Lem-Homotopy-Bicart}, we have the following homotopy bicartesian diagram
	\[
	\begin{tikzcd}
		C(\tT_{X,I_\bB}) \arrow[r, "\iota_{12}^*"] \arrow[d, "\iota_3^*"'] & C(\tT_{X}) \arrow[d, "\alpha_Y"] \\
		C(\bB) \arrow[r, "\beta_Y"] & C(Y).
	\end{tikzcd}
	\]
	Thus $\iota_{12}^*$ is a $B_\infty$-quasi-isomorphism. 
	It is clear that the following diagram 
	\[
	\begin{tikzcd}[column sep=large]
		& C(\tT_X) \arrow[dl,equal] \arrow[d,"\theta_X"] \\
		C(\tT_X) & C(\tT_{X,I_\bB}) \arrow[l,two heads,"\iota_{1i}^*"'] \\
	\end{tikzcd}	
	\]	
	is commutative, i.e., $\iota_{1i}^*\theta_X=1_{C(\tT_{X})}$, for $i=2,3$. 
	Hence $\theta_X$ is a $B_\infty$-quasi-isomorphism. 
	Furthermore, $\iota_{13}^*$ is also  a $B_\infty$-quasi-isomorphism in $B_\infty$, and $\iota_{12}^*=\iota_{13}^*$ in ${\rm Ho} B_\infty$.
\end{proof}

Analogous to Lemma~\ref{Lem-ThetaX-BInfQis}, we have the following lemma.

\begin{lemma} \label{Lem-Theta'X-BInfQis}
	Let $\aA$ and $\bB$ be small dg categories, and $X$ a dg $\aA$-$\bB$-bimodule. Then the graded linear map
	\[
	\theta'_X:=\begin{pmatrix}
		1_{C(\aA)} & 0 & 0 \\
		1_{C(\aA)} & 0 & 0 \\
		0 & 1_{C(\bB)} & 0 \\
		\kappa_\aA & 0 & 0 \\
		0 & 0 & 1_{D(X)} \\
		0 & 0 & 1_{D(X)} \\
		0 & 0 & \kappa_X^r
	\end{pmatrix}:C(\tT_X) \rightarrow C(\tT_{I_\aA,X})
	\]
	is a $B_\infty$-quasi-isomorphism. Moreover, the $B_\infty$-morphism $\iota_{13}^*$ and $\iota_{23}^*$ induced by the fully faithful dg functors $\iota_{13},\iota_{23}:\tT_X\hookrightarrow\tT_{I_\aA,X}$ are $B_\infty$-quasi-isomorphisms in $B_\infty$ and equal in ${\rm Ho}B_\infty$. 
\end{lemma}

\begin{remark}
	The $B_\infty$-morphism $\theta_\aA^2=\theta_{I_\aA}\theta_\aA=\theta'_{I_\aA}\theta_\aA$.
	So we can apply Lemma~\ref{Lem-ThetaA-BInfQis} and Lemma~ \ref{Lem-ThetaX-BInfQis} or Lemma~\ref{Lem-Theta'X-BInfQis} to prove Lemma~\ref{Lem-ThetaA^2-BInfQis}. 
\end{remark}

\medspace

\noindent Y. Han and X.K. Wang \\
\\
KLMM, Academy of Mathematics and Systems Science \\
Chinese Academy of Sciences \\ 
Beijing 100190, China \\ 
\\
School of Mathematical Sciences \\
University of Chinese Academy of Sciences\\ 
Beijing 100049, China \\ 
\\
E-mail: hany@iss.ac.cn (Y. Han); wangxukun@amss.ac.cn (X.K. Wang)

\end{document}